\documentclass[smallextended,referee,envcountsect]{svjour3}
\smartqed
\usepackage{graphicx}

\usepackage{amsmath}
\usepackage{amssymb}
\usepackage{mathtools}
\usepackage{hyperref}
\usepackage{paralist}
\usepackage{color}
\usepackage{makecell}
\usepackage{multirow}
\usepackage{multicol} 
\usepackage{booktabs} 

\newcommand{\Id}{\mathbb{I}}
\newcommand{\R}{\mathbb{R}}
\newcommand{\Rext}{\R\cup\{+\infty\}}

\newcommand{\set}[1]{\left\{#1\right\}}
\newcommand{\sets}[1]{\{#1\}}

\newcommand{\norms}[1]{\Vert#1\Vert}

\newcommand{\iprods}[1]{\langle #1\rangle}
\newcommand{\Eproof}{\hfill $\square$}
\newcommand{\prox}{\mathrm{prox}}
\newcommand{\proj}{\mathrm{proj}}
\newcommand{\diag}[1]{\mathrm{diag}(#1)}

\newcommand{\dom}[1]{\mathrm{dom}(#1)}
\newcommand{\gra}[1]{\mathrm{gra}(#1)}
\newcommand{\ran}[1]{\mathrm{ran}(#1)}
\newcommand{\zer}[1]{\mathrm{zer}(#1)}
\newcommand{\mbf}[1]{\mathbf{#1}}
\newcommand{\mbb}[1]{\mathbb{#1}}
\newcommand{\mcal}[1]{\mathcal{#1}}
\newcommand{\Ac}{\mathcal{A}}

\newcommand{\Xc}{\mathcal{X}}
\newcommand{\Gc}{\mathcal{G}}
\newcommand{\Hc}{\mathcal{H}}
\newcommand{\Lc}{\mathcal{L}}
\newcommand{\Qc}{\mathcal{Q}}
\newcommand{\Rc}{\mathcal{R}}
\newcommand{\Tc}{\mathcal{T}}
\newcommand{\Ec}{\mathcal{E}}
\newcommand{\Fc}{\mathcal{F}}

\newcommand{\Nc}{\mathcal{N}}
\newcommand{\Pc}{\mathcal{P}}
\newcommand{\Vc}{\mathcal{V}}

\newcommand{\intx}[1]{\mathrm{int}\left(#1\right)}

\newcommand{\BigO}[1]{\mathcal{O}\left(#1\right)}
\newcommand{\BigOs}[1]{\mathcal{O}\big(#1\big)}
\newcommand{\SmallOs}[1]{o\big(#1\big)}
\newcommand{\SmallO}[1]{o\left(#1\right)}
\newcommand{\ul}[1]{\underline #1}

\newcommand{\mytb}[1]{\textbf{#1}}
\newcommand{\myti}[1]{\textit{#1}}
\newcommand{\myblue}[1]{{\color{blue}{#1}}}
\newcommand{\myred}[1]{{\color{red}{#1}}}

\newcommand{\myeqc}[1]{ {\tiny\textcircled{#1}} }

\newcommand{\rmark}[1]{{\color{black}#1}}
\newcommand{\bmark}[1]{{\color{black}#1}}

\newcommand{\beforesec}{\vspace{-3ex}}
\newcommand{\aftersec}{\vspace{-2ex}}
\newcommand{\beforesubsec}{\vspace{-4ex}}
\newcommand{\aftersubsec}{\vspace{-2ex}}

\begin{document}

\title{Accelerated Extragradient-Type Methods: Generalization and Unified Convergence Analysis}
\titlerunning{Accelerated EG-Type Methods: Generalization and Unified Analysis}


\author{Quoc Tran-Dinh\vspace{0.25ex}\\
\newline {Department of Statistics and Operations Research}\\
\newline The University of North Carolina at Chapel Hill\\
318 Hanes Hall, UNC-Chapel Hill, NC 27599-3260.\\
\newline \textit{Email:} \url{quoctd@email.unc.edu}.}

\author{Quoc Tran-Dinh \and Nghia Nguyen-Trung}

\institute{Quoc Tran-Dinh \and Nghia Nguyen-Trung \at
             Department of Statistics and Operations Research\\ 
             The University of North Carolina at Chapel Hill\\
             318 Hanes Hall, UNC-Chapel Hill, NC 27599-3260\\
             \textit{Corresponding author:} quoctd@email.unc.edu
}

\date{} 

\maketitle

\vspace{-5ex}
\begin{abstract}
\small
This paper provides a comprehensive study for two types of extragradient methods: anchored extragradient-type and Nesterov's accelerated extragradient-type for solving [non]linear inclusions (and, in particular, equations), primarily under the Lipschitz continuity and the co-hypomonotonicity assumptions. 
It consists of the following main contributions.
First, we unify and generalize a class of anchored extragradient methods for monotone inclusions to a wider range of schemes encompassing existing algorithms as special cases. 
We establish $\BigOs{1/k}$ last-iterate convergence rates on the residual norm of the underlying mapping for this general framework and then specialize it to obtain convergence guarantees for specific instances, where $k$ denotes the iteration counter. 
Second, we extend our approach to a class of anchored Tseng's forward-backward-forward splitting (FBFS) methods to obtain a broader class of algorithms for solving co-hypomonotone inclusions, covering a class of nonmonotone problems. 
Again, we prove $\BigOs{1/k}$ last-iterate convergence rates for this general scheme and specialize it to obtain convergence results for existing and new variants. 
Third, we develop a new class of moving anchor FBFS algorithms that achieves both $\BigOs{1/k}$ and $\SmallOs{1/k}$ convergence rates, and the convergence of iterate sequences.
Fourth, we generalize and unify Nesterov's accelerated extra-gradient method to a new class of algorithms that covers existing schemes as special instances while generating new variants. 
For these schemes, we can prove $\BigO{1/k}$ last-iterate convergence rates for the residual norm under co-hypomonotonicity. 
Fifth, we propose another novel class of Nesterov's accelerated extragradient methods to solve inclusions. 
Interestingly, these algorithms achieve both $\BigO{1/k}$ and $\SmallO{1/k}$ last-iterate convergence rates, and also the convergence of iterate sequences under co-hypomonotonicity and Lipschitz continuity. 
Finally, we provide a set of numerical experiments encompassing different scenarios to validate our results.
\end{abstract}

\section{Introduction}\label{sec:intro}
\aftersec
\noindent\mytb{Overview.}
The \textit{generalized equation}, also known as the \textit{nonlinear inclusion}, serves as a versatile mathematical tool for modeling a wide range of problems in computational mathematics and related disciplines. It encompasses optimization problems (both constrained and unconstrained), minimax optimization, variational inequalities, complementarity problems, game theory, and fixed-point problems, see, e.g., \cite{Bauschke2011,reginaset2008,Facchinei2003,phelps2009convex,Rockafellar2004,Rockafellar1976b,ryu2016primer}.
This model has found direct applications in diverse fields such as operations research, economics, uncertainty quantification, and transportation, see, e.g., \cite{Ben-Tal2009,giannessi1995variational,harker1990finite,Facchinei2003,Konnov2001}.

In recent years, there has been a surge of interest in minimax problems, a special case of generalized equations, driven by their applications in machine learning and robust optimization. This is particularly evident in the areas of generative adversarial networks (GANs), adversarial training, and distributionally robust optimization, see, e.g., \cite{arjovsky2017wasserstein,Ben-Tal2009,goodfellow2014generative,levy2020large,madry2018towards,rahimian2019distributionally}.
Furthermore, minimax models have emerged as valuable tools in online learning and reinforcement learning, see, e.g., \cite{arjovsky2017wasserstein,azar2017minimax,bhatia2020online,goodfellow2014generative,jabbar2021survey,levy2020large,lin2022distributionally,madry2018towards,rahimian2019distributionally,wei2021last}.
The growing prominence of these applications has spurred renewed research efforts in nonlinear inclusions, fixed-point problems, and operator equations.

\beforesubsec
\subsection{\mytb{Problem statements, equivalent forms, and special cases}}\label{subsec:prob_stat}
\aftersubsec
\noindent\textbf{$\mathrm{(a)}$~Problem statement.}
In this paper, we consider the following \myti{generalized equation} (also known as a [composite] \myti{[non]linear inclusion}):
\begin{equation}\label{eq:NI}
\textrm{Find $x^{\star}\in\dom{\Phi }$ such that:} \quad 0 \in \Phi x^{\star} \equiv Fx^{\star} + Tx^{\star},
\tag{NI}
\end{equation}
where $F : \R^p\to\R^p$ is a single-valued operator, $T : \R^p\rightrightarrows 2^{\R^p}$ is a set-valued (or multivalued) mapping from $\R^p$ to $2^{\R^p}$ (the set of all subsets of $\R^p$), $\Phi := F + T$, and $\dom{\Phi } := \dom{F}\cap\dom{T}$ is the domain of $\Phi$, which is the intersection of the domains of $F$ and $T$.
We denote $\zer{\Phi}$ the solution set of \eqref{eq:NI}.

In this paper, we work with finite-dimensional Euclidean spaces $\R^p$ and $\R^n$.
Nevertheless, we believe that most results presented in this paper can be extended to Hilbert spaces.
This work is also a continuation of our recent paper \cite{tran2024revisiting} on non-accelerated extragradient-type methods. 

\vspace{0.5ex}
\noindent\textbf{$\mathrm{(b)}$~Special cases.}
Problem \eqref{eq:NI} covers the following special cases.
\begin{compactenum}
\item[$\mathrm{(i)}$] 
If $T = 0$, then \eqref{eq:NI} reduces to a \myti{[non]linear equation}:
\begin{equation}\label{eq:NE}
\textrm{Find $x^{\star}\in\dom{F}$ such that:} \quad Fx^{\star} = 0.
\tag{NE}
\end{equation}
Though \eqref{eq:NE} is a special case of \eqref{eq:NI}, in many cases,  \eqref{eq:NI} can be reformulated equivalently to \eqref{eq:NE} via, e.g., Moreau-Yosida's approximation or forward-backward splitting residual operator, see, e.g., \cite{tran2022accelerated}.
\item[$\mathrm{(ii)}$] 
Note that \eqref{eq:NE} is equivalent to the following fixed-point problem:
\begin{equation}\label{eq:FixPoint}
\textrm{Find $x^{\star}\in\dom{G}$ such that:} \quad x^{\star} = Gx^{\star},
\tag{FixP}
\end{equation}
where $G : \R^p \to\R^p$ is a given operator.
Clearly, by defining $Fx := x - Gx$, then \eqref{eq:FixPoint} is equivalent to \eqref{eq:NE}.
Alternatively, we can rewrite \eqref{eq:NI} equivalently to Kakutani's fixed-point problem: \myti{Find $x^{\star}$ such that $x^{\star} \in \Psi{x^{\star}}$} in the context of multivalued mappings, where $\Psi := \Id - \Phi$.
\item[$\mathrm{(iii)}$] 
If $T := \partial{g}$, the subdifferential of a proper, closed, and convex function $g : \R^p \to \Rext$, then \eqref{eq:NI} reduces a \myti{mixed variational inequality problem} (MVIP):
Find $x^{\star}\in\R^p$ such that:
\begin{equation}\label{eq:MVIP}
\iprods{Fx^{\star}, x - x^{\star}} + g(x) - g(x^{\star}) \geq 0, \ \forall x \in \R^p.
\tag{MVIP}
\end{equation}
In particular, if $T = \Nc_{\Xc}$, the normal cone of a nonempty, closed, and convex set $\Xc$ in $\R^p$ (i.e. $g = \delta_{\Xc}$, the indicator of $\Xc$), then \eqref{eq:MVIP} reduces to the classical (Stampacchia's) \myti{variational inequality problem} (VIP):
\begin{equation}\label{eq:VIP}
\textrm{Find $x^{\star}\in\Xc$ such that:} \quad \iprods{Fx^{\star}, x - x^{\star}} \geq 0, \ \textrm{for all} \ x \in \Xc.
\tag{VIP}
\end{equation}
While \eqref{eq:VIP} can be viewed as a primal VIP (or a strong VIP), its dual (or weak) form can be written as
\begin{equation}\label{eq:DVIP}
\textrm{Find $x^{\star}\in\Xc$ such that:} \quad \iprods{Fx, x - x^{\star}} \geq 0, \ \textrm{for all} \ x \in \Xc,
\tag{DVIP}
\end{equation}
which is known as Minty's variational inequality problem.
If $F$ is monotone (see the definition in Section~\ref{sec:background}) then both problems \eqref{eq:VIP} and \eqref{eq:MVIP} are equivalent, i.e. their solution sets are identical, see \cite{Facchinei2003,Konnov2001}.
\item[$\mathrm{(iv)}$] 
One important special case of \eqref{eq:NI} or \eqref{eq:VIP} is the optimality condition of the following general minimax problem:
\begin{equation}\label{eq:minimax_prob}
\min_{u \in\R^m} \max_{v\in \R^n} \Big\{ \Lc(u, v) := \varphi(u) + \Hc(u, v) - \psi(v) \Big\},
\end{equation}
where $\varphi : \R^m\to\Rext$ and $\psi : \R^n\to\Rext$ are often proper, closed, and convex functions, and $\Hc : \R^m\times\R^n\to\R$ is a joint objective function, often assumed to be differentiable, but not necessarily convex-concave.
If we denote $x := [u, v]$ as the concatenation of $u$ and $v$, and define $T := [\partial{\varphi}, \partial{\psi}]$ and $F := [\nabla_u{\Hc}(u, v), -\nabla_v{\Hc}(u, v)]$, then the optimality condition of \eqref{eq:minimax_prob} is exactly covered by \eqref{eq:NI} as a special case.
\end{compactenum}
Since \eqref{eq:NI} covers the above problems as special cases or equivalent forms, our methods developed in this paper can be customized to solve these problems.

\beforesubsec
\subsection{\mytb{Our goals and contributions}}\label{subsec:contribution}
\aftersubsec
Before discussing related work, let us state our goals and contributions in this paper.
We also compare our work with the most related and recent works.

\vspace{0.5ex}
\noindent\textbf{$\mathrm{(a)}$~Our goals.}
Our first goal is to unify and generalize both Halpern's and Nesterov's accelerated extragradient-type methods, covering existing algorithms as special cases. 
Our second goal is to  develop new classes of accelerated extragradient algorithms to solve \eqref{eq:NI} with ``better'' theoretical guarantees. 
We establish convergence rates for these generalized methods under standard or weaker assumptions. 
Unlike classical methods using gap functions or a distance to the solution set $\zer{\Phi}$, our convergence metric is the residual norm (\emph{cf.} \eqref{eq:res_norm}), which is applicable to both monotone and nonmonotone problems, and requires mild assumptions.
We also prove the convergence of iterate sequences to a solution of \eqref{eq:NI}.
Our approach allows for the generation of new variants through flexible choices for a general search direction, denoted by $u^k$.

\vspace{0.5ex}
\noindent\textbf{$\mathrm{(b)}$~Our contributions.}
Our contributions are detailed below, and Table  \ref{tbl:survey_results} offers a brief overview of our results and their distinction from prior results.
\begin{compactitem}
\item[$\mathrm{1.}$] \textbf{\textit{Generalized Anchored EG.}} 
We generalize the extra-anchored gradient method to a broader class of algorithms for solving \eqref{eq:NI}. 
Under a ``monotonicity" assumption on $F$ and a $3$-cyclically monotone assumption of $T$, we prove $\BigOs{1/k}$-last-iterate sublinear convergence rates for this generalized scheme, where $k$ denotes the iteration counter. 
This framework subsumes several common variants and known schemes in the literature as special instances  (see Table~\ref{tbl:survey_results}), while allowing us to derive new variants.

\item[$\mathrm{2.}$] \textbf{\textit{Generalized Fast EG.}}
We generalize the fast EG method  in \cite{lee2021fast}  for solving \eqref{eq:NI} to a broader class of schemes and establish its  $\BigOs{1/k}$-sublinear last-iterate convergence rate under a ``co-hypomonotonicity" assumption of $\Phi$ and the Lipschitz continuity of $F$. 
This generalization encompasses the fast EG method in \cite{lee2021fast} and the past-fast EG in \cite{cai2022baccelerated} as special instances.

\item[$\mathrm{3.}$] \textbf{\textit{Generalized Moving Anchor EG.}}
We propose  a class of generalized moving anchor EG methods  for solving \eqref{eq:NI}.
This class covers the methods in \cite{alcala2023moving,yuan2024symplectic} as special cases and again allows us to derive new variants, including a new variant of the optimistic gradient scheme.
We establish both $\BigOs{1/k}$ and $\SmallOs{1/k}$ last-iterate convergence rates under co-hypomonotonicity of $\Phi$ and Lipschitz continuity of $F$. 
We also prove the convergence of iterate sequences to a solution of \eqref{eq:NI}.

\item[$\mathrm{4.}$] \textbf{\textit{Generalized Nesterov's Accelerated EG.}}
We study a generalized form of Nesterov's accelerated EG method, inspired by the connection to Halpern's methods in \cite{tran2022connection}. 
We provide a unified $\BigO{1/k}$-last-iterate convergence rate for this generalized method, which covers Tseng's accelerated variant and Nesterov's accelerated forward-reflected-backward method (also known as an accelerated optimistic gradient scheme).

\item[$\mathrm{5.}$]
\textbf{\textit{Generalized Nesterov's Accelerated EG with Better Guarantees.}}
We develop a new class of Nesterov's accelerated EG methods using different correction terms to solve \eqref{eq:NI}. 
These methods appear to have better convergence guarantees than the previous one, and encompass \cite{bot2022fast,sedlmayer2023fast} as special cases but under the co-hypomonotone setting. 
We prove both $\BigO{1/k}$ and $\SmallO{1/k}$ convergence rates through flexible parameter updates. 
In addition, we also establish the convergence of iterate sequences to a solution of \eqref{eq:NI}, which have not been proven in the previous methods.
\end{compactitem}

\begin{table}[ht!]
\vspace{-0ex}
\newcommand{\cell}[1]{{\!\!}{#1}{\!\!}}
\begin{center}
\caption{Summary of existing results and our methods (marked by \myblue{blue} color)}\label{tbl:survey_results}
\vspace{-2ex}
\begin{small}
\resizebox{\textwidth}{!}{  
\begin{tabular}{|c|c|c|c|c|c|c|} \hline
{\!\!\!} \cell{Methods} {\!\!\!} & Choice of $u^k$ & \cell{Assumptions} & \cell{Add. Ass.} & \cell{Rates} & \myblue{ISC} & {\!\!\!} \cell{References} {\!\!\!} \\ \hline
\multicolumn{7}{|c|}{ For solving \eqref{eq:NE}} \\ \hline
\mytb{EAG}  & $u^k := Fx^k$ & $F$ is \mytb{mono} & None & $\BigOs{1/k}$  & No &   \cite{yoon2021accelerated} \\ \hline
\mytb{EAG-mv}  & $u^k := Fx^k$ & $F$ is \mytb{chm} & None & $\BigOs{1/k}$ & No &  \cite{alcala2023moving} \\ \hline
\mytb{FEG}  & $u^k := Fx^k$ & $F$ is \mytb{chm} & None & $\BigOs{1/k}$   & No &  \cite{lee2021fast} \\ \hline
\myblue{\mytb{GEAG}} & \myblue{$\mathrm{lin}(Fx^k, Fx^{k-1}, Fy^{k-1})$} &  \myblue{ $F$ is \mytb{mono}} & \myblue{ None}  & \myblue{ $\BigOs{1/k}$ } &  No & \myred{\mytb{Ours}} \\ \hline
\mytb{NesEG} & $u^k := Fx^k$ & $F$ is \mytb{mono} & None & $\BigOs{1/k}$   & \myblue{Yes} &  \cite{bot2022fast,tran2022connection} \\ \hline
{\!\!\!} \mytb{PEAG}/\mytb{APEG} {\!\!\!} &  $u^k := Fy^{k-1}$ & $F$ is \mytb{mono} & None & $\BigOs{1/k}$   & No &  \cite{tran2021halpern,tran2022connection} \\ \hline
\myblue{\mytb{GFEG}} &  \myblue{$\mathrm{lin}(Fx^k, Fx^{k-1}, Fy^{k-1})$} & \myblue{ $F$ is \mytb{chm}} & \myblue{ None}  & \myblue{ $\BigOs{1/k}$ } &  No & \myred{\mytb{Ours}} \\ \hline
\myblue{\mytb{GAEG}} &  \myblue{$\mathrm{lin}(Fx^k, Fx^{k-1}, Fy^{k-1})$} & \myblue{ $F$ is \mytb{chm}} & \myblue{ None}  & \myblue{ $\BigOs{1/k}$ } &  No & \myred{\mytb{Ours}} \\ \hline
\myblue{\mytb{GAEG$_{+}$}} &  \myblue{$\mathrm{lin}(Fx^k, Fx^{k-1}, Fy^{k-1})$} & \myblue{ $F$ is \mytb{chm}} & \myblue{ None}  & {\!\!\!\!\!\!\!} \myblue{ $\BigOs{1/k}$, \myred{$\SmallO{1/k}$} } {\!\!\!\!\!\!\!} & \myblue{Yes} & \myred{\mytb{Ours}} \\ \hline
\multicolumn{7}{|c|}{ For solving \eqref{eq:NI}, \eqref{eq:MVIP}, and \eqref{eq:VIP}} \\ \hline
\mytb{EAG} & $u^k := Fx^k$ &  $F$ is \mytb{mono} & $T := \Nc_{\Xc}$ & $\BigOs{1/k}$  & No & \cite{cai2022accelerated} \\ \hline
\mytb{SEG} & $u^k := Fx^k$ & $\Phi$ is \mytb{chm} & None & {\!\!\!\!\!} $\BigOs{1/k}, \myred{\SmallO{1/k}}$ {\!\!\!\!\!}  & \myblue{Yes} & \cite{yuan2024symplectic} \\ \hline
{\!\!} \mytb{FEG}/\mytb{AEG} {\!\!} & $u^k := Fx^k$ & $\Phi$ is \mytb{chm} & None & $\BigOs{1/k}$   & No & \cite{cai2022accelerated,tran2023extragradient} \\ \hline
{\!\!\!} \mytb{PEAG}/\mytb{APEG} {\!\!\!} & $u^k := Fy^{k-1}$ & $\Phi$ is \mytb{chm} & None & $\BigOs{1/k}$  &  No &  \cite{cai2022baccelerated,tran2023extragradient} \\ \hline
\myblue{\mytb{GEAG}} &  \myblue{$\mathrm{lin}(Fx^k, Fx^{k-1}, Fy^{k-1})$}  &  \myblue{$F$ is \mytb{mono}} & \myblue{$T$ is \mytb{$3$-cm}} & $\BigOs{1/k}$  & No &  \myred{\mytb{Ours}} \\ \hline
\myblue{\mytb{GFEG}} &  \myblue{$\mathrm{lin}(Fx^k, Fy^{k-1}, {\color{green}v^{k-1}})$}  & \myblue{ $\Phi$ is \mytb{chm}} & \myblue{ None}  & \myblue{ $\BigOs{1/k}$ } & No &  \myred{\mytb{Ours}} \\ \hline
\myblue{\mytb{GFEG-mv}} &  \myblue{$\mathrm{lin}(Fx^k, Fy^{k-1}, {\color{green}v^{k-1}})$}  & \myblue{ $\Phi$ is \mytb{chm}} & \myblue{ None}  & {\!\!\!\!\!\!} \myblue{ $\BigOs{1/k}$, \myred{$\SmallO{1/k}$} } {\!\!\!\!\!\!} & \myblue{Yes} &  \myred{\mytb{Ours}} \\ \hline
\mytb{NesEG} & $u^k := Fx^k$ & $F$ is \mytb{mono} & $T := \Nc_{\Xc}$ & $\BigOs{1/k}$   &  \myblue{Yes} & \cite{sedlmayer2023fast,tran2022connection} \\ \hline
\myblue{\mytb{GAEG}} & \myblue{$\mathrm{lin}(Fx^k, Fy^{k-1}, {\color{green}v^{k-1}})$}  & \myblue{ $F$ is \mytb{chm}} & \myblue{ None}  & \myblue{ $\BigOs{1/k}$ } & No &  \myred{\mytb{Ours}} \\ \hline
\myblue{\mytb{GAEG$_{+}$}} &  \myblue{$\mathrm{lin}(Fx^k, Fy^{k-1}, {\color{green}v^{k-1}})$} & \myblue{ $F$ is \mytb{chm}} & \myblue{ None}  & {\!\!\!\!\!\!} \myblue{ $\BigOs{1/k}$, \myred{$\SmallO{1/k}$} } {\!\!\!\!\!\!} & \myblue{Yes} & \myred{\mytb{Ours}} \\ \hline
\end{tabular}}
\end{small}
\end{center}
\vspace{-1ex}
{\footnotesize
\textbf{Abbreviations:} 
\mytb{EAG} $=$ extra-anchored gradient;
\mytb{FEG} $=$ fast extragradient;
\mytb{EAG-mv} $=$ moving anchor fast extragradient;
\mytb{PEAG} $=$ past extra-anchored gradient;
\mytb{SEG} $=$ symplectic extragradient;
\mytb{AEG} $=$ Nesterov's accelerated extragradient; 
\mytb{APEG} $=$ Nesterov's accelerated past extragradient; 
and the prefix \mytb{G} is for \mytb{Generalized} (i.e. our methods).
In addition, {\color{green}$v^{k-1}$} is an appropriate direction depending on the method;
\mytb{mono} $=$ monotone;
\mytb{chm} $=$ co-hypomonotone;
and \mytb{3-cm} $=$ $3$-cyclically monotone.
Note that ``Add. Ass.'' means ``Additional Assumptions''; and ``\myblue{ISC}'' abbreviates for ``iterate sequence convergence''.
}
\vspace{-2ex}
\end{table}

\noindent\textbf{$\mathrm{(b)}$~Comparison.}
We believe that our results in this paper are new and significant.
Indeed, let us compare them with the most related works here.
\begin{compactitem}[$\bullet$]
\item
First, one key novelty lies in our search directions, denoted by $u^k$, as illustrated in Figure~\ref{fig:EG_illustration}. 
While classical extragradient methods and their accelerated variants specifically employ two search directions: $Fx^k$ and $Fy^k$ evaluated at two distinct sequences  $\sets{x^k}$ and $\sets{y^k}$, past-extragradient and optimistic gradient-type schemes reduce the number of $F$ evaluations by replacing $Fx^k$ by $Fy^{k-1}$, enabling one to eliminate the sequence $\sets{x^k}$.
 Our GEG-type methods utilize $u^k$ and $Fy^k$, where $u^k$ is a flexible linear combination of $Fx^k$, $Fy^{k-1}$, and potentially a prior computed  quantity  {\color{green}$v^{k-1}$}.
 This flexibility expands the spectrum of our methods and offers a potential for improving practical performance by searching for the ``best" direction $u^k$ among various linear combinations.

\begin{figure}[hpt!]
\vspace{-0ex}
\centering
\includegraphics[width=\linewidth]{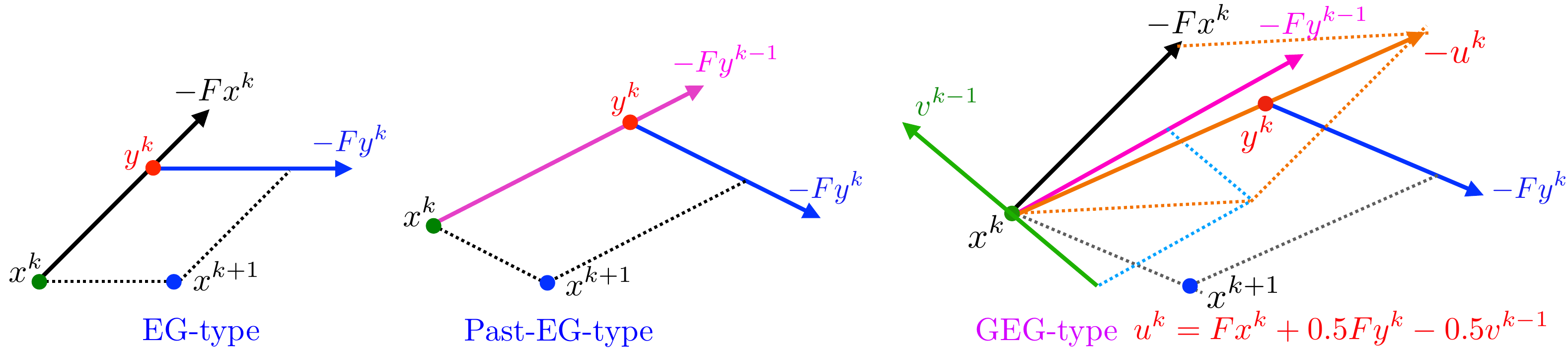}
\caption{
The search directions in existing methods (EG- and Past-EG-type) and our schemes (GEG-type):  with $u^k := \alpha_1 Fx^k + \alpha_2 Fy^{k-1} + \alpha_3 v^{k-1}$ for $\alpha_1 = 1$, $\alpha_2=0.5$, and $\alpha_3 = -0.5$.
}
\label{fig:EG_illustration}
\vspace{-4ex}
\end{figure}
\item 
Second, our generalized EAG method in Section~\ref{sec:EAG4NI} is new because it employs a general direction $u^k$ as described above and accommodates $3$-cyclically monotone operators $T$ (covering normal cone and subdifferential mappings as special cases, but not necessarily maximally monotone). 
Our method subsumes several existing methods, including those in \cite{cai2022accelerated,tran2021halpern,yoon2021accelerated}, and can generate new variants by flexibly choosing $u^k$ to satisfy the condition \eqref{eq:EAG4NI_uk_cond}. 
However, its applicability is currently limited to monotone $F$ and $3$-cyclically monotone $T$, while extensions to non-monotone cases remain open.
Similarly, our generalized anchored FBFS method in Section~\ref{sec:FEG4NI} is novel and general, encompassing the algorithms from \cite{cai2022baccelerated,lee2021fast,tran2023extragradient} as special cases. 
It also allows one to generate new variants through flexible choices of $u^k$ under the condition \eqref{eq:FEG4NI_u_cond}.
Nevertheless, the convergence of iterates is still unknown. 

\item
Third, our moving-anchor FBFS method in Section~\ref{sec:DFBFS4NI} is more general and different from the one in  \cite{alcala2023moving} and a very recent work in \cite{yuan2024symplectic} due to the anchor point update.
Note that the moving anchor trick is very similar to classical restarting techniques in optimization \cite{Odonoghue2012} and in fixed-point iterations \cite{lu2024restarted}.
While \cite{alcala2023moving} only achieves $\BigOs{1/k}$ convergence rates under different settings, \cite{yuan2024symplectic} was able to prove a sharper rate, i.e. $\SmallOs{1/k}$, and the convergence of the iterates.
We achieve all of these convergence results for a much more general class of methods compared to \cite{alcala2023moving,yuan2024symplectic}, covering new variants.

\item
Fourth, our generalized Nesterov's accelerated method in Section~\ref{sec:AEG4NI} extends the first algorithm in \cite{tran2023extragradient} to a broader class of algorithms. 
This includes novel variants, such as the case where $u^k := Fy^{k-1}$. 
Notably, when $u^k := Fy^{k-1}$, our method covers those in \cite{bot2022fast,sedlmayer2023fast} as special cases, while relaxing monotone assumption of $\Phi$ to a co-hypomonotone one.
It also allows $T$  to be a general (e.g., maximally monotone)  mapping instead of a normal cone as in \cite{sedlmayer2023fast}.
However, we can only prove $\BigO{1/k}$ rates of the residual norm $\norms{Fx^k + \xi^k}$.

\item 
Finally, our algorithm in Section~\ref{sec:NGEAG4NI} appears to be novel both in terms of its algorithmic form and its convergence analysis. 
It differs from existing methods, even in special cases such as solving \eqref{eq:NE} under monotonicity assumptions. 
We establish both $\BigO{1/k}$ and $\SmallO{1/k}$ last-iterate convergence rates under co-hypomonotonicity of $\Phi$ and Lipschitz continuity of $F$. 
Our method subsumes those in \cite{bot2022fast,sedlmayer2023fast} as special cases and allows for the creation of new variants by the selection of $u^k$.
Furthermore, our analysis is more intricate compared to the analyses in \cite{bot2022fast,sedlmayer2023fast} and \cite{yuan2024symplectic} due to the generalization of our scheme and the co-hypomonotonicity of $\Phi$.
\end{compactitem}

\vspace{-0.5ex}
\beforesubsec
\subsection{\mytb{Related work}}\label{subsec:related_work}
\aftersubsec
Theory and numerical methods for \eqref{eq:NI} and its special cases have been extensively studied for many decades, as evidenced in monographs such as \cite{Bauschke2011,Facchinei2003,minty1962monotone,Rockafellar1997} and the references therein.
In our recent work \cite{tran2024revisiting}, we studied several non-accelerated extragradient-type methods. 
This paper focuses on both anchored (aka Halpern's) and Nesterov's accelerated counterparts of the extragradient method, which are expected to achieve faster convergence rates. 
Therefore, we only review related work along this line. 

\noindent\mytb{$\mathrm{(a)}$~From non-accelerated to accelerated methods.}
For the EG method and its variants employing constant stepsizes, as investigated in our recent work \cite{tran2024revisiting}, their convergence rate on the residual norm $\norms{Fx^k + \xi^k}$ for $\xi^k \in Tx^k$ is $\BigOs{1/\sqrt{k}}$.
This rate is unimprovable for constant parameters, as demonstrated in \cite{golowich2020last}. 
To enhance the convergence rates of EG-type methods from  $\BigOs{1/\sqrt{k}}$ to $\BigOs{1/k}$, it is essential to utilize variable parameters, including stepsizes, or [dual] averaging techniques \cite{Nemirovskii2004,Nesterov2007a}. 
For variable parameters, there are at least two distinct approaches to developing accelerated methods for solving \eqref{eq:NE} and \eqref{eq:NI}. 
The first approach relies on Halpern's fixed-point iteration \cite{halpern1967fixed}, while the second leverages Nesterov's accelerated techniques. 
Although these approaches were developed independently, they exhibit a connection as discussed in \cite{tran2022connection}.

\noindent\mytb{$\mathrm{(b)}$~Halpern's fixed-point iteration-type methods.}
The Halpern fixed-point iteration, a classical method for approximating fixed points of nonexpansive operators (equivalently, finding roots of co-coercive operators), can be viewed as a blend of fixed-point iteration and adaptive Tikhonov regularization, as discussed in \cite{boct2024extra}. 
Alternatively, it is also known as an anchored method revolves around a fixed initial point (i.e. anchor point). 
While extensively studied in fixed-point theory, the first demonstration of its $\BigO{1/k}$ last-iterate convergence rate was achieved by Sabach \& Shtern in \cite{sabach2017first} and subsequently by Lieder in \cite{lieder2021convergence}. 
Notably, this $\BigO{1/k}$ rate on the residual norm is shown to be unimprovable in \cite{lieder2021convergence} through a constructive example.

In the last few years, we have seen significant advancements of Halpern's fixed-point method, as evidenced by many works such as \cite{diakonikolas2020halpern,lee2021fast,tran2021halpern,yoon2021accelerated}.
Diakonikolas \emph{et al.} \cite{diakonikolas2020halpern} effectively leveraged Halpern's technique for solving both equations and variational inequalities.
A seminal contribution by Yoon and Ryu \cite{yoon2021accelerated} extended Halpern's iteration to the  EG method for solving \eqref{eq:NE}, resulting in the ``extra-anchored gradient" (EAG) method.
Remarkably, EAG maintains the $\BigO{1/k}$ last-iterate convergence rate while only requiring monotonicity and Lipschitz continuity of $F$.
Building upon EAG, Lee and Kim \cite{lee2021fast} extended its applicability to the co-hypomonotone setting of \eqref{eq:NE}, achieving the same convergence rates.
Tran-Dinh and Luo  \cite{tran2021halpern} further expanded this line of research by applying the EAG framework to the past-extragradient method introduced by Popov \cite{popov1980modification}, leading to the ``past extra-anchored gradient" (PEAG) method with the same $\BigO{1/k}$ last-iterate convergence rates (up to a constant factor).
Recent works, such as \cite{cai2022accelerated} and \cite{cai2022baccelerated}, have significantly extended these results in \cite{lee2021fast,tran2021halpern,yoon2021accelerated} to develop methods for solving \eqref{eq:VIP} and \eqref{eq:NI}, all while preserving the desirable $\BigO{1/k}$ last-iterate convergence rates for $\norms{Fx^k + \xi^k}$.
Unlike Nesterov's accelerated methods, both $\SmallOs{1/k}$ convergence rates and the convergence of iterates remain open in these works.
 
We suspect that Halpern-type methods suffer from a key limitation: parameter selection, typically $\tau_k = \frac{1}{k+\nu}$ for some $\nu > 1$ (e.g., $\nu = 2$), restricts their flexibility and can potentially hinder their performance. 
In addition, similar to original Nesterov's accelerated methods, this choice of $\tau_k$ may not allow us to prove faster convergence rates (e.g., $\SmallO{1/k}$ rates) as well as the convergence of iterates as discussed in \cite{chambolle2015convergence}.
Moreover, the fixed anchor point contributes to all the iterates, potentially affects the overall performance. 
Recent work by Yuan and Zhang \cite{yuan2024symplectic} addresses these issues by adapting existing Halpern's accelerated schemes \cite{cai2022accelerated,lee2021fast,yoon2021accelerated} and the moving anchor EG method in \cite{alcala2023moving} with a new parameter choice and a forward update for the anchor point as in \cite{alcala2023moving}.
This approach, derived from a symplectic discretization of an associated ODE, achieves both $\BigO{1/k}$ and $\SmallO{1/k}$ last-iterate convergence rates on $\norms{Fx^k + \xi^k}$ under co-hypomonotonicity while enabling a more flexible parameter choice. 
However, the focus of \cite{yuan2024symplectic} remains on the extragradient method, leaving extensions to past-extragradient and other schemes as an open research question.
In this paper, we  will also address this open question.

\noindent\mytb{$\mathrm{(c)}$~Nesterov's accelerated-type methods.}
Alternatively, Nesterov's accelerated method \cite{Nesterov1983} stands as a significant breakthrough in convex optimization over the past few decades. 
Though initially introduced in 1983, its widespread recognition surged with seminal works by Nesterov \cite{Nesterov2005c} and Beck and Teboulle \cite{Beck2009}. 
This powerful technique has since been extensively explored and applied across diverse fields, encompassing proximal-point, coordinate gradient, stochastic gradient, primal-dual, operator splitting, conditional gradient, Newton-type, and high-order methods. 
While the majority of research on accelerated methods has focused on convex optimization, recent years have witnessed a surge of interest in extending these techniques to the realm of monotone equations and inclusions. 
Early contributions in this direction include the works of Attouch et al. \cite{attouch2020convergence,attouch2019convergence}, Bot \emph{et al.} \cite{bot2022fast,bot2022bfast}, Kim and Fessler \cite{kim2021accelerated}, and Maing\'{e} \cite{mainge2021accelerated,mainge2021fast}, among many others.
In this paper, we go beyond the monotonicity in prior works to tackle a class of co-hypomonotone operators, possibly covering nonmonotone applications.

Developing accelerated methods for  inclusions of the form \eqref{eq:NI} encounters significant  challenges compared to convex optimization, necessitating a fundamental shift in the construction of an appropriate Lyapunov function, serving as a key metric for convergence analysis. 
Existing research often leverages proximal-point methods, which were extended to accelerated schemes in \cite{guler1992new}. 
An alternative approach to designing accelerated methods involves employing the ``performance estimation problem" technique pioneered in \cite{drori2014performance} and further explored in \cite{taylor2017exact,taylor2017smooth}. 
This technique has found successful applications in various works, including \cite{gupta2022branch,kim2021accelerated,ryu2020operator}. 
Notably, Nesterov's accelerated methods for diverse problem classes have been demonstrated to exhibit ``optimal" convergence rates, meaning their upper bounds on convergence rates (or ``oracle" complexity) align with the corresponding lower bounds in a specific sense, as evidenced in \cite{Nesterov2004,woodworth2016tight}, making them theoretically unimprovable.

\noindent\mytb{$\mathrm{(d)}$~Other connections.}
It has been demonstrated that Halpern's and Nesterov's accelerated methods can be interpreted as discretizations of corresponding dynamical systems, a perspective frequently observed in classical gradient methods \cite{Su2014,suh2024continuous,yuan2024symplectic}. 
This viewpoint has spurred extensive research into new variants and extensions, as shown in works such as \cite{attouch2016rate,bot2022fast,bot2022bfast,shi2021understanding,wibisono2016variational}. 
Furthermore, connections between Nesterov's method and other algorithms have been explored. 
For instance, \cite{attouch2022ravine} establishes an equivalence between Nesterov's accelerated methods and Ravine's methods, originally proposed in 1961. 
Recent studies, such as \cite{partkryu2022,tran2022connection}, have revealed relationships between Nesterov's accelerated schemes and Halpern's fixed-point iterations \cite{halpern1967fixed}, demonstrating their equivalence under specific conditions. 
Building upon this understanding, \cite{tran2023extragradient,tran2022connection} have leveraged this perspective to develop novel Nesterov's accelerated variants for solving \eqref{eq:NE} and \eqref{eq:NI}, including extragradient methods. 

\beforesubsec
\subsection{\mytb{Paper organization}}\label{subsec:prob_stat}
\aftersubsec
Our paper is organized in such a way that each section from Section~\ref{sec:EAG4NI} to Section~\ref{sec:NGEAG4NI} presents one class of methods and its convergence theory and special instances. 
Each section is relatively independent, and mostly does not depend on other sections.
It starts with the algorithm framework, special instances, key technical lemmas for convergence analysis, convergence theorems, and the convergence of special cases. 
Technical lemmas show the flow of our analysis.
For clarity of presentation, we defer technical proofs to the appendix.

More specifically, the remainder of this paper is structured as follows. 
\begin{compactenum}[1.]
\item Section \ref{sec:background} provides a foundation by reviewing essential concepts and relevant prior results.
One important concept is the ``co-hypomonotonicity'', see \cite{bauschke2020generalized}.

\item Section~\ref{sec:EAG4NI} introduces a novel class of anchored extragradient methods designed to address monotone inclusions \eqref{eq:NI}, along with their accompanying convergence guarantees and special instances.

\item Building upon this, Section \ref{sec:FEG4NI} generalizes and explores a class of anchored Tseng's forward-backward-forward splitting schemes for solving \eqref{eq:NI}. 
This section encompasses a unified convergence analysis of the proposed method, deriving convergence rates for specific instances. 

\item Section~\ref{sec:DFBFS4NI} proposes a class of  moving anchor FBFS schemes for solving \eqref{eq:NI} and establishes its convergence guarantees under the co-hypomonotonicity. 
 
\item Section~\ref{sec:AEG4NI} delves into a generalized class of Nesterov's accelerated schemes and analyzes their convergence rates.

\item Distinct from Sections~\ref{sec:DFBFS4NI} and \ref{sec:AEG4NI}, Section~\ref{sec:NGEAG4NI} introduces a novel generalized  Nesterov's accelerated EG  framework. 
Notably, we establish both $\BigO{1/k}$ and $\SmallO{1/k}$ convergence rates for our method and its specific instances, while also demonstrating the convergence of iterates to a solution of \eqref{eq:NI}.

\item Finally, Section \ref{sec:AEG4NI_numerical_experiments} presents a set of numerical experiments to illustrate the theoretical aspects and practical ability of our proposed methods.
\end{compactenum}
Although the technical proofs in this paper are both innovative and substantially distinct from  prior works, we have elected to relocate all of them to the appendix. This decision was made primarily to enhance the overall presentation and readability of the main body of the paper. 
The proofs are organized in each corresponding section as outlined in the main text of this paper.

\beforesec
\section{Background and Preliminary Results}\label{sec:background}
\aftersec
We recall several concepts which will be used in this paper.
These concepts and properties are well-known and can be found, e.g., in \cite{Bauschke2011,Facchinei2003,Rockafellar2004,Rockafellar1970,ryu2016primer}.

\beforesubsec
\subsection{\bf Basic Concepts, Monotonicity, and Lipschitz Continuity}\label{subsec:basic_concepts}
\aftersubsec
We work with finite dimensional Euclidean spaces $\R^p$ and $\R^n$ equipped with standard inner product $\iprods{\cdot, \cdot}$ and Euclidean norm $\norms{\cdot}$.
For a multivalued mapping $T : \R^p \rightrightarrows 2^{\R^p}$, $\dom{T} = \set{x \in\R^p : Tx \not= \emptyset}$ denotes its domain, $\ran{T} := \bigcup_{x \in \dom{T}}Tx$ is its range, and $\gra{T} = \set{(x, y) \in \R^p\times \R^p : y \in Tx}$ stands for its graph, where $2^{\R^p}$ is the set of all subsets of $\R^p$.
The inverse mapping of $T$ is defined as $T^{-1}y := \sets{x \in \R^p : y \in Tx}$.
We say that $T$ is \emph{closed} if $\gra{T}$ is closed.
For a proper, closed, and convex function $f : \R^p\to\Rext$, $\dom{f} := \sets{x \in \R^p : f(x) < +\infty}$ denotes the domain of $f$, $\partial{f}$ denotes the subdifferential of $f$, and $\nabla{f}$ stands for the [sub]gradient of $f$.

\vspace{1ex}
\noindent\textbf{$\mathrm{(a)}$~\textit{Monotonicity and comonotonicity.}}
For a multivalued mapping $T : \R^p \rightrightarrows 2^{\R^p}$ and $\rho \in \R$, we say that $T$ is $\rho$-comonotone (see \cite{bauschke2020generalized}) if 
\begin{equation*}
\begin{array}{ll}
\iprods{u - v, x - y} \geq \rho\norms{u  - v}^2, \quad \forall (x, u), (y, v)  \in \gra{T}.
\end{array}
\end{equation*}
If $T$ is single-valued, then this condition reduces to $\iprods{Tx - Ty, x - y} \geq \rho\norms{Tx - Ty}^2$ for all $x, y\in\dom{T}$.
\begin{compactitem}
\item If $\rho = 0$, then we say that $T$ is monotone.
\item If $\rho > 0$, then $T$ is called $\rho$-co-coercive. 
In particular, if $\rho = 1$, then $T$ is firmly nonexpansive.
\item If $\rho < 0$, then $T$ is called $\vert\rho\vert$-co-hypomonotone, see  \cite{bauschke2020generalized,combettes2004proximal}.
\end{compactitem}
Clearly, $T$ is $\vert\rho\vert$-co-hypomonotone iff $T^{-1}$ is $\vert\rho\vert$-hypomonotone, see \cite{bauschke2020generalized}.
Note that a co-hyopomonotone operator can also be nonmonotone.
A simple example is $Tx = \mbf{Q}x + \mbf{q}$ for a given invertible symmetric, but not necessarily positive definite matrix $\mbf{Q}$, and $\mbf{q}\in\R^p$.
In this case, we have $\rho = -\lambda_{\min}(\mbf{Q}^{-1}) > 0$.

We say that $T$ is maximally $\rho$-comonotone if $\gra{T}$ is not properly contained in the graph of any other $\rho$-comonotone operator.
If $\rho = 0$, then we say that $T$ is maximally monotone.
Note that $T$ is maximally monotone, then $\eta T$ is also maximally monotone for any $\eta > 0$, and if $T$ and $F$ are maximally monotone, and $\dom{T}\cap\intx{\dom{F}} \not=\emptyset$, then $F + T$ is maximally monotone.
For a proper, closed, and convex function $f : \R^p\to\Rext$, the subdifferential $\partial{f}$ of $f$ is maximally monotone. 

\vspace{1ex}
\noindent\textbf{$\mathrm{(b)}$~\textit{Cyclic monotonicity.}}
We say that a mapping $T$ is $m$-cyclically monotone ($m\geq 2$) if $\sum_{i=1}^m\iprods{u^i, x^i - x^{i+1}} \geq 0$ for all $(x^i, u^i) \in \gra{T}$ and $x_1 = x_{m+1}$ (see \cite{Bauschke2011}).
We say that $T$ is cyclically monotone if it is $m$-cyclically monotone for every $m \geq 2$.
If $T$ is $m$-cyclically monotone, then it is also $\hat{m}$-cyclically monotone for any $2 \leq \hat{m} \leq m$.
Since a  $2$-cyclically monotone operator $T$ is monotone, any $m$-cyclically monotone operator $T$ is $2$-cyclically monotone, and thus is also monotone.
An $m$-cyclically monotone operator $T$ is called maximally $m$-cyclically monotone if $\gra{T}$ is not properly contained into the graph of any other $m$-cyclically monotone operator. 

As proven in \cite[Theorem 22.18]{Bauschke2011} that $T$ is maximally cyclically monotone iff $T = \partial{f}$, the subdifferential of a proper, closed, and convex function $f$.
On the one hand, there exist maximally monotone operators (e.g., rotation operators) that are not $3$-cyclically monotone, see \cite[Example 22.15]{Bauschke2011}.
On the other hand, as indicated in  \cite[Example 2.16]{bartz2007fitzpatrick}, there exist maximally $3$-cyclically monotone operators that are not maximally monotone.
These classes of operators are intersected with each other, and one is not a subclass of the other.

\vspace{1ex}
\noindent\textbf{$\mathrm{(c)}$~\textit{Lipschitz continuity and contraction.}}
A multivalued mapping $T$ is called $L$-Lipschitz continuous if $\sup\set{\norms{u - v} : u \in Tx, \ v \in Ty }\leq L\norms{x - y}$ for all $x, y \in\dom{T}$, where $L \geq 0$ is the Lipschitz constant. 
If $T$ is single-valued, then this condition reduces to $\norms{Tx - Ty} \leq L\norms{x - y}$ for all $x, y\in\dom{T}$.
If $L = 1$, then we say that $T$ is nonexpansive, while if $L \in [0, 1)$, then we say that $T$ is $L$-contractive, and $L$ is its contraction factor.
If $T$ is $\rho$-co-coercive with $\rho > 0$, then $T$ is also $L$-Lipschitz continuous with the Lipschitz constant $L := \frac{1}{\rho}$.
However, the reverse statement is not true in general.
For a continuously differentiable function $f : \R^p \to \R$, we say that $f$ is $L$-smooth if its gradient $\nabla{f}$ is $L$-Lipschitz continuous on $\dom{f}$.
If $f$ is convex and $L$-smooth, then $\nabla{f}$ is $\frac{1}{L}$-co-coercive and vice versa, see, e.g., \cite{Nesterov2004}.

\vspace{1ex}
\noindent\textbf{$\mathrm{(d)}$~\textit{Normal cone.}}
Given a nonempty, closed, and convex set $\Xc$ in $\R^p$, the normal cone of $\Xc$ is defined as $\Nc_{\Xc}(x) := \sets{w \in \R^p : \iprods{w, x - y} \geq 0, \ \forall y\in\Xc}$ if $x\in\Xc$ and $\Nc_{\Xc}(x) = \emptyset$, otherwise.
If $f := \delta_{\Xc}$, the indicator of a convex set $\Xc$,   then we have $\partial{f} = \Nc_{\Xc}$.
Moreover, $J_{\partial{f}}$ reduces to the projection onto $\Xc$.

\vspace{1ex}
\noindent\textbf{$\mathrm{(e)}$~\textit{Resolvent and proximal operators.}}
Given a multivalued operator $T$, the operator $J_Tx := \set{y \in \R^p : x \in y + Ty}$ is called the resolvent of $T$, denoted by $J_Tx = (\Id + T)^{-1}x$, where $\Id$ is the identity mapping.
If $T$ is monotone, then evaluating $J_T$ requires solving a strongly monotone inclusion $0 \in y - x + Ty$.
Hence, $J_T$ is well-defined and single-valued.
If $T = \partial{f}$, the subdifferential of proper, closed, and convex function $f$, then $J_T$ reduces to the proximal operator of $f$, denoted by $\prox_f$, which can be computed as $\prox_f(x) := \mathrm{arg}\min_{y}\sets{f(y) + (1/2)\norms{y-x}^2}$.
In particular, if $T = \Nc_{\Xc}$, the normal cone of a closed and convex set $\Xc$, then $J_T$ is the projection onto $\Xc$, denoted by $\proj_{\Xc}$.
If $T$ is maximally monotone, then $\mathrm{ran}(\Id + T) = \R^p$ (by Minty's theorem) and $T$ is firmly nonexpansive (and thus nonexpansive).

\beforesubsec
\subsection{\bf Exact Solutions and Approximate Solutions}\label{subsec:exact_and_approx_sols}
\aftersubsec
\vspace{-0.5ex}
There are different metrics to characterize exact and approximate solutions of \eqref{eq:NI}.
The most obvious one is the residual norm of $\Phi$, which is defined as
\begin{equation}\label{eq:res_norm}
r(x) := \min_{\xi \in Tx}\norms{Fx + \xi}, \quad x \in \dom{\Phi}.
\vspace{-0.5ex}
\end{equation}
Clearly, $r(x^{\star}) = 0$ iff $x^{\star} \in \zer{\Phi}$, a solution of \eqref{eq:NI}.
If $T = 0$, then $r(x) = \norms{Fx}$.
Given a tolerance $\epsilon > 0$, if $r(\hat{x}) \leq \epsilon$, then we say that $\hat{x}$ is an $\epsilon$-approximate solution to \eqref{eq:NI}.
The algorithms presented in this paper use this metric as the main tool to characterize approximate solutions to \eqref{eq:NI}.

Other metrics often used for monotone \eqref{eq:VIP}, a special case of \eqref{eq:NI}, are gap and restricted gap functions \cite{Facchinei2003,Konnov2001,Nesterov2007a}, which are respectively defined as
\begin{equation}\label{eq:gap_func}
\mbb{G}(x) := \max_{y \in \Xc}\iprods{Fy, y - x} \quad \text{and} \quad \mbb{G}_{\mcal{B}}(x) := \max_{y\in\Xc\cap \mcal{B}}\iprods{Fy, y  - x},
\vspace{-0.5ex}
\end{equation}
where $\mcal{B}$ is a given nonempty, closed, and bounded convex set.
Note that, under the monotonicity of $F$, $\mbb{G}(x) \geq 0$ for all $x\in\Xc$, and $\mbb{G}(x^{\star}) = 0$ iff $x^{\star}$ is a solution to \eqref{eq:VIP}.
Therefore, to characterize an $\epsilon$-approximate solution $\tilde{x}$ to \eqref{eq:VIP}, we can impose a condition $\mbb{G}(\tilde{x}) \leq \epsilon$.

For the restricted gap function $\mbb{G}_{\mcal{B}}$, if $x^{\star}$ is a solution of \eqref{eq:VIP} and $x^{\star} \in \mcal{B}$, then $\mbb{G}_{\mcal{B}}(x^{\star}) = 0$.
Conversely, if $\mbb{G}_{\mcal{B}}(x^{\star}) = 0$ and $x^{\star} \in \intx{\mcal{B}}$, the interior of $\mcal{B}$, then $x^{\star}$ is a solution of \eqref{eq:VIP} in $\mcal{B}$ (see \cite[Lemma 1]{Nesterov2007a}).
For \eqref{eq:DVIP}, we can also define similar dual gap functions and restricted dual gap functions \cite{Nesterov2007a}.
Gap functions have widely been used in the literature to characterize approximate solutions generated by many numerical methods for solving monotone \eqref{eq:VIP} or \eqref{eq:DVIP}, see, e.g., \cite{chen2017accelerated,Cong2012,Facchinei2003,Konnov2001,Nemirovskii2004,Nesterov2007a} as concrete examples.

If $J_{\eta T}$ is well-defined and single-valued for some $\eta > 0$, and $F$ is single-valued, then we can use  the following forward-backward splitting residual:
\begin{equation}\label{eq:FB_residual}
\Gc_{\eta }(x) := \tfrac{1}{\eta}\left(x - J_{\eta T}(x - \eta Fx)\right), 
\vspace{-0.5ex}
\end{equation}
to characterize solutions of \eqref{eq:NI},  where $F$ is single-valued and $J_{\eta T}$ is the resolvent of $\eta T$ for any $\eta > 0$.
It is clear that $\Gc_{\eta}(x^{\star}) = 0$ iff $x^{\star} \in \zer{\Phi}$.
In addition, if $J_{\eta T}$ is firmly nonexpansive, then we also have 
\begin{equation}\label{eq:FBR_bound2}
\norms{\Gc_{\eta }(x)} \leq \norms{Fx + \xi}, \quad (x, \xi) \in \gra{T}.
\end{equation}
Hence, for a given tolerance $\epsilon > 0$, if $\norms{\Gc_{\eta}(\tilde{x})} \leq \epsilon$, then we can say that $\tilde{x}$ is an $\epsilon$-approximate solution of \eqref{eq:NI}.
If $T := \Nc_{\Xc}$, i.e. \eqref{eq:NI} reduces to \eqref{eq:VIP}, then $\Gc_{1}(x)$ reduces to the classical natural map $\Pi_{F,\Xc}(x) = x - \proj_{\Xc}(x - Fx)$ of \eqref{eq:VIP}, and $r_n(x) := \norms{\Gc_{1}(x) } = \norms{\Pi_{F,\Xc}(x)}$ is the corresponding natural residual at $x$.
From \eqref{eq:FBR_bound2}, we have $r_n(x) \leq \norms{Fx + \xi}$ for any $\xi \in \Nc_{\Xc}(x)$.
The natural residual is a fundamental metric to study generalized equations of the form \eqref{eq:NI} and \eqref{eq:VIP}, see, e.g., \cite{Facchinei2003} for more details.

\begin{remark}\label{re:gap_function}
Note that characterizing an [approximate] solution via gap or restricted gap function for \eqref{eq:VIP} requires the monotonicity of $F$ to upper bound $\mbb{G}(x^k)$ by $\mbb{G}(x^k) \leq \max_{y \in \Xc}\iprods{Fx^k, x^k - y}$.
For the nonmonotone case, including co-hypomonotonicity, these gap functions are not applicable in general.
\end{remark}

\beforesec
\section{A Class of Extra-Anchored Gradient Methods for Monotone \eqref{eq:NI}}\label{sec:EAG4NI}
\aftersec
In this section, we generalize the extra-anchored gradient method (EAG) in \cite{yoon2021accelerated} for \eqref{eq:NE} and in \cite{cai2022accelerated} for \eqref{eq:VIP} to a more general class of algorithms and for monotone inclusions of the form \eqref{eq:NI}.
We provide a unified convergence rate analysis for our scheme and derive special instances.

\beforesubsec
\subsection{\mytb{A Class of Extra-Anchored Gradient Methods for \eqref{eq:NI}}}\label{subsec:EAG4NI}
\aftersubsec
\noindent\textbf{$\mathrm{(a)}$~\textit{The proposed method.}}
We propose the following scheme to solve \eqref{eq:NI}.
Starting from $x^0 \in \dom{\Phi}$, at each iteration $k\geq 0$, given $u^k$, we update
\begin{equation}\label{eq:EAG4NI}
\arraycolsep=0.2em
\left\{\begin{array}{lcl}
y^k &:= & J_{\hat{\eta}_k T}\left( \tau_kx^0 + (1-\tau_k)x^k - \hat{\eta}_k u^k \right),   \vspace{1ex}\\
x^{k+1} &:= & J_{\eta T}\left( \tau_kx^0 + (1-\tau_k)x^k - \eta Fy^k \right),
\end{array}\right.
\tag{GEAG}
\end{equation}
where $\tau_k \in (0, 1)$, $\hat{\eta}_k > 0$, and $\eta > 0$ are given and will be determined later, and $u^k \in \mathbb R^p$ is a user-defined direction satisfying the following condition:
\begin{equation}\label{eq:EAG4NI_uk_cond}
\norms{u^k \! - \! Fx^k}^2  \leq  \kappa \norms{Fx^k \! - \! Fy^{k-1}}^2 + \hat{\kappa}\norms{x^k \! - \! y^{k-1} \! + \! \hat{\eta}_{k-1}(Fx^{k-1} \! - \! u^{k-1})}^2,
\end{equation}
for given parameters $\kappa \geq 0$ and $\hat{\kappa} \geq 0$, $y^{-1} := x^0$, and $u^{-1} = u^0 := Fx^0$.
This condition looks slightly technical, but we will explain later how it is constructed.
Our generalization in \eqref{eq:EAG4NI} consists of the following new points:
\begin{compactitem}
\item[$\mathrm{(i)}$] We assume that $T$ is maximally $3$-cyclically monotone.
This choice covers  $T = \Nc_{\Xc}$, the normal cone of $\Xc$, and $T = \partial{g}$, the subdifferential of a convex function $g$ as special cases, but it is more general than the subdifferentials, and is not identical to the class of maximally monotone operators.
\item[$\mathrm{(ii)}$] We use a user-defined direction $u^k$ that satisfies \eqref{eq:EAG4NI_uk_cond}, which covers $u^k := Fx^k$ and $u^k := Fy^{k-1}$ as special instances.
\end{compactitem}
\vspace{0.75ex}
\noindent\textbf{$\mathrm{(b)}$~\textit{Three instances.}}
By different choices of $u^k$, \eqref{eq:EAG4NI} covers both existing and new variants of the extra-anchored gradient (EAG) methods for solving monotone inclusions of the form \eqref{eq:NI}.
Here, let us consider at least three instances of \eqref{eq:EAG4NI} as follows.
\begin{compactitem}
\item[$\mathrm{(i)}$] \textbf{Variant 1 (EAG).} If we choose $u^k := Fx^k$, then \eqref{eq:EAG4NI_uk_cond} holds with $\kappa = 0$ and $\hat{\kappa} = 0$.
Clearly, this variant of \eqref{eq:EAG4NI} covers the method in  \cite{cai2022accelerated} as a special case when $T = \Nc_{\Xc}$.
If, in addition, $T = 0$, then this variant reduces to the extra-anchored gradient (EAG) scheme from \cite{yoon2021accelerated}.
In fact, \ref{eq:EAG4NI} purely generalizes \cite{yoon2021accelerated} from \eqref{eq:NE} to \eqref{eq:NI} with a general extra-gradient direction $u^k$, determined appropriately, and a $3$-cyclically monotone $T$.

\item[$\mathrm{(ii)}$] \textbf{Variant 2 (Popov's past EAG).} If  $u^k := Fy^{k-1}$, then \eqref{eq:EAG4NI_uk_cond} holds with $\kappa = 1$ and $\hat{\kappa} = 0$.
This variant of \eqref{eq:EAG4NI} covers the past-extra-anchored gradient method in \cite{tran2021halpern} as a special case with $T := 0$.
However, this variant has not yet been studied for \eqref{eq:VIP} and a general monotone inclusion  \eqref{eq:NI}.

\item[$\mathrm{(iii)}$] \textbf{Variant 3 (Generalized EAG).} If we choose 
\begin{equation}\label{eq:EAG4NI_general_uk}
u^k := (1-\alpha) Fx^k + \alpha Fy^{k-1} + \hat{\alpha} \big(x^k - y^{k-1} + \hat{\eta}_{k-1}(Fx^{k-1} - u^{k-1})\big),
\end{equation}
for any $\alpha, \hat{\alpha}  \in \R$, then $u^k$ satisfies \eqref{eq:EAG4NI_uk_cond} with $\kappa :=  \alpha^2 \geq 0$ and $\hat{\kappa} := \hat{\alpha}^2 \geq 0$.
Clearly, this choice takes the advantages of both \textbf{Variant 1} and \textbf{Variant 2}, while still covers other possibilities generated by the term $v^{k-1} := x^k - y^{k-1} + \hat{\eta}_{k-1}(Fx^{k-1} - u^{k-1})$.
In particular, if $T = 0$, then it is trivial to show that the term $v^k$ reduces to $v^{k-1} := \hat{\eta}_{k-1} Fx^{k-1} - \eta Fy^{k-1}$.
In this case, we get $u^k = (1-\alpha)Fx^k + (\alpha - \hat{\alpha}  \eta)Fy^{k-1} +  \hat{\alpha}  \hat{\eta}_{k-1}Fx^{k-1}$, which is a linear combination of $Fx^k$, $Fy^{k-1}$, and $Fx^{k-1}$.
\end{compactitem}
Note that the per-iteration cost of  \textbf{Variant 1} and \textbf{Variant 3} is essentially the same with two evaluations $Fx^k$ and $Fy^k$ of $F$, while  \textbf{Variant 2} only requires one evaluation $Fy^k$ of $F$.
Note also that, due to the flexibility of choosing $\alpha$ and also $\hat{\alpha}$,  \textbf{Variant 3} covers  \textbf{Variant 1} and \textbf{Variant 2} as special cases.

If $T = 0$, then our methods presented above can be applied to solve \eqref{eq:NE}.
If $T = \Nc_{\Xc}$, the normal cone of a convex set $\Xc$, then they can be utilized to solve \eqref{eq:VIP}.
If $T := \partial{g}$, the subdifferential of a convex function $g$, then our methods reduce to the variants for solving \eqref{eq:MVIP}.
Obviously, our methods can also be specified to solve \eqref{eq:FixPoint} and \eqref{eq:minimax_prob}.
However, we omit the details of these applications and leave them open to the readers.

\beforesubsec
\subsection{\mytb{Key Estimates for Convergence Analysis}}\label{subsec:EAG4NI_key_estimates}
\aftersubsec
To facilitate the form of condition \eqref{eq:EAG4NI_uk_cond} and to analyze the convergence of \eqref{eq:EAG4NI}, we define the following quantities:
\begin{equation}\label{eq:EAG4NI_ex2}
\arraycolsep=0.3em
\begin{array}{llcl}
& w^k := Fx^k + \xi^k, & & \quad \hat{w}^k := Fy^{k-1} + \xi^k, \vspace{1ex} \\
& \tilde{w}^k := Fx^k + \zeta^k, & \quad  \textrm{and}  \quad & \quad \tilde{z}^k := u^k + \zeta^k,
\end{array}
\end{equation}
for some $\xi^k \in Tx^k$ and $\zeta^k \in Ty^k$.
Then, we can easily see that  $w^k - \hat{w}^k = Fx^k - Fy^{k-1}$ and $\tilde{w}^k - \tilde{z}^k = Fx^k - u^k$.

\vspace{0.75ex}
\noindent\textbf{$\mathrm{(a)}$~\textit{Equivalent form.}}
Using these quantities and the fact that $w = J_{\eta T}(v)$ iff $v - w \in \eta Tw$, we can equivalently rewrite \eqref{eq:EAG4NI} as
\begin{equation}\label{eq:EAG4NI_reform}
\arraycolsep=0.2em
\left\{\begin{array}{lcl}
y^k &:= & \tau_kx^0 + (1-\tau_k)x^k - \hat{\eta}_k\tilde{z}^k, \vspace{1ex}\\
x^{k+1} &:= & \tau_kx^0 + (1-\tau_k)x^k - \eta \hat{w}^{k+1}.
\end{array}\right.
\end{equation}
Moreover,  by \eqref{eq:EAG4NI_ex2},  \eqref{eq:EAG4NI_reform}, and $\hat{\eta}_{k-1} := \eta(1-\tau_{k-1})$ from \eqref{eq:EAG4NI_param_update} below, we get
\begin{equation*} 
\arraycolsep=0.0em
\begin{array}{lcl}
x^k - y^{k-1} + \hat{\eta}_{k-1}(Fx^{k-1} - u^{k-1}) & \overset{\tiny\eqref{eq:EAG4NI_ex2}, \eqref{eq:EAG4NI_reform}}{ = } & -\eta \hat{w}^k + \hat{\eta}_{k-1}\tilde{z}^{k-1} + \hat{\eta}_{k-1}(\tilde{w}^{k-1} - \tilde{z}^{k-1}) \vspace{1ex}\\
& = & -\eta[ \hat{w}^k - (1-\tau_{k-1})\tilde{w}^{k-1}].
\end{array} 
\end{equation*}
Therefore, the condition \eqref{eq:EAG4NI_uk_cond} is equivalent to the following form:
\begin{equation}\label{eq:EAG4NI_uk_cond2}
\arraycolsep=0.2em
\begin{array}{ll}
\norms{u^k  -   Fx^k}^2  \leq  \kappa \norms{Fx^k  -  Fy^{k-1}}^2 + \hat{\kappa}\eta^2\norms{ \hat{w}^k - (1-\tau_{k-1})\tilde{w}^{k-1} }^2.
\end{array}
\end{equation}
This condition will be used in our analysis below, and it also explains how the condition \eqref{eq:EAG4NI_uk_cond} is constructed.

\vspace{0.75ex}
\noindent\textbf{$\mathrm{(b)}$~\textit{Lyapunov function.}}
To establish the convergence of \eqref{eq:EAG4NI}, we define the following functions:
\begin{equation}\label{eq:EAG4NI_potential_func}
\arraycolsep=0.2em
\begin{array}{lcl}
\Vc_k & := &  \frac{a_k}{2} \norms{w^k}^2 + b_k\iprods{w^k, x^k - x^0} + \frac{b_0}{\eta}\norms{x^0 - x^{\star}}^2,  \vspace{1ex}\\
\Lc_k &:= & \Vc_k + \frac{ c_k}{2} \norms{u^k - Fx^k}^2,
\end{array}
\end{equation}
where $x^{\star} \in \zer{\Phi}$, $a_k$, $b_k$ and $c_k$  are nonnegative parameters, determined later.
Note that $\Vc_k$ has been widely used in the literature, including  \cite{cai2022accelerated,lee2021fast,tran2023extragradient,yoon2021accelerated} as a Lyapunov function to analyze convergence of extra-anchored gradient-type methods.
Since we study the generalized scheme \eqref{eq:EAG4NI}, we need to add one more term $\frac{c_k}{2}\norms{u^k - Fx^k}^2$ to $\Vc_k$ to handle the difference between $u^k$ and $Fx^k$.

\vspace{0.75ex}
\noindent\textbf{$\mathrm{(c)}$~\textit{Descent property.}}
Now, we state the following descent property of $\Vc_k$.

\begin{lemma}\label{le:EAG4NI_key_estimate1}
For \eqref{eq:NI}, assume that $\zer{\Phi} \neq\emptyset$,  $F$ is $L$-Lipschitz continuous and monotone, and $T$ is maximally $3$-cyclically monotone.
Let $\sets{(x^k, y^k)}$ be generated by \eqref{eq:EAG4NI} satisfying \eqref{eq:EAG4NI_uk_cond} and $\Vc_k$ be defined by \eqref{eq:EAG4NI_potential_func}.
Suppose that 
\begin{equation}\label{eq:EAG4NI_param_cond}
\arraycolsep=0.3em
\begin{array}{ll}
\tau_k \in (0, 1), \quad \hat{\eta}_k := \eta (1-\tau_k), \quad a_k := \frac{\eta b_k(1-\tau_k)}{\tau_k}, \quad\text{and} \quad  b_{k+1} := \frac{b_k}{1-\tau_k}.
\end{array}
\end{equation}
Then, for any $\omega > 0$, $r > 0$, and $c > 0$, let $M_c := (1+c)(1+\omega)L^2$,  we have
\begin{equation}\label{eq:EAG4NI_key_estimate1}
\hspace{-1ex}
\arraycolsep=0.3em
\begin{array}{lcl}
\Vc_k - \Vc_{k+1} &\geq & \frac{\eta b_{k+1}[ \tau_{k+1} - \tau_k(1-\tau_{k+1}) ]}{2\tau_k\tau_{k+1}} \norms{w^{k+1}}^2 +  \frac{\omega \eta b_{k+1}}{2\tau_k}\norms{w^{k+1} - \hat{w}^{k+1}}^2 \vspace{1ex}\\
&& + {~} \frac{(1 - r) a_k}{2} \norms{w^k - \tilde{w}^k}^2 - \frac{(c + r M_c \eta^2)a_k}{2 rc}\norms{Fx^k - u^k}^2 \vspace{1ex}\\
&& + {~}   \frac{\eta b_{k+1}(1 - M_c\eta^2)}{2\tau_k} \norms{\hat{w}^{k+1} - (1-\tau_k)\tilde{w}^k}^2.
\end{array}
\hspace{-3ex}
\end{equation}
\end{lemma}

\beforesubsec
\subsection{\mytb{Convergence Analysis of \eqref{eq:EAG4NI} and Its Special Instances}}\label{subsec:EAG4NI_convergence}
\aftersubsec
\noindent\textbf{$\mathrm{(a)}$~\textit{The convergence of \eqref{eq:EAG4NI}.}}
For $\kappa$ and $\hat{\kappa}$ from \eqref{eq:EAG4NI_uk_cond}, let $r \in (0, 1]$ and
\begin{equation}\label{eq:EAG4NI_eta_upper_bound}
\arraycolsep=0.2em
\begin{array}{lcl}
\bar{\eta} := \frac{\sqrt{r}}{\sqrt{(1 + r)(r + 2\kappa)L^2 + 2\kappa\hat{\kappa}}}.
\end{array}
\end{equation}
Now, we are ready to state the convergence of \eqref{eq:EAG4NI}.

\begin{theorem}\label{th:EAG4NI_convergence}
For \eqref{eq:NI}, assume that $\zer{\Phi} \neq\emptyset$,  $F$ is $L$-Lipschitz continuous and monotone, and $T$ is maximally $3$-cyclically monotone.
Let $\sets{(x^k, y^k)}$ be generated by \eqref{eq:EAG4NI} such that $u^k$ satisfies \eqref{eq:EAG4NI_uk_cond} and 
\begin{equation}\label{eq:EAG4NI_param_update}
\arraycolsep=0.2em
\begin{array}{lcl}
\tau_k := \frac{1}{k + \nu}, \quad \eta_k :=  \eta \in \left(0,  \bar{\eta}\right], \quad\text{and} \quad  \hat{\eta}_k := \eta(1-\tau_k),
\end{array}
\end{equation}
where $\nu > 1$ is given and $\bar{\eta}$ is given in \eqref{eq:EAG4NI_eta_upper_bound}.
Then, for any $x^{\star}\in\zer{\Phi}$ and $\xi^k \in Tx^k$, the following result holds:
\begin{equation}\label{eq:EAG4NI_convergence1}
\norms{Fx^k + \xi^k}^2 \leq  \frac{4\norms{x^0 - x^{\star}}^2 + \eta^2 \norms{Fx^0 + \xi^0}^2}{\eta^2(k + \nu - 1)^2}.
\end{equation}
\end{theorem}

\noindent\textbf{$\mathrm{(b)}$~\textit{Two special instances.}}
Next, we derive the convergence of the two special instances of \eqref{eq:EAG4NI}: the extra-anchored gradient (EAG) method with $u^k := Fx^k$ (\textbf{Variant 1}), and the past-extra-anchored gradient (PEAG) method with  $u^k := Fy^{k-1}$ (\textbf{Variant 2}).

\begin{corollary}\label{co:EAG4NI_Fxk_Fyk-1}
For \eqref{eq:NI}, assume that $\zer{\Phi} \neq\emptyset$,  $F$ is $L$-Lipschitz continuous and monotone, and $T$ is maximally $3$-cyclically monotone.
Let $\sets{(x^k, y^k)}$ be generated by \eqref{eq:EAG4NI} using either $u^k := Fx^k$ or $u^k := Fy^{k-1}$.
\begin{compactitem}
\item[$\mathrm{(i)}$~\textbf{Variant 1.}] We choose $u^k := Fx^k$ and the parameters as follows:
\begin{equation}\label{eq:EAG4NI_Fxk_param_update}
\arraycolsep=0.2em
\begin{array}{lcl}
\tau_k := \frac{1}{k+\nu}, \quad \eta_k :=  \eta \in \big(0,  \frac{1}{L}\big], \quad\text{and} \quad  \hat{\eta}_k := \eta(1-\tau_k), \quad (\nu > 1).
\end{array}
\end{equation}
\item[$\mathrm{(ii)}$~\textbf{Variant 2.}] We choose $u^k := Fy^{k-1}$ and the parameters as follows:
\begin{equation}\label{eq:EAG4NI_Fyk-1_param_update}
\arraycolsep=0.2em
\begin{array}{lcl}
\tau_k := \frac{1}{k+\nu}, \quad \eta_k :=  \eta \in \big(0,  \frac{1}{L\sqrt{6}}\big], \quad\text{and} \quad  \hat{\eta}_k := \eta(1-\tau_k).
\end{array}
\end{equation}
\end{compactitem}
Then, for  $\xi^k \in Tx^k$ and $x^{\star}\in\zer{\Phi}$, the following result holds:
\begin{equation}\label{eq:EAG4NI_Fxk_Fyk-1_convergence1}
\norms{Fx^k + \xi^k}^2 \leq  \frac{4\norms{x^0 - x^{\star}}^2 + \eta^2\norms{Fx^0 + \xi^0}^2}{\eta^2(k + \nu - 1)^2}.
\end{equation}
\end{corollary}

\begin{proof}
$\mathrm{(i)}$~If $u^k := Fx^k$, then  \eqref{eq:EAG4NI_uk_cond} holds with $\kappa = \hat{\kappa} = 0$. 
From \eqref{eq:EAG4NI_eta_upper_bound}, we have $\bar{\eta} := \frac{1}{L\sqrt{1+r}}$.
Letting $r\downarrow 0^{+}$, the condition $\eta \in (0, \bar{\eta}]$ in Theorem~\ref{th:EAG4NI_convergence} becomes $\eta \in \big(0, \frac{1}{L}\big]$ as in \eqref{eq:EAG4NI_Fxk_param_update}.

$\mathrm{(ii)}$~If $u^k = Fy^{k-1}$, then \eqref{eq:EAG4NI_uk_cond} holds with $\kappa = 1$ and $\hat{\kappa} = 0$. 
In this case, from  \eqref{eq:EAG4NI_eta_upper_bound}, we have $\bar{\eta} := \frac{\sqrt{r}}{L\sqrt{(1+r)(2+r)}}$.
If we choose $r := 1$, then we get $\bar{\eta} := \frac{1}{L\sqrt{6}}$, leading to the choice of $\eta$ as in \eqref{eq:EAG4NI_Fyk-1_param_update}.

In both cases, we can conclude that the bound \eqref{eq:EAG4NI_Fxk_Fyk-1_convergence1} is a direct consequence of \eqref{eq:EAG4NI_convergence1} from Theorem \ref{th:EAG4NI_convergence}. 
\Eproof
\end{proof}

\beforesec
\section{A Class of Anchored FBFS Methods for Inclusion \eqref{eq:NI} }\label{sec:FEG4NI}
\aftersec
In this section, we develop another class of Halpern-type methods by generalizing the fast extragradient method in \cite{lee2021fast} to solve \eqref{eq:NI}.
Note that this method is different from \eqref{eq:EAG4NI} as it is rooted from Tseng's forward-backward-forward splitting (FBFS) method \cite{tseng2000modified} instead of the extragradient method from \cite{korpelevich1976extragradient}.
In addition, we relax our assumption from monotonicity  to co-hypomonotonicity.

\beforesubsec
\subsection{\mytb{A Class of  Anchored FBFS Methods} }\label{subsec:FEG4NI}
\aftersubsec
First, for given $x^k \in \dom{\Phi}$ and $u^k \in \R^p$, we define the following elements:
\begin{equation}\label{eq:FEG4NI_w_defs}
w^k := Fx^k + \xi^k, \quad \hat{w}^k := Fy^{k-1} + \xi^k, \quad \text{and} \quad z^k := u^k + \xi^k,
\end{equation}
for some $\xi^k \in Tx^k$.

\noindent\textbf{$\mathrm{(a)}$~\textit{The proposed method.}}
Our algorithm is described as follows.
Starting from $x^0\in\dom{\Phi}$, at each iteration $k \geq 0$, for given $(x^k, u^k)$, we update
\begin{equation}\label{eq:FEG4NI}
\hspace{-3ex}
\arraycolsep=0.1em
\left\{\begin{array}{lcl}
y^k        &:= &  x^k + \tau_k(x^0 - x^k) -  ( \hat{\eta}_k - \beta_k) z^k, \vspace{1ex}\\
x^{k+1} &:= & x^k + \tau_k(x^0 - x^k) - \eta \hat{w}^{k+1} + \beta_kz^k,
\end{array}\right.
\hspace{-3ex}
\tag{GFEG}
\end{equation}
where $\tau_k \in (0, 1)$, $\eta  > 0$,  $\hat{\eta}_k > 0$, and $\beta_k \geq 0$ are given, determined later.
The term $u^k$ in $z^k$ is chosen such that it satisfies the following condition:
\begin{equation}\label{eq:FEG4NI_u_cond}
\arraycolsep=0.2em
\begin{array}{lcl}
\norms{ Fx^k - u^k }^2  & \leq & \kappa\norms{ Fx^k - Fy^{k-1} }^2 + \hat{\kappa}\norms{ w^k - w^{k-1} }^2, 
\end{array}
\end{equation}
for given constants $\kappa \geq 0$ and $\hat{\kappa} \geq 0$, $y^{-1} = x^{-1} := x^0$, and $u^0 := Fx^0$.

\vspace{0.75ex}
\noindent\textbf{$\mathrm{(b)}$~\textit{Three special instances.}}
We consider three special  instances of \eqref{eq:FEG4NI}.
\begin{compactitem}
\item[$\mathrm{(i)}$] \textbf{Variant 1: The anchored FBFS method.} If we choose $u^k := Fx^k$, then \eqref{eq:FEG4NI_u_cond} holds with $\kappa = 0$ and $\hat{\kappa} = 0$.
Moreover, we can see that \eqref{eq:FEG4NI} covers the following existing methods as special cases.
 
\item If $T = 0$, $\beta_k  := 0$, and $\hat{\eta}_k = \eta = \eta_k$ (fixed or varying), then \eqref{eq:FEG4NI} reduces to the extra-anchored gradient (EAG) scheme for solving \eqref{eq:NE} in \cite{yoon2021accelerated} under the monotonicity of $F$ as
\begin{equation*}
\arraycolsep=0.2em 
\left\{\begin{array}{lcl}
y^k & := &  \tau_kx^0 + (1-\tau_k)x^k - \eta_kFx^k, \vspace{1ex}\\
x^{k+1} & = &  \tau_kx^0 + (1-\tau_k)x^k - \eta_kFy^k.
\end{array}\right.
\end{equation*}
\item If $T = 0$, $\beta_k := 2\rho(1-\tau_k)$, and $\hat{\eta}_k := \eta(1-\tau_k)$, then \eqref{eq:FEG4NI} reduces to the fast extragradient (FEG) variant for solving \eqref{eq:NE} in \cite{lee2021fast}, but under the co-hypomonotonicity of $F$.

\item If $T$ is a maximally monotone operator (e.g., $T := \Nc_{\Xc}$, the normal cone of a nonempty, closed, and convex set $\Xc$), $\beta_k := 2\rho(1-\tau_k)$ and $\hat{\eta}_k := \eta(1-\tau_k)$, then  \eqref{eq:FEG4NI} is exactly the variant of EAG studied in \cite{cai2022accelerated}.

\item[$\mathrm{(ii)}$]\textbf{Variant 2: Anchored FRBS  method.} 
If we choose $u^k := Fy^{k-1}$, then \eqref{eq:FEG4NI_u_cond} holds with $\kappa = 1$ and $\hat{\kappa} = 0$.
In this case, \eqref{eq:FEG4NI} reduces to
\begin{equation}\label{eq:FEG4NI_Fyk-1}
\arraycolsep=0.2em
\left\{\begin{array}{lcl}
y^k &:= &   \tau_kx^0 + (1-\tau_k)x^k  - ( \hat{\eta}_k - \beta_k) \hat{w}^k \vspace{1ex}\\
x^{k+1} &:= &  \tau_kx^0 + (1-\tau_k)x^k - \eta \hat{w}^{k+1} + \beta_k \hat{w}^k.
\end{array}\right.
\end{equation}
This scheme is a variant of \mytb{Popov's past extra-gradient} method in \cite{popov1980modification}.
We call it the \mytb{anchored FRBS} scheme that covers variants studied in \cite{cai2022baccelerated,tran2023extragradient}.
This scheme can also be viewed as a Halpern's accelerated variant of the \textbf{forward-reflected backward splitting (FRBS)} or the \mytb{optimistic gradient} method in the literature, see also \cite{daskalakis2018training,mertikopoulos2019optimistic,mokhtari2020convergence}.

When $T = 0$, as discussed in Part 1 of our work \cite{tran2024revisiting}, we can view \eqref{eq:FEG4NI} as a Halpern's acceleration of the \mytb{forward-reflected-backward splitting} scheme in \cite{malitsky2020forward}, which is also equivalent to   the \mytb{reflected gradient} method in \cite{malitsky2015projected}, the \mytb{reflected forward-backward splitting} scheme in \cite{cevher2021reflected}, and the \mytb{golden ratio} method in \cite{malitsky2019golden} for solving \eqref{eq:NE}.
This scheme can also be viewed as a Halpern's accelerated variant of the \mytb{optimistic gradient} method in the literature, see also \cite{daskalakis2018training,mertikopoulos2019optimistic,mokhtari2020convergence}.
Hence, \eqref{eq:FEG4NI} is sufficiently general to cover several existing methods as its instances.

\item[$\mathrm{(iii)}$]\textbf{Variant 3: Generalization.} 
We can construct a generalized direction:
\begin{equation}\label{eq:FEG4NI_generalized_uk}
\arraycolsep=0.2em
\begin{array}{lcl}
u^k & := & \alpha Fx^k + \hat{\alpha} Fy^{k-1} + (1 - \alpha - \hat{\alpha}) Fx^{k-1} + \hat{\alpha}(\xi^k  - \xi^{k-1}),
\end{array}
\end{equation}
for any given constants $\alpha, \hat{\alpha} \in \R$.

Clearly, $u^k$ is an affine combination of $Fx^k$, $Fy^{k-1}$, $Fx^{k-1}$, and $\xi^k -\xi^{k-1}$. 
Moreover, we can verify that $u^k$ satisfies \eqref{eq:FEG4NI_u_cond} with $\kappa = (1+m)\hat{\alpha}^2$ and $\hat{\kappa}  = (1+m^{-1})(1-\alpha - \hat{\alpha})^2$ by Young's inequality for any $m > 0$.
Generally, $u^k$ requires  $Fx^k$ and/or $Fx^{k-1}$, and thus \eqref{eq:FEG4NI} needs at most two evaluations of $F$ at each iteration $k$, which is the same cost as \textbf{Variant 1}.
\end{compactitem}
In fact, \eqref{eq:FEG4NI} is rooted from Tseng's forward-backward-forward splitting method \cite{tseng2000modified} instead of the extragradient method \cite{korpelevich1976extragradient}  because it only requires one resolvent evaluation $J_{\eta T}$ per iteration.
Recently, \cite{tran2023extragradient} provides an elementary convergence analysis for \eqref{eq:FEG4NI}, which relies on the technique in \cite{yoon2021accelerated}.
In this paper, we generalize these variants to \eqref{eq:FEG4NI}, which covers a wide range of variants, including \cite{lee2021fast,tran2023extragradient,yoon2021accelerated} as special instances.

In particular, if $T = 0$, then \eqref{eq:FEG4NI} solves \eqref{eq:NE}, and  \eqref{eq:FEG4NI_u_cond} reduces to
\begin{equation*} 
\arraycolsep=0.2em
\begin{array}{lcl}
\norms{ Fx^k - u^k }^2  & \leq & \kappa\norms{ Fx^k - Fy^{k-1} }^2 + \hat{\kappa}\norms{ Fx^k - Fx^{k-1} }^2. 
\end{array}
\end{equation*}
In addition, $u^k$ in \eqref{eq:FEG4NI_generalized_uk} reduces to $u^k :=  \alpha Fx^k + \hat{\alpha} Fy^{k-1} + (1 - \alpha - \hat{\alpha}) Fx^{k-1}$.

\vspace{0.75ex}
\noindent\textbf{$\mathrm{(c)}$~\textit{The implementation of \eqref{eq:FEG4NI}}.}
Since $x^{k+1}$ are in both sides of  line 2  of \eqref{eq:FEG4NI}, by using the resolvent  $J_{\eta T}$ of $T$, we can rewrite \eqref{eq:FEG4NI} as
\begin{equation}\label{eq:FEG4NI_impl}
\arraycolsep=0.1em
\left\{\begin{array}{lcl}
y^k &:= &  x^k + \tau_k(x^0 - x^k) -  ( \hat{\eta}_k - \beta_k)(u^k + \xi^k), \vspace{1ex}\\
x^{k+1}  & := & J_{\eta T}\left( y^k - \eta Fy^k  +  \hat{\eta}_k (u^k + \xi^k)  \right), \vspace{1ex}\\
\xi^{k+1} & := &  \frac{1}{\eta}\big[ y^k +\hat{\eta}_k (u^k + \xi^k) - x^{k+1} \big] - Fy^k,
\end{array}\right.
\end{equation}
where $w^k := Fx^k + \xi^k$, $\xi^0 \in Tx^0$ is arbitrary.

Generally, we do not require $J_{\eta T}$ to be single-valued, but only assume that $\ran{J_{\eta T}} \subseteq\dom{F} = \R^p$, and $\dom{J_{\eta T}} = \R^p$ so that the iterates are well-defined.
However, to simplify our analysis, we assume that $J_{\eta T}$ is single-valued.
In this case, we say that $J_{\eta T}$ is \textbf{well-defined}.

Since $w^k := Fx^k + \xi^k$, the condition~\eqref{eq:FEG4NI_u_cond} can be rewritten as 
\begin{equation*}
\norms{ Fx^k - u^k }^2 \leq \kappa\norms{ Fx^k - Fy^{k-1} }^2 + \hat{\kappa}\norms{ Fx^k - Fx^{k-1} + \xi^k - \xi^{k-1} }^2.
\end{equation*}
We have provided three concrete choices of $u^k$ as discussed above, but any direction $u^k$ satisfying this condition still works.

\begin{remark}\label{re:FEG4NI_u_cond2}
In fact, we can relax our condition~\eqref{eq:FEG4NI_u_cond} to the following one:
\begin{equation*}
\norms{ Fx^k - u^k }^2 \leq \kappa\norms{ Fx^k - Fy^{k-1} }^2 + \hat{\kappa}\norms{ w^k  - w^{k-1} }^2 + \bar{\kappa}\norms{\hat{w}^k - (1-\tau_{k-1})w^{k-1}}^2,
\end{equation*}
for given constants $\kappa, \hat{\kappa}, \bar{\kappa} \in \R_{+}$.
Then, our analysis bellow still goes through.
Nevertheless, analyzing this case is relatively involved as we need to process an additional parameter $\bar{\kappa}$.
Hence, we do not include it in this paper.
\end{remark}

\beforesubsec
\subsection{\mytb{Key Estimates for Convergence Analysis}}\label{subsec:FEG4NI_key_estimates}
\aftersubsec
\noindent\mytb{$\mathrm{(a)}$~\textit{Lyapunov function.}}
Similar to Section~\ref{sec:EAG4NI}, we will use the following Lyapunov function to analyze the convergence of \eqref{eq:FEG4NI}:
\begin{equation}\label{eq:FEG4NI_Lyapunov_func}
\begin{array}{lcl}
\Lc_k &:= & \frac{ a_k }{ 2 }  \norms{w^k}^2 + b_k \iprods{w^k, x^k - x^0} + \frac{c_k}{2} \norms{u^k - Fx^k}^2,  
\end{array}
\end{equation}
where  $a_k$, $b_k$, and $c_k$ are given nonnegative parameters, determined later.
Compared to $\Lc_k$ in \eqref{eq:EAG4NI_potential_func} of Section~\ref{sec:EAG4NI}, we drop the last term $\frac{b_0}{\eta}\norms{x^0 - x^{\star} }^2$ here.

\vspace{0.75ex}
\noindent\mytb{$\mathrm{(b)}$~\textit{Key lemma.}}
The following lemma provides a key step to analyze the convergence of \eqref{eq:FEG4NI}.

\begin{lemma}\label{le:FEG4NI_key_estimate1}
For~\eqref{eq:NI}, suppose that $F$ is $L$-Lipschitz continuous, $J_{\eta T}$ is well-defined, and $\Phi$ is $\rho$-co-hypomonotone.
Let $\sets{(x^k, y^k)}$ be generated by \eqref{eq:FEG4NI} starting from $x^0 \in \dom{\Phi}$ such that $u^k$ satisfies \eqref{eq:FEG4NI_u_cond}.
For a fixed $\eta > 0$ and a given $b_k > 0$, suppose that 
\begin{equation}\label{eq:FEG4NI_eta_and_bk}
\arraycolsep=0.2em
\begin{array}{lcl}
\tau_k \in (0, 1), \quad \eta_k := \eta(1-\tau_k), \quad \text{and} \quad b_{k+1} := \frac{b_k}{1-\tau_k}.
\end{array}
\end{equation}
For any  $\mu \in [0, 1]$, $r > 0$,  and $\gamma > 0$, let us denote
\begin{equation}\label{eq:EAG4NI_Lyapunov_coefficients}
\hspace{-3ex}
\arraycolsep=0.1em
\left\{ \begin{array}{lclllcl}
a_{k+1} &:= & \frac{b_{k+1}[ \eta - \beta_k(1 + \gamma \tau_k ) ]}{\tau_k}, & \quad & \hat{a}_{k} & := & \frac{ b_k[ \eta (1-\tau_k)^2 - \beta_k(1 - 2\tau_k + \gamma \tau_k(1 - \tau_k) ) ] }{\tau_k(1-\tau_k)}, \vspace{1ex}\\
\alpha_{k+1} & := & \frac{(1-\mu)\eta b_{k+1} }{\mu\tau_k },  & & \hat{\alpha}_{k+1} & := & \frac{b_{k+1 } [ \beta_k(1 - \gamma + \gamma \tau_k) - 2\rho(1-\tau_k) ]}{\tau_k }, \vspace{1ex}\\
\delta_{k+1} &:= &  \frac{\eta b_{k+1}  [ \mu - (1 + r)L^2\eta^2 ] }{\mu\tau_k } &\quad \text{and}{~} \quad & c_k &:= & \frac{ b_k [ (1 + r) L^2 \eta^3 \gamma (1-\tau_k)^2 + r \mu \beta_k ]}{r\mu\gamma \tau_k(1-\tau_k)}.
\end{array}\right.
\hspace{-6ex}
\end{equation}
Then the function
\begin{equation}\label{eq:FEG4NI_Vhat_func}
\hspace{-2ex}
\arraycolsep=0.2em
\begin{array}{lcl}
\Pc_k & := &  \frac{ \hat{a}_k }{2} \norms{w^k}^2 +  b_k\iprods{w^k, x^k - x^0} +  \frac{ c_k }{ 2 } \norms{Fx^k - u^k}^2
\end{array}
\hspace{-2ex}
\end{equation}
satisfies the following property:
\begin{equation}\label{eq:FEG4NI_key_estimate1}
\hspace{-2ex}
\arraycolsep=0.2em
\begin{array}{lcl}
\Pc_k & \geq & \frac{ a_{k+1} }{ 2 } \norms{w^{k+1}}^2 + b_{k+1}\iprods{w^{k+1}, x^{k+1} - x^0}  +  \frac{ \alpha_{k+1} }{ 2}  \norms{w^{k+1} - \hat{w}^{k+1}}^2   \vspace{1ex}\\
&& + {~}  \frac{ \hat{\alpha}_{k+1} }{ 2}  \norms{w^{k+1} - w^k}^2 +   \frac{\delta_{k+1} }{2} \norms{\hat{w}^{k+1} - (1-\tau_k)w^k}^2.
\end{array}
\hspace{-2ex}
\end{equation}
\end{lemma}

\noindent\mytb{$\mathrm{(c)}$~\textit{Descent property.}}
Next, we establish a descent property of $\Lc_k$ in \eqref{eq:FEG4NI_Lyapunov_func}.

\begin{lemma}\label{le:FEG4NI_key_est2}
Under the same settings as in Lemma~\ref{le:FEG4NI_key_estimate1}, assume that $a_k$, $\hat{a}_k$, $b_k$, $\alpha_k$, $\hat{\alpha}_k$, $c_k$, and $\delta_k$  given in \eqref{eq:EAG4NI_Lyapunov_coefficients} satisfy:
\begin{equation}\label{eq:FEG4NI_key_est2_cond}
\arraycolsep=0.2em
\begin{array}{lcl}
 a_k \geq \hat{a}_k \geq 0, \quad \alpha_k \geq \kappa c_k, \quad  \text{and} \quad \hat{\alpha}_k \geq \hat{\kappa} c_k.
\end{array}
\end{equation}
Then, $\Lc_k$ defined by \eqref{eq:FEG4NI_Lyapunov_func} satisfies the following inequality:
\begin{equation}\label{eq:FEG4NI_key_est2}
\arraycolsep=0.2em
\begin{array}{lcl}
\Lc_k \geq \Lc_{k+1} + \frac{\delta_{k+1} }{2}\norms{ \hat{w}^{k+1} - (1-\tau_k) w^k }^2.
\end{array}
\end{equation}
\end{lemma}

\begin{lemma}\label{le:FEG4NI_choice_of_params}
Under the same settings as in Lemma~\ref{le:FEG4NI_key_est2}, suppose that for a fixed $\nu > 2$, $\tau_k$ and $\beta_k$ are chosen as follows:
\begin{equation}\label{eq:FEG4NI_choice_of_params}
\arraycolsep=0.2em
\begin{array}{lcl}
\tau_k := \frac{1}{k + \nu} \quad \text{and} \quad { \beta_k := \beta (1 - \tau_k)},
\end{array}
\end{equation}
where $\beta$ satisfies three conditions for given $r > 0$, $\gamma > 0$, and $\mu \in [0, 1]$:
\begin{equation}\label{eq:FEG4NI_choice_of_beta}
\arraycolsep=0.2em
\left\{\begin{array}{lcl}
\beta &\geq & \frac{(\nu-2)\gamma} {(\nu-2)\gamma(1-\gamma) - \nu \hat{\kappa} } \Big[ \frac{ (\nu -1)  (1 + r)   \hat{\kappa} L^2 \eta^3 }{(\nu - 2) r \mu} + 2\rho \Big], \vspace{1ex}\\
\beta  & \leq & \frac{(\nu - 1) \gamma \eta }{\nu \kappa\mu } \Big[ 1 - \mu  - \frac{ (1 + r)  \kappa  L^2 \eta^2}{r} \Big], \vspace{1ex}\\
\beta & \leq & \frac{\nu \eta}{\nu + \gamma}. 
\end{array}\right.
\end{equation}
provided that the right-hand sides of these conditions are well-defined.

Suppose further that \eqref{eq:FEG4NI_choice_of_beta} 
holds for a given choice of $r$, $\mu$, and $\gamma$.
Then, the parameters $\tau_k$, $\eta$, $\hat{\eta}_k$, and $\beta_k$ satisfy \eqref{eq:FEG4NI_key_est2_cond}, and thus \eqref{eq:FEG4NI_key_est2} still holds.
\end{lemma}

\beforesubsec
\subsection{\mytb{Convergence Guarantees of \eqref{eq:FEG4NI} and Its Special Cases} }\label{subsec:FEG4NI_key_estimates}
\aftersubsec
\noindent\textbf{$\mathrm{(a)}$~\textit{The convergence of \eqref{eq:FEG4NI}.}}
Now, we state the main convergence result of \eqref{eq:FEG4NI} for a general direction $u^k$ satisfying the condition \eqref{eq:FEG4NI_u_cond}.

\begin{theorem}\label{th:FEG4NI_convergence}
For~\eqref{eq:NI}, suppose that $\zer{\Phi} \neq \emptyset$, $F$ is $L$-Lipschitz continuous, $\Phi$ is $\rho$-co-hypomonotone,  and $J_{\eta T}$ is well-defined and single-valued.
Let $\sets{(x^k, y^k)}$ be generated by \eqref{eq:FEG4NI} such that $u^k$ satisfies \eqref{eq:FEG4NI_u_cond}.

Suppose further that  $\tau_k$, $\hat{\eta}_k$, and $\beta_k$ are updated as follows:
\begin{equation}\label{eq:FEG4NI_param_update}
\arraycolsep=0.2em
\begin{array}{ll}
& \tau_k    :=  \frac{1}{k + \nu}, \quad \hat{\eta}_k := \eta(1-\tau_k), \quad \text{and} \quad \beta_k := \beta (1 - \tau_k),
\end{array}
\end{equation}
where $\nu > 2$ is given, and $\eta$ and $\beta$ are chosen as follows:
\begin{compactitem}
\item[$\mathrm{(i)}$] If $\kappa = 0$ and $\hat{\kappa} = 0$ in \eqref{eq:FEG4NI_u_cond}, and $2L\rho \leq 1$, then we choose 
\begin{equation*}
\arraycolsep=0.2em
\begin{array}{lcl}
0 <  \eta \leq \frac{1}{L} \quad \text{and} \quad 2\rho \leq \beta < \eta.
\end{array}
\end{equation*}
\item[$\mathrm{(ii)}$] If $\kappa = 0$ and $ 0 < \hat{\kappa} < \frac{\nu - 2}{4\nu}$ in \eqref{eq:FEG4NI_u_cond}, and  $L\rho < \frac{(\nu - 2  - 4\nu \hat{\kappa})\sigma_1}{12(\nu - 2)}$, then we choose 
\begin{equation*}
\arraycolsep=0.2em
\begin{array}{lcl}
0 < \eta \leq \frac{\sigma_1}{L} \quad \text{and} \quad \frac{2(\nu - 2)}{\nu - 2 - 4\nu \hat{\kappa}}\big[ \frac{ 2(\nu - 1) \hat{\kappa} \eta \sigma_1^2}{\nu - 2} + 2\rho \big] \leq \beta < \frac{2\eta}{3},
\end{array}
\end{equation*}
where $\sigma_1^2 := \min\left\{\frac{1}{2}, \frac{\nu - 2  - 4\nu \hat{\kappa}}{12(\nu - 1)\hat{\kappa}} \right\}$.
\item[$\mathrm{(iii)}$] If $\kappa > 0$ and $\hat{\kappa} = 0$ in \eqref{eq:FEG4NI_u_cond}, and $L \rho \leq \frac{\nu - 1}{32\nu \sqrt{\kappa + 1}}$, then we choose 
\begin{equation*}
\arraycolsep=0.2em
\begin{array}{lcl}
0 <  \eta \leq \frac{1}{2L\sqrt{\kappa+1}} \quad \text{and} \quad 4\rho \leq \beta \leq \frac{(\nu - 1)\eta}{4\nu}.
\end{array}
\end{equation*}
\item[$\mathrm{(iv)}$] If $\kappa > 0$ and $0 < \hat{\kappa} <  \frac{\nu - 2}{4\nu }$ in \eqref{eq:FEG4NI_u_cond}, and $L\rho \leq \frac{(\nu - 1)(\nu - 2 - 4\nu \hat{ \kappa})\sigma_2}{32\nu(\nu-2)}$, then we choose 
\begin{equation*}
\arraycolsep=0.2em
\begin{array}{lcl}
0 <  \eta \leq \frac{1}{2L\sqrt{\kappa+1}} \quad \text{and} \quad \frac{2(\nu - 2)}{\nu - 2 - 4\nu \hat{\kappa}}\big[ \frac{ 2(\nu - 1) (\kappa + 1) \hat{\kappa}\eta \sigma_2^2 }{\nu - 2} + 2\rho \big] \leq \beta \leq  \frac{(\nu - 1)  \eta }{4\nu},
\end{array}
\end{equation*}
where $\sigma_2^2 := \frac{1}{\kappa + 1}\min\left\{\frac{1}{2}, \frac{\nu - 2 - 4\nu \hat{ \kappa}}{32\nu \hat{ \kappa}} \right\}$.
\end{compactitem}
Then, in all cases, we have the following bound:
\begin{equation}\label{eq:FEG4NI_convergence}
\norms{ Fx^k + \xi^k }^2  \leq   \norms{ Fx^k + \xi^k }^2 + \psi \norms{Fx^k - u^k}^2 \leq    \frac{\Rc_0^2}{(k + \nu  - 1)^2},
\end{equation}
where $\psi :=  0$ if $\kappa = \hat{\kappa} = 0$, and $\psi := \frac{8(L^2\eta^3 + \beta)}{2\eta - 3\beta}$, otherwise;  $x^{\star} \in \zer{\Phi}$, and 
\begin{equation*}
\arraycolsep=0.0em
\Rc_0^2 := \left\{\begin{array}{lll}
& \frac{ 4(\nu - 1) }{(\eta - \beta)^2  } \norms{x^0 - x^{\star}}^2 +  \frac{4(\nu - 1)[\eta(\nu - 1) - \beta(\nu - 2)] }{\eta - \beta}\norms{Fx^0 + \xi^0}^2, &{\!\!\!\!}\textrm{for Case $\mathrm{(i)}$}, \vspace{1ex}\\
& \frac{ 8(\nu - 1) }{(2\eta - 3\beta)^2  } \norms{x^0 - x^{\star}}^2 +  \frac{4[2\eta(\nu - 1)^2 - \beta(\nu -2)(2\nu - 1)] }{2\eta - 3\beta}\norms{Fx^0 + \xi^0}^2, &~ \textrm{otherwise}.
\end{array}\right.
\end{equation*}
\end{theorem}

\begin{remark}\label{re:FEG4NI_remark1}
Theorem~\ref{th:FEG4NI_convergence} proves the $\BigO{1/k^2}$-last-iterate convergence rate of \eqref{eq:FEG4NI} under the condition \eqref{eq:FEG4NI_u_cond}.
While the parameter $\kappa > 0$ is arbitrary,  $\hat{\kappa}$ must be in a certain range $0 \leq \hat{\kappa} < \frac{\nu - 2}{4\nu}$.
Letting $\nu \to +\infty$, we get $0 \leq \hat{\kappa} < \frac{1}{4}$.
Moreover, we also obtain $L\rho < \frac{1}{ 32\sqrt{2(\kappa + 1)} } \cdot \min\big\{ 1, \frac{1}{4\sqrt{\hat{\kappa}}} \big\}$ when $\kappa > 0$.
\end{remark}

\noindent\textbf{$\mathrm{(b)}$~\textit{Two special instances.}}
Now, we consider two special cases of \eqref{eq:FEG4NI} when $u^k := Fx^k$, corresponding to the extra-anchored-gradient (also called fast extragradient) method in \cite{lee2021fast,yoon2021accelerated} and $u^k := Fy^{k-1}$, corresponding to the past extra-anchored gradient method in \cite{tran2021halpern}. 

\begin{corollary}\label{co:FEG4NI_variant1}
For~\eqref{eq:NI}, suppose that $\zer{\Phi} \neq \emptyset$, $F$ is $L$-Lipschitz continuous,  $J_{\eta T}$ is singled-valued and well-defined, and $\Phi$ is $\rho$-co-hypomonotone.
Let $\sets{(x^k, y^k)}$ be generated by \eqref{eq:FEG4NI} starting from $x^0 \in \dom{\Phi}$ and using $u^k := Fx^k$.
Suppose further that $2L\rho \le 1$, and for any fixed $\nu > 2$, $\eta$ and $\beta$ are respectively chosen such that $ \eta \in (0,  \frac{1}{L}]$ and $\beta \in [2\rho, \eta)$, and 
\begin{equation}\label{eq:FEG4NI_param_update1}
\arraycolsep=0.2em
\begin{array}{ll}
& \tau_k    :=  \frac{1}{k + \nu}, \quad   \hat{\eta}_k := \eta (1-\tau_k), \quad \text{and} \quad  \beta_k  :=   \beta(1-\tau_k).
\end{array}
\end{equation}
Then, the following result holds:
\begin{equation}\label{eq:FEG4NI_convergence1}
\norms{ Fx^k + \xi^k }^2 \leq \frac{4 \hat{\Rc}_0^2}{(\eta - \beta)( k + \nu - 1)^2},
\end{equation}
where $\hat{\Rc}_0^2 := \frac{\nu - 1}{\eta - \beta}\norms{x^0 - x^{\star}}^2 + \big[ \eta(\nu - 1)^2 - \beta(\nu - 2)(\nu - 1)\big]\norms{Fx^0 + \xi^0}^2$.
\end{corollary}

\begin{proof}
Since $u^k := Fx^k$, \eqref{eq:FEG4NI_u_cond} holds with $\kappa = \hat{\kappa} = 0$.
Applying Theorem~\ref{th:FEG4NI_convergence} with $\kappa = \hat{\kappa} = 0$,  we obtain \eqref{eq:FEG4NI_param_update1} from \eqref{eq:FEG4NI_param_update}, and \eqref{eq:FEG4NI_convergence} reduces to \eqref{eq:FEG4NI_convergence1}.
\Eproof
\end{proof}

\begin{corollary}\label{co:FEG4NI_variant2}
For~\eqref{eq:NI}, suppose that $\zer{\Phi} \neq \emptyset$, $F$ is $L$-Lipschitz continuous,  $J_{\eta T}$ is well-defined, and $\Phi$ is $\rho$-co-hypomonotone.
Let $\sets{(x^k, y^k)}$ be generated by \eqref{eq:FEG4NI} starting from $x^0 \in \dom{\Phi}$ and $y^{-1} := x^0$ using $u^k := Fy^{k-1}$.
Suppose further that for a given $\nu > 2$, we have  $L\rho \le \frac{\nu - 1}{32\sqrt{2}\nu }$, and $\eta$ and $\beta$ are respectively chosen such that $\eta \in \left(0, \frac{1}{2\sqrt{2}L}\right]$ and $\beta \in \left[4\rho, \frac{(\nu - 1)\eta}{4\nu }\right]$, and
\begin{equation}\label{eq:FEG4NI_param_update2}
\arraycolsep=0.2em
\begin{array}{ll}
& \tau_k    :=  \frac{1}{k + \nu}, \quad   \hat{\eta}_k := \eta (1-\tau_k), \quad \text{and} \quad \beta_k := \beta(1 - \tau_k).
\end{array}
\end{equation}
Then, the following result holds:
\begin{equation}\label{eq:FEG4NI_convergence2}
\norms{ Fx^k + \xi^k }^2 \leq \norms{ Fx^k + \xi^k }^2 + \psi \cdot \norms{Fx^k - Fy^{k-1}}^2 \leq  \frac{\tilde{\Rc}_0^2 }{(k + \nu - 1)^2},
\end{equation}
where $\tilde{\Rc}_0^2 := \frac{8(\nu - 1)}{(2\eta - 3\beta)^2} \norms{x^0 - x^{\star}}^2 + \frac{4[2\eta(\nu - 1)^2 - \beta(\nu - 2)(2\nu - 1)]}{ 2\eta - 3\beta} \norms{Fx^0 + \xi^0}^2$ for $x^{\star} \in \zer{\Phi}$ and $\psi :=  \frac{8(L^2\eta^3 + \beta)}{2\eta - 3\beta}$.
\end{corollary}

\begin{proof}
Since $u^k := Fy^{k-1}$, \eqref{eq:FEG4NI_u_cond} holds with $\kappa = 1$ and $\hat{\kappa} = 0$.
Applying Theorem~\ref{th:FEG4NI_convergence} with $\kappa = 1$ and $\hat{\kappa} = 0$, we obtain the update rule \eqref{eq:FEG4NI_param_update2} from \eqref{eq:FEG4NI_param_update}, and \eqref{eq:FEG4NI_convergence} reduces to \eqref{eq:FEG4NI_convergence2}.
\Eproof
\end{proof}

\begin{remark}\label{re:comparision_between_FEG4NI_and_GNEA4NI}
The analysis of Theorem~\ref{th:FEG4NI_convergence} is much more complicated than that of Theorem~\ref{th:EAG4NI_convergence} due to the co-hypomonotonicity of $\Phi$.
In addition, we have not tried to optimally select the parameters in our analysis.
The ranges of $\hat{\kappa}$, $L\rho$, and $\eta$ can be improved by refining our analysis, but we omit it here.
\end{remark}

\beforesec
\section{A Class of Moving Anchor FBFS Methods}\label{sec:DFBFS4NI}
\aftersec
\noindent\textbf{Motivation.} 
The anchored EG-type methods developed in Sections~\ref{sec:EAG4NI} and \ref{sec:FEG4NI} have at least three aspects that can be improved.
\begin{compactitem}
\item First, the fixed anchor point $x^0$ is propagated over all the iterations.
Although the coefficient $\tau_k$ of $x^0$ is decreasing when $k$ is growing, this term is still slowing down the actual performance of the algorithm.
\item Second, the update rule $\tau_k := \frac{1}{k+\nu}$ of $\tau_k$ is tight, and has less freedom to improve  the practical performance of  algorithm.
\item Third, we are unable to prove faster convergence rates (i.e. $\SmallO{1/k}$ rates) and the convergence of the iterate sequences (see \cite{chambolle2015convergence}).
\end{compactitem}
Here, we explore the ideas of restarting strategies in optimization \cite{Odonoghue2012} and recently moving anchor trick in \cite{alcala2023moving} to propose a class of moving anchor FBFS methods for solving \eqref{eq:NI}. 
Our method can resolve the above three aspects.

\beforesubsec
\subsection{\mytb{A Class of Moving Anchor FBFS Methods} }\label{subsec:DFBFS4NI}
\aftersubsec
First, for given $x^k \in \dom{\Phi}$ and $u^k \in \R^p$, we recall the following elements:
\begin{equation}\label{eq:DFEG4NI_w_defs}
w^k := Fx^k + \xi^k, \quad \hat{w}^k := Fy^{k-1} + \xi^k, \quad \text{and} \quad z^k := u^k + \xi^k,
\end{equation}
for some $\xi^k \in Tx^k$.

\vspace{0.75ex}
\noindent\textbf{$\mathrm{(a)}$~\textit{The proposed method.}}
Instead of fixing $x^0$, we update $x^0$ by a forward step similar to those in \cite{alcala2023moving,yuan2024symplectic}, but using the direction $z^k$, not $w^{k+1}$.
More specifically, our algorithm is described as follows.
Starting from $x^0\in\dom{\Phi}$, we set $\bar{x}^0 := x^0$, and at each iteration $k \geq 0$, for given $x^k$ and $u^k$, we update
\begin{equation}\label{eq:DFEG4NI}
\hspace{-3ex}
\arraycolsep=0.1em
\left\{\begin{array}{lcl}
y^k &:= & x^k + \tau_k(\bar{x}^k - x^k) - (\hat{\eta}_k - \beta_k)z^k, \vspace{1ex}\\
x^{k+1} &:= &  x^k + \tau_k(\bar{x}^k - x^k)  - \eta \hat{w}^{k+1} + \beta_kz^k, \vspace{1ex}\\
\bar{x}^{k+1} &:= & \bar{x}^k -  \gamma  z^k,
\end{array}\right.
\hspace{-3ex}
\tag{GFEG$_{+}$}
\end{equation}
where $\tau_k := \frac{1}{t_k} \in (0, 1]$ for $t_k > 1$, $\eta  > 0$,  $\hat{\eta}_k > 0$, $\beta_k \geq 0$, and $\gamma > 0$ are given parameters, determined later.
The direction $u^k$ in $z^k$ from \eqref{eq:DFEG4NI_w_defs} is chosen such that it satisfies the following condition:
\begin{equation}\label{eq:DFEG4NI_u_cond}
\arraycolsep=0.2em
\begin{array}{lcl}
\norms{ Fx^k - u^k }^2  & \leq & \kappa\norms{ Fx^k - Fy^{k-1} }^2 + \hat{\kappa}\eta^2\norms{ \hat{w}^k - (1-\tau_{k-1})w^{k-1} }^2, 
\end{array}
\end{equation}
for given constants $\kappa \geq 0$ and $\hat{\kappa} \geq 0$, $y^{-1} = x^{-1} := x^0$, and $u^0 := Fx^0$.

The scheme \eqref{eq:DFEG4NI} looks similar to \eqref{eq:FEG4NI}, but it has two main differences.
First, $\bar{x}^k$ is updated by $\bar{x}^{k+1} :=  \bar{x}^k -  \gamma  z^k$ instead of fixing at $\bar{x}^k = x^0$.
Second, the parameters $\tau_k$ and $\beta_k$ are updated differently.

\vspace{0.75ex}
\noindent\textbf{$\mathrm{(b)}$~\textit{Three special instances.}}
Similar to \eqref{eq:FEG4NI}, our scheme \eqref{eq:DFEG4NI} covers a class of methods by choosing different direction $u^k$.
Again, we consider at least three special  instances of \eqref{eq:DFEG4NI}.
\begin{compactitem}
\item[$\mathrm{(i)}$] \textbf{Variant 1: Moving anchor FBFS method.} 
If we choose $u^k := Fx^k$, then \eqref{eq:DFEG4NI_u_cond} holds with $\kappa = 0$ and $\hat{\kappa} = 0$.
This variant is similar to the moving anchor extragradient (EAG-V) scheme in \cite{alcala2023moving} when $T = 0$ as well as the symplectic extragradient method in  \cite{yuan2024symplectic}.
However, our last line $\bar{x}^{k+1} := \bar{x}^k - \gamma z^k$ of \eqref{eq:DFEG4NI} is different from $\bar{x}^{k+1} := \bar{x}^k - \gamma w^{k+1}$ in \cite{alcala2023moving,yuan2024symplectic}.
Though this change is not essential in deterministic methods, it is very useful when developing randomized and stochastic variants of \eqref{eq:DFEG4NI}.

\item[$\mathrm{(ii)}$]\textbf{Variant 2: Moving anchor past-FBFS  method.} 
If we choose $u^k := Fy^{k-1}$, then \eqref{eq:DFEG4NI_u_cond} holds with $\kappa = 1$ and $\hat{\kappa} = 0$, and \eqref{eq:DFEG4NI} reduces to
\begin{equation}\label{eq:DFEG4NI_Fyk-1}
\arraycolsep=0.2em
\left\{\begin{array}{lcl}
y^k &:= &   \tau_k\bar{x}^k + (1-\tau_k)x^k  - ( \hat{\eta}_k - \beta_k) \hat{w}^k \vspace{1ex}\\
x^{k+1} &:= &  \tau_k\bar{x}^k + (1-\tau_k)x^k - \eta \hat{w}^{k+1} + \beta_k \hat{w}^k, \vspace{1ex}\\
\bar{x}^{k+1} & := & \bar{x}^k - \gamma\hat{w}^k.
\end{array}\right.
\vspace{-0.5ex}
\end{equation}
This scheme is new.
It can be viewed as a variant of the forward-reflected-backward splitting, or an optimistic gradient method in the literature.
One main advantage here is that it saves one evaluation of $F$ per iteration.

\item[$\mathrm{(iii)}$]\textbf{Variant 3: Generalization.} 
We can construct a generalized direction:
\begin{equation}\label{eq:DFEG4NI_generalized_uk}
\arraycolsep=0.2em
\begin{array}{lcl}
u^k & := & \alpha Fx^k + (1 - \alpha + \rmark{\eta} \hat{\alpha}) Fy^{k-1} \rmark{-} \rmark{\eta}\hat{\alpha}(1-\tau_{k-1})Fx^{k-1} \vspace{1ex}\\
&& + {~} \rmark{\eta}\hat{\alpha}[\xi^k  - (1-\tau_{k-1})\xi^{k-1}],
\end{array}
\vspace{-0.5ex}
\end{equation}
for any given constants $\alpha, \hat{\alpha} \in \R$.

Clearly, $u^k$ is an affine combination of $Fx^k$, $Fy^{k-1}$, $Fx^{k-1}$, $\xi^k$, and $\xi^{k-1}$. 
Moreover, we can verify that $u^k$ satisfies \eqref{eq:DFEG4NI_u_cond} with $\kappa = (1+m)(1-\alpha)^2$ and $\hat{\kappa}  = (1+m^{-1})\hat{\alpha}^2$ by Young's inequality for any $m > 0$.
Generally, $u^k$ requires  $Fx^k$ and/or $Fy^{k-1}$, and thus \eqref{eq:DFEG4NI} needs at most two evaluations of $F$ at each iteration $k$.
\end{compactitem}

\vspace{0.75ex}
\noindent\textbf{$\mathrm{(c)}$~\textit{The implementation of \eqref{eq:DFEG4NI}}.}
If the resolvent $J_{\eta T}$ of $\eta T$ is well-defined and single-valued, then we can rewrite  \eqref{eq:DFEG4NI} equivalently to
\begin{equation}\label{eq:DFEG4NI_impl}
\arraycolsep=0.1em
\left\{\begin{array}{lcl}
y^k &:= &  x^k + \tau_k(\bar{x}^k - x^k) -  ( \hat{\eta}_k - \beta_k)(u^k + \xi^k), \vspace{1ex}\\
x^{k+1}  & := & J_{\eta T}\left( y^k - \eta Fy^k  +  \hat{\eta}_k (u^k + \xi^k)  \right), \vspace{1ex}\\
\xi^{k+1} & := &  \frac{1}{\eta}\big[ y^k +\hat{\eta}_k (u^k + \xi^k) - x^{k+1} \big] - Fy^k, \vspace{1ex}\\
\bar{x}^{k+1} &:= & \bar{x}^k - \gamma(u^k + \xi^k),
\end{array}\right.
\end{equation}
where $\xi^0 \in Tx^0$ is arbitrary given, $\bar{x}^0 := x^0$, and $u^0 := Fx^0$.
In fact, we only require $J_{\eta T}$ to be well-defined and not necessarily single-valued.

\beforesubsec
\subsection{\mytb{Key Estimates for Convergence Analysis}}\label{subsec:DFEG4NI_key_estimates}
\aftersubsec
\noindent\mytb{$\mathrm{(a)}$~\textit{Lyapunov function.}}
To establish the convergence of \eqref{eq:DFEG4NI}, we define the following functions:
\begin{equation}\label{eq:DFEG4NI_potential_func}
\arraycolsep=0.2em
\begin{array}{lcl}
\Vc_k & := &  \frac{a_k}{2} \norms{w^k}^2 + t_{k-1}\iprods{w^k, x^k - \bar{x}^k} + \frac{(1-\mu)}{2\gamma}\norms{\bar{x}^k - x^{\star}}^2,  \vspace{1ex}\\
\Lc_k &:= & \Vc_k + \frac{\lambda \eta t_k^2}{2} \norms{u^k - Fx^k}^2,
\end{array}
\end{equation}
where $x^{\star} \in \zer{\Phi}$, $a_k$, $t_k$, and $\lambda$  are nonnegative parameters, determined later.
Note that $\Vc_k$ is similar to the one in \eqref{eq:DFEG4NI_potential_func}, but $x^0$ is replaced by $\bar{x}^k$.
Again, since we study the generalized scheme \eqref{eq:DFEG4NI}, we need to add one more term $\frac{\lambda\eta t_k^2}{2}\norms{u^k - Fx^k}^2$ to $\Vc_k$ to handle the mismatch between $u^k$ and $Fx^k$.

\vspace{0.75ex}
\noindent\mytb{$\mathrm{(b)}$~\textit{Descent property and boundedness.}}
First, we state the following key lemma for our convergence analysis.

\begin{lemma}\label{le:DFEG4NI_descent_pro}
For~\eqref{eq:NI}, suppose that $F$ is $L$-Lipschitz continuous and $\Phi$ is $\rho$-co-hypomonotone.
Given $\mu \in [0, 1)$ and $c_1 \in \big[0, \frac{1}{2}\big)$, let $\sets{(x^k, \bar{x}^k, z^k)}$ be generated by  \eqref{eq:DFEG4NI} starting from $x^0 \in \dom{\Phi}$ such that $u^k$ satisfies \eqref{eq:DFEG4NI_u_cond}.
Suppose that the parameters are updated as $\tau_k := \frac{1}{t_k}$ and
\begin{equation}\label{eq:DFEG4NI_param}
\arraycolsep=0.2em
\begin{array}{lcl}
t_k := t_{k-1} + \mu, \quad \hat{\eta}_k := \eta(1-\tau_k), \quad \textrm{and} \quad \beta_k := - \gamma \tau_k + \frac{2\rho(1-\tau_k)}{1-2c_1}.
\end{array}
\end{equation}
Given any $\omega \geq 0$ and $c > 0$, define $M^2 := (1+\omega)(1+c)L^2$ and $e^k := u^k - Fx^k$.
Then $\Vc_k$ defined by \eqref{eq:DFEG4NI_potential_func} satisfies 
\begin{equation}\label{eq:DFEG4NI_descent_property}
\hspace{-0.25ex}
\arraycolsep=0.2em
\begin{array}{lcl}
\Vc_k & \geq & \Vc_{k+1} +  \frac{(a_k - \hat{a}_k)}{2}\norms{w^k}^2  + \frac{\omega\eta  t_k^2}{2} \norms{w^{k+1} - \hat{w}^{k+1}}^2 \vspace{1ex}\\
&& + {~}  \frac{\rmark{\eta}(1 - M^2\eta^2)}{2}\norms{t_k\hat{w}^{k+1} - (t_k-1)w^k }^2  +  (1-\mu)\iprods{w^k, x^k - x^{\star} } \vspace{1ex}\\
&&  + {~} (1-\mu)\iprods{e^k, \bar{x}^k - x^{\star}} -  \frac{b_k}{2} \norms{e^k}^2,
\end{array}
\hspace{-2ex}
\end{equation}
where, for any $c_2 > 0$ and $\hat{c} > 0$, $a_{k+1}$, $\hat{a}_k$, and $b_k$ are respectively given by
\begin{equation}\label{eq:DFEG4NI_coefs}
\arraycolsep=0.1em
\left\{\begin{array}{lcl}
a_{k+1} &:= &  t_k\big[  \eta t_k - \frac{2\rho (t_k-1)}{1- 2c_1} \big], \vspace{1ex}\\
\hat{a}_k &:= &  \eta(t_k-1)^2 + 2\gamma(1+c_2)(t_k - 1) - \frac{2\rho(t_k-1)[ t_k - 2(1+c_2)]}{1-2c_1} \vspace{1ex}\\
&& + {~}  (1+\hat{c})(1-\mu)\gamma, \vspace{1ex}\\
b_k &:= & \frac{M^2\eta^3(t_k-1)^2}{c} +   \frac{(t_k-1)}{2(1-2c_1)}\big[ \frac{2\rho t_k}{c_1} + \frac{2\rho + \gamma(1-2c_1)}{c_2} \big] + \frac{(1+\hat{c})(1-\mu)\gamma}{\hat{c}}.
\end{array}\right.
\end{equation}
\end{lemma}

We also need to show that $\Vc_k$ in  \eqref{eq:DFEG4NI_potential_func} is bounded from bellow.

\begin{lemma}\label{le:DFEG4NI_Vk_lowerbound}
Let $\Vc_k$ be defined by \eqref{eq:DFEG4NI_potential_func} and $\Phi$ be $\rho$-co-hypomonotone.
Then, for any $\tilde{c} > 0$, we have
\begin{equation}\label{eq:DFEG4NI_Vk_lower_bound}
\arraycolsep=0.2em
\begin{array}{lcl}
\Vc_k & \geq & \frac{[(1-\mu)\tilde{c}\eta - \gamma]}{2\tilde{c}\gamma \eta }\norms{\bar{x}^k - x^{\star}}^2 + \frac{(a_k - \tilde{c}\eta t_{k-1}^2 - 2\rho t_{k-1})}{2}\norms{w^k}^2.
\end{array}
\end{equation}
\end{lemma}

\beforesubsec
\subsection{\mytb{Convergence Analysis for The General Case}}\label{subsec:DFEG4NI_monotone_case}
\aftersubsec
\noindent\textbf{$\mathrm{(a)}$ The $\BigOs{1/k}$ convergence rates.}
First, we establish the convergence of \eqref{eq:DFEG4NI} for the general case when $u^k \neq Fx^k$.
Then, we investigate a special case $u^k := Fx^k$ to expand the range of $L\rho$.

\begin{theorem}\label{th:DFEG4NI_convergence1}
For~\eqref{eq:NI}, suppose that $\zer{\Phi} \neq \emptyset$, $F$ is $L$-Lipschitz continuous, and $\Phi$ is $\rho$-co-hypomonotone.
Let $\sets{(x^k, y^k)}$ be generated by \eqref{eq:DFEG4NI} such that $u^k$ satisfies \eqref{eq:DFEG4NI_u_cond}.
Suppose further that  $\tau_k$, $\hat{\eta}_k$, and $\beta_k$ are updated as
\begin{equation}\label{eq:DFEG4NI_param_update1}
\arraycolsep=0.1em
\begin{array}{ll}
& \tau_k    := \frac{1}{t_k} = \frac{1}{\mu(k+r)}, \quad \hat{\eta}_k := \eta(1-\tau_k), \ \  \text{and} \ \ \beta_k :=  - \gamma \tau_k + 4\rho (1 - \tau_k),
\end{array}
\end{equation}
where $r \geq 1 + \frac{2}{\mu}$ and $\mu \in (0, \frac{1}{2})$ are given, and $\eta$ and $\gamma$ are chosen such that
\begin{equation}\label{eq:DFEG4NI_eta_choice1}
\arraycolsep=0.2em
\begin{array}{ll}
8\rho < \eta \leq \bar{\eta} := \frac{1}{\sqrt{2(1 + 4\kappa)L^2 + 4\hat{\kappa}}} \quad \textrm{and} \quad \gamma := \frac{(1-2\mu)\eta}{6},
\end{array}
\end{equation}
provided that $128[ (1 + 4\kappa)L^2 + 2\hat{\kappa} ]\rho^2 < 1$.
Then, we have
\begin{equation}\label{eq:DFEG4NI_th41_convergence1}
\norms{ Fx^k + \xi^k }^2  \leq  \frac{ 6 e^{\frac{\theta}{r}} \cdot \Rc_0^2}{ \eta \mu^2(k + r - 1)^2},
\end{equation}
where $\theta := \frac{(1-2\mu)(1-\mu)^2}{2\rmark{\nu}(2-\mu)\mu^2}$ and $\Rc_0^2 := \frac{\eta\mu^2(r-1)^2}{2}\norms{Fx^0 + \xi^0}^2 + \frac{3(1-\mu)}{(1-2\mu)\eta}\norms{x^0 - x^{\star}}^2$.
Moreover, we also have
\begin{equation}\label{eq:DFEG4NI_th41_summable1}
\arraycolsep=0.2em
\begin{array}{lcl}
\sum_{k=0}^{\infty} [ \mu(k+r) -1] \norms{Fx^k + \xi^k}^2 & \leq & \frac{4 \Rc_0^2}{(3-2\mu)(\eta - 8\rho)}\big(  1 + \frac{\theta}{r} e^{\frac{\theta}{r}} \big).
\end{array}
\end{equation}
\end{theorem}

The choice of $\eta$ and $\gamma$ and the update rule of $\beta_k$ can be improved to get a larger range by optimizing the constants in Lemma~\ref{le:DFEG4NI_descent_pro}.
In Theorem~\ref{th:DFEG4NI_convergence1}, we only provide one choice of these constants, which may not be the best one.
For instance, if $u^k := Fx^k$, then we have $\kappa = \hat{\kappa} = 0$, and the range of $L\rho$ in Theorem \ref{th:DFEG4NI_convergence1} reduces to $8\sqrt{2}L\rho < 1$, which is much worse than $2L\rho < 1$ stated in Theorem~\ref{th:DFEG4NI_convergence2} below.
Though Theorem~\ref{th:DFEG4NI_convergence1} provides a unified result for all choices of $u^k$ satisfying \eqref{eq:DFEG4NI_u_cond}, for specific choice of $u^k$, it is still possible to improve the range of parameters.
Let us consider the following examples.
\begin{compactitem}
\item[$\mathrm{(i)}$] If $u^k := Fx^k$, then we separate its analysis in Subsection~\ref{subsec:DFEG4NI_case2}.
\item[$\mathrm{(ii)}$] If $u^k := Fy^{k-1}$, then $\kappa = 1$ and $\hat{\kappa} = 0$, the range of $L\rho$ reduces to $8\sqrt{10}L\rho < 1$.
Certainly, this range can be further improved by refining our analysis in Theorem~\ref{th:DFEG4NI_convergence1}.
\item[$\mathrm{(iii)}$] If $\rho = 0$, i.e. $\Phi$ is monotone, then we can improve the range of $\eta$ and $\gamma$ respectively to, e.g., $0 < \eta \leq \frac{1}{\sqrt{2(1+3\kappa)L^2 + 3\hat{\kappa}}}$ and $\gamma := \frac{(1-\mu)\eta}{3}$.
\end{compactitem}

\vspace{0.5ex}
\noindent\mytb{$\mathrm{(b)}$ The $\SmallOs{1/k^2}$ Convergence Rates and Convergence of Iterates.} 
Now, we prove  $\SmallOs{1/k^2}$ convergence rates of \eqref{eq:DFEG4NI} and the convergence of iterate sequences to a solution of \eqref{eq:NI}.

\begin{theorem}\label{th:DFEG4NI_small_o_rates}
Under the conditions and settings of Theorem~\ref{th:DFEG4NI_convergence1}, we additionally assume that $\eta < \bar{\eta}$.
Then
\begin{equation}\label{eq:DFEG4NI_small_o_rates_summable_results}
\arraycolsep=0.2em
\begin{array}{lcl}
\sum_{k=0}^{\infty} (k+r)^2 \norms{Fx^{k+1} - Fy^k}^2 & < & + \infty, \vspace{1ex}\\
\sum_{k=0}^{\infty} (k+r)^2 \norms{u^k - Fx^k}^2 & < & + \infty, \vspace{1ex}\\
\sum_{k=0}^{\infty} (k+r) \norms{y^k - x^k}^2 & < & + \infty, \vspace{1ex}\\
\sum_{k=0}^{\infty} (k+r) \norms{x^{k+1} - x^k}^2 & < & + \infty.
\end{array}
\end{equation}
We also have
\begin{equation}\label{eq:DFEG4NI_small_o_rates}
\arraycolsep=0.2em
\begin{array}{lcl}
\lim_{k\to\infty} k^2\norms{Fx^k + \xi^k}^2 & = & 0, \vspace{1ex}\\
\rmark{\lim_{k\to\infty} k^2\norms{y^k - x^k}^2} & = & 0, \vspace{1ex}\\
\rmark{\lim_{k\to\infty} k^2\norms{x^{k+1} - x^k}^2} & = & 0.
\end{array}
\end{equation}
Moreover, if $T$ is closed, then all three sequences $\sets{x^k}$, $\sets{y^k}$, and $\sets{\bar{x}^k}$ converge to a solution $x^{\star} \in \zer{\Phi}$ of \eqref{eq:NI}.
\end{theorem}

The limits in \eqref{eq:DFEG4NI_small_o_rates} show that $\norms{Fx^k + \xi^k}^2 = \SmallOs{1/k^2}$, $\norms{x^{k+1} - x^k}^2 =  \SmallOs{1/k^2}$, and $\norms{y^k - x^k}^2 =  \SmallOs{1/k^2}$, stating a faster convergence rate than $\BigOs{1/k^2}$ in Theorem~\ref{th:DFEG4NI_convergence1}.
However, this faster rate is achieved when $k$ is sufficiently large, while the $\BigOs{1/k^2}$ bound holds for all $k \geq 0$.

Since $F$ is continuous, the closedness of $T$ implies the closedness of $\gra{\Phi}$, which  guarantees that if $(x^k, v^k) \in \gra{\Phi}$ and $(x^k, v^k)$ converges to $(x, v)$, then $(x, v) \in \gra{\Phi}$.
In particular, this condition holds if $T$ is maximally monotone \cite[Proposition 20.38]{Bauschke2011}.

\beforesubsec
\subsection{\mytb{Special Case -- Moving Anchor Extragradient Method}}\label{subsec:DFEG4NI_case2}
\aftersubsec
The choice $u^k := Fx^k$ leads to $e^k := u^k - Fx^k = 0$, making our analysis simpler.
Moreover, it allows us to expand the range of $L\rho$ to $2L\rho < 1$.
Hence, we state it separately in this subsection.
This variant is similar to the ones in \cite{alcala2023moving,yuan2024symplectic}, but the update of $\bar{x}^k$ uses $w^k$ instead of $w^{k+1}$ as in  \cite{alcala2023moving,yuan2024symplectic}.

\begin{theorem}\label{th:DFEG4NI_convergence2}
For~\eqref{eq:NI}, suppose that $\zer{\Phi} \neq \emptyset$, $F$ is $L$-Lipschitz continuous, and $\Phi$ is  $\rho$-co-hypomonotone such that $2L\rho < 1$.
Let $\sets{(x^k, y^k, \bar{x}^k)}$ be generated by \eqref{eq:DFEG4NI} using $u^k := Fx^k$ and
\begin{equation}\label{eq:DFEG4NI_param_update2}
\arraycolsep=0.2em
\begin{array}{ll}
& \tau_k    := \frac{1}{t_k} = \frac{1}{\mu(k+r)}, \quad \hat{\eta}_k := \eta(1-\tau_k), \ \ \text{and} \ \ \beta_k :=   - \gamma\tau_k + 2\rho(1-\tau_k),
\end{array}
\end{equation}
where $r \geq \frac{1}{\mu}$ and $\mu \in (0, 1)$ are given, and $\eta$ and $\gamma$ are chosen such that
\begin{equation}\label{eq:DFEG4NI_eta_choice2}
\arraycolsep=0.2em
\begin{array}{ll}
2\rho < \eta \leq \frac{1}{L} \quad \textrm{and} \quad 0 < \gamma < (1-\mu)(\eta - 2\rho).
\end{array}
\end{equation}
Then, we have
\begin{equation}\label{eq:DFEG4NI_convergence2}
\arraycolsep=0.2em
\begin{array}{ll}
& \norms{ Fx^k + \xi^k }^2   \leq    \dfrac{\rmark{2}(1-\mu)\Rc_0^2}{[(1-\mu)(\eta-2\rho) - \gamma]\mu^2(k + r - 1)^2}, \vspace{1ex}\\
&{\displaystyle\sum_{k=0}^{\infty}} \rmark{[2\mu(k+r) - 1 - \mu]} \norms{Fx^k + \xi^k}^2  \leq   \dfrac{\rmark{2}\Rc_0^2}{ \rmark{(1-\mu)(\eta-2\rho) - \gamma}},
\end{array}
\end{equation}
where $\Rc_0^2 := \frac{\mu(r-1)[ \mu(r-1)(\eta - 2\rho) + 2\rho]}{2}\norms{Fx^0 + \xi^0}^2 + \frac{1-\mu}{2\gamma}\norms{x^0 - x^{\star}}^2$.
\end{theorem}

Similar to Theorem~\ref{th:DFEG4NI_small_o_rates}, we can prove the  summable results in \eqref{eq:DFEG4NI_small_o_rates_summable_results}, the $\SmallO{1/k^2}$ convergence rates in \eqref{eq:DFEG4NI_small_o_rates}, and the convergence of the iterate sequences $\sets{x^k}$, $\sets{y^k}$, and $\sets{\bar{x}^k}$ for this special case.
However, we omit the details since it repeats the proof of Theorem~\ref{th:DFEG4NI_small_o_rates} using Theorem~\ref{th:DFEG4NI_convergence2}.

\beforesec
\section{Nesterov's Accelerated Extragradient Methods for Inclusions \eqref{eq:NI}}\label{sec:AEG4NI}
\aftersec
Sections~\ref{sec:EAG4NI}, \ref{sec:FEG4NI}, and \ref{sec:DFBFS4NI} have studied Halpern's accelerated (or anchored) methods.
Now, we will focus on Nesterov's accelerated algorithms in the next two sections.
In this section, we unify and generalize Nesterov's accelerated extragradient method from \cite{tran2023extragradient} to a broader class of algorithms for solving \eqref{eq:NI}.

\beforesubsec
\subsection{\textbf{A Class of Nesterov's Accelerated Extragradient Algorithms}}
\aftersubsec
\noindent\textbf{$\mathrm{(a)}$~\textit{The proposed method.}}
First, for given $x^k \in \dom{\Phi}$, $u^k \in \R^p$, and $\xi^k \in Tx^k$, we recall the following quantities from \eqref{eq:FEG4NI_w_defs}:
\begin{equation}\label{eq:NesEAG4NI_w_defs}
w^k := Fx^k + \xi^k, \quad \hat{w}^k := Fy^{k-1} + \xi^k, \quad \text{and} \quad z^k := u^k + \xi^k.
\end{equation}
Next, starting from $x^0 \in \dom{\Phi}$, let $u^0 := Fx^0$ and $y^{-1} = \hat{x}^{-1} = \hat{x}^0 := x^0$.
At each iteration $k \ge 0$, we update the sequence $\sets{(x^k, \hat{x}^k,  y^k)}$ as 
\begin{equation}\label{eq:NesEAG4NI}
\arraycolsep=0.2em
\left\{\begin{array}{lcl}
x^{k+1} & :=  & y^k - \eta ( \hat{w}^{k+1} - \gamma_k z^k ), \vspace{1ex}\\
\hat{x}^{k+1} &:= & x^{k+1} - \lambda z^{k+1}, \vspace{1ex}\\
y^{k+1} &:= & \hat{x}^{k+1} + \theta_k( \hat{x}^{k+1} - \hat{x}^k )  +  \nu_k(y^k - \hat{x}^{k+1} ),
\end{array}\right.
\tag{GAEG}
\end{equation}
where $\eta > 0$, $\gamma_k \in [0, 1]$, $\lambda > 0$, $\theta_k  \in (0, 1)$, and $\nu_k \geq 0$ are given parameters (determined later), and $u^k \in \mathbb R^p$ is a search direction satisfying
\begin{equation}\label{eq:NesEAG4NI_u_cond}
\norms{Fx^k - u^k}^2 \le \kappa\norms{Fx^k - Fy^{k-1}}^2 + \hat{\kappa}\norms{w^k - w^{k-1}}^2,
\end{equation}
for given constants $\kappa \geq 0$ and $\hat{\kappa} \ge 0$, and $w^{-1} := w^0$.
In particular, if $T = 0$, then $w^k - w^{k-1} = Fx^k - Fx^{k-1}$, and thus $u^k$ depends on $Fx^k$, $Fy^{k-1}$, and $Fx^{k-1}$.
Note that we allow $\kappa \geq 0$ to be arbitrary, but $\hat{\kappa}$ will be chosen in a certain range (\emph{cf.} Theorem~\ref{th:NesEAG4NI_convergence}).

\vspace{0.75ex}
\noindent\textbf{$\mathrm{(b)}$~\textit{Three special instances.}}
As before, our method \eqref{eq:NesEAG4NI} possibly covers a wide class of schemes by instantiating different directions $u^k$, which satisfy \eqref{eq:NesEAG4NI_u_cond}.
We consider at least the following three special choices of $u^k$.
\begin{compactitem}
\item[$\mathrm{(i)}$]\textbf{Variant 1.} 
If we choose $u^k := Fx^k$, then $z^k = w^k$ and \eqref{eq:NesEAG4NI_u_cond} holds with $\kappa = \hat{\kappa} = 0$.
In this case, \eqref{eq:NesEAG4NI} reduces to 
\begin{equation}\label{eq:NesEAG4NI_Fxk}
\arraycolsep=0.2em
\left\{\begin{array}{lcl}
x^{k+1} & :=  & y^k - \eta ( \hat{w}^{k+1} - \gamma_k w^k ), \vspace{0.8ex}\\
\hat{x}^{k+1} &:= & x^{k+1}  - \lambda w^{k+1}, \vspace{0.8ex}\\
y^{k+1} &:= & \hat{x}^{k+1} + \theta_k( \hat{x}^{k+1} - \hat{x}^k )  +  \nu_k(y^k - \hat{x}^{k+1} ).
\end{array}\right.
\end{equation}
Clearly, \eqref{eq:NesEAG4NI_Fxk} is exactly the accelerated extragradient method in \cite{tran2023extragradient}.

\item[$\mathrm{(ii)}$]\textbf{Variant 2.} 
If we choose $u^k := Fy^{k-1}$, then $z^k = \hat{w}^k$ and \eqref{eq:NesEAG4NI_u_cond} holds with $\kappa = 1$ and $\hat{\kappa} = 0$.
Then, \eqref{eq:NesEAG4NI} reduces to 
\begin{equation}\label{eq:NesEAG4NI_Fyk-1}
\arraycolsep=0.2em
\left\{\begin{array}{lcl}
x^{k+1} & :=  & y^k - \eta ( \hat{w}^{k+1}  - \gamma_k  \hat{w}^k ), \vspace{0.8ex}\\
\hat{x}^{k+1} &:= & x^{k+1}  - \lambda \hat{w}^{k+1}, \vspace{0.8ex}\\
y^{k+1} &:= & \hat{x}^{k+1} + \theta_k( \hat{x}^{k+1} - \hat{x}^k )  +  \nu_k(y^k - \hat{x}^{k+1} ).
\end{array}\right.
\end{equation}
This variant seems to be new and it is different from the second algorithm in \cite{tran2023extragradient}.
This method is also different from the one in \cite{sedlmayer2023fast} when specified to solve the monotone \eqref{eq:VIP}.
However, we only obtain $\BigO{1/k}$ convergence rate, while \cite{sedlmayer2023fast} can achieve both $\BigO{1/k}$ and $\SmallO{1/k}$ convergence rates.

\item[$\mathrm{(iii)}$]\textbf{Variant 3: Generalization.} 
Let $v^{k-1} := w^{k-1} - \xi^k$.
We can construct 
\begin{equation*} 
\arraycolsep=0.2em
\begin{array}{lcl}
u^k := \alpha Fx^k + \hat{\alpha} Fy^{k-1} + (1 - \alpha - \hat{\alpha} )v^{k-1} 
\end{array} 
\end{equation*}
as an affine combination of $Fx^k$, $Fy^{k-1}$ and $v^{k-1}$ for given constants $\alpha, \hat{\alpha} \in \R$.
Then, by Young's inequality, for any $m > 0$, one can show that
\begin{equation*} 
\arraycolsep=0.2em
\begin{array}{lcl}
\norms{Fx^k - u^k}^2 &= & \norms{\hat{\alpha}(Fx^k - Fy^{k-1}) + (1 - \alpha - \hat{\alpha} )(Fx^k - v^{k-1})}^2 \vspace{1ex}\\
&\leq & \kappa \norms{Fx^k - Fy^{k-1}}^2 + \hat{\kappa} \norms{w^k - w^{k-1}}^2,
\end{array} 
\end{equation*}
where $\kappa := (1+m)\hat{\alpha}^2$ and $\hat{\kappa}  := (1 \! +\! m^{-1})(1 \! - \! \alpha \! - \! \hat{\alpha})^2$.
Thus $u^k$ satisfies \eqref{eq:NesEAG4NI_u_cond}. 
\end{compactitem}

\vspace{0.75ex}
\noindent\textbf{$\mathrm{(c)}$~\textit{The implementation of \eqref{eq:NesEAG4NI}.}}
To implement \eqref{eq:NesEAG4NI}, we assume that the resolvent $J_{\eta T}$ of $\eta T$ is \textbf{well-defined} (see Section~\ref{sec:EAG4NI}) for any $\eta > 0$.
Generally, we do not require $J_{\eta T}$ to be single-valued, but to avoid complicating our presentation, we again assume that $J_{\eta T}$ is single-valued.

Using the resolvent $J_{\eta T}$, we can rewrite \eqref{eq:NesEAG4NI} equivalently to
\begin{equation}\label{eq:NesEAG4NI_reform}
\arraycolsep=0.2em
\left\{\begin{array}{lcl}
x^{k+1} & :=  & J_{\eta T}\big( y^k - \eta ( Fy^k - \gamma_k (u^k + \xi^k) )  \big), \vspace{0.8ex}\\
\xi^{k+1} &:= & \frac{1}{\eta} (y^k - x^{k+1} ) - Fy^k + \gamma_k (u^k + \xi^k), \vspace{0.8ex}\\
\hat{x}^{k+1} &:= & x^{k+1} - \lambda (u^{k+1} + \xi^{k+1}), \vspace{0.8ex}\\
y^{k+1} &:= & \hat{x}^{k+1} + \theta_k( \hat{x}^{k+1} - \hat{x}^k )  +  \nu_k(y^k - \hat{x}^{k+1} ).
\end{array}\right.
\end{equation}
This scheme requires one evaluation $Fy^k$ of $F$, one evaluation $J_{\eta T}$ of $\eta T$, and one construction of $u^k$.
Under the condition \eqref{eq:NesEAG4NI_u_cond}, constructing $u^k$ may require at most one more evaluation of $F$.
In particular, if $u^k := Fx^k$, then we need one more evaluation $Fx^k$ of $F$.
If $u^k := Fy^{k-1}$, then we can reuse the value $Fy^{k-1}$ evaluated from the previous iteration $k-1$.

\beforesubsec
\subsection{\textbf{Key Estimates for Convergence Analysis}}
\aftersubsec
\noindent\textbf{$\mathrm{(a)}$~\textit{Key lemma.}}
To analyze \eqref{eq:NesEAG4NI}, we first state the following lemma.

\begin{lemma}\label{le:NesEAG4NI_key_estimate1}
For \eqref{eq:NI}, suppose that $\Phi$ is $\rho$-co-hypomonotone, $F$ is $L$-Lipschitz continuous, and $\zer{\Phi} \neq\emptyset$.
Given $\lambda > 0$, $\hat{c} > 0$, $\beta > 0$, $\mu \geq 0$, and $t_0 > 0$, let $\sets{(x^k,  \hat{x}^k, y^k )}$ be generated by \eqref{eq:NesEAG4NI} using the following parameters:
\begin{equation}\label{eq:NesEAG4NI_para_choice}
\begin{array}{ll}
&  \eta := (1+2\hat{c})\lambda + 2\rho + 2\beta + 2\mu, \quad\quad  t_{k+1} := t_k + 1, \vspace{1ex} \\
& \gamma_k := \frac{t_k - 1}{t_k},  \quad  \theta_k := \frac{t_k - 1}{t_{k+1}}, \quad \text{ and } \quad \nu_k := \frac{t_k}{t_{k+1}}.
\end{array}
\end{equation}
For any $x^{\star} \in \zer{\Phi}$, we consider the following function:
\begin{equation}\label{eq:NesEAG4NI_Pk_func}
\begin{array}{lcl}
\Pc_k & := &  \frac{a_k}{2}\norms{w^k}^2 + b_k\iprods{w^k,  \hat{x}^k - y^k} + \norms{\hat{x}^k - x^{\star} + t_k(y^k - \hat{x}^k) }^2,
\end{array}
\end{equation}
where $a_k := \frac{b_k}{t_k}\big[ \lambda t_k +   (1+2\hat{c})\lambda + 2\rho + 2\mu  \big]$ and $b_{k+1} := \frac{b_k}{ \theta_k }$ for $b_k > 0$.

Then, for any $c > 0$ and $M_c := (1+c)L^2$, we have
\begin{equation}\label{eq:NesEAG4NI_key_est1}
\arraycolsep=0.2em
\begin{array}{lcl}
\Pc_k - \Pc_{k+1}  &\geq &    \mu b_k \norms{w^{k+1} - w^k}^2 + \frac{ b_k (1 - M_c\eta^2) }{ 2M_c\eta\gamma_k } \norms{w^{k+1} - \hat{w}^{k+1}}^2 \vspace{1ex}\\
&& - {~}   \frac{ b_k}{ 2 } \big(  \frac{ \lambda }{ \hat{c} } + \frac{\eta^2\gamma_k }{ \beta  } + \frac{ \eta \gamma_k}{c} \big) \norms{u^k - Fx^k}^2 \vspace{1ex}\\
&& - {~}  \frac{ b_k }{2} \big(  \frac{ \lambda }{ \hat{c} }  + \frac{\lambda^2 }{\beta \gamma_k  } \big) \norms{u^{k+1} - Fx^{k+1}}^2.
\end{array} 
\end{equation}
\end{lemma}

\noindent\textbf{$\mathrm{(b)}$~\textit{Lyapunov function and descent property.}} 
Next, for given $\Pc_k$ defined in \eqref{eq:NesEAG4NI} we consider the following Lyapunov function:
\begin{equation}\label{eq:NesEAG4NI_Lyapunov_func}
\begin{array}{lcl}
\Lc_k &  := & \Pc_k + \frac{c_k}{2} \norms{u^k - Fx^k}^2 \vspace{1ex} \\
& = & \frac{a_k}{2}\norms{w^k}^2 +  b_k\iprods{w^k,  \hat{x}^k - y^k} + \norms{\hat{x}^k - x^{\star} + t_k(y^k - \hat{x}^k) }^2 \vspace{1ex}\\
&& + {~} \frac{c_k}{2} \norms{u^k - Fx^k}^2,
\end{array}
\end{equation}
where $a_k := \frac{b_k}{t_k} \left[ \lambda t_k +   (1+2\hat{c})\lambda + 2\rho + 2\mu  \right]$ given in Lemma~\ref{le:NesEAG4NI_key_estimate1}, $b_{k+1} := \frac{b_k}{ \theta_k }$ for $b_k > 0$, and  $c_k :=   b_k \big(  \frac{ \lambda }{ \hat{c} } + \frac{\eta^2\gamma_k }{ \beta  } + \frac{ \eta \gamma_k}{c} \big)$.

Then, we can prove the following descent property of  $\Lc_k$.

\begin{lemma}\label{le:NesEAG4NI_descent_property}
Under the same settings as in Lemma~\ref{le:NesEAG4NI_key_estimate1}, if $u^k$ satisfies \eqref{eq:NesEAG4NI_u_cond} and $t_0 > 1$, then $\Lc_k$ defined by \eqref{eq:NesEAG4NI_Lyapunov_func} satisfies 
\begin{equation}\label{eq:NesEAG4NI_descent_property}
\hspace{-1ex}
\begin{array}{lcl}
\Lc_k - \Lc_{k+1}  &\geq &   \frac{ b_k t_k }{2t_{k-1}} \big[  \frac{  1  }{ M_c\eta} - \eta - \kappa \big( \frac{ 2\lambda }{\hat{c}}  + \frac{\lambda^2 }{\beta   }  +  \frac{\eta^2 }{ \beta  } + \frac{ \eta   }{c  } \big) \big]  \norms{w^{k+1} - \hat{w}^{k+1}}^2 \vspace{1ex}\\
&& + {~}  \frac{b_kt_k}{2t_{k-1}} \big[ \frac{2\mu(t_0-1)}{t_0} -  \hat{\kappa} \big( \frac{ 2\lambda }{\hat{c}} + \frac{\lambda^2 }{\beta   }  +  \frac{\eta^2 }{ \beta  } + \frac{ \eta   }{c  } \big)  \big] \norms{w^{k+1} - w^k}^2.
\end{array}
\end{equation}
\end{lemma}

\vspace{0.75ex}
\noindent\textbf{$\mathrm{(c)}$~\textit{Lower bound of $\Lc_k$}.}
Finally, we can lower bound $\Lc_k$ as follows.

\begin{lemma}\label{le:NesNAEG4NI_lower_bound_Lk}
Under the same settings as in Lemma~\ref{le:NesEAG4NI_descent_property} and $b_0 := \frac{3\lambda t_0(t_0-1)}{2}$,  $\Lc_k$ defined by \eqref{eq:NesEAG4NI_Lyapunov_func} is lower bounded by 
\begin{equation}\label{eq:NesEAG4NI_lower_bound_Lk}
\begin{array}{lcl}
\Lc_k & \geq & \frac{b_k}{2t_k}\big[ \frac{\lambda t_k }{4} + (2\hat{c} - \frac{1}{2})\lambda + 2\mu  \big] \norms{w^k}^2 \vspace{1ex}\\
&& + {~} \frac{b_k}{2t_k}\big[  \frac{\lambda}{\hat{c}} (t_k - 4\hat{c}) + (\frac{\eta}{\beta} + \frac{1}{c})\eta t_{k-1}  \big]\norms{u^k - Fx^k}^2.
\end{array}
\end{equation}
\end{lemma}

\beforesubsec
\subsection{\textbf{Convergence Guarantees of \eqref{eq:NesEAG4NI} and Its Special Cases}}
\aftersubsec
\noindent\textbf{$\mathrm{(a)}$~\textit{The convergence of \eqref{eq:NesEAG4NI}.}}
Given $\kappa \geq 0$, $\sigma := \frac{5}{24}$, and $L$ and $\rho$ such that  $4(1+\sigma)\sqrt{2(5\kappa + 1)} L\rho < 1$, we define 
\begin{equation}\label{eq:NesEAG4NI_lambda_bar}
\hspace{-3ex}
\arraycolsep=0.2em
\begin{array}{ll}
\bar{\lambda} := \frac{2\bar{c} }{ \bar{b} + \sqrt{\bar{b}^2 + 4\bar{c}}}, \ \  \text{where} \ \ \bar{b} :=  \frac{(1+\sigma)(49\kappa + 8)\rho}{29\kappa + 4}  \ \textrm{and} \ \bar{c} := \frac{1 - 32(1+\sigma)^2( 5\kappa + 1) L^2\rho^2}{8L^2(29\kappa + 4)}.
\end{array}
\hspace{-3ex}
\end{equation}
Now, we are ready to prove the convergence of \eqref{eq:NesEAG4NI} as follows.

\begin{theorem}\label{th:NesEAG4NI_convergence}
For \eqref{eq:NI}, suppose that $\zer{\Phi} \ne \emptyset$, $F$ is $L$-Lipschitz continuous, $\Phi$ is $\rho$-co-hypomonotone, and $J_{\eta T}$ is well-defined.
Let $\sets{(x^k, \hat{x}^k, y^k )}$ be generated by \eqref{eq:NesEAG4NI} using $u^k$ satisfying \eqref{eq:NesEAG4NI_u_cond} with $\kappa \geq 0$ and $\hat{\kappa} \geq 0$.
Given $r  > 1$ and $\bar{\lambda}$ defined in \eqref{eq:NesEAG4NI_lambda_bar}, suppose that  
\begin{equation}\label{eq:NesEAG4NI_para_choice}
\hspace{-2ex}
\arraycolsep=0.2em
\left\{\begin{array}{ll}
& 0 \leq \hat{\kappa} \leq \frac{r - 1}{58 r },  \ \quad L\rho < \frac{ 6 }{29 \sqrt{2(5\kappa + 1)} }, \ \quad 0 < \lambda \leq \bar{\lambda}, \ \quad  \eta := 4\lambda + \frac{29}{6} \rho,   \vspace{1ex}\\
&  t_k := k + r, \qquad  \gamma_k := \frac{t_k - 1 }{t_k}, \qquad  \theta_k := \frac{t_k - 1}{t_{k+1}}, \quad \textrm{ and } \quad  \nu_k := \frac{t_k}{t_{k+1}}.
\end{array}\right.
\hspace{-2ex}
\end{equation}
Then, for any $x^{\star} \in \zer{\Phi}$, the following bound holds:
\begin{equation}\label{eq:NesEAG4NI_convergence}
\norms{Fx^k + \xi^k}^2 \leq \frac{ \Rc_0^2 }{3 \lambda^2 (k+ r - 1)(k + r + 2)},
\end{equation}
where $\Rc_0^2 :=  \lambda(r - 1)\big[ 12(r + 2)\lambda + 29\rho \big] \norms{Fx^0 + \xi^0}^2 + 16 \norms{x^0 - x^\star}^2$.
\end{theorem}

\begin{remark}\label{re:NesEAG4NI_no_sequence_convergence}
The momentum parameter $\theta_k$ in \eqref{eq:NesEAG4NI_para_choice} is exactly $\theta_k = \frac{k+r-1}{k+r+1}$, aligned with the one in standard Nesterov's accelerated methods.
However, this choice does not allow us to prove faster convergence rates (i.e. $\SmallO{1/k}$ rates) as well as the convergence of iterate sequences as recognized in \cite{attouch2016rate,chambolle2015convergence}.
\end{remark}

\noindent\textbf{$\mathrm{(b)}$~\textit{Special instances.}}
Finally, we consider two special cases of  \eqref{eq:NesEAG4NI} corresponding to $u^k = Fx^k$ and $u^k = Fy^{k-1}$. 
The proof of this result is similar to Theorem \ref{th:NesEAG4NI_convergence}, and we provide its details in the appendix.

\begin{corollary}\label{cor:GAEG4NI_Fx^k_Fyk-1_convergence}
For \eqref{eq:NI}, suppose that $\zer{\Phi} \ne \emptyset$, $F$ is $L$-Lipschitz continuous, $\Phi$ is $\rho$-co-hypomonotone, and $J_{\eta T}$ is well-defined.
Let  $\sets{(x^k, \hat{x}^k, y^k )}$ be generated by  \eqref{eq:NesEAG4NI}.
\begin{compactitem}
\item[\textbf{$\mathrm{(i)}$}] If we choose $u^k := Fx^k$, then we require $2L\rho < 1$, and choose $0 < \lambda \leq \frac{1}{L} - 2\rho$, $\eta := \lambda + 2\rho$, and other parameters as in \eqref{eq:NesEAG4NI_para_choice} with $r > 2$.  
Then
\begin{equation}\label{eq:NesEAG4NI_Fx^k_convergence}
\norms{Fx^k + \xi^k}^2 \leq \frac{\Rc_0^2 }{ \lambda^2(k+r - 2)(k+r - 1)},
\end{equation}
where $\Rc_0^2 := 4\lambda(r-1)[ (r + 1)\lambda  + 2 \rho ] \norms{Fx^0 + \xi^0}^2 + \frac{16}{3} \norms{x^0 - x^\star}^2$.

\item[\textbf{$\mathrm{(ii)}$}] If we choose $u^k := Fy^{k-1}$, then we require $8\sqrt{3} L\rho < 1$, and choose $0 < \lambda \leq \bar{\lambda}$, $\eta := 3\lambda + 4 \rho$, and other parameters as in \eqref{eq:NesEAG4NI_para_choice}, 
where 
\begin{equation*} 
\hspace{-1ex}
\begin{array}{lcl}
\bar{\lambda} :=  \frac{2\bar{c}}{\bar{b} + \sqrt{\bar{b}^2 + 4\bar{c}}} \quad \text{with} \quad \bar{b} := \frac{272 \rho}{123} \geq 0 \quad \textrm{and} \quad \bar{c} := \frac{1 - 129L^2\rho^2}{164L^2}  > 0.
\end{array}
\end{equation*}
In particular, if $\rho = 0$ $($i.e. $\Phi$ is monotone$)$, then we choose $0 < \lambda \leq \bar{\lambda} := \frac{1}{2\sqrt{41}L}$ and $\eta := 3\lambda$.
Then, the following bound holds:
\begin{equation}\label{eq:NesEAG4NI_Fy^k-1_convergence}
\norms{Fx^k + \xi^k}^2 \leq \frac{\Rc_0^2 }{ \lambda^2(k+r - 1)(k+r)},
\end{equation}
where $\Rc_0^2 := 2\lambda(r-1)[ (2r + 3)\lambda  + 4 \rho ] \norms{Fx^0 + \xi^0}^2 + \frac{16}{3} \norms{x^0 - x^\star}^2$.
\end{compactitem}
\end{corollary}

\begin{remark}\label{re:NesEAG4IN_para_choice}
The range of $L\rho$ in Theorem~\ref{th:NesEAG4NI_convergence} and in Part (b) of Corollary~\ref{cor:GAEG4NI_Fx^k_Fyk-1_convergence} can potentially be improved by properly choosing parameters $c$, $\hat{c}$, $\beta$, and $\mu$ in our proof.
Here, we have not tried to optimize these ranges of $L\rho$, $\lambda$, and $\eta$.
\end{remark}

\beforesec
\section{New Accelerated EG-Type Methods for \eqref{eq:NI}: Better Guarantees}\label{sec:NGEAG4NI}
\aftersec
In this section, we develop a novel class of generalized Nesterov's accelerated extragradient methods to solve \eqref{eq:NI}, which can achieve faster convergence rates than the one in Section~\ref{sec:AEG4NI}.
In addition, we can prove the convergence of the iterate sequences, which is still open in Section~\ref{sec:AEG4NI}.
Note that this framework is different from the moving anchor FBFS method in Section~\ref{sec:DFBFS4NI} as it relies on Nesterov's acceleration techniques.
A detailed comparison between two types of methods can be found in \cite{TranDinh2025a}.

\beforesubsec
\subsection{\textbf{A Class of Nesterov's Accelerated Extragradient Algorithms}}\label{subsec:NGEAG4NI}
\aftersubsec
\noindent\textbf{$\mathrm{(a)}$~\textit{The proposed method.}}
First, for given $x^k \in \dom{\Phi}$, $u^k \in \R^p$, and $\xi^k \in Tx^k$, we recall the following quantities from the previous sections:
\begin{equation}\label{eq:NGAEG4NI_w_quantities}
w^k := Fx^k + \xi^k, \quad \hat{w}^k := Fy^{k-1} + \xi^k, \quad\text{and} \quad z^k := u^k + \xi^k.
\end{equation}
Next, starting from $x^0 \in \dom{\Phi}$, we set $u^0 := Fx^0$ and $y^{-1} = y^0 :=  x^0$, and at each iteration $k \ge 0$, we update the iterate sequence $\sets{(x^k, y^k) }$ as
\begin{equation}\label{eq:NGEAG4NI}
\arraycolsep=0.2em
\left\{\begin{array}{lcl}
x^{k+1} & := & y^k - \eta d^k, \vspace{1ex} \\
y^{k+1} & := & x^{k+1} + \theta_k(x^{k+1} - x^k) - p^k,
\end{array}\right.
\tag{GAEG$_{+}$}
\end{equation}
where $d^k$ and $p^k$ are respectively defined by
\begin{equation}\label{eq:NGEAG4NI_dir}
\arraycolsep=0.2em
\left\{\begin{array}{lcl}
d^k & := & \hat{w}^{k+1} - \gamma_k z^k, \vspace{1ex}\\
p^k &:= & \eta_k z^{k+1} - \lambda_k \hat{w}^{k+1} + \nu_k z^k.
\end{array}\right.
\end{equation}
The involved parameters in \eqref{eq:NGEAG4NI} and \eqref{eq:NGEAG4NI_dir}  are updated by
\begin{equation}\label{eq:NGEAG4NI_params}
\arraycolsep=0.2em
\begin{array}{ll}
&t_{k+1} := t_k + 1, \qquad   \theta_k := \frac{t_k - r- \mu}{t_{k+1}}, \qquad  \gamma_k := \frac{t_k-r+1}{t_k},  \vspace{1ex}\\
&\eta_k := \frac{(\eta-\beta)t_k - \delta}{t_{k+1}}, \qquad \lambda_k  := \frac{\eta t_k}{t_{k+1}}, \quad  \textrm{and} \quad \nu_k := \frac{\beta t_k}{t_{k+1}},
\end{array}
\end{equation}
where $\eta > \beta \geq 0$, $r > 0$, and $\mu > 0$ are determined later and $\delta := (r-1)\beta + \frac{(r-2)(\eta-\beta)}{\mu+1}$.
Moreover, the direction $u^k$ in \eqref{eq:NGAEG4NI_w_quantities} satisfies the following condition:
\begin{equation}\label{eq:NGEAG4NI_u_cond}
\norms{Fx^k - u^k}^2 \le \kappa\norms{Fx^k - Fy^{k-1}}^2 + \hat{\kappa}\norms{d^{k-1}}^2,
\end{equation}
for given constants $\kappa \geq 0$ and $\hat{\kappa} \ge 0$, and $d^{-1} := 0$.
Here, we allow $\kappa$ to be arbitrary, but we require $\hat{\kappa}$ to be sufficiently small, determined later.

\begin{remark}\label{re:choice_of_parameter}
At first glance, one may question: why are the parameters updated as in \eqref{eq:NGEAG4NI_params}?
Let us explain.
The choice of $\theta_k$ is suggested by Nesterov's accelerated methods in convex optimization and monotone inclusions, which is often of the form $\theta_k = \frac{k - s}{k+2}$ for $s \geq 0$.
The direction $d^k$ can be written as $d^k =   \big(1 - \frac{1}{\gamma_k}\big)\hat{w}^{k+1} + \gamma_k(\hat{w}^{k+1} - z^k)$, which can be viewed as a generalization of the forward-reflected-backward splitting step.
Here, $\gamma_k(\hat{w}^{k+1} - z^k)$ can be seen as an extrapolation or a correction term.
The parameters $\eta_k$, $\lambda_k$, and $\nu_k$ in $p^k$ satisfy $\eta_k - \lambda_k + \nu_k = \frac{\delta}{t_{k+1}}$ as often seen in Nesterov's accelerated methods for monotone inclusions, see, e.g., \cite{attouch2020convergence,attouch2019convergence,bot2022fast,bot2022bfast,kim2021accelerated,mainge2021accelerated,mainge2021fast}.
However, the specific form of parameters in \eqref{eq:NGEAG4NI_params} comes from the derivations of our convergence analysis, when we enforce a descent direction of an appropriate Lyapunov function (defined later).
This step is rather technical, and we only provide the update rule \eqref{eq:NGEAG4NI_params} instead of presenting the detailed derivations  in this paper.
\end{remark}

\begin{remark}\label{re:compare_to_Sect5}
Note that \eqref{eq:NesEAG4NI} studied in Section~\ref{sec:AEG4NI} can be written into the form \eqref{eq:NGEAG4NI}.
However, the choice of parameters is significantly different.
\end{remark}

\noindent\textbf{$\mathrm{(b)}$~\textit{Three special instances.}}
The generalized method \eqref{eq:NGEAG4NI} possibly covers a wide class of schemes by instantiating different directions $u^k$, which satisfy \eqref{eq:NGEAG4NI_u_cond}.
As before, we consider the following three special choices of $u^k$.
\begin{compactitem}
\item[$\mathrm{(i)}$~\textbf{Variant 1: EG-Type Variant.}] 
If we choose $u^k := Fx^k$, then \eqref{eq:NGEAG4NI_u_cond} automatically holds with $\kappa = \hat{\kappa} = 0$.
Thus \eqref{eq:NGEAG4NI} reduces to
\begin{equation}\label{eq:NGEAG4NI_Fxk}
\arraycolsep=0.2em
\left\{\begin{array}{lcl}
x^{k+1} & :=  & y^k - \eta(\hat{w}^{k+1} - \gamma_k w^k ), \vspace{1ex}\\
y^k &:= & x^{k+1} + \theta_k ( x^{k+1} - x^k)  - \eta_k w^{k+1} + \lambda_k \hat{w}^{k+1} - \nu_k w^k,
\end{array}\right.
\tag{AEG}
\end{equation}
where $\gamma_k$, $\theta_k$, $\eta_k$, $\lambda_k$, and $\nu_k$ are updated as in \eqref{eq:NGEAG4NI_params}.
Clearly, \eqref{eq:NGEAG4NI_Fxk} can be viewed as an accelerated extragradient (AEG) method for solving \eqref{eq:NI}.
Note that our method \eqref{eq:NGEAG4NI_Fxk} is different from existing works such as \cite{yuan2024symplectic} due to both the update rule and the choice of parameters.
As in any EG method,  \eqref{eq:NGEAG4NI_Fxk} requires two evaluations $Fx^k$ and $Fy^k$ of $F$, but we only need one resolvent $J_{\eta T}$ of $T$ at each iteration.

\item[$\mathrm{(ii)}$~\textbf{Variant 2: Past-EG-Type Variant.}]
If we choose $u^k := Fy^{k-1}$, then \eqref{eq:NGEAG4NI_u_cond} holds with $\kappa = 1$ and $\hat{\kappa} = 0$.
Clearly, \eqref{eq:NGEAG4NI} reduces to the following variant:
\begin{equation}\label{eq:NGEAG4NI_Fyk-1}
\arraycolsep=0.2em
\left\{\begin{array}{lcl}
x^{k+1} & :=  & y^k - \eta ( \hat{w}^{k+1} - \gamma_k \hat{w}^k ), \vspace{1ex}\\
y^k &:= & x^{k+1} + \theta_k ( x^{k+1} - x^k)  -  (\eta_k - \lambda_k) \hat{w}^{k+1} - \nu_k \hat{w}^k.
\end{array}\right.
\tag{OG}
\end{equation}
This scheme can be viewed as an accelerated [Popov's] past-extragradient (or equivalently, optimistic gradient) method for solving \eqref{eq:NI}.
Unlike \eqref{eq:NGEAG4NI_Fxk}, \eqref{eq:NGEAG4NI_Fyk-1} only requires one evaluation $Fy^k$  and one resolvent $J_{\eta T}$ of $T$ at each iteration, saving one evaluation of $F$.

\item[$\mathrm{(iii)}$~\textbf{Variant 3: Generalization.}]
We can construct $u^k$ as follows:
\begin{equation}\label{eq:NGEAG4NI_u_dir}
u^k := \hat{\alpha} d^{k-1} + (1 - \alpha) Fx^k + \alpha Fy^{k-1}.
\end{equation}
Clearly, $u^k$ is  a linear combination of $Fx^k$, $Fy^{k-1}$ and $d^{k-1}$ for $\alpha, \hat{\alpha} \in \R$.
Then, $u^k$ satisfies \eqref{eq:NGEAG4NI_u_cond} with $\kappa = (1+m)\alpha^2$ and $\hat{\kappa}  = (1+m^{-1})\hat{\alpha}^2$ by Young's inequality for any $m > 0$.
However, to guarantee small $\hat{\kappa}$ as in \eqref{eq:NGEAG4NI_u_cond}, we  choose $\hat{\alpha}$ sufficiently small.
When $\alpha  \in (0, 1]$, \eqref{eq:NGEAG4NI} requires two evaluations $Fx^k$ and $Fy^{k-1}$ of $F$ and one resolvent $J_{\eta T}$ as in \eqref{eq:NGEAG4NI_Fxk}.
\end{compactitem}

\vspace{0.75ex}
\noindent\textbf{$\mathrm{(c)}$~\textit{Comparison.}}
Let us compare our variant \eqref{eq:NGEAG4NI_Fyk-1} with some related works \cite{bot2022fast,sedlmayer2023fast,yuan2024symplectic}.
If $T = 0$, then our scheme \eqref{eq:NGEAG4NI_Fyk-1} reduces to 
\begin{equation}\label{eq:NGEAG4NE_Fyk-1}
\arraycolsep=0.2em
\left\{\begin{array}{lcl}
x^{k+1} & :=  & y^k - \eta (Fy^k - \gamma_k Fy^{k-1} ), \vspace{1ex}\\
y^k &:= & x^{k+1} + \theta_k ( x^{k+1} - x^k)  -  (\eta_k - \lambda_k) Fy^k - \nu_k Fy^{k-1}.
\end{array}\right.
\end{equation}
This scheme now solves the equation \eqref{eq:NE} as \cite{bot2022fast}.
However, the $\gamma_k$ parameter and last term $\nu_k Fy^{k-1}$ in \eqref{eq:NGEAG4NE_Fyk-1} make our variant here different  from \cite[Algorithm 2]{bot2022fast} and allow us to handle the co-hypomonotonicity of $F$ instead of the monotonicity of $F$ as in \cite{bot2022fast}. 
In particular, when $\beta = 0$ in \eqref{eq:NGEAG4NI_params}, \eqref{eq:NGEAG4NE_Fyk-1} does not reduce to \cite[Algorithm 2]{bot2022fast}.
If $T = \Nc_{\Xc}$, the normal cone of a convex set $\Xc$, then our variant \eqref{eq:NGEAG4NI_Fyk-1} is also different from \cite{sedlmayer2023fast} even when we set $\beta = 0$.

\vspace{0.75ex}
\noindent\textbf{$\mathrm{(d)}$~\textit{The implementation of \eqref{eq:NGEAG4NI}.}}
To use the resolvent $J_{\eta T}$ of $T$, we can rewrite \eqref{eq:NGEAG4NI} equivalently to the following scheme:
\begin{equation}\label{eq:NGEAG4NI_implementation}
\arraycolsep=0.2em
\left\{\begin{array}{lcl}
x^{k+1} & := & J_{\eta T}\big( y^k - \eta Fy^k + \eta\gamma_kz^k \big), \vspace{1ex} \\
\xi^{k+1} &:= & \frac{1}{\eta}\big( y^k - \eta Fy^k + \eta\gamma_kz^k  - x^{k+1}\big), \vspace{1ex}\\
z^{k+1} & := & u^{k+1} + \xi^{k+1}, \vspace{1ex}\\
y^{k+1} & := & x^{k+1} + \theta_k(x^{k+1} - x^k) - \eta_k z^{k+1} + \lambda_k(Fy^k + \xi^{k+1}) - \nu_k z^k.
\end{array}\right.
\end{equation}
As before, we assume that $J_{\eta T}$ is well-defined and single-valued.

\beforesubsec
\subsection{\textbf{Key Estimates for Convergence Analysis of \eqref{eq:NGEAG4NI}}}\label{subsec:NGEAG4NI_key_estimate}
\aftersubsec
In this subsection, we first establish some key bounds for our convergence analysis of \eqref{eq:NGEAG4NI} and its instances.
These bounds are independent of the choice of $u^k$, and cover all three instances discussed in (i), (ii), and (iii).

\noindent\textbf{$\mathrm{(a)}$~\textit{Auxiliary functions.}}
First, given $r$, $\mu$, $\eta$, $\beta$, and $t_k$ in \eqref{eq:NGEAG4NI_params}, let us define the following parameters:
\begin{equation}\label{eq:NGEAG4NI_params2}
\arraycolsep=0.2em
\begin{array}{lcl}
 \psi := \frac{\eta-\beta}{\mu + 1}, \quad \omega := \frac{\mu(\eta-\beta)}{\mu + 1}, \quad \textrm{and} \quad  c_k := \big[ \psi(t_k - r + 1) - (r-1)\beta \big]t_k.
\end{array}
\end{equation}
Next, given $c_k$ in \eqref{eq:NGEAG4NI_params2},  we define the following two functions:
\begin{equation}\label{eq:NGEAG4NI_Pk_and_Ek}
\hspace{-2ex}
\arraycolsep=0.2em
\begin{array}{lcl}
\Pc_k &:= & \norms{r(x^k - x^{\star}) + t_k(y^k - x^k)}^2 + r \mu \norms{x^k - x^{\star}}^2 + 2c_k \iprods{z^k, y^k - x^k}, \vspace{1ex} \\
\bmark{\Ec_k} &:= & \eta^2 t_k^2 \norms{d^k}^2 - t_{k+1}^2\norms{p^k}^2  +  2\eta c_k\iprods{z^k, d^k} + 2c_{k+1}\iprods{z^{k+1}, p^k}.
\end{array}
\hspace{-2ex}
\end{equation}

\noindent\textbf{$\mathrm{(b)}$~\textit{Expanding $\Pc_k - \Pc_{k+1}$.}}
We expand $\Pc_k - \Pc_{k+1}$ as in the following lemma.

\begin{lemma}\label{le:NGEAG4NI_descent_property1}
Let $\sets{(x^k, y^k)}$ be computed by \eqref{eq:NGEAG4NI}, and $\Pc_k$ and $\bmark{\Ec_k}$ be defined by \eqref{eq:NGEAG4NI_dir}.
Suppose that the parameters  are updated by \eqref{eq:NGEAG4NI_params} and \eqref{eq:NGEAG4NI_params2}.
Then
\begin{equation}\label{eq:NGEAG4NI_key_property1}
\arraycolsep=0.2em
\begin{array}{lcl}
\Pc_k - \Pc_{k+1} & = & \bmark{\Ec_k} +  \mu(2t_k - r - \mu) \norms{x^{k+1} - x^k}^2\vspace{1ex} \\
&& + {~} 2r [(\eta-\beta)t_k - \delta]  \iprods{ z^{k+1}, x^{k+1} - x^{\star}} \vspace{1ex}\\
&& - {~} 2r  [(\eta-\beta)t_k - \eta(r-1)] \iprods{ z^k, x^k - x^{\star}} \vspace{1ex}\\
&& + {~} 2 t_k \big[ \omega (t_k - \mu - r) + \mu \eta \big]  \iprods{ z^{k+1} - z^k, x^{k+1} - x^k} \vspace{1ex} \\
&&+ {~} 2 \mu \eta t_k \iprods{\hat{w}^{k+1} - z^{k+1}, x^{k+1} - x^k}.
\end{array}
\end{equation}
\end{lemma}

\noindent\textbf{$\mathrm{(c)}$~\textit{Lower bounding $\bmark{\Ec_k}$}.}
First, let us  expand $\bmark{\Ec_k}$ defined by \eqref{eq:NGEAG4NI_Pk_and_Ek} as follows.

\begin{lemma}\label{le:NGEAG4NI_Ek_simplification}
Under the same setting as in Lemma~\ref{le:NGEAG4NI_descent_property1}, $\bmark{\Ec_k}$ defined by \eqref{eq:NGEAG4NI_Pk_and_Ek} can be expressed as
\begin{equation}\label{eq:NAEG4NI_Ek_simplification}
\hspace{-1ex}
\arraycolsep=0.2em
\begin{array}{lcl}
\bmark{\Ec_k} & = & \big[ (1-\mu)\psi t_k  - \delta \big][(\eta - \beta)t_k - \delta] \norms{z^{k+1}}^2  \vspace{1ex} \\
&& + {~} \big\{ \eta(t_k - r + 1)\big[ (\eta - 2\psi)(t_k - r + 1) + 2(r-1)\beta \big] - \beta^2t_k^2 \big\} \norms{z^k}^2 \vspace{1ex}\\
&& + {~} 2\eta \omega t_k^2   \iprods{z^{k+1}, \hat{w}^{k+1}} -  2 \beta \omega t_k^2 \iprods{z^{k+1}, z^k} \vspace{1ex} \\
&& - {~} 2 \eta\omega t_k(t_k - r + 1) \iprods{\hat{w}^{k+1}, z^k}.
\end{array}
\hspace{-1ex}
\end{equation}
\end{lemma}

Now, we can lower bound $\Ec_k$  as in the following lemma.

\begin{lemma}\label{le:NGEAG4NI_Ek_lower_bound}
Under the same setting as in Lemma~\ref{le:NGEAG4NI_descent_property1} and the $L$-Lipschitz continuity of $F$, $\bmark{\Ec_k}$ defined by \eqref{eq:NGEAG4NI_Pk_and_Ek} satisfies
\begin{equation}\label{eq:NAEG4NI_key_estimate2b}
\arraycolsep=0.2em
\begin{array}{lcl}
\bmark{\Ec_k} & \geq & \Lambda_{k+1} \norms{z^{k+1}}^2 - \Lambda_k \norms{z^k}^2 + S_k\norms{z^{k+1}}^2  +   \beta \omega t_k^2 \norms{z^{k+1} - z^k}^2  \vspace{1ex} \\
&& + {~} \eta \omega(1 - M^2\eta^2)t_k^2 \norms{ d^k}^2 +  \eta\omega \hat{\phi} t_k^2 \norms{ w^{k+1} - \hat{w}^{k+1}}^2 \vspace{1ex}\\
&& + {~} \eta \omega \phi t_k^2 \norms{z^{k+1} - \hat{w}^{k+1}}^2 -  \frac{\eta\omega (1+\phi)(1+c_1)t_k^2 }{c_1}\norms{z^{k+1} - w^{k+1} }^2,
\end{array}
\end{equation}
where $M^2 :=  \big[ (1 + \phi)(1+c_1) + \hat{\phi} \big] L^2$ for any $\phi \geq 0$, $\hat{\phi} \geq 0$, and $c_1 > 0$, and 
\begin{equation}\label{eq:NGEAG4NI_ak_and_ahat_k}
\arraycolsep=0.2em
\left\{\begin{array}{lcl}
\Lambda_k &:= & \frac{[(\eta - \beta)(t_k - r + 1) - (r-1)\beta ]^2 + \mu(r-1)^2\eta\beta}{\mu + 1}, \vspace{1ex}\\
\Gamma &:= & \frac{\mu(r-2)^2\eta^2 + \mu[(\mu - 1)r^2 - 2(\mu - 3)r + \mu - 7]\eta \beta - \mu[\mu r^2 - 2(\mu - 1)r + \mu - 3]\beta^2 }{(\mu + 1)^2}, \vspace{1ex}\\
S_k &:= & \frac{2\mu(r-2)(\eta - \beta)^2 }{(\mu+1)^2} t_k - \Gamma.
\end{array}\right.
\end{equation}
\end{lemma}

\beforesubsec
\subsection{\textbf{Convergence Analysis of The AEG Method}}\label{subsec:NGEAG4NI_convergence_results1}
\aftersubsec
Let us first establish the convergence of the instance \eqref{eq:NGEAG4NI_Fxk} (the accelerated extragradient scheme) of \eqref{eq:NGEAG4NI} with $u^k := Fx^k$, which is less complicated than the generalization case when $u^k \neq Fx^k$ in \eqref{eq:NGEAG4NI_u_cond}.
We separate this case since we can leverage larger stepsizes and a wider range of $L\rho$ (i.e. $L\rho < \frac{1}{2}$).

\vspace{0.75ex}
\noindent\textbf{$\mathrm{(a)}$~\textit{The Lyapunov function and technical lemmas.}}
Given $c_k$ in \eqref{eq:NGEAG4NI_params2} and $\Lambda_k$ in \eqref{eq:NGEAG4NI_ak_and_ahat_k}, consider the following Lyapunov function:
\begin{equation}\label{eq:NGEAG4NI_Lyapunov_func1}
\hspace{-2ex}
\arraycolsep=0.2em
\begin{array}{lcl}
\Lc_k &: = & \norms{r(x^k - x^{\star}) + t_k(y^k - x^k)}^2 + r \mu \norms{x^k - x^{\star}}^2 + 2c_k \iprods{w^k, y^k - x^k} \vspace{1ex}\\
&& + {~} \Lambda_k\norms{w^k}^2 + 2r  [(\eta-\beta)t_k - \eta(r-1)] \iprods{w^k, x^k - x^{\star}}.
\end{array}
\hspace{-2ex}
\end{equation}
First, we prove the following key bound to establish convergence of \eqref{eq:NGEAG4NI_Fxk}.

\begin{lemma}\label{le:NGEAG4NI_V1_key_estimate1}
For \eqref{eq:NI}, suppose that $\zer{\Phi} \neq\emptyset$, $F$ is $L$-Lipschitz continuous, and $\Phi$ is $\rho$-co-hypomonotone.
Let $\sets{(x^k, y^k)}$ be generated by \eqref{eq:NGEAG4NI_Fxk} $($i.e. $u^k := Fx^k$$)$ using the parameters in \eqref{eq:NGEAG4NI_params} and \eqref{eq:NGEAG4NI_params2}.
Let $\Lc_k$ be defined by \eqref{eq:NGEAG4NI_Lyapunov_func1}.
Then
\begin{equation}\label{eq:NGEAG4NI_V1_key_estimate1}
\hspace{-2ex}
\arraycolsep=0.2em
\begin{array}{lcl}
\Lc_k - \Lc_{k+1} & \geq &  \mu\big(2t_k - r - \mu - \frac{\mu\eta}{\hat{\phi}\omega}\big) \norms{x^{k+1} - x^k}^2 +  \eta \omega(1 - M^2\eta^2)t_k^2 \norms{ d^k}^2 \vspace{1ex} \\
&& + {~}  t_k\big[ \omega(\beta - 2\rho)t_k + 2 \rho(\mu \omega + r\omega  - \mu \eta) \big]  \norms{ w^{k+1} - w^k}^2 \vspace{1ex}\\
&& + {~}  \big[ S_k - 2\rho \omega  r(r-2)  \big] \norms{w^{k+1}}^2.
\end{array}
\hspace{-2ex}
\end{equation}
\end{lemma}

Next, we lower bound the Lyapunov function $\Lc_k$ in \eqref{eq:NGEAG4NI_Lyapunov_func1} as follows.

\begin{lemma}\label{le:NGEAG4NI_V1_key_estimate2}
Under the same setting as in Lemma~\ref{le:NGEAG4NI_V1_key_estimate1}, we have
\begin{equation}\label{eq:NGEAG4NI_V1_key_estimate2}
\hspace{-2ex}
\arraycolsep=0.2em
\begin{array}{lcl}
\Lc_k &\geq & \norms{r(x^k - x^{\star}) + t_k(y^k - x^k) + [\psi(t_k - r + 1) - (r-1)\beta] w^k}^2  \vspace{1ex}\\
&& + {~} \sigma_k \norms{w^k}^2 + r \mu \norms{x^k - x^{\star}}^2,
\end{array}
\hspace{-2ex}
\end{equation}
where $\sigma_k = \mu\big[ \psi(t_k - r + 1) - r\rho\big]^2 - \mu\big[r^2\rho^2 - \frac{(r-1)^2\beta(\eta-\beta)}{\mu+1}\big]$.

Moreover, we also have 
\begin{equation}\label{eq:NGEAG4NI_V1_key_estimate3}
\hspace{-2ex}
\arraycolsep=0.2em
\begin{array}{lcl}
\Lc_0 &\leq & (C_0 + r^2 + r\mu)\norms{x^0 - x^{\star}}^2 + (C_0 + \Lambda_0)\norms{w^0}^2,
\end{array}
\hspace{-2ex}
\end{equation}
where $C_0 := r  [(\eta-\beta)t_0 - \eta(r-1)]$ provided that $t_0 \geq \frac{\eta(r-1)}{\eta-\beta}$.
\end{lemma}

\noindent\textbf{$\mathrm{(b)}$~\textit{The $\BigO{1/k}$-last iterate convergence rates and summable bounds.}}
For given $\hat{\Gamma} := \frac{2\Gamma}{r-2} + 4\rho \omega r$ and $\hat{\phi} := \frac{1 - L^2\eta^2}{2L^2\eta^2} > 0$, we define
\begin{equation}\label{eq:NGEAG4NI_V1_R02}
\arraycolsep=0.2em
\begin{array}{lcl}
\Rc_0^2 & := & (C_0 + r^2 + r\mu)\norms{x^0 - x^{\star}}^2 + (C_0 + \Lambda_0)\norms{w^0}^2, \vspace{1ex}\\
t_0 &:= & \max\Big\{ \frac{r + 1}{2} + \frac{\eta}{ 2\hat{\phi} \omega}, \  \frac{\hat{\Gamma}}{(\eta-\beta)^2}, \ \frac{4\rho[\eta - \omega(r+1)]}{\omega(\beta-2\rho)}, \ \frac{\eta(r-1)}{\eta-\beta} \Big\}, 
\end{array}
\end{equation}
where $C_0 := r  [(\eta-\beta)t_0 - \eta(r-1)] \geq 0$.

Now, we are ready to prove the convergence of \eqref{eq:NGEAG4NI_Fxk} as follows.

\begin{theorem}\label{th:NGEAG4NI_V1_convergence}
For \eqref{eq:NI}, suppose that $\zer{\Phi} \neq\emptyset$, $F$ is $L$-Lipschitz continuous, and $\Phi$ is $\rho$-co-hypomonotone such that $2L\rho < 1$.
Let $\sets{(x^k, y^k)}$ be generated by \eqref{eq:NGEAG4NI_Fxk} $($i.e. $u^k := Fx^k$$)$ using the parameters in \eqref{eq:NGEAG4NI_params} and \eqref{eq:NGEAG4NI_params2} with $\mu := 1$.
Moreover, we choose $r > 2$, and $t_k$, $\beta$, $\eta$, and $\hat{\phi}$ such that 
\begin{equation}\label{eq:NGEAG4NI_V1_params0}
\arraycolsep=0.2em
\begin{array}{ll}
t_k := k + t_0, \quad 2\rho < \beta < \eta < \frac{1}{L}, \quad \text{and} \quad \hat{\phi} := \frac{1 - L^2\eta^2}{2L^2\eta^2} > 0.
\end{array}
\end{equation}
Then, the following summable bounds hold:
\begin{equation}\label{eq:NGEAG4NI_V1_result1}
\arraycolsep=0.2em
\begin{array}{lcl}
\sum_{k=0}^{\infty} \big(2k + 2t_0 - r - 1 - \frac{\eta}{\hat{\phi}\omega}\big) \norms{x^{k+1} - x^k}^2 & \leq & \Rc_0^2, \vspace{1ex}\\
\sum_{k=0}^{\infty}  \eta\omega(1 - L^2\eta^2) (k + t_0)^2 \norms{d^k}^2 & \leq & 2\Rc_0^2, \vspace{1ex}\\
\sum_{k=0}^{\infty} \big[(\eta-\beta)^2(k+t_0) - \hat{\Gamma} \big] \norms{w^{k+1} }^2 & \leq & \frac{2}{r-2}\Rc_0^2, \vspace{1ex}\\
\sum_{k=0}^{\infty} (k + t_0)^2 \norms{w^{k+1} - w^k}^2 & \leq & \frac{2}{\omega(\beta - 2\rho)}\Rc_0^2,
\end{array}
\end{equation}
where $\omega := \frac{\eta-\beta}{2} > 0$, $\hat{\Gamma} := \frac{2\Gamma}{r-2} + 4\rho \omega r$, and $\Rc_0^2$ and $t_0$ are given in \eqref{eq:NGEAG4NI_V1_R02}.

Furthermore, we also have
\begin{equation}\label{eq:NGEAG4NI_V1_result2}
\norms{Fx^k + \xi^k}^2 \leq \frac{4\Rc_0^2}{(\eta - \beta)^2(k + t_0 - r +1)^2}, \quad \text{where} \quad \xi^k \in Tx^k.
\end{equation}
\end{theorem}

\begin{remark}\label{re:NGEAG4NI_V1_choice_of_stepsize}
We highlight that we can extend our results in Theorem~\ref{th:NGEAG4NI_V1_convergence} to cover the extreme cases $\eta := \frac{1}{L}$ and $\beta := 2\rho$.
However, the analysis is different and thus we omit it here to avoid overloading the paper.
\end{remark}

\noindent\textbf{$\mathrm{(c)}$~\textit{The $\SmallO{1/k}$-convergence rates.}}
Next, we establish $\SmallO{1/k}$ rates of \eqref{eq:NGEAG4NI_Fxk}.

\begin{theorem}\label{th:NGEAG4NI_V1_convergence2}
Under the same conditions and settings as in Theorem~\ref{th:NGEAG4NI_V1_convergence}, we have the following results:
\begin{equation}\label{eq:NGEAG4NI_V1_result3}
\arraycolsep=0.2em
\begin{array}{lcl}
\lim_{k\to\infty} k^2 \norms{Fx^k + \xi^k}^2  = 0, \vspace{1ex}\\
\lim_{k\to\infty} k^2 \norms{x^{k+1} - x^k}^2  = 0, \vspace{1ex}\\
\lim_{k\to\infty} k^2 \norms{y^k - x^k}^2  = 0.
\end{array}
\end{equation}
These expressions show that $\norms{Fx^k + \xi^k} = \SmallO{1/k}$, $ \norms{x^{k+1} - x^k} = \SmallO{1/k}$, and $\norms{y^k - x^k} = \SmallO{1/k}$, respectively for \eqref{eq:NGEAG4NI_Fxk}, where $\xi^k \in Tx^k$.
\end{theorem}

\noindent\textbf{$\mathrm{(d)}$~\textit{The convergence of iterate sequences.}}
Finally, we prove the convergence of $\sets{x^k}$ and also $\sets{y^k}$ to a solution $x^{\star} \in \zer{\Phi}$.

\begin{theorem}\label{th:NGEAG4NI_V1_convergence3}
Under the same conditions and settings as in Theorem~\ref{th:NGEAG4NI_V1_convergence},  if, additionally, $T$ is closed $($in particular, $T$ is maximally monotone$)$, then both $\sets{x^k}$ and $\sets{y^k}$ generated by \eqref{eq:NGEAG4NI_Fxk} converge to $x^{\star} \in \zer{\Phi}$.
\end{theorem}

\begin{remark}\label{re:maximal_monotonicity_of_T}
Theorem~\ref{th:NGEAG4NI_V1_convergence3} requires $T$ to be closed.
Clearly, if $T$ is maximally monotone, then by \cite[Proposition 20.37]{Bauschke2011}, $T$ is closed.
If $T$ is upper semi-continuous on $\dom{T}$ and has closed values, then $T$ is also closed \cite{Konnov2001}. 
\end{remark}

\beforesubsec
\subsection{\textbf{Convergence Analysis of The Generalized AEG Method}}\label{subsec:NGEAG4NI_convergence_results2}
\aftersubsec
In this subsection, we will investigate the convergence of the generalized scheme \eqref{eq:NGEAG4NI} under the condition~\eqref{eq:NGEAG4NI_u_cond}.
Our analysis is new and requires several additional technical steps compared to Subsection~\ref{subsec:NGEAG4NI_convergence_results1}.

\vspace{0.75ex}
\noindent\textbf{$\mathrm{(a)}$~\textit{Technical lemmas.}}
Since $z^k$ can be different from $w^k$, our first step is to process the product term $\iprods{z^{k+1} - z^k, x^{k+1} - x^k}$ from \eqref{eq:NGEAG4NI_key_property1}.

\begin{lemma}\label{le:NGEAG4NI_key_estimate3}
Denote $a_k := t_k \big[\omega (t_k - r - \mu) + \mu \eta \big]$.
We consider the quantity:
\begin{equation}\label{eq:NGEAG4NI_Fk_quantity}
\begin{array}{lcl}
\Fc_k & := & 2a_k \iprods{z^{k+1} - z^k, x^{k+1} - x^k}.
\end{array}
\end{equation}
Then, for any $\Delta > 0$, $r > 2$,  $\mu \geq 1$, and $\eta \geq \frac{(r+\mu-1)\beta}{r-2}$, we have
\begin{equation}\label{eq:NGEAG4NI_key_estimate3}
\arraycolsep=0.2em
\begin{array}{lcl}
\Fc_k & \geq &  2a_{k+1}\iprods{z^{k+1} - w^{k+1}, y^{k+1} - x^{k+1}}  - 2a_k\iprods{z^k - w^k, y^k - x^k} \vspace{1ex}\\
&& - {~} 2\omega t_k [(\eta + \beta)t_k + \delta]  \norms{z^{k+1} - w^{k+1}}^2 -  \frac{ \eta \omega t_k^2 }{2} \norms{ z^{k+1} - \hat{w}^{k+1}}^2 \vspace{1ex}\\
&& - {~} \frac{ \beta \omega t_k^2}{ 2} \norms{ z^{k+1} - z^k }^2 - \frac{ \delta\omega t_k }{ 2}  \norms{ z^{k+1}}^2 - \frac{\eta\omega t_k^2  }{\Delta r}\norms{d^k}^2 \vspace{1ex}\\
&& - {~} \Delta \eta r\omega t_k^2  \norms{z^k - w^k}^2 +  2a_k \iprods{w^{k+1} - w^k, x^{k+1} - x^k} \vspace{1ex}\\
&& + {~} 2 b_k \iprods{z^{k+1} - w^{k+1}, x^{k+1} - x^k},
\end{array}
\end{equation}
where $b_k :=  \omega (r + \mu - 1) (t_k - r - \mu) + \mu \eta (r + \mu)$.
\end{lemma}

Now, we combine Lemmas~\ref{le:NGEAG4NI_descent_property1}, \ref{le:NGEAG4NI_Ek_lower_bound}, and \ref{le:NGEAG4NI_key_estimate3} to prove the following result.

\begin{lemma}\label{le:NGEAG4NI_key_bound100}
Suppose that $\Phi$ is $\rho$-co-hypomonotone and $F$ is $L$-Lipschitz continuous.
Let $\sets{(x^k, y^k)}$ be generated by \eqref{eq:NGEAG4NI} using the update rules \eqref{eq:NGEAG4NI_params} and \eqref{eq:NGEAG4NI_params2}.
Let us consider the following function:
\begin{equation}\label{eq:NGEAG4NI_Gk_func}
\begin{array}{lcl}
\Gc_k &:= & \Lambda_k\norms{z^k}^2 +  \norms{r(x^k - x^{\star}) + t_k(y^k - x^k)}^2 + r \mu \norms{x^k - x^{\star}}^2 \vspace{1ex}\\
&& + {~} 2r\big[(\eta - \beta)t_k - \eta(r-1)\big]\iprods{z^k, x^k - x^{\star}} \vspace{1ex}\\
&& + {~} 2c_k \iprods{z^k, y^k - x^k}  + 2a_k \iprods{z^k - w^k, y^k - x^k}. 
\end{array}
\end{equation}
Then, for any $\Delta > 0$, $r > 2$, $\mu \geq 1$, and $\eta \geq \frac{(r+\mu-1)\beta}{r-2}$, we have
\begin{equation}\label{eq:NGEAG4NI_key_bound100}
\hspace{-2ex}
\arraycolsep=0.2em
\begin{array}{lcl}
\Gc_k - \Gc_{k+1} & \geq & \big[  S_k - \frac{\delta \omega t_k}{2} - 4\rho \omega r(r-2) \big]  \norms{z^{k+1}}^2 - \frac{ \omega r^2(r-2)^2 }{ \eta t_k^2} \norms{x^{k+1} - x^{\star} }^2 \vspace{1ex}\\
&& + {~} \mu \big( t_k - r - \mu - \frac{2\mu\eta}{\omega} \big) \norms{x^{k+1} - x^k}^2  \vspace{1ex}\\
&& + {~} \eta\omega\big(1 - \frac{1}{\Delta r} - M^2\eta^2 \big) t_k^2  \norms{ d^k}^2 +  \eta\omega \hat{\phi} t_k^2 \norms{ w^{k+1} - \hat{w}^{k+1}}^2 \vspace{1ex}\\
&& + {~} \eta\omega ( \phi - 1 ) t_k^2 \norms{z^{k+1} - \hat{w}^{k+1}}^2 - \big( \Delta r \eta + 8\rho \big) \omega t_k^2  \norms{z^k - w^k}^2  \vspace{1ex}\\
&& - {~} \hat{\Theta}_k \norms{z^{k+1} - w^{k+1} }^2 +  \frac{ (\beta - 8\rho) \omega t_k^2}{2} \norms{z^{k+1} - z^k}^2,
\end{array}
\hspace{-3ex}
\end{equation}
where $\hat{\Theta}_k :=   \omega t_k \big[   (\eta (5+2\phi) + 2 \beta +  8\rho ) t_k  + 2\delta + \omega(r + \mu - 1)^2 \big]  + 4 \rho\omega r(r-2)$.
\end{lemma}

\noindent\textbf{$\mathrm{(b)}$~\textit{The Lyapunov function and its descent property.}}
Now, we define the following Lyapunov function to analyze our generalized scheme \eqref{eq:NGEAG4NI}:
\begin{equation}\label{eq:NGEAG4NI_V2_Lyapubov_func}
\arraycolsep=0.2em
\begin{array}{lcl}
\hat{\Lc}_k &:= & \norms{r(x^k - x^{\star}) + t_k(y^k - x^k)}^2 + r \mu \norms{x^k - x^{\star}}^2 +  \Lambda_k\norms{z^k}^2 \vspace{1ex}\\
&& + {~} 2r\big[(\eta - \beta)t_k - \eta(r-1)\big]\iprods{z^k, x^k - x^{\star}} \vspace{1ex}\\
&& + {~} 2c_k \iprods{z^k, y^k - x^k}  + 2a_k \iprods{z^k - w^k, y^k - x^k} + \alpha_k\norms{z^k - w^k}^2,
\end{array}
\end{equation}
where $\alpha_k := ( \Delta r \eta  + 8\rho )\omega t_k^2$ for some $\Delta > 0$ determined later, $c_k$ is in \eqref{eq:NGEAG4NI_params2}, and $\Lambda_k$ is in \eqref{eq:NGEAG4NI_ak_and_ahat_k}.
We prove the following result.

\begin{lemma}\label{le:NGEAG4NI_V2_descent_of_Lyapunov}
For \eqref{eq:NI}, suppose that $\Phi$ is $\rho$-co-hypomonotone and $F$ is $L$-Lipschitz continuous.
Let $\sets{(x^k, y^k)}$ be generated by \eqref{eq:NGEAG4NI} using the update rules \eqref{eq:NGEAG4NI_params} and \eqref{eq:NGEAG4NI_params2} and $u^k$ satisfying \eqref{eq:NGEAG4NI_u_cond}.
Let $\hat{\Lc}_k$ be defined by \eqref{eq:NGEAG4NI_V2_Lyapubov_func}.
Given $r > 2$, $\mu \geq 1$, $\epsilon \geq 0$, and $\hat{\epsilon} \geq 0$, we choose $\beta$, $\eta$, and $t_k$ such that
\begin{equation}\label{eq:NGEAG4NI_V2_param2}
\arraycolsep=0.2em
\begin{array}{ll}
\beta := 8\rho + 2\epsilon, \quad  \frac{(r+\mu-1)\beta}{r-2} \leq \eta, \quad \text{and} \quad t_k \geq \bar{t}_0. 
\end{array}
\end{equation}
where $\bar{t}_0 := \frac{D + \sqrt{D^2 + 4 \eta E}}{2 \eta}$ with  $D := 2\delta + \omega(r + \mu - 1)^2 + 2\Delta r \eta + 2\beta + 2\rmark{\hat{\epsilon} \eta}$ and $E := 4 \rho r(r-2) + \Delta r \eta + 8\rho + \hat{\epsilon} \eta $.
Then, we have 
\begin{equation}\label{eq:NGEAG4NI_V2_Lyapubov_descent}
\hspace{-2ex}
\arraycolsep=0.2em
\begin{array}{lcl}
\hat{\Lc}_k - \hat{\Lc}_{k+1} & \geq & \frac{\omega[3(r-2)\psi - (r-1)\beta]t_k - \hat{\Gamma}}{2}    \norms{z^{k+1}}^2 - \frac{ \omega r^2(r-2)^2 }{\eta t_k^2} \norms{x^{k+1} - x^{\star} }^2 \vspace{1ex}\\
&& + {~}  \eta\omega (\hat{\phi} - \kappa \Theta ) t_k^2 \norms{ w^{k+1} - \hat{w}^{k+1}}^2\vspace{1ex}\\ 
&& + {~} \mu \big(t_k - r - \mu - \frac{2\mu\eta}{\omega} \big) \norms{x^{k+1} - x^k}^2  \vspace{1ex}\\
&& + {~} \eta \omega  \big(1 - \frac{1}{\Delta r} - \hat{\kappa}\Theta - M^2\eta^2\big) t_k^2 \norms{ d^k}^2\vspace{1ex}\\
&& + {~} \hat{\epsilon} \omega \eta t_{k+1}^2 \norms{z^{k+1} - w^{k+1}}^2 + \epsilon\omega t_k^2 \norms{z^{k+1} - z^k}^2,
\end{array}
\hspace{-2ex}
\end{equation}
where $\hat{\Gamma} := 2\Gamma + 8\rho \omega r(r-2)$ and  $\Theta := 10 + 2\phi  + \Delta r + \hat{\epsilon}$.
\end{lemma}

\noindent\textbf{$\mathrm{(c)}$~\textit{The lower bound of $\hat{\Lc}_k$.}}
We also need to  lower bound $\hat{\Lc}_k$ as follows.

\begin{lemma}\label{le:NGEAG4NI_Lyapunov_lower_bound}
Under the same setting as in Lemma~\ref{le:NGEAG4NI_V2_descent_of_Lyapunov}, we have
\begin{equation}\label{eq:NGEAG4NI_Lyapunov_lower_bound}
\begin{array}{lcl}
\hat{\Lc}_k & \geq &  r(\mu - 1) \norms{x^k - x^{\star}}^2 + \frac{(\eta - \beta)^2t_k^2}{2(\mu+1)^2}\norms{z^k}^2 \vspace{1ex}\\
&& + {~} \big( \frac{\Delta r \eta}{\omega} - 4r - 1 \big) \omega^2 t_k^2\norms{z^k - w^k}^2,
\end{array}
\end{equation}
provided that $\mu \ge 1$, $r > 2$, $t_k \geq r$, and $t_k \geq \hat{t}_0$, in which $\hat{t}_0$ is defined by
\begin{equation}\label{eq:NGEAG4NI_Lyapunov_lower_bound_t0_hat}
\arraycolsep=0.2em
\begin{array}{lcl}
\hat{t}_0 &:= & \begin{cases}
\frac{B + \sqrt{B^2 - AC}}{A}, &\text{if } B^2 \ge AC,\\
0, &\text{otherwise,}
\end{cases}
\end{array}
\end{equation}
where 
\begin{equation*} 
\arraycolsep=0.2em
\left\{\begin{array}{lcl}
A &:=& \frac{(\eta - \beta)^2}{(\mu + 1)^2} \left( \mu - \frac{1}{2r} - \frac{1}{2} \right) > 0, \vspace{1ex}\\ 
B &:=& \frac{(\eta - \beta)}{(\mu + 1)^2}\big[ \frac{(r-1)[2r\mu (\eta - \beta) - (\eta + \mu \beta)]}{2r} + \rho (\mu + 1)(2r\mu - 1) \big], \vspace{1ex}\\
C &:=& (r - 1) \big\{ \frac{(r-1)(\eta + \mu \beta)}{(\mu + 1)^2}\big[ \left(\mu - \frac{1}{2r}\right)\eta - \frac{(2r + 1)\mu \beta}{2r} \big] - 2\rho [ \eta - (2r + 1)\omega ] \big\}.
\end{array}\right.
\end{equation*}
In particular, if we choose $\Delta := 3$, then we have
\begin{equation}\label{eq:NGEAG4NI_Lyapunov_lower_bound2}
\begin{array}{lcl}
\hat{\Lc}_k \geq r(\mu - 1) \norms{x^k - x^{\star}}^2 + \frac{(\eta - \beta)^2t_k^2}{4(\mu+1)^2}\norms{w^k}^2.
\end{array}
\end{equation}
\end{lemma}

%
\noindent\textbf{$\mathrm{(d)}$~\textit{The $\BigO{1/k}$ convergence rates and summable bounds.}}
Fix $r > 2$ and $\hat{\epsilon} \geq 0$, suppose that the constant $\hat{\kappa}$ in \eqref{eq:NGEAG4NI_u_cond} satisfies
\begin{equation}\label{eq:NGEAG4NI_V2_kappa_hat_cond}
\arraycolsep=0.2em
\begin{array}{lcl}
0 \leq \hat{\kappa} < \frac{3r - 1}{3r(12 + 3r + \hat{\epsilon})}.
\end{array}
\end{equation}
Given $\hat{\kappa}$ as in \eqref{eq:NGEAG4NI_V2_kappa_hat_cond}, $\bar{t}_0$ defined in Lemma~\ref{le:NGEAG4NI_V2_descent_of_Lyapunov}, and $\hat{t}_0$ defined by \eqref{eq:NGEAG4NI_Lyapunov_lower_bound_t0_hat}, we denote the following two constants:
\begin{equation}\label{eq:NGEAG4NI_V2_Psi_cond}
\hspace{-2ex}
\arraycolsep=0.2em
\left\{\begin{array}{lcl}
\Psi & := & \frac{1}{\sqrt{4 +  12 \kappa  + 3\kappa r + \kappa \hat{\epsilon}} } \big[1 - \frac{1}{3r} - \hat{\kappa}(12 + 3r + \hat{\epsilon})\big] \in (0, 1], \vspace{1ex}\\
t_0 &:= & \max\Big\{  \bar{t}_0, \  \hat{t}_0, \  r,  \ \frac{\eta(r - 1)}{\eta - \beta}, \ \frac{\mu}{\omega}(r-1)[ \eta - (2r+1)\omega ], \ \frac{(r-2)\sqrt{\omega r}}{\sqrt{\eta (\mu - 1)}} \Big\}.
\end{array}\right.
\hspace{-2ex}
\end{equation}
Note that if $\hat{\kappa} = 0$, (e.g., when $u^k := Fy^{k-1}$ as in \eqref{eq:NGEAG4NI_Fyk-1}, then \eqref{eq:NGEAG4NI_V2_kappa_hat_cond} automatically holds, and $\Psi$ in \eqref{eq:NGEAG4NI_V2_Psi_cond} reduces to $\Psi :=  \frac{3r - 1}{3r \sqrt{4 +  \kappa(12 + 3r + \hat{\epsilon})}}$.

Now, we are ready to prove the convergence of \eqref{eq:NGEAG4NI} as follows.

\begin{theorem}\label{th:NGEAG4NI_V2_convergence1}
For \eqref{eq:NI}, suppose that $\zer{\Phi}\neq\emptyset$, $\Phi$ is $\rho$-co-hypomonotone, and $F$ is $L$-Lipschitz continuous.
Fix $r > 2$ and $\mu > 1$,  let $\sets{(x^k, y^k)}$ be generated by \eqref{eq:NGEAG4NI} using the update rules \eqref{eq:NGEAG4NI_params} and \eqref{eq:NGEAG4NI_params2} and $u^k$ such that \eqref{eq:NGEAG4NI_u_cond} holds with $\hat{\kappa}$ as in \eqref{eq:NGEAG4NI_V2_kappa_hat_cond}.
Suppose that $L\rho < \frac{(r-2)\Psi}{8(r+\mu-1)}$ for  $\Psi$ given by \eqref{eq:NGEAG4NI_V2_Psi_cond}.
Let us choose $t_k$, $\beta$, and $\eta$ such that 
\begin{equation}\label{eq:NGEAG4NI_V2_choice_of_constants}
\arraycolsep=0.0em
\begin{array}{ll}
& t_k = k + t_0, \quad \beta := 8\rho + 2\epsilon, \ \ \text{and} \ \ \frac{(r+\mu-1)\beta}{r-2} \leq \eta \leq \bar{\eta} := \frac{\Psi}{L},
\end{array}
\end{equation}
where $t_0$ is given in \eqref{eq:NGEAG4NI_V2_Psi_cond} and  $\epsilon \geq 0$ is given.
Then, we have
\begin{equation}\label{eq:NGEAG4NI_V2_convergence1}
\arraycolsep=0.2em
\begin{array}{lcl}
\sum_{k=0}^{\infty} (k+t_0) \norms{x^{k+1} - x^k}^2 & < & +\infty, \vspace{1ex}\\
\sum_{k=0}^{\infty} (k + t_0)^2 \norms{z^k - w^k}^2  & < & +\infty, \vspace{1ex}\\
 \sum_{k=0}^{\infty} (k + t_0)^2 \norms{z^{k+1} - z^k}^2  & < & +\infty, \vspace{1ex}\\
\sum_{k=0}^{\infty} (\bar{\eta}^2 - \eta^2) (k + t_0)^2 \norms{d^k}^2 & < & +\infty, \vspace{1ex}\\
\sum_{k=0}^{\infty} (k + t_0) \norms{z^{k} }^2 & < & +\infty, \vspace{1ex}\\
\sum_{k=0}^{\infty} (k + t_0) \norms{w^k }^2 & < & +\infty.
\end{array}
\end{equation}
Moreover, for $\Rc_0^2$ given in \eqref{eq:NGEAG4NI_V1_R02}, we also have 
\begin{equation}\label{eq:NGEAG4NI_V2_convergence2}
\norms{Fx^k + \xi^k}^2  \leq \frac{4\Omega(\mu+1)^2 \Rc_0^2}{(\eta - 8\rho)^2(k+t_0)^2}, \quad \textrm{where} \quad \xi^k \in Tx^k,
\end{equation}
where $\Omega :=  \left( \frac{1 + t_0 + c}{1 + t_0 - c} \right)^{\frac{c}{2}} \exp\left(\frac{c^2}{(1 + t_0)^2 - c^2}\right)  < \infty$ with $c^2 := \frac{\omega r (r - 2)^2}{\eta (\mu - 1)}$.
\end{theorem}

\noindent\textbf{$\mathrm{(e)}$~\textit{The $\SmallO{1/k}$ last-iterate convergence rates.}}
Our next step is to establish $\SmallO{1/k}$ convergence rates of \eqref{eq:NGEAG4NI}.

\begin{theorem}\label{th:NGEAG4NI_V2_convergence2}
Under the same conditions and settings as in Theorem~\ref{th:NGEAG4NI_V2_convergence1} and  we additionally choose $\eta < \bar{\eta}$, then we have the following limits:
\begin{equation}\label{eq:NGEAG4NI_V2_convergence3}
\arraycolsep=0.2em
\begin{array}{lcl}
\lim_{k\to\infty} k^2 \norms{Fx^k + \xi^k}^2  = 0, \vspace{1ex}\\
\lim_{k\to\infty} k ^2 \norms{x^{k+1} - x^k}^2  = 0, \vspace{1ex}\\
\lim_{k\to\infty} k^2 \norms{y^k - x^k}^2  = 0.
\end{array}
\end{equation}
These expressions show that $\norms{Fx^k + \xi^k} = \SmallO{1/k}$, $ \norms{x^{k+1} - x^k} = \SmallO{1/k}$, and $\norms{y^k - x^k} = \SmallO{1/k}$, respectively for \eqref{eq:NGEAG4NI}, where $\xi^k \in Tx^k$.
\end{theorem}

\noindent\textbf{$\mathrm{(f)}$~\textit{The convergence of iterate sequences.}}
Finally, we prove the convergence of the iterate sequences $\sets{x^k}$ and $\sets{y^k}$ generated by \eqref{eq:NGEAG4NI}  to a solution $x^{\star} \in \zer{\Phi}$.
The proof of this result is rather involved and hence, we divide it into several technical lemmas (see Appendix~\ref{apdx:sec:NGEAG4NI}).

\begin{theorem}\label{th:NGEAG4NI_V2_convergence3}
Under the same conditions and settings as in Theorem~\ref{th:NGEAG4NI_V2_convergence1},  if we additionally choose $\eta < \bar{\eta}$ and assume that $T$ is closed, then both $\sets{x^k}$ and $\sets{y^k}$ generated by \eqref{eq:NGEAG4NI} converge to $x^{\star} \in \zer{\Phi}$.
\end{theorem}

\begin{remark}\label{re:NGEAG4NI_optimal_choice_of_params}
We have used several times of Young's inequality in our analysis in Subsection~\ref{subsec:NGEAG4NI_convergence_results2}.
Our analysis as well as the ranges of $L\rho$ and other parameters can be improved by appropriately choosing the constant $c$ whenever we applying  Young's inequality of the form $\norms{u - v}^2 \leq (1+c)\norms{u}^2 + \frac{1+c}{c}\norms{v}^2$.
\end{remark}

\beforesec
\section{Numerical Experiments}\label{sec:AEG4NI_numerical_experiments}
\aftersec
In this section, we provide a number of experiments to validate our algorithms and their theoretical results.
All the algorithms are implemented in Python running on a single node of a Linux server (called \texttt{Longleaf}) with the configuration: AMD EPYC 7713 64-Core Processor, 512KB cache, and 64GB RAM.

\beforesubsec
\subsection{\mytb{Mathematical model: Quadratic minimax optimization}}\label{subsec:numexp_quad_minimax}
\aftersubsec
The mathematical model we use for our experiments is the following simplex constrained quadratic minimax problem:
\begin{align}\label{prob:minimax}
	\min_{u \in \mathbb{R}^{p_1}}\max_{v\in\mathbb{R}^{p_2}}\Big\{ \mathcal{L}(u,v) = f(u) + \Hc(x, y) - g(v) \Big\},
\end{align}
where $\Hc(x, y) := \frac{1}{2}u^{\top}A u + b^{\top}u + u^{\top}L v - \frac{1}{2}v^{\top}B v - c^{\top}v$ such that  $A \in\mathbb{R}^{p_1\times p_1}$ and $B \in \mathbb{R}^{p_2\times p_2}$ are symmetric matrices, $b \in\mathbb{R}^{p_1}$, $c \in\mathbb{R}^{p_2}$ are given vectors, and $L \in \mathbb{R}^{p_1\times p_2}$ is a given matrix. 
The functions  $f(u) = \delta_{\Delta_{p_1}}(u)$ and $g(v) = \delta_{\Delta_{p_2}}(v)$ are added to handle simplex constraints $u \in \Delta_{p_1}$ and $v \in \Delta_{\rmark{p_2}}$, where $\Delta_{p_1}$ and $\Delta_{p_2}$ are the standard simplexes in $\mathbb R^{p_1}$ and $\mathbb R^{p_2}$, respectively, and $\delta_{\mathcal X}$ is the indicator function of a closed convex set $\mathcal X$. 

First, we denote $x \coloneqq [u,v] \in \mathbb R^p$ for $p := p_1 + p_2$, which is the concatenation of the primal variable $u \in \mathbb R^{p_1}$ and its dual variable $v \in \mathbb R^{p_2}$. 
Next, we define $\mathbf{F} \coloneqq \left[ [A, L], [-L^\top, B] \right]$ as the KKT (Karush-Kuhn-Tucker) matrix in $\mathbb R^{p\times p}$ constructed from the four blocks $A, L, -L^\top$, and $B$, and $\mathbf{f} \coloneqq [b, c] \in \mathbb R^p$. 
The operator $F: \mathbb R^p \to \mathbb R^p$ is then defined as $Fx \coloneqq \mathbf{F}x + \mathbf{f}$. 
Then, we denote $T \coloneqq [\partial\delta_{\Delta_{p_1}}, \partial\delta_{\Delta_{p_2}}]$ the maximally monotone mapping constructed from the subdifferentials of $\delta_{\Delta_{p_1}}$ and $\delta_{\Delta_{p_2}}$. 
Finally, the optimality condition of \eqref{prob:minimax} becomes $0 \in Fx + Tx$ covered by \eqref{eq:NI}.

\beforesubsec
\subsection{\mytb{Numerical experiments}}\label{subsec:num_experiments}
\aftersubsec
\noindent\textbf{$\mathrm{(a)}$~\textit{Data generation.}}
Our aim is to verify the proposed algorithms and their corresponding theoretical results.
In what follows, all matrices and vectors in $F$ are generated randomly from the standard normal distribution. 
We generate the matrix $A = Q D Q^\top$, where $Q$ is an orthonormal matrix obtained from the QR factorization of a random matrix, and $D = \diag{d_{1}, \dots, d_{p_1}}$ is the diagonal matrix formed from $d_{1}, \dots, d_{p_1}$ randomly generated and then clipped by a lower bound $\ul{d}$, i.e. $d_{j} \coloneqq \max\{d_{j}, \ul{d}\}$. 
The matrix $B$ is also generated by the same way. 
The matrix $L$ and vectors $b$ and $c$ are also randomly generated. 

\vspace{0.75ex}
\noindent\textbf{$\mathrm{(b)}$~\textit{Experiment setup.}}
We perform four different experiments. 
In \textit{Experiment 1} and \textit{Experiment 2}, we choose $\ul{d} = 0.1$ (monotone), and run the four sets of algorithms on 10 problem instances for each case: $p = 1000$ and $p = 2000$, respectively. 
In \textit{Experiment 3} and \textit{Experiment 4}, we choose $\ul{d} = -10^{-3}$ (possibly nonmonotone) and run the same four sets of algorithms on 10 problem instances for each case: $p = 1000$ and $p = 2000$, respectively.
Then, we report the mean of the relative operator norm $\frac{\norms{\Gc_\eta x^k}}{\norms{\Gc_\eta x^0}}$ over 10 problem instances, where $\Gc_\eta \coloneqq \eta^{-1} \left( x - J_{\eta T}(x-\eta Fx) \right)$ is defined in \eqref{eq:FB_residual}. 

\vspace{0.75ex}
\noindent\textbf{$\mathrm{(c)}$~\textit{Algorithms and parameters.}}
In each \textit{Experiment}, we examine the following four sets of algorithms.
In the first set, we consider five variants using the direction $u^k := Fx^k$ of \eqref{eq:EAG4NI}, \eqref{eq:FEG4NI}, \eqref{eq:DFEG4NI}, \eqref{eq:NesEAG4NI}, and \eqref{eq:NGEAG4NI}.
Next, in the second one, we consider five variants of \eqref{eq:EAG4NI}, \eqref{eq:FEG4NI}, \eqref{eq:DFEG4NI}, \eqref{eq:NesEAG4NI}, and \eqref{eq:NGEAG4NI} using $u^k := Fy^{k-1}$ (in fact, \eqref{eq:FEG4NI} uses a slightly different direction $u^k = Fy^{k-1} + \xi^k - \xi^{k-1}$ due to the choice $\alpha = 0$, $\hat{\alpha} = 1$).
Then, in the third set, we investigate different variants of \eqref{eq:DFEG4NI} by fixing $r = 3$ and choosing different values of $\mu$.
Finally, in the fourth set, we study different variants of \eqref{eq:NGEAG4NI} by modifying the choice of its corresponding parameters.
The stepsize of each algorithm is tuned manually to obtain the best possible performance for all algorithms, but not necessarily optimal. 
The starting points are always chosen as $x^0 \coloneqq 0.01\cdot\texttt{ones}(p)$.

\vspace{0.75ex}
\noindent\textbf{$\mathrm{(d)}$~\textit{Results.}} 
The numerical results of \textit{Experiment 1} and \textit{Experiment 2} are reported in Figures \ref{fig:constr_minimax_Fxk_1}, \ref{fig:constr_minimax_Fyk-1_1}, \ref{fig:constr_minimax_DFEG_1}, and \ref{fig:constr_minimax_GAEG_1}.

\begin{figure}[hpt!]
	\vspace{-3ex}
	\centering
	\includegraphics[width=\linewidth]{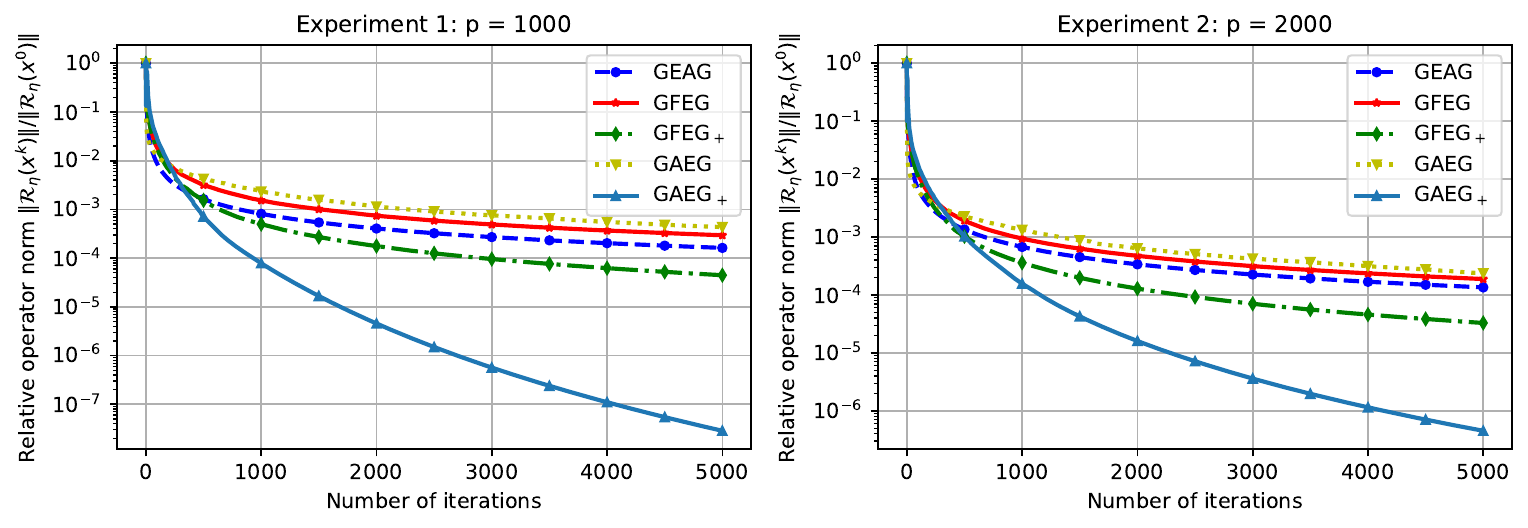}
	\caption{The behaviors of the first set of algorithms for solving \eqref{eq:NI} in \textit{Experiments 1} and \textit{2} when we choose $u^k := Fx^k$. 
	The plot reveals the mean of 10 problem instances.}
	\label{fig:constr_minimax_Fxk_1}
	\vspace{-3ex}
\end{figure}
\begin{figure}[hpt!]
	\vspace{-4ex}
	\centering
	\includegraphics[width=\linewidth]{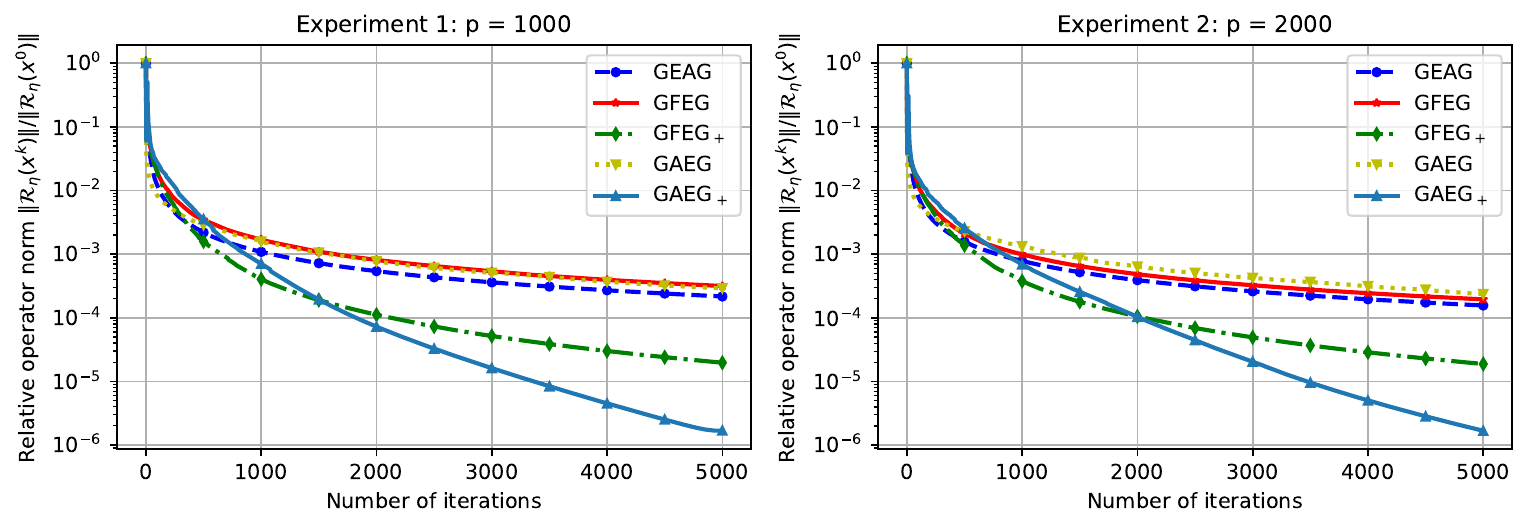}
	\caption{The behaviors of the second set of algorithms for solving \eqref{eq:NI} in \textit{Experiments 1} and \textit{2} when we choose $u^k := Fy^{k-1}$. 
	The plot reveals the mean of 10 problem instances.}
	\label{fig:constr_minimax_Fyk-1_1}
	\vspace{-3ex}
\end{figure}
\begin{figure}[hpt!]
	\vspace{-0ex}
	\centering
	\includegraphics[width=\linewidth]{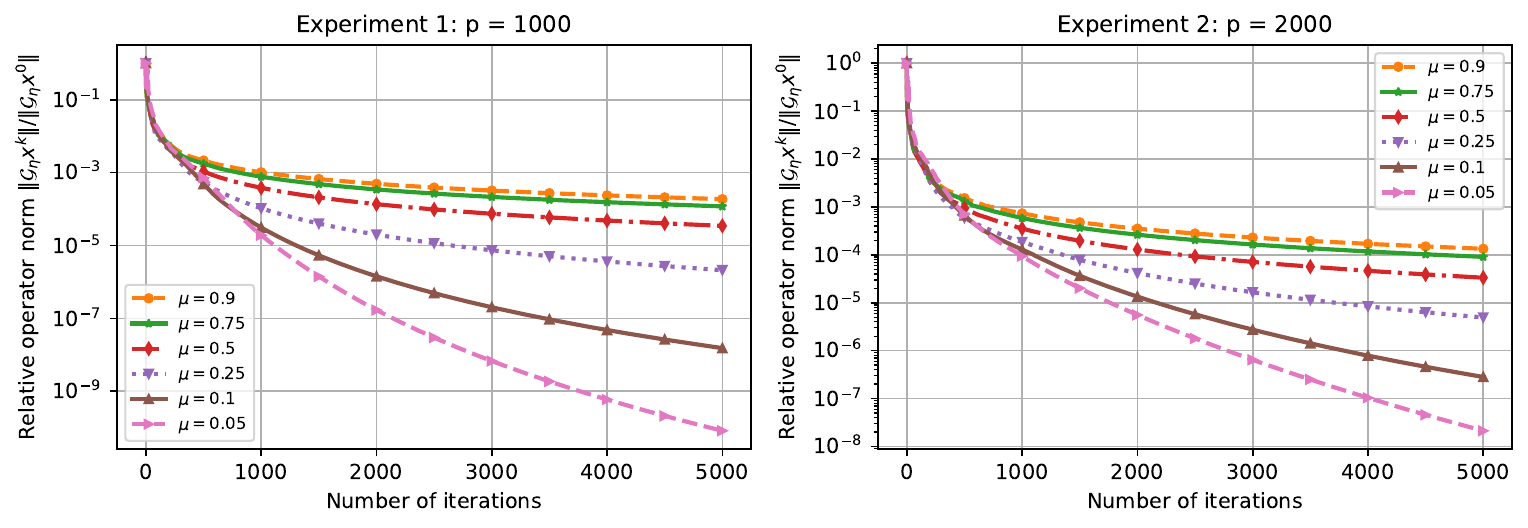}
	\caption{The behaviors of  \eqref{eq:DFEG4NI} for solving \eqref{eq:NI} in \textit{Experiments 1} and \textit{2}. 
	The plot reveals the mean of 10 problem instances. 
	The legend presents the values of   $\mu$ used in the corresponding instant of \eqref{eq:DFEG4NI}. }
	\label{fig:constr_minimax_DFEG_1}
	\vspace{-3ex}
\end{figure}
\begin{figure}[hpt!]
	\vspace{-0ex}
	\centering
	\includegraphics[width=\linewidth]{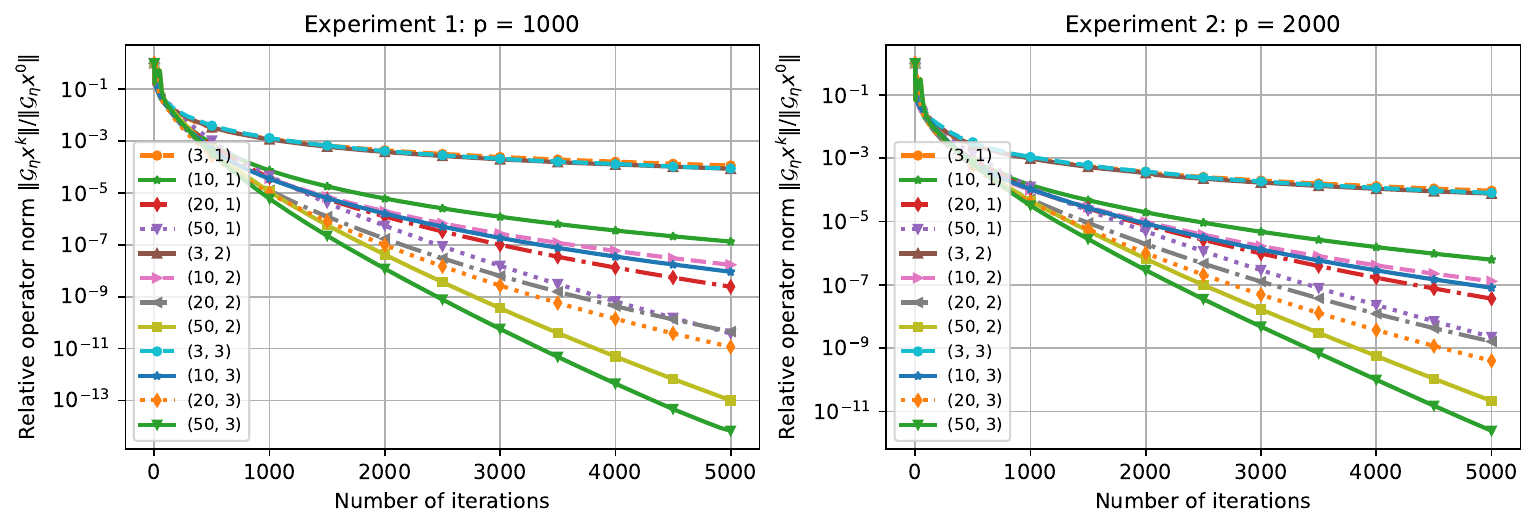}
	\caption{The behaviors of  \eqref{eq:NGEAG4NI} for solving \eqref{eq:NI} in \textit{Experiments 1} and \textit{2}. 
	The plot reveals the mean of 10 problem instances. 
	The legend presents the values of $r$ and $\mu$ used in the corresponding instant of \eqref{eq:NGEAG4NI}. For example, $(3, 1)$ means the solid orange line with filled circle marker corresponds to the instance of \eqref{eq:NGEAG4NI} using $r = 3$ and $\mu = 1$.}
	\label{fig:constr_minimax_GAEG_1}
	\vspace{-3ex}
\end{figure}

From Figure \ref{fig:constr_minimax_Fxk_1}, we can see that in both experiments,  \eqref{eq:EAG4NI}, \eqref{eq:FEG4NI}, and \eqref{eq:NesEAG4NI} have a similar performance within the accuracy of $10^{-4}$ after 5000 iterations. 
Our \eqref{eq:DFEG4NI} gives a better performance than these three schemes, but  \eqref{eq:NGEAG4NI} strongly outperforms its competitors when providing the accuracy of $10^{-7}$ after the same number of iterations. 
In fact, other experiments reveal that \eqref{eq:NGEAG4NI} can reach the accuracy of $10^{-14}$ after $5000$ iterations by appropriately tuning the parameters (see Figure \ref{fig:constr_minimax_GAEG_1}). 

A similar observation can also be seen from Figure \ref{fig:constr_minimax_Fyk-1_1}, which presents the behaviors of the second set of algorithms using $u^k := Fy^{k-1}$.
However, \eqref{eq:DFEG4NI} now performs better and \eqref{eq:NGEAG4NI} is slightly slowed down.

Next, we test \eqref{eq:DFEG4NI} with $u^k := Fx^k$ by fixing $r = 3$ and  choosing $6$ different values of $\mu$ ranging from $0.05$ to $0.05$.
Figure~\ref{fig:constr_minimax_DFEG_1} shows the results of the two experiments, where we observe that when $\mu$ is decreasing, this method converges faster and achieves better accuracy.

Finally, in Figure \ref{fig:constr_minimax_GAEG_1}, we fixed the values of $\eta$ and $\beta$ and examine the effect of $r$ and $\mu$ to the performance of \eqref{eq:NGEAG4NI}. The numerical results show that larger values of $r$ and $\mu$ can significantly improve the performance of \eqref{eq:NGEAG4NI}, which helps this algorithm become more effective and outperform the remaining competitors.

\vspace{0.75ex}
\noindent\textbf{$\mathrm{(e)}$~\textit{Experiments with nonmonotone problems.}}
We also perform a similar test with four sets of experiments as in \textit{Experiment 1} and \textit{Experiment 2}, but on a class of possibly nonmonotone problem instances of \eqref{prob:minimax}.
We denote these experiments by \textit{Experiment 3} and \textit{Experiment 4}.

\begin{figure}[hpt!]
	\vspace{-0ex}
	\centering
	\includegraphics[width=\linewidth]{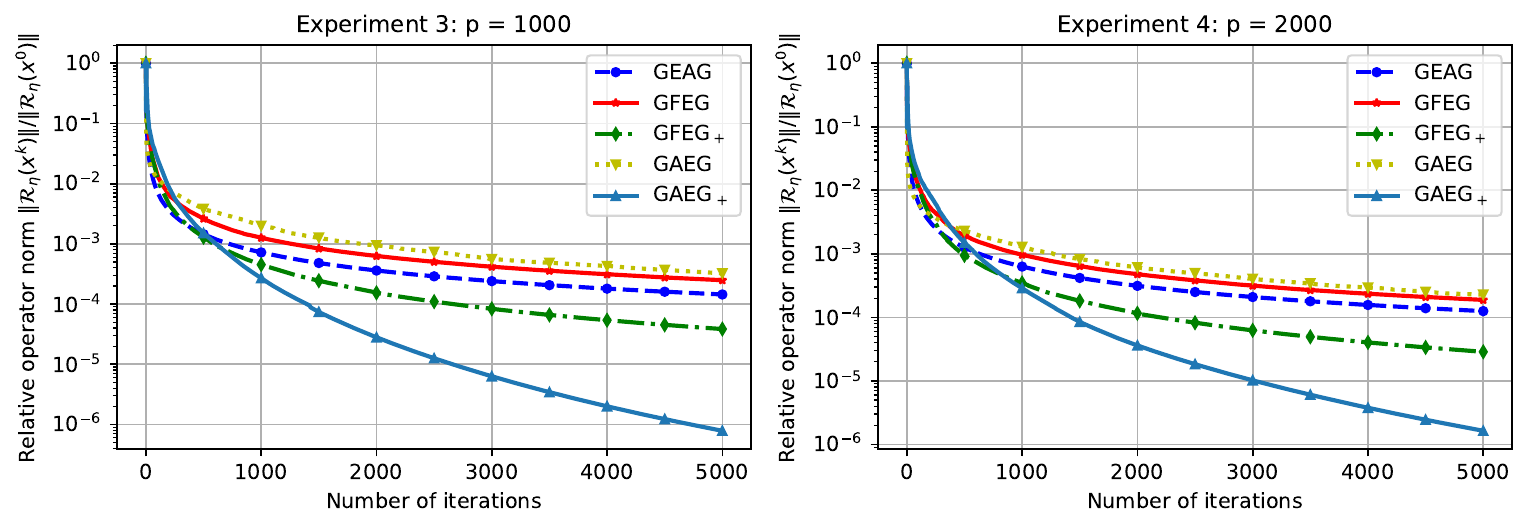}
	\caption{The behaviors of the first set of algorithms for solving \eqref{eq:NI} in \textit{Experiment 3} and \textit{Experiment 4}. 
	The plot reveals the mean of 10 problem instances.}
	\label{fig:constr_minimax_Fxk_2}
	\vspace{-3ex}
\end{figure}
\begin{figure}[hpt!]
	\vspace{-0ex}
	\centering
	\includegraphics[width=\linewidth]{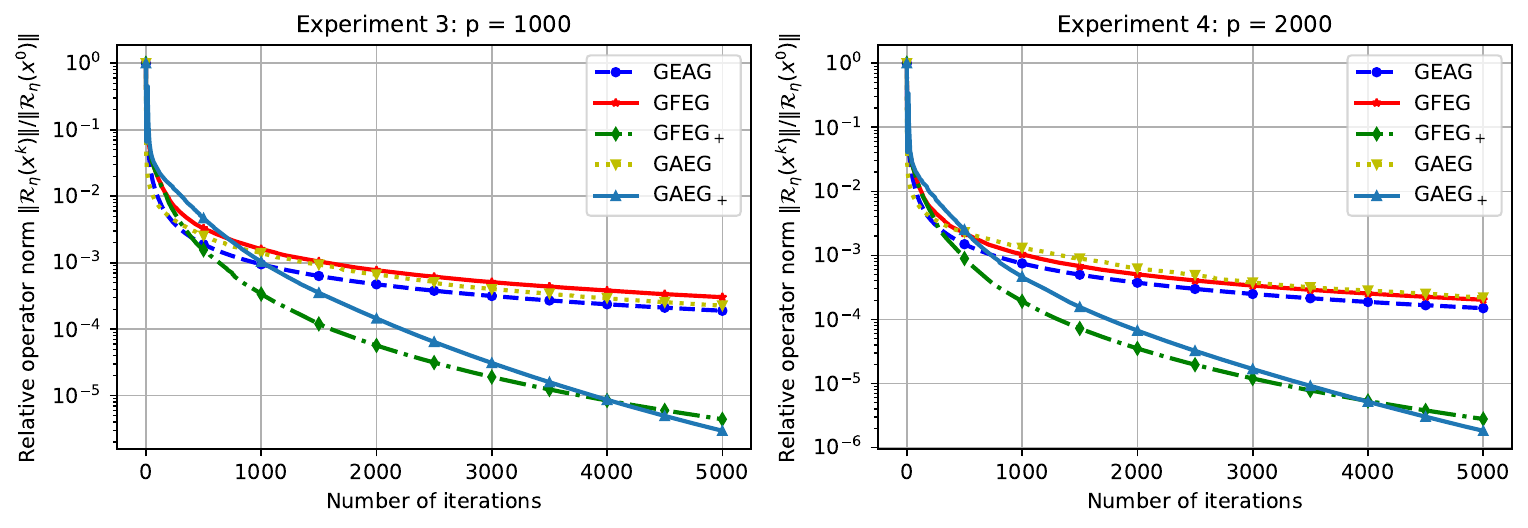}
	\caption{The behaviors of the second set of algorithms for solving \eqref{eq:NI} in \textit{Experiment 3} and \textit{Experiment 4}. 
	The plot reveals the mean of 10 problem instances.}
	\label{fig:constr_minimax_Fyk-1_2}
	\vspace{-1ex}
\end{figure}
\begin{figure}[hpt!]
	\vspace{-0ex}
	\centering
	\includegraphics[width=\linewidth]{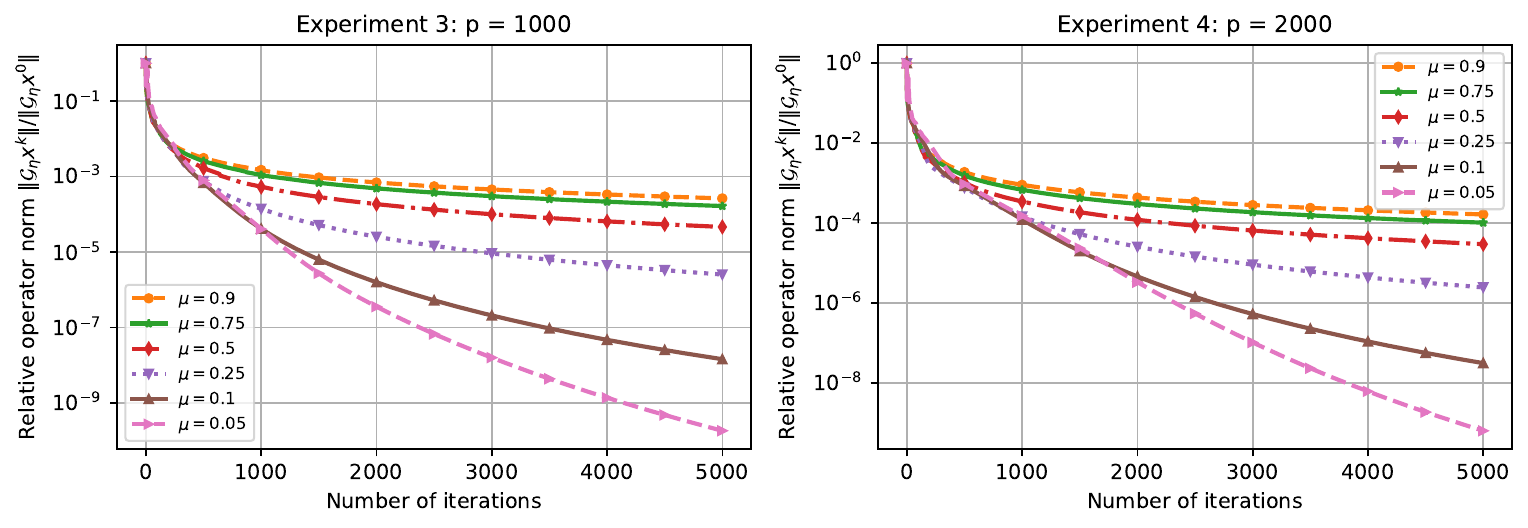}
	\caption{The behaviors of \eqref{eq:DFEG4NI} for solving \eqref{eq:NI} in \textit{Experiments 3} and \textit{4}. 
	The plot reveals the mean of 10 problem instances. 
	The legend presents the values of  $\mu$ used in the corresponding instant of \eqref{eq:DFEG4NI}.}
	\label{fig:constr_minimax_DFEG_2}
	\vspace{-1ex}
\end{figure}
\begin{figure}[hpt!]
	\vspace{-0ex}
	\centering
	\includegraphics[width=\linewidth]{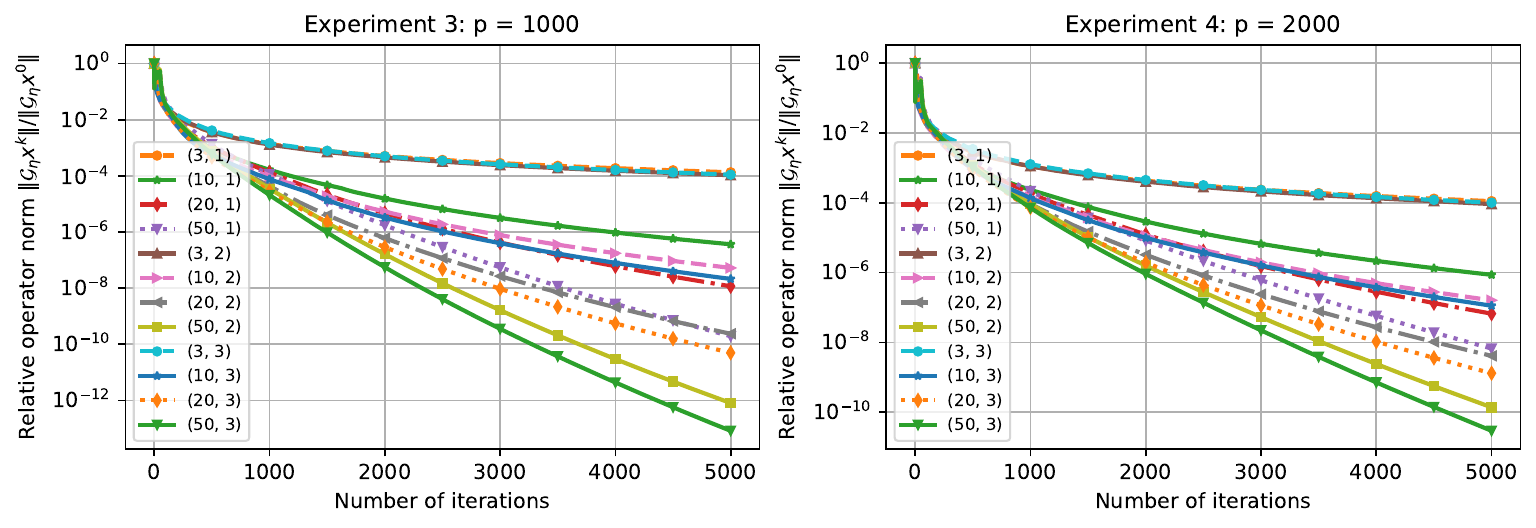}
	\caption{The behaviors of \eqref{eq:NGEAG4NI} for solving \eqref{eq:NI} in \textit{Experiments 3} and \textit{4}. 
	The plot reveals the mean of 10 problem instances. 
	The legend presents the values of $r$ and $\mu$ used in the corresponding instant of \eqref{eq:NGEAG4NI}. 
	For example, $(3, 1)$ means the solid orange line with filled circle marker corresponds to the instance of \eqref{eq:NGEAG4NI} using $r = 3$ and $\mu = 1$.}
	\label{fig:constr_minimax_GAEG_2}
	\vspace{-3ex}
\end{figure}

Despite the relaxation of the monotonicity of the involved operator, the results presented in Figures \ref{fig:constr_minimax_Fxk_2}, \ref{fig:constr_minimax_Fyk-1_2},  \ref{fig:constr_minimax_DFEG_2}, and \ref{fig:constr_minimax_GAEG_2} are still consistent with what were observed in the monotone scenarios in \textit{Experiment 1} and \textit{Experiment 2}.
These findings reinforce the reliability of the accelerated extragradient-type methods studied in this paper across diverse circumstances and futher establish their potential for applications where monotonicity is not guaranteed.

\vspace{1ex}
\noindent\textbf{Acknowledgements.}
This work is  partially supported by the National Science Foundation (NSF), grant no. NSF-RTG DMS-2134107 and the Office of Naval Research (ONR), grant No. N00014-23-1-2588 (2023-2026).

\appendix
\normalsize

\vspace{-0.5ex}
\beforesec
\section{Appendix 1: Useful Technical Lemmas}\label{apdx:sec:technical_results}
\aftersec
We need the following lemmas for our convergence analysis in the paper.

\begin{lemma}[see Lemma 5.31, \cite{Bauschke2011}]\label{le:A1_descent}
Let $\sets{u_k}$, $\sets{v_k}$, and $\sets{\gamma_k}$ be nonnegative sequences in $\R$ such that
\begin{equation}\label{eq:lm_A1_cond}
u_{k+1} \leq (1 + \gamma_k)u_k - v_k, \quad \forall k \geq 0.
\end{equation}
If $A := \sum_{k=0}^{\infty}\gamma_k < +\infty $ and $\omega := \prod_{k=0}^{\infty}(1 + \gamma_k) \in  [1, +\infty)$, then
\begin{equation}\label{eq:lm_A1_result1}
\arraycolsep=0.2em
\begin{array}{lcl}
u_{k+1} \leq  u_0\prod_{l=0}^k(1 + \gamma_l) + \sum_{l=0}^k\varepsilon_l\prod_{j=l+1}^{k}(1 + \gamma_j) \leq \omega(u_0 + B_k).
\end{array}
\vspace{-0.5ex}
\end{equation}
Moreover,  $\sets{u_k}$ converges $($i.e. $\lim_{k\to\infty}u_k$ exists$)$ and
\begin{equation}\label{eq:lm_A1_result2}
\arraycolsep=0.2em
\begin{array}{lcl}
\sum_{k=0}^{\infty} v_k & \leq & (1 + \omega A) u_0.
\end{array}
\end{equation}
\end{lemma}

\begin{lemma}[see \cite{TranDinh2025a}]\label{le:A2_sum_bounds}
For any $r > 0$ and $c > 0$, let $A_k := \sum_{l=0}^{\rmark{k}}\frac{1}{(l+r)^2}$ and $\omega_k := \prod_{l=0}^k\big( 1 + \frac{c}{(l+r)^2} \big)$.
Then,  we have $A_k \leq \frac{1}{r}$ and $\omega_k \leq e^{c/r}$ for $k\geq 0$.
\end{lemma}

\vspace{-0.5ex}
\beforesec
\section{The Proof of Technical Results in Section~\ref{sec:EAG4NI}}\label{apdx:sec:EAG4NI}
\aftersec
\vspace{-0.5ex}
This section provides the full proof of Lemmas and Theorems in Section~\ref{sec:EAG4NI}.

\vspace{-0.5ex}
\beforesubsec
\subsection{\mytb{The Proof of Lemma~\ref{le:EAG4NI_key_estimate1}}}\label{apdx:le:EAG4NI_key_estimate1}
\aftersubsec
\vspace{-0.5ex}
Since $T$ is maximally $3$-cyclically monotone, for any $\xi^{k+1} \in Tx^{k+1}$, $\xi^k \in Tx^k$, and $\zeta^k \in Ty^k$, we have
\begin{equation*} 
\arraycolsep=0.2em
\begin{array}{lcl}
\iprods{\xi^{k+1}, x^{k+1} - x^k} + \iprods{\xi^k, x^k - y^k} + \iprods{\zeta^k, y^k - x^{k+1}} \geq 0.
\end{array}
\end{equation*}
By the monotonicity of $F$, we also have $\iprods{Fx^{k+1} - Fx^k, x^{k+1} - x^k} \geq 0$.
Summing up this inequality and the last one, and then using $w^{k+1} = Fx^{k+1} + \xi^{k+1}$, $w^k := Fx^k + \xi^k$, and $\tilde{w}^k := Fx^k + \zeta^k$ from \eqref{eq:EAG4NI_ex2}, we obtain
\begin{equation}\label{eq:EAG4NI_proof1} 
\arraycolsep=0.2em
\begin{array}{lcl}
\iprods{w^{k+1}, x^{k+1} - x^k} - \iprods{\tilde{w}^k, x^{k+1} - x^k} + \iprods{w^k - \tilde{w}^k, x^k - y^k} \geq 0.
\end{array}
\end{equation}
Now, let us denote $e^k := u^k - Fx^k = \tilde{z}^k - \tilde{w}^k$.
Then, we get $\tilde{z}^k = \tilde{w}^k + e^k$. 
By this representation, we can derive from  \eqref{eq:EAG4NI_reform} that
\begin{equation*} 
\arraycolsep=0.2em
\left\{\begin{array}{lcl}
x^{k+1} - x^k &= & -\frac{\tau_k}{1 - \tau_k}(x^{k+1} - x^0) - \frac{ \eta }{1-\tau_k}\hat{w}^{k+1}, \vspace{1ex}\\
x^{k+1} - x^k &= & - \tau_k(x^k - x^0) - \eta \hat{w}^{k+1}, \vspace{1ex}\\
x^k - y^k &= & \tau_k(x^k - x^0) + \hat{\eta}_k\tilde{w}^k + \hat{\eta}_k e^k.
\end{array}\right.
\end{equation*}
Substituting these relations into \eqref{eq:EAG4NI_proof1}, then using Young's inequality in $\myeqc{1}$ for any $r > 0$, and rearranging terms, we arrive at
\begin{equation*} 
\arraycolsep=0.2em
\begin{array}{lcl}
\Tc_{[1]} & := & \tau_k\iprods{w^k, x^k - x^0}  -   \frac{\tau_k}{1-\tau_k}\iprods{w^{k+1}, x^{k+1} - x^0} \vspace{1ex}\\
& \geq &  \frac{\eta }{1-\tau_k}\iprods{w^{k+1}, \hat{w}^{k+1}} -  \eta \iprods{\tilde{w}^k, \hat{w}^{k+1}} - \hat{\eta}_k\iprods{w^k, \tilde{w}^k} + \hat{\eta}_k\norms{\tilde{w}^k}^2 \vspace{1ex}\\
&& - {~} \hat{\eta}_k\iprods{e^k, w^k - \tilde{w}^k} \vspace{1ex} \\
& \overset{\tiny\myeqc{1}}{\geq} & \frac{ \eta }{1-\tau_k}\iprods{w^{k+1}, \hat{w}^{k+1}} -  \eta \iprods{\tilde{w}^k, \hat{w}^{k+1}} - \hat{\eta}_k\iprods{w^k, \tilde{w}^k} + \hat{\eta}_k\norms{\tilde{w}^k}^2 \vspace{1ex}\\
&& - {~} \frac{r\hat{\eta}_k}{2}\norms{w^k - \tilde{w}^k}^2 - \frac{\hat{\eta}_k}{2r}\norms{e^k}^2.
\end{array}
\end{equation*}
Multiplying this inequality by $\frac{b_k}{\tau_k}$ and  using $b_{k+1} = \frac{b_k}{1-\tau_k}$ from \eqref{eq:EAG4NI_param_cond}, we can show that
\begin{equation}\label{eq:EAG4NI_proof2}  
\arraycolsep=0.2em
\begin{array}{lcl}
\Tc_{[2]}  &:=&  b_k\iprods{w^k, x^k - x^0} - b_{k+1}\iprods{w^{k+1}, x^{k+1} - x^0} \vspace{1ex}\\
&\geq & \frac{\eta b_{k+1} }{\tau_k}\iprods{w^{k+1} - \tilde{w}^k, \hat{w}^{k+1}} + \eta b_{k+1}  \iprods{\tilde{w}^k, \hat{w}^{k+1}} - \frac{b_k\hat{\eta}_k}{\tau_k}\iprods{w^k, \tilde{w}^k} \vspace{1ex}\\
&& + {~} \frac{b_k\hat{\eta}_k}{\tau_k}\norms{\tilde{w}^k}^2 - \frac{r b_k\hat{\eta}_k}{2\tau_k}\norms{w^k - \tilde{w}^k}^2 - \frac{b_k\hat{\eta}_k}{2r \tau_k}\norms{e^k}^2.
\end{array}
\end{equation}
Next, from  \eqref{eq:EAG4NI_reform}, we have $x^{k+1} - y^k = -\eta \hat{w}^{k+1} + \hat{\eta}_k\tilde{z}^k = -\eta \hat{w}^{k+1} + \hat{\eta}_k\tilde{w}^k + \hat{\eta}_k e^k$.
Using this expression, the $L$-Lipschitz continuity of $F$, and Young's inequality in $\myeqc{2}$ for any $c > 0$, we can derive that
\begin{equation*}
\arraycolsep=0.2em
\begin{array}{lcl}
\norms{w^{k+1} - \hat{w}^{k+1}}^2 & = & \norms{Fx^{k+1} - Fy^k}^2  \leq  L^2\norms{x^{k+1} - y^k}^2 \vspace{1ex}\\
&=& L^2\norms{\eta \hat{w}^{k+1} - \hat{\eta}_k\tilde{w}^k - \hat{\eta}_k e^k}^2 \vspace{1ex}\\
& \overset{\tiny\myeqc{2} }{ \leq }& (1+c)L^2\norms{\eta \hat{w}^{k+1} - \hat{\eta}_k\tilde{w}^k}^2 + \frac{(1+c)L^2\hat{\eta}_k^2}{c} \norms{e^k}^2.
\end{array}
\end{equation*}
Multiplying both sides of this inequality by $1+\omega$ for some $\omega > 0$ and rearranging the result, we get
\begin{equation*}
\arraycolsep=0.2em
\begin{array}{lcl}
0 &\geq& \norms{w^{k+1} - \hat{w}^{k+1}}^2 + \omega\norms{w^{k+1} - \hat{w}^{k+1}}^2 - (1+c)(1+\omega)L^2\norms{\eta \hat{w}^{k+1} - \hat{\eta}_k\tilde{w}^k}^2 \vspace{1ex}\\
&& - {~} \frac{(1+c)}{c}(1+\omega)L^2\hat{\eta}_k^2\norms{e^k}^2.
\end{array}
\end{equation*}
For simplicity of notations, we denote $M_c := (1+c)(1+\omega)L^2$.
Using this $M_c$ and expanding the last expression, we obtain
\begin{equation*}
\arraycolsep=0.2em
\begin{array}{lcl}
0 &\geq& \omega\norms{w^{k+1} - \hat{w}^{k+1}}^2 + \norms{w^{k+1}}^2 - 2\iprods{w^{k+1} - \tilde{w}^k, \hat{w}^{k+1}} + (1 - M_c\eta^2)\norms{\hat{w}^{k+1}}^2 \vspace{1ex}\\
&& - {~} 2(1 - M_c \eta \hat{\eta}_k)\iprods{\tilde{w}^k, \hat{w}^{k+1}} - M_c \hat{\eta}_k^2\norms{\tilde{w}^k}^2 - \frac{M_c\hat{\eta}_k^2}{c}\norms{e^k}^2.
\end{array}
\end{equation*}
Multiplying this inequality by $\frac{\eta b_{k+1}}{2\tau_k}$, adding the result to \eqref{eq:EAG4NI_proof2}, and using $\hat{\eta}_k = \eta (1-\tau_k)$ and $b_k = b_{k+1} (1-\tau_k)$ from \eqref{eq:EAG4NI_param_cond}, we can prove that
\begin{equation*} 
\arraycolsep=0.2em
\begin{array}{lcl}
\Tc_{[2]}  &:=&  b_k\iprods{w^k, x^k - x^0} - b_{k+1}\iprods{w^{k+1}, x^{k+1} - x^0} \vspace{1ex}\\
&\geq &  \frac{\eta b_{k+1} }{2\tau_k}\norms{w^{k+1}}^2 + \frac{\eta b_{k+1}(1 - M_c\eta^2)}{2\tau_k}\norms{\hat{w}^{k+1}}^2 + \frac{\omega \eta b_{k+1}}{2\tau_k}\norms{w^{k+1} - \hat{w}^{k+1}}^2 \vspace{1ex}\\
&& - {~} \frac{r \eta b_{k+1} (1-\tau_k)^2}{2\tau_k}\norms{w^k - \tilde{w}^k}^2 + \frac{\eta b_{k+1}(1-\tau_k)^2(2 - M_c\eta^2)}{2\tau_k}\norms{\tilde{w}^k}^2 \vspace{1ex}\\
&& - {~} \frac{\eta b_{k+1}(1 - M_c\eta^2)(1 - \tau_k)}{\tau_k}  \iprods{\tilde{w}^k, \hat{w}^{k+1}} - \frac{\eta b_{k+1} (1-\tau_k)^2}{\tau_k}\iprods{w^k, \tilde{w}^k} \vspace{1ex}\\
& & - {~} \frac{\eta b_{k+1} (1-\tau_k)^2(c + r M_c \eta^2)}{2rc \tau_k}\norms{e^k}^2. 
\end{array}
\end{equation*}
Utilizing the following two identities
\begin{equation*}
\arraycolsep=0.2em
\begin{array}{lcl}
2(1-\tau_k)\iprods{\tilde{w}^k, \hat{w}^{k+1}} &=& (1-\tau_k)^2\norms{\tilde{w}^k}^2 + \norms{\hat{w}^{k+1}}^2 -  \norms{\hat{w}^{k+1} - (1-\tau_k)\tilde{w}^k}^2, \vspace{1ex}\\
2\iprods{w^k, \tilde{w}^k} &=& \norms{w^k}^2 + \norms{\tilde{w}^k}^2 - \norms{w^k - \tilde{w}^k}^2,
\end{array}
\end{equation*}
we can further lower bound $\Tc_{[2]}$ above as 
\begin{equation*} 
\arraycolsep=0.2em
\begin{array}{lcl}
\Tc_{[2]} &\geq &  \frac{\eta b_{k+1} }{2\tau_k}\norms{w^{k+1}}^2  - \frac{\eta b_{k+1}(1-\tau_k)^2}{2\tau_k}\norms{w^k}^2  + \frac{\omega \eta b_{k+1}}{2\tau_k}\norms{w^{k+1} - \hat{w}^{k+1}}^2 \vspace{1ex}\\
&& + {~}   \frac{\eta b_{k+1}(1 - M_c\eta^2)}{2\tau_k} \norms{\hat{w}^{k+1} - (1-\tau_k)\tilde{w}^k}^2 + \frac{(1 - r)\eta b_{k+1}(1-\tau_k)^2}{2\tau_k}\norms{w^k - \tilde{w}^k}^2 \vspace{1ex}\\
& & - {~} \frac{\eta b_{k+1} (1-\tau_k)^2(c + r M_c \eta^2)}{2rc \tau_k}\norms{e^k}^2. 
\end{array}
\end{equation*}
Substituting $a_k := \frac{\eta b_k(1-\tau_k)}{\tau_k} = \frac{\eta b_{k+1}(1-\tau_k)^2}{\tau_k}$ and $\Vc_k$ from \eqref{eq:EAG4NI_potential_func} into the last estimate, we can show that
\begin{equation*} 
\arraycolsep=0.2em
\begin{array}{lcl}
\Vc_k - \Vc_{k+1} &\geq & \frac{\eta b_{k+1}}{2}\big( \frac{1}{\tau_k}  - \frac{ 1 - \tau_{k+1} }{\tau_{k+1} } \big) \norms{w^{k+1}}^2 +  \frac{\omega \eta b_{k+1}}{2\tau_k}\norms{w^{k+1} - \hat{w}^{k+1}}^2 \vspace{1ex}\\
&& + {~}   \frac{\eta b_{k+1}(1 - M_c\eta^2)}{2\tau_k} \norms{\hat{w}^{k+1} - (1-\tau_k)\tilde{w}^k}^2 \vspace{1ex}\\
& & + {~} \frac{(1 - r)a_k}{2}\norms{w^k - \tilde{w}^k}^2  -  \frac{(c + r M_c \eta^2)a_k}{2rc}\norms{e^k}^2,
\end{array}
\end{equation*}
which proves \eqref{eq:EAG4NI_key_estimate1} after replacing $e^k := u^k - Fx^k$.
\Eproof

\beforesubsec
\subsection{\mytb{The Proof of Theorem~\ref{th:EAG4NI_convergence}}}\label{apdx:th:EAG4NI_convergence}
\aftersubsec
First, employing  $\Lc_k$ from \eqref{eq:EAG4NI_potential_func}, and substituting $\tau_k := \frac{1}{k + \nu}$ and $c_k :=  \frac{(c + r M_c \eta^2)a_k }{rc} = \frac{(c + r M_c \eta^2) \eta b_k(k + \nu  - 1) }{rc}$ into \eqref{eq:EAG4NI_key_estimate1}, we can prove that
\begin{equation*} 
\hspace{-1ex}
\arraycolsep=0.1em
\begin{array}{lcl}
\Lc_k - \Lc_{k+1} &\geq &  \frac{\omega \eta b_{k+1}(k + \nu )}{2}\norms{w^{k+1}  \! - \! \hat{w}^{k+1}}^2 +  \frac{(1 - M_c\eta^2)\eta b_{k+1}(k + \nu )}{2} \norms{\hat{w}^{k+1} \! - \! (1-\tau_k)\tilde{w}^k}^2 \vspace{1ex}\\
&& + {~} \frac{(1-r)a_k}{2} \norms{w^k - \tilde{w}^k}^2 -   \frac{ (c + r M_c \eta^2)\eta b_{k+1} (k + \nu)}{2rc} \norms{ u^{k+1} - Fx^{k+1} }^2.
\end{array}
\hspace{-3ex}
\end{equation*}
Next, combining  the last estimate and \eqref{eq:EAG4NI_uk_cond2}, and using $w^{k+1} - \hat{w}^{k+1} = Fx^{k+1} - Fy^k$, we can show that
\begin{equation*} 
\hspace{-1ex}
\arraycolsep=0.1em
\begin{array}{lcl}
\Lc_k - \Lc_{k+1} &\geq &  \frac{\eta b_{k+1}(k + \nu)}{2} \big[ \omega -  \frac{\kappa(c + rM_c\eta^2)}{rc}  \big]  \norms{w^{k+1} - \hat{w}^{k+1}}^2 + \frac{(1-r)a_k}{2}  \norms{w^k - \tilde{w}^k}^2 \vspace{1ex}\\
&& + {~}   \frac{\eta b_{k+1}(k + \nu)}{2}  \big[   1 -  M_c\eta^2 - \frac{\hat{\kappa}\eta^2 (c + rM_c\eta^2)}{rc} \big] \norms{\hat{w}^{k+1} - (1-\tau_k)\tilde{w}^k}^2.
\end{array}
\hspace{-3ex}
\end{equation*}
If we choose $\eta$ and $r$ such that $r \in (0, 1]$, $M_c\eta^2 \leq 1$, and 
\begin{equation}\label{eq:EAG4NI_para_cond2}
\begin{array}{lcl}
\omega \geq \frac{\kappa(c + rM_c\eta^2)}{rc} \quad \text{and} \quad M_c\eta^2 + \frac{\hat{ \kappa} \eta^2 (c + rM_c\eta^2)}{rc} \leq 1,
\end{array}
\end{equation}
then the last estimate leads to $\Lc_{k+1} \leq \Lc_k$ for all $k \geq 0$.

The first condition of \eqref{eq:EAG4NI_para_cond2} holds if $0 < \eta \leq \frac{\sqrt{c (r\omega - \kappa)} }{\sqrt{ r\kappa M_c}}$, where $\kappa \leq r\omega$ due to the choice of $\omega$.
The second condition of \eqref{eq:EAG4NI_para_cond2} holds if $( M_c + \hat{\kappa}\omega)\eta^2 \leq 1$, which is equivalent to $\eta \leq \frac{1}{\sqrt{M_c + \hat{\kappa}\omega}}$.
Now, for any $r \in (0, 1]$, let us choose $\omega := \frac{2\kappa}{r}$ and $c := r$.
Then, we have $M_c + \hat{\kappa}\omega \geq \frac{ r\kappa M_c}{c(r\omega - \kappa)} = M_c$.
Therefore, if we choose $0 < \eta \leq \bar{\eta} := \frac{\sqrt{r}}{\sqrt{(1+r)(r + 2\kappa)L^2 + 2\kappa\hat{\kappa}}}$, then $\eta$ satisfies both conditions above.
Clearly, this $\bar{\eta}$ is given by \eqref{eq:EAG4NI_param_update}. 

Finally, since $b_{k+1} = \frac{b_k}{1-\tau_k} = \frac{b_k(k + \nu)}{k + \nu - 1}$, by induction, we obtain $b_k = \frac{b_0(k+ \nu -1)}{\nu -1}$ for some $b_0 > 0$.
Moreover, we also have
\begin{equation*}
\arraycolsep=0.2em
\begin{array}{lcl}
a_k = \frac{\eta b_0(k+\nu-1)^2}{\nu-1}  \quad \text{and} \quad c_k = \frac{\eta b_0 (c + rM_c\eta^2)(k+\nu-1)^2}{rc(\nu-1)}.
\end{array}
\end{equation*}
Using these expressions, $\iprods{w^k, x^k - x^{\star}} \geq 0$ for any $x^{\star} \in \zer{\Phi}$, and Young's inequality in $\myeqc{1}$, we can lower bound $\Vc_k$ from \eqref{eq:EAG4NI_potential_func} as
\begin{equation*}
\arraycolsep=0.2em
\begin{array}{lcl}
\Vc_k & = & \frac{b_0\eta(k+\nu-1)^2}{2(\nu-1)}\norms{w^k}^2 + \frac{b_0(k+\nu-1)}{\nu-1}\big[ \iprods{w^k, x^{\star} - x^0} + \iprods{w^k, x^k - x^{\star}} \big] \vspace{1ex}\\
&& + {~} \frac{b_0}{\eta(\nu-1)}\norms{x^0 - x^{\star}}^2 \vspace{1ex}\\
& \overset{\tiny\myeqc{1}}{\geq} & \frac{b_0\eta(k+\nu-1)^2}{4(\nu-1)}\norms{w^k}^2.
\end{array}
\end{equation*}
Since we have chosen $u^0 := Fx^0$, from \eqref{eq:EAG4NI_potential_func}, we have $\Lc_0 := \Vc_0 + \frac{c_0}{2}\norms{Fx^0 - u^0}^2 = \frac{\eta b_0}{4(\nu-1)}\norms{w^0}^2 + \frac{ b_0 }{\eta(\nu-1)}\norms{x^0 - x^{\star}}^2 = \frac{b_0}{4\eta(\nu-1)}\big[ \eta^2 \norms{w^0}^2 + 4\norms{x^0 - x^{\star} }^2 \big]$.
Putting these bounds together, we can derive that
\begin{equation*}
\arraycolsep=0.2em
\begin{array}{lcl}
\frac{b_0\eta(k+\nu-1)^2}{4(\nu-1)}\norms{w^k}^2 & \leq & \Vc_k \leq \Lc_k \leq \Lc_0 =  \frac{b_0}{4\eta(\nu-1)}\big[ \eta^2 \norms{w^0}^2 + 4\norms{x^0 - x^{\star} }^2 \big],
\end{array}
\end{equation*}
which is exactly \eqref{eq:EAG4NI_convergence1} after substituting $w^0 := Fx^0 + \xi^0$ and $w^k := Fx^k + \xi^k$ for $\xi^0 \in Tx^0$ and $\xi^k \in Tx^k$, respectively.
\Eproof

\beforesec
\section{The Proof of Technical Results in Section~\ref{sec:FEG4NI}}\label{apdx:sec:FEG4NI}
\aftersec
This appendix provides the full proof of Lemmas and Theorems in Section~\ref{sec:FEG4NI}.

\beforesubsec
\subsection{\mytb{The Proof of Lemma~\ref{le:FEG4NI_key_estimate1}}}\label{apdx:le:FEG4NI_key_estimate1}
\aftersubsec
First, let us denote $e^k := u^k - Fx^k = z^k - w^k$.
Then, from the second line of \eqref{eq:FEG4NI}, we can easily show that
\begin{equation}\label{eq:FEG4NI_proof1}
\arraycolsep=0.2em
\left\{\begin{array}{lcl}
x^{k+1} - x^k &= & -\frac{\tau_k}{1 - \tau_k}(x^{k+1} - x^0) - \frac{ \eta }{1-\tau_k}\hat{w}^{k+1} + \frac{\beta_k}{1-\tau_k}w^k + \frac{\beta_k}{1-\tau_k}e^k , \vspace{1ex}\\
x^{k+1} - x^k &= & - \tau_k(x^k - x^0) - \eta \hat{w}^{k+1} + \beta_kw^k + \beta_k e^k.
\end{array}\right.
\end{equation}
Next, since $\Phi$ is $\rho$-co-hypomonotone, we have $\iprods{w^{k+1}, x^{k+1} - x^k} - \iprods{w^k, x^{k+1} - x^k} \geq -\rho\norms{w^{k+1} - w^k}^2$.
Substituting \eqref{eq:FEG4NI_proof1} into this inequality, we can derive
\begin{equation*}
\arraycolsep=0.1em
\begin{array}{lcl}
\tau_k\iprods{w^k, x^k - x^0}  & - &  \frac{\tau_k}{1-\tau_k}\iprods{w^{k+1}, x^{k+1} - x^0}  \geq   -\rho\norms{w^{k+1} - w^k}^2 \vspace{1ex}\\
&& + {~} \frac{\eta }{1-\tau_k}\iprods{w^{k+1}, \hat{w}^{k+1}} - \frac{\beta_k}{1-\tau_k}\iprods{w^{k+1}, w^k} - \eta \iprods{w^k, \hat{w}^{k+1}}  \vspace{1ex}\\
&& + {~} \beta_k\norms{w^k}^2 - \frac{\beta_k}{1-\tau_k}\iprods{w^{k+1} - (1-\tau_k)w^k, e^k}.
\end{array}
\end{equation*}
Now, multiplying this inequality by $\frac{b_k}{\tau_k}$ and using $b_{k+1} = \frac{b_k}{1-\tau_k}$, and then applying Young's inequality $\iprods{u, v} \leq \frac{\gamma}{2}\norms{u}^2 + \frac{1}{2\gamma}\norms{v}^2$ for $\gamma > 0$ to $\myeqc{1}$, we get
\begin{equation*} 
\arraycolsep=0.2em
\begin{array}{lcl}
\hat{\Tc}_{[1]} &:= & b_k\iprods{w^k, x^k - x^0} -  b_{k+1} \iprods{w^{k+1}, x^{k+1} - x^0} \vspace{1ex}\\
& \overset{\tiny\myeqc{1}}{\geq} &   \frac{ \eta b_k }{\tau_k( 1-\tau_k)}\iprods{w^{k+1}, \hat{w}^{k+1}} -  \frac{b_k\beta_k}{\tau_k( 1 - \tau_k) } \iprods{w^{k+1}, w^k}  - \frac{ \eta b_k}{\tau_k} \iprods{w^k, \hat{w}^{k+1}}   \vspace{1ex}\\
&& + {~}  \frac{b_k\beta_k}{\tau_k}\norms{w^k}^2 - \frac{b_k \rho }{\tau_k} \norms{w^{k+1} - w^k}^2  - \frac{b_k\beta_k}{2\gamma \tau_k(1-\tau_k)}\norms{e^k}^2 \vspace{1ex}\\
&& - {~} \frac{\gamma b_k\beta_k}{2\tau_k(1-\tau_k)}\norms{w^{k+1} - (1-\tau_k)w^k}^2.
\end{array}
\end{equation*}
Then, utilizing the following two elementary identities:
\begin{equation*}
\arraycolsep=0.1em
\begin{array}{lcl}
2\iprods{w^{k+1}, w^k} & = &  \norms{w^{k+1}}^2 + \norms{w^k}^2 - \norms{w^{k+1} - w^k}^2, \vspace{1ex}\\
\norms{w^{k+1} \! - \! (1-\tau_k)w^k}^2 & = &  \tau_k \norms{w^{k+1}}^2 - \tau_k(1 - \tau_k) \norms{w^k}^2 + (1 \! - \! \tau_k) \norms{w^{k+1} \! - \! w^k}^2, 
\end{array}
\end{equation*}
the last inequality leads to
\begin{equation}\label{eq:FEG4NI_proof2}
\hspace{-2ex}
\arraycolsep=0.1em
\begin{array}{lcl}
\hat{\Tc}_{[1]} &:= & b_k\iprods{w^k, x^k - x^0} -  b_{k+1} \iprods{w^{k+1}, x^{k+1} - x^0} \vspace{1ex}\\
& \geq & \frac{ \eta b_k }{\tau_k( 1-\tau_k)}\iprods{w^{k+1} - w^k, \hat{w}^{k+1}} + \frac{ \eta b_k }{1-\tau_k} \iprods{w^k, \hat{w}^{k+1}}  \vspace{1ex}\\
&& - {~} \frac{b_k\beta_k (1 + \gamma \tau_k) }{2\tau_k(1-\tau_k)}\norms{w^{k+1}}^2 +  \frac{b_k\beta_k[ 1 - 2\tau_k + \gamma \tau_k(1-\tau_k)] }{2 \tau_k(1-\tau_k)}\norms{w^k}^2  \vspace{1ex}\\
&& + {~}  \frac{b_k[\beta_k (1 - \gamma + \gamma \tau_k) - 2\rho(1 - \tau_k) ]}{2 \tau_k(1-\tau_k) }  \norms{w^{k+1} - w^k}^2 - \frac{b_k\beta_k}{2\gamma \tau_k(1-\tau_k)}\norms{e^k}^2.
\end{array}
\hspace{-2ex}
\end{equation}
Since $x^{k+1} - y^k = \hat{\eta}_k w^k  -\eta \hat{w}^{k+1} + \hat{\eta}_ke^k$ from   \eqref{eq:FEG4NI}, by the Lipschitz continuity of $F$ and Young's inequality in $\myeqc{2}$, for any $r > 0$, we have
\begin{equation*}
\arraycolsep=0.2em
\begin{array}{lcl}
\norms{w^{k+1} - \hat{w}^{k+1}}^2 &= & \norms{Fx^{k+1} - Fy^k}^2 \leq L^2\norms{x^{k+1} - y^k}^2 \vspace{1ex}\\
& = &  L^2\norms{ \hat{\eta}_kw^k  -\eta \hat{w}^{k+1} + \hat{\eta}_ke^k }^2 \vspace{1ex}\\
& \overset{\tiny\myeqc{2}}{\leq} & (1+r)L^2\norms{\eta  \hat{w}^{k+1} - \hat{\eta}_kw^k}^2 + \frac{(1+r)L^2\hat{\eta}_k^2}{r}\norms{e^k}^2.
\end{array}
\end{equation*}
Since $\norms{w^{k+1} - \hat{w}^{k+1}}^2 = \mu\norms{w^{k+1} - \hat{w}^{k+1}}^2 + (1-\mu)\norms{w^{k+1} - \hat{w}^{k+1}}^2$ for any $\mu \in [0, 1]$, partly expanding the last estimate yields
\begin{equation*}
\arraycolsep=0.2em
\begin{array}{lcl}
\mu \norms{w^{k+1}}^2 & + &  \mu \norms{\hat{w}^{k+1}}^2 - 2 \mu \iprods{w^{k+1}, \hat{w}^{k+1}} + (1-\mu)\norms{w^{k+1} - \hat{w}^{k+1}}^2 \vspace{1ex}\\
& \leq & (1+r)L^2\eta^2 \norms{\hat{w}^{k+1}}^2 + (1+r)L^2\hat{\eta}_k^2\norms{w^k}^2 \vspace{1ex}\\
&& - {~} 2(1+r)L^2\eta \hat{\eta}_k\iprods{w^k, \hat{w}^{k+1}}  +  \frac{ (1+r) L^2\hat{\eta}_k^2 }{r}\norms{e^k}^2.
\end{array}
\end{equation*}
Rearranging this inequality, and multiplying the result by $\frac{\eta b_k }{2\mu\tau_k(1-\tau_k)}$, we get
\begin{equation*}
\arraycolsep=0.2em
\begin{array}{lcl}
\frac{(1 + r) L^2\eta b_k\hat{\eta}_k^2}{2r\mu \tau_k (1- \tau_k) }\norms{e^k}^2 & \geq &   \frac{(1-\mu)\eta b_k }{2\mu\tau_k(1-\tau_k)}\norms{w^{k+1} - \hat{w}^{k+1}}^2 - \frac{(1+r)L^2\eta b_k\hat{\eta}_k^2 }{2\mu\tau_k (1-\tau_k)}  \norms{w^k}^2   \vspace{1ex}\\
&&   + {~}  \frac{\eta b_k }{2 \tau_k(1-\tau_k)}  \norms{w^{k+1}}^2  +  \frac{\eta b_k [ \mu  - (1+r)L^2\eta^2 ] }{2\mu\tau_k(1-\tau_k)} \norms{\hat{w}^{k+1}}^2 \vspace{1ex}\\
&& - {~} \frac{\eta b_k [ \mu - (1+r)L^2\eta \hat{\eta}_k ] }{\mu\tau_k(1-\tau_k)} \iprods{w^k, \hat{w}^{k+1}} \vspace{1ex}\\
&& - {~} \frac{\eta b_k }{\tau_k(1-\tau_k)} \iprods{w^{k+1} - w^k, \hat{w}^{k+1}}.
\end{array}
\end{equation*}
Adding the last inequality to $\hat{\Tc}_{[1]}$ in \eqref{eq:FEG4NI_proof2}, we can derive that
\begin{equation*} 
\arraycolsep=0.2em
\begin{array}{lcl}
\hat{\Tc}_{[2]} &:= & b_k\iprods{w^k, x^k - x^0} - b_{k+1}\iprods{w^{k+1}, x^{k+1} - x^0}  \vspace{1ex}\\
&& + {~} \frac{b_k}{2\tau_k(1-\tau_k)}\big[ \frac{(1+r)L^2 \eta \hat{\eta}_k^2 }{r\mu } + \frac{\beta_k}{\gamma } \big] \norms{e^k}^2 \vspace{1ex}\\
& \geq &  \frac{b_k[ \eta - \beta_k(1 + \gamma \tau_k)] }{2 \tau_k(1-\tau_k)} \norms{w^{k+1}}^2 + \frac{b_k [ \mu \beta_k(1 - 2\tau_k + \gamma \tau_k(1-\tau_k) ) - (1+r)L^2\eta \hat{\eta}_k^2 ] }{2 \mu \tau_k(1-\tau_k)}  \norms{w^k}^2 \vspace{1ex}\\
&& + {~}  \frac{ \eta b_k  [\mu  - (1+r)L^2\eta^2 ] }{2\mu\tau_k(1-\tau_k)}\norms{\hat{w}^{k+1}}^2 - \frac{ \eta b_k [ \mu (1-\tau_k)  - (1+r)L^2\eta \hat{\eta}_k ] }{\mu \tau_k(1 - \tau_k) } \iprods{w^k, \hat{w}^{k+1}}  \vspace{1ex}\\
&&  + {~} \frac{(1-\mu)\eta b_k}{2\mu\tau_k(1-\tau_k)}\norms{w^{k+1} - \hat{w}^{k+1}}^2   +   \frac{b_k[ \beta_k(1 - \gamma + \gamma \tau_k) - 2\rho(1-\tau_k) ]}{2 \tau_k(1-\tau_k) }  \norms{w^{k+1} - w^k}^2.
\end{array}
\end{equation*}
Finally, substituting $\hat{\eta}_k = \eta (1-\tau_k)$ and $b_{k+1} = \frac{b_k}{1 - \tau_k}$ into $\hat{\Tc}_{[2]}$, then using 
\begin{equation*}
\arraycolsep=0.2em
\begin{array}{lcl}
2(1-\tau_k)\iprods{w^k, \hat{w}^{k+1}} = (1-\tau_k)^2\norms{w^k}^2 + \norms{\hat{w}^{k+1}}^2 - \norms{\hat{w}^{k+1} - (1-\tau_k)w^k}^2
\end{array}
\end{equation*}
and rearranging the result, we eventually arrive at
\begin{equation*} 
\arraycolsep=0.1em
\begin{array}{lcl}
\hat{\Tc}_{[3]} &:= &  \frac{ b_k[ \eta (1-\tau_k)^2 - \beta_k(1 - 2\tau_k + \gamma \tau_k(1 - \tau_k) ) ] }{2 \tau_k(1-\tau_k)} \norms{w^k}^2 +  b_k\iprods{w^k, x^k - x^0} \vspace{1ex}\\
&&  + {~} \frac{ b_k [ (1 + r) L^2 \eta^3 \gamma (1-\tau_k)^2 +  r\mu \beta_k ]}{2r\mu\gamma \tau_k(1-\tau_k)} \norms{e^k}^2 \vspace{1ex}\\
& \geq & \frac{b_{k+1}[ \eta - \beta_k(1 + \gamma \tau_k ) ]}{2 \tau_k}\norms{w^{k+1}}^2 + b_{k+1}\iprods{w^{k+1}, x^{k+1} - x^0} \vspace{1ex}\\
&&  + {~} \frac{ (1-\mu) \eta b_{k+1}}{2\mu\tau_k }\norms{w^{k+1} - \hat{w}^{k+1}}^2    +  \frac{ \eta b_{k+1} [ \mu - (1 + r)L^2\eta^2 ] }{2\mu\tau_k } \norms{\hat{w}^{k+1} - (1-\tau_k)w^k}^2 \vspace{1ex}\\
&& + {~}   \frac{b_{k+1 } [ \beta_k(1 - \gamma + \gamma \tau_k) - 2\rho(1-\tau_k) ]}{2 \tau_k }  \norms{w^{k+1} - w^k}^2.
\end{array}
\end{equation*}
Using the notations from \eqref{eq:EAG4NI_Lyapunov_coefficients} in $\hat{\Tc}_{[3]}$ and rearranging the result yields \eqref{eq:FEG4NI_key_estimate1}.
\Eproof

\beforesubsec
\subsection{\mytb{The Proof of Lemma~\ref{le:FEG4NI_key_est2}}}\label{apdx:le:FEG4NI_key_est2}
\aftersubsec
Substituting \eqref{eq:FEG4NI_u_cond} into \eqref{eq:FEG4NI_key_estimate1} with a notice that $w^k - \hat{w}^k = Fx^k - Fy^{k-1}$, and using \eqref{eq:FEG4NI_key_est2_cond}, we can derive
\begin{equation*} 
\arraycolsep= 0.1em
\begin{array}{lcl}
\Lc_k  &  \overset{ \tiny\eqref{eq:FEG4NI_Lyapunov_func} }{ = } & \frac{a_k}{2} \norms{w^k}^2  + b_k \iprods{ w^k, x^k - x^0 } + \frac{c_k}{2}\norms{Fx^k - u^k}^2 \vspace{1ex} \\ 
& \overset{ \tiny \eqref{eq:FEG4NI_key_est2_cond} }{ \geq } & \frac{ \hat{a}_k }{ 2}  \norms{w^k}^2 +  b_k\iprods{w^k, x^k - x^0} + \frac{c_k }{2} \norms{Fx^k - u^k}^2 \vspace{0ex}\\ 
& \overset{\tiny \eqref{eq:FEG4NI_key_estimate1} }{ \geq } & \frac{ a_{k+1} }{2}  \norms{w^{k+1}}^2 + b_{k+1}\iprods{w^{k+1}, x^{k+1} - x^0} + \frac{\alpha_{k+1} }{2} \norms{w^{k+1} - \hat{w}^{k+1}}^2  \vspace{1ex}\\
&&  + {~}   \frac{\hat{\alpha}_{k+1}  }{2}  \norms{w^{k+1} - w^k}^2 +  \frac{ \delta_{k+1} }{ 2 }  \norms{\hat{w}^{k+1} - (1-\tau_k)w^k}^2 \vspace{1ex}\\
& \overset{\tiny \eqref{eq:FEG4NI_key_est2_cond}}{\geq } & \frac{a_{k+1}}{2 } \norms{w^{k+1}}^2 + b_{k+1}\iprods{w^{k+1}, x^{k+1} - x^0}  + \frac{\kappa c_{k+1} }{2} \norms{Fx^{k+1} - Fy^k}^2   \vspace{1ex}\\
&& + {~}   \frac{\hat{\kappa}c_{k+1} }{2}   \norms{w^{k+1} - w^k}^2 +  \frac{ \delta_{k+1}  }{2} \norms{\hat{w}^{k+1} - (1-\tau_k)w^k}^2 \vspace{1ex}\\
&  \overset{ \tiny\eqref{eq:FEG4NI_u_cond} }{ \geq } &  \frac{a_{k+1}}{2 } \norms{w^{k+1}}^2 + b_{k+1}\iprods{w^{k+1}, x^{k+1} - x^0}  + \frac{c_{k+1} }{2} \norms{Fx^{k+1} - u^{k+1}}^2   \vspace{1ex}\\
&& + {~}  \frac{ \delta_{k+1} }{2} \norms{\hat{w}^{k+1} - (1-\tau_k)w^k}^2 \vspace{1ex}\\
&  \overset{ \tiny\eqref{eq:FEG4NI_Lyapunov_func} }{ = } & \Lc_{k+1} +  \frac{ \delta_{k+1} }{ 2 } \norms{\hat{w}^{k+1} - (1-\tau_k)w^k}^2,
\end{array}
\end{equation*}
which proves \eqref{eq:FEG4NI_key_est2}.
\Eproof

\beforesubsec
\subsection{\mytb{The Proof of Lemma~\ref{le:FEG4NI_choice_of_params}}}\label{apdx:le:FEG4NI_choice_of_params}
\aftersubsec
Let us choose $\tau_k := \frac{1}{k + \nu}$ for some $\nu > 1$ and {$\beta_k := \beta(1- \tau_k)$} as in \eqref{eq:FEG4NI_choice_of_params}.
We verify the first condition $a_k \geq \hat{a}_k$ of  \eqref{eq:FEG4NI_key_est2_cond}, which is equivalent to
\begin{equation*}
\arraycolsep=-0.1em
\begin{array}{lcl}
\eta\big(\frac{1}{\tau_{k-1}} - \frac{1}{\tau_k} + 1\big) \geq  \beta_{k-1}\big( \frac{1}{\tau_{k-1}} + \gamma\big) - \beta_k\big[  \frac{1-2\tau_k}{\tau_k(1-\tau_k)} + \gamma \big].
\end{array}
\end{equation*}
Since  $\tau_k := \frac{1}{k + \nu}$ and $\beta_k := \beta(1-\tau_k)$, this condition automatically holds.

We still need $\hat{a}_k \geq 0$, which is guaranteed if $\eta \geq  \beta\big( 1 - \frac{\tau_k}{1-\tau_k} + \gamma \tau_k \big)$.
Since this holds for all $k\geq 0$,  we impose a stricter condition $\beta \leq \frac{\nu \eta }{\nu + \gamma}$ as in \eqref{eq:FEG4NI_choice_of_beta}.

Next, the second condition $\alpha_k \geq \kappa c_k$ of \eqref{eq:FEG4NI_key_est2_cond} is equivalent to
\begin{equation*}
\arraycolsep=-0.1em
\begin{array}{lcl}
\beta \leq \frac{ \gamma \eta (k + \nu - 1)}{\kappa\mu(k + \nu)} \Big[ 1 - \mu  - \frac{ \kappa  (1 + r) L^2 \eta^2 }{r} \Big],
\end{array}
\end{equation*}
provided that $ \kappa  (1 + r) L^2 \eta^2 < (1-\mu)r$.
Since this condition must hold for all $k \geq 0$, we impose $\beta \leq \frac{ \gamma \eta (\nu - 1)}{ \kappa\mu \nu} \left[ 1 - \mu  - \frac{ \kappa  (1 + r) L^2 \eta^2}{r} \right]$ as in \eqref{eq:FEG4NI_choice_of_beta}.

Finally, the third condition $\hat{\alpha}_k \geq \kappa c_k$  of \eqref{eq:FEG4NI_key_est2_cond} becomes 
\begin{equation*}
\arraycolsep=-0.1em
\begin{array}{lcl}
\beta \Big[ 1- \gamma + \gamma \tau_{k-1} - \frac{\hat{\kappa} \tau_{k-1}}{\gamma \tau_k(1-\tau_{k-1})} \Big] \geq \frac{\hat{\kappa}(1+r)L^2\eta^3}{r\mu} \cdot \frac{(1-\tau_k)\tau_{k-1}}{(1-\tau_{k-1})\tau_k} + 2\rho.
\end{array}
\end{equation*}
Since it must hold for all $k \geq 0$, we impose 
\begin{equation*} 
\arraycolsep=0.2em
\begin{array}{lcl}
\beta &\geq& \frac{(\nu - 2)\gamma} { (\nu - 2)\gamma(1-\gamma) - \nu \hat{\kappa} } \Big[ \frac{(\nu - 1) \hat{\kappa}  (1 + r) L^2 \eta^3 }{(\nu - 2) r\mu } + 2\rho \Big],
\end{array} 
\end{equation*}
provided that  $(\nu - 2)\gamma(1-\gamma) > \nu \hat{\kappa}$.
This is the first condition in  \eqref{eq:FEG4NI_choice_of_beta}, which guarantees  the third condition of \eqref{eq:FEG4NI_key_est2_cond} to hold.

Overall, we have shown that  $\tau_k$ and $\beta_k$ chosen by \eqref{eq:FEG4NI_choice_of_params}  satisfy \eqref{eq:FEG4NI_key_est2_cond}, where $\beta$ satisfies \eqref{eq:FEG4NI_choice_of_beta}.
Therefore, we still obtain \eqref{eq:FEG4NI_key_est2}.
\Eproof

\beforesubsec
\subsection{\mytb{The Proof of Theorem~\ref{th:FEG4NI_convergence}}}\label{apdx:th:FEG4NI_convergence}
\aftersubsec
First, for the right-hand side of \eqref{eq:FEG4NI_choice_of_beta} to be well-defined and for $\delta_{k+1} \geq 0$ in \eqref{eq:FEG4NI_key_est2}, we need to impose the following conditions:
\begin{equation}\label{eq:FEG4NI_param_choice3}
\arraycolsep=0.2em
\begin{array}{ll}
L^2\eta^2 \leq   \frac{(1-\mu)r}{(1+r)\kappa}, \quad L^2\eta^2 \leq \frac{\mu}{1 + r}, \ \ \text{and} \ \  \nu \hat{\kappa} < (\nu - 2)(1-\gamma)\gamma.
\end{array}
\end{equation}
Furthermore, for $\nu > 2$, we can impose a stricter condition $\beta < \frac{\eta}{1+\gamma} < \frac{\nu \eta}{\nu + \gamma}$ to guarantee the third line of \eqref{eq:FEG4NI_choice_of_beta}.
Now, we consider the following two cases.

\noindent\textbf{Case 1.} If $\kappa = 0$, then the first condition of \eqref{eq:FEG4NI_param_choice3} automatically holds.
The second condition of  \eqref{eq:FEG4NI_param_choice3}  becomes $\eta \leq \frac{\sqrt{\mu}}{ L \sqrt{1+r}}$.
Therefore, we consider the following two sub-cases:
\begin{compactitem}
\item[(i)] If $\hat{\kappa} = 0$, then the last condition of \eqref{eq:FEG4NI_param_choice3} holds automatically.
Moreover, we can choose
\begin{equation*} 
\arraycolsep=0.2em
\begin{array}{ll}
\frac{2\rho}{1-\gamma} \leq \beta < \frac{\eta}{1+\gamma}
\end{array}
\end{equation*}
to guarantee \eqref{eq:FEG4NI_choice_of_beta} in Lemma~\ref{le:FEG4NI_choice_of_params}.
Let $\gamma \to 0^{+}$, $r \to 0^{+}$ and $\mu = 1$, the last two conditions reduce to $2\rho \leq \beta < \eta$.
However, since $\eta \leq \frac{\sqrt{\mu}}{ L \sqrt{1+r}} = \frac{1}{L}$, we finally get $2\rho \leq \beta < \eta \leq \frac{1}{L}$,  provided that $2L\rho \leq 1$.

\item[(ii)] If $\hat{\kappa} > 0$, then the last condition of \eqref{eq:FEG4NI_param_choice3} becomes $0 < \hat{\kappa} < \frac{(\nu - 2)(1-\gamma)\gamma}{\nu}$.
We can choose $\gamma := \frac{1}{2}$ to obtain $0 < \hat{\kappa} < \frac{\nu - 2}{4\nu}$, provided that $\nu > 2$.
In this case, if we choose $\mu = 1$ and $r = 1$, then \eqref{eq:FEG4NI_choice_of_beta} in Lemma~\ref{le:FEG4NI_choice_of_params} holds if
\begin{equation*} 
\arraycolsep=0.2em
\begin{array}{ll}
\frac{2(\nu - 2)}{\nu - 2 - 4\nu \hat{\kappa}}\big[ \frac{ 2(\nu - 1)  \hat{\kappa}L^2\eta^3 }{\nu - 2} + 2\rho \big] \leq \beta < \frac{2\eta}{3} < \frac{2\nu \eta}{2\nu + 1},
\end{array}
\end{equation*}
leading to 
\begin{equation*} 
\arraycolsep=0.2em
\begin{array}{lcl}
2\rho  < \frac{(\nu - 2  - 4\nu \hat{\kappa}) \eta}{3(\nu - 2)} -  \frac{ 2(\nu - 1) \hat{\kappa} L^2 \eta^3 }{\nu - 2} =  \frac{ [(\nu - 2  - 4\nu \hat{\kappa}) - 6(\nu - 1) \hat{\kappa} L^2 \eta^2] \eta }{3(\nu - 2)}. 
\end{array}
\end{equation*}
If we impose $L^2\eta^2 \leq \frac{\nu - 2  - 4\nu \hat{\kappa}}{12(\nu - 1)\hat{\kappa}}$, then to guarantee that $\eta \leq \frac{\sqrt{\mu}}{L\sqrt{1+r}} = \frac{1}{L\sqrt{2}}$, we require $\eta \le \frac{\sigma_1}{L}$, where $\sigma_1^2 := \min\left\{\frac{1}{2}, \frac{\nu  - 2  - 4\nu \hat{\kappa}}{12(\nu - 1)\hat{\kappa}} \right\}$.
Using this bound, the last inequality holds if $2\rho < \frac{(\nu - 2  - 4\nu \hat{\kappa})\sigma_1}{6(\nu - 2)L}$, or equivalently, $L\rho < \frac{(\nu - 2  - 4\nu \hat{\kappa})\sigma_1}{12(\nu - 2)}$.
Moreover, we can choose
\begin{equation*} 
\arraycolsep=0.2em
\begin{array}{ll}
\frac{2(\nu - 2)}{\nu - 2 - 4\nu \hat{\kappa}}\big[ \frac{ 2(\nu - 1) \hat{\kappa} \eta \sigma_1^2}{\nu - 2} + 2\rho \big] \leq \beta < \frac{2\eta}{3}.
\end{array}
\end{equation*}
\end{compactitem}
\noindent\textbf{Case 2.}  
If $\kappa > 0$, then we can choose $r = 1$ and $\mu := \frac{r}{\kappa + r} = \frac{1}{\kappa + 1}$ such that $L^2\eta^2 \leq \frac{(1-\mu)r}{(1+r)\kappa} = \frac{\mu}{1 + r} = \frac{r}{(1+r)(\kappa+r)}$.
Hence, we get $\eta \leq \frac{\sqrt{r}}{L \sqrt{(1+r)(\kappa + r)}} = \frac{1}{L\sqrt{2(\kappa + 1)}}$.
We instead impose the following condition $\eta \leq \frac{1}{2L\sqrt{\kappa + 1}}$, which is stricter than the last one.
Then, we have $1 - \mu  - \frac{ (1 + r)  \kappa  L^2 \eta^2}{r}  \geq \frac{\kappa}{2(\kappa + 1)}$.

Now, we consider the following two sub-cases.
\begin{compactitem}
\item[(i)] If $\hat{\kappa} = 0$, then the last condition of \eqref{eq:FEG4NI_param_choice3} automatically holds.
In addition, we can choose $\gamma := \frac{1}{2}$, and thus \eqref{eq:FEG4NI_choice_of_beta} of Lemma~\ref{le:FEG4NI_choice_of_params} holds if
\begin{equation*} 
\arraycolsep=0.2em
\begin{array}{ll}
4\rho \leq \beta \leq \frac{(\nu - 1)\eta}{4\nu}  < \frac{2\eta}{3} < \frac{2\nu \eta}{2\nu + 1}.
\end{array}
\end{equation*}
Moreover, since $\eta \leq \frac{1}{2L\sqrt{\kappa + 1}}$, we eventually get $4\rho \leq \beta \leq \frac{(\nu - 1)\eta}{4\nu } \leq \frac{1}{8L\nu \sqrt{\kappa + 1}}$, provided that $L \rho \leq \frac{\nu - 1}{32\nu \sqrt{\kappa + 1}}$.

\item[(ii)] If $\hat{\kappa} > 0$, then the last condition \eqref{eq:FEG4NI_param_choice3} imposes that $0 < \hat{\kappa} < \frac{(\nu - 2)(1-\gamma)\gamma}{\nu}$.
Again, we can choose $\gamma := \frac{1}{2}$ to obtain $0 < \hat{\kappa} < \frac{\nu - 2}{4\nu }$, provided that $\nu > 2$.
In this case, \eqref{eq:FEG4NI_choice_of_beta} of Lemma~\ref{le:FEG4NI_choice_of_params} holds if
\begin{equation*} 
\arraycolsep=0.2em
\begin{array}{ll}
\frac{2(\nu - 2)}{\nu - 2  - 4\nu \hat{\kappa}}\big[ \frac{ 2(\nu - 1) (\kappa + 1) \hat{\kappa} L^2 \eta^3 }{\nu - 2} + 2\rho \big] \leq \beta \leq  \frac{(\nu - 1)  \eta }{ 4\nu }  < \frac{2\eta}{3} < \frac{2\nu \eta}{2\nu + 1},
\end{array}
\end{equation*}
leading to
\begin{equation*} 
\arraycolsep=0.2em
\begin{array}{ll}
2\rho \leq \frac{(\nu - 1)(\nu - 2 - 4\nu\hat{ \kappa}) \eta}{8\nu(\nu - 2)} - \frac{2(\nu - 1)(\kappa + 1)\hat{ \kappa}L^2\eta^3}{\nu - 2} = \frac{(\nu - 1)\left[\nu - 2 - 4\nu \hat{\kappa} - 16\nu (\kappa+1)\hat{\kappa}L^2\eta^2\right]\eta}{8\nu (\nu - 2)}.
\end{array}
\end{equation*}
If we enforce $L^2\eta^2 \leq \frac{\nu - 2 - 4\nu \hat{\kappa}}{32\nu(\kappa+1)\hat{\kappa}}$, then to guarantee $\eta \leq \frac{1}{2L\sqrt{\kappa + 1}}$, we require $\eta \leq \frac{\sigma_2}{L}$, where $\sigma_2^2 := \frac{1}{\kappa + 1}\min\left\{\frac{1}{2}, \frac{\nu - 2 - 4\nu \hat{ \kappa}}{32\nu \hat{ \kappa}} \right\}$.
Using this bound, the last inequality holds if $2\rho \leq \frac{(\nu - 1)(\nu - 2 - 4\nu \hat{\kappa})\sigma_2}{16\nu(\nu - 2)L}$, or equivalently, $L\rho \leq \frac{(\nu - 1)(\nu - 2 - 4\nu \hat{ \kappa})\sigma_2}{32\nu(\nu - 2)}$.
Moreover, we can choose
\begin{equation*} 
\arraycolsep=0.2em
\begin{array}{ll}
\frac{2(\nu - 2)}{\nu - 2 - 4\nu \hat{\kappa}}\big[ \frac{ 2(\nu - 1) (\kappa + 1) \hat{\kappa}\eta \sigma_2^2 }{\nu - 2} + 2\rho \big] \leq \beta \leq  \frac{(\nu - 1)  \eta }{4\nu }.
\end{array}
\end{equation*}
\end{compactitem}
Overall, under the choice of parameters as discussed in the two cases above, we can guarantee the conditions in \eqref{eq:FEG4NI_key_est2} of Lemma~\ref{le:FEG4NI_choice_of_params}.
Moreover, in all cases, we also have $\beta < \frac{\eta}{1+\gamma}$, or equivalently, $\eta - (1+\gamma)\beta > 0$.

Next, using the update rules \eqref{eq:FEG4NI_param_update}, we have $\beta_k = \frac{\beta(k + \nu - 1)}{k + \nu}$ and $b_k = \frac{b_0 (k+\nu-1)}{\nu-1}$ for all $k \geq 0$.
Furthermore, we can easily show that
\begin{equation*}
\arraycolsep=0.2em
\begin{array}{lcl}
a_k & := & \frac{b_k[\eta - \beta_{k-1}(1+\gamma\tau_{k-1})]}{\tau_{k-1}} \geq \frac{b_0[\eta - (1+\gamma)\beta](k + \nu - 1)^2}{(s-1)} + 2\rho b_k.
\end{array}
\end{equation*} 
Since $x^\star \in \zer{\Phi}$, and $w^k \in \Phi x^k = Fx^k + Tx^k$, we have $\iprods{w^k, x^k - x^\star} \ge -\rho \norms{w^k}^2$ due to the $\rho$-co-hypomonotonicity of $\Phi$. Using this bound and the elementary inequality $uv \le su^2 + \frac{1}{4s}v^2$ for some $s > 0$, we can show that
\begin{equation}\label{eq:FEG4NI_th31_proof5}
\arraycolsep=0.2em
\begin{array}{lcllcl}
\hat{\Pc}_k & := & \frac{a_k}{2} \norms{w^k}^2 + b_k\iprods{w^k, x^k - x^0} \vspace{1ex}\\
& = &  \frac{a_k}{2} \norms{w^k}^2 + b_k\iprods{w^k, x^{\star} - x^0} + b_k\iprods{w^k, x^k - x^{\star}}  \vspace{1ex}\\
& \geq &  ( \frac{a_k}{2} - \rho b_k) \norms{w^k}^2 - b_k\norms{w^k}\norms{x^0 - x^{\star}} \vspace{1ex}\\
&\geq &  \left( \frac{a_k}{2} - \rho b_k - \frac{[\eta- (1+\gamma)\beta]b_k^2}{4b_0} \right) \norms{w^k}^2 - \frac{b_0}{\eta - (1+\gamma)\beta}\norms{x^0 - x^{\star}}^2 \vspace{1ex}\\
& \geq &  \frac{b_0[\eta - (1+\gamma)\beta](k + \nu -1)^2}{4(\nu - 1)}\norms{w^k}^2 - \frac{b_0}{\eta - (1+\gamma)\beta}\norms{x^0 - x^{\star}}^2.
\end{array}
\end{equation}
Here, we have used $a_k \ge \frac{b_0[\eta - (1+\gamma)\beta](k + \nu - 1)^2}{(\nu - 1)} + 2\rho b_k$ and $b_k := \frac{b_0(k + \nu - 1)}{\nu - 1}$ in the last equality.

Finally, from \eqref{eq:FEG4NI_key_est2} of Lemma~\ref{le:FEG4NI_key_est2}, by induction, we have 
\begin{equation*}
\hat{\Pc}_k + \tfrac{c_k}{2} \norms{Fx^k - u^k}^2 = \Lc_k \leq \Lc_0.
\end{equation*}
Due to the choice of $y^{-1} = x^{-1} = x^0$ and $u^0 := Fx^0$, we get 
\begin{equation*}
\arraycolsep=0.2em
\begin{array}{lcl}
\Lc_0 = \frac{a_0}{2}\norms{Fx^0 + \xi^0}^2 = \frac{b_0 \left[\eta(\nu - 1)^2 - \beta(\nu - 2)(\nu + \gamma-1)\right]}{\nu - 1}\norms{Fx^0 + \xi^0}^2.
\end{array}
\end{equation*}
Therefore, we eventually obtain
\begin{equation*}
\arraycolsep=0.2em
\begin{array}{lcl}
\Lc_k \leq \frac{b_0 \left[\eta(\nu - 1)^2 - \beta(\nu - 2)(\nu + \gamma - 1)\right]}{\nu - 1}\norms{Fx^0 + \xi^0}^2.
\end{array}
\end{equation*}
We also note that
\begin{equation*}
\arraycolsep=0.2em
\begin{array}{lcl}
c_k &= & \frac{ b_k [ (1 + r) L^2 \eta^3 \gamma (1-\tau_k)^2 + r \mu \beta_k ]}{r\mu\gamma \tau_k(1-\tau_k)} = \frac{ b_0(k + \nu - 1) [ (1 + r) L^2 \eta^3 \gamma (k + \nu - 1) + r \mu \beta (k + \nu) ]}{r\mu\gamma (\nu - 1)} \vspace{1ex} \\
& \geq  & \frac{ b_0[ (1 + r) L^2 \eta^3 \gamma + r \mu \beta ] (k + \nu - 1)^2 }{r\mu\gamma (\nu - 1)} = b_0C_0(k + \nu - 1)^2,
\end{array}
\end{equation*}
where $C_0 := \frac{[ (1 + r) L^2 \eta^3 \gamma + r \mu \beta ] }{r\mu\gamma(\nu - 1)}$.

We further note that since either $r = 0$ or $r = 1$, $\gamma = 0$ or $\gamma = \frac{1}{2}$, and $\mu = 1$ or $\mu = \frac{1}{1 + \kappa}$, we can set $C_0 := 0$ if $r=0$ (corresponding to Case (i)), and evaluate $C_0 \geq \frac{2(L^2\eta^3 + \beta)}{s-1}$ if $r = 1$ (corresponding to the other cases).

Combining $\Lc_k \leq \frac{b_0 \left[\eta(\nu - 1)^2 - \beta(\nu - 2)(\nu + \gamma - 1)\right]}{\nu - 1}\norms{Fx^0 + \xi^0}^2$ and \eqref{eq:FEG4NI_th31_proof5}, and using the above coefficients $a_k$, $\psi$, and $b_0 > 0$, we can show that
\begin{equation*}
\arraycolsep=0.2em
\begin{array}{lcl}
\norms{Fx^k + \xi^k}^2  +  \psi\norms{Fx^k - u^k}^2  &  \leq &  \frac{1}{( k + \nu - 1)^2 } \cdot \Big[ \frac{ 4(\nu - 1)\norms{x^0 - x^{\star}}^2 }{[\eta - (1+\gamma)\beta]^2  } \vspace{1ex}\\
&& + {~}  \frac{4[\eta(\nu - 1)^2 - \beta(\nu - 2)(\nu + \gamma-1)] }{[\eta - (1+\gamma)\beta] }\norms{Fx^0 + \xi^0}^2 \Big].
\end{array}
\end{equation*}
Here, $\psi := 0$ if $\kappa = \hat{\kappa} = 0$ and $\psi := \frac{4(L^2\eta^3 + \beta)}{\eta - (1+\gamma)\beta}$, otherwise.
Since $\eta - (1+\gamma)\beta > 0$, this inequality implies \eqref{eq:FEG4NI_convergence}.
We note that, for Case (i), we set $\gamma = 0$, and for other cases, we have used $\gamma = \frac{1}{2}$.
Hence, the final term on the  right-hand side of the last inequality reduces to $\Rc_0^2$ defined in \eqref{eq:FEG4NI_convergence} of the theorem.
\Eproof

\beforesec
\section{The Proof of Technical Results in Section~\ref{sec:DFBFS4NI}}\label{apdx:sec:DFBFS4NI}
\aftersec
This appendix presents the full proof of the results in Section~\ref{sec:DFBFS4NI}.

\beforesubsec
\subsection{\mytb{The Proof of Lemma~\ref{le:DFEG4NI_descent_pro}}}\label{apdx:le:DFEG4NI_descent_pro}
\aftersubsec
First, since $\tau_k := \frac{1}{t_k}$ and $e^k := u^k - Fx^k = z^k - w^k$, we can rewrite the second line of \eqref{eq:DFEG4NI} as
\begin{equation*} 
\arraycolsep=0.2em
\begin{array}{lcl}
t_kx^{k+1} = \bar{x}^k + (t_k - 1)x^k - \eta t_k \hat{w}^{k+1}  + \beta_k t_k w^k + \beta_kt_k e^k.
\end{array}
\end{equation*}
Rearranging this expression in two different ways, and using $\bar{x}^k = \bar{x}^{k+1} + \gamma z^k = \bar{x}^{k+1} + \gamma w^k + \gamma e^k$ from the third line of \eqref{eq:DFEG4NI}, we get
\begin{equation*} 
\arraycolsep=0.2em
\left\{\begin{array}{lcl}
t_k(t_k-1)(x^k - x^{k+1}) &= &  (t_k-1) (x^k - \bar{x}^k) + \eta t_k(t_k-1)\hat{w}^{k+1} \vspace{1ex}\\
&& - {~} \beta_k t_k(t_k-1) w^k - \beta_k t_k(t_k-1) e^k, \vspace{1ex}\\
t_k(t_k-1)(x^k - x^{k+1}) & = &  t_k(x^{k+1} - \bar{x}^k) + \eta t_k^2 \hat{w}^{k+1} -  \beta_k t_k^2 w^k - \beta_k t_k^2 e^k \vspace{1ex}\\
& = & t_k(x^{k+1} - \bar{x}^{k+1}) + \eta t_k^2 \hat{w}^{k+1} - t_k(\beta_kt_k + \gamma )w^k \vspace{1ex}\\
&& - {~} t_k(\beta_kt_k + \gamma )e^k.
\end{array}\right.
\end{equation*}
Next, by the $\rho$-co-hypomonotonicity of $\Phi$, we have 
\begin{equation*} 
\arraycolsep=0.2em
\begin{array}{lcl}
\Tc_{[1]} &:= &  t_k(t_k-1)\iprods{w^k, x^k - x^{k+1}} - t_k(t_k-1)\iprods{w^{k+1}, x^k - x^{k+1}} \vspace{1ex}\\
& \geq & - \rho t_k(t_k-1) \norms{w^{k+1} - w^k}^2.
\end{array}
\end{equation*}
Substituting the above two expressions into $\Tc_{[1]}$, we can show that
\begin{equation*}
\arraycolsep=0.2em
\begin{array}{lcl}
\Tc_{[2]} &:= & (t_k-1)\iprods{w^k, x^k - \bar{x}^k} - t_k\iprods{w^{k+1}, x^{k+1} - \bar{x}^{k+1} } \vspace{1ex}\\
& \geq & \eta t_k^2 \iprods{\hat{w}^{k+1}, w^{k+1} }  - \eta t_k(t_k-1) \iprods{\hat{w}^{k+1}, w^k } - \rho t_k(t_k-1) \norms{w^{k+1} - w^k}^2 \vspace{1ex}\\
&& - {~} t_k(\beta_kt_k + \gamma )\iprods{w^{k+1}, w^k} + \beta_k t_k(t_k-1)\norms{w^k}^2 \vspace{1ex}\\
&& - {~} t_k\iprods{ (\beta_kt_k + \gamma ) w^{k+1} - \beta_k (t_k - 1)w^k, e^k}.
\end{array}
\end{equation*}
By Young's inequality, for any $c_1 > 0$ and $c_2 > 0$,  one can prove that
\begin{equation*}
\arraycolsep=0.15em
\begin{array}{lcl}
\Tc_{[3]} &:= &  -\iprods{ (\beta_kt_k + \gamma ) w^{k+1} - \beta_k (t_k - 1)w^k, e^k} \vspace{1ex}\\
& = & - (\beta_kt_k + \gamma)\iprods{w^{k+1} - w^k, e^k} - (\gamma + \beta_k) \iprods{w^k, e^k} \vspace{1ex}\\
& \geq & -c_1(\beta_kt_k + \gamma)\norms{w^{k+1} - w^k}^2 - c_2(\gamma + \beta_k)\norms{w^k}^2 - \big( \frac{\beta_kt_k + \gamma}{4c_1} + \frac{\gamma + \beta_k}{4c_2}\big) \norms{e^k}^2 \vspace{1ex}\\
& = & -c_1(\beta_kt_k + \gamma)\norms{w^{k+1}}^2 - \big[ c_1(\beta_kt_k + \gamma) + c_2(\gamma + \beta_k) \big]\norms{w^k}^2 \vspace{1ex}\\
&& + {~} 2c_1(\beta_kt_k + \gamma)\iprods{w^{k+1}, w^k} - \big( \frac{\beta_kt_k + \gamma}{4c_1} + \frac{\gamma + \beta_k}{4c_2}\big) \norms{e^k}^2.
\end{array}
\end{equation*}
Substituting $\Tc_{[3]}$ into $\Tc_{[2]}$, we can derive that
\begin{equation*}
\arraycolsep=0.2em
\begin{array}{lcl}
\Tc_{[4]} &:= & (t_k-1)\iprods{w^k, x^k - \bar{x}^k} - t_k\iprods{w^{k+1}, x^{k+1} - \bar{x}^{k+1} } \vspace{1ex}\\
& \geq & \eta t_k^2 \iprods{\hat{w}^{k+1},  w^{k+1} }  - \eta t_k(t_k-1) \iprods{\hat{w}^{k+1}, w^k } \vspace{1ex}\\
&& - {~} t_k \big[ \rho (t_k-1)   + c_1(\beta_kt_k + \gamma) + c_2(\gamma + \beta_k) - \beta_k(t_k-1) \big] \norms{w^k}^2 \vspace{1ex}\\
&& - {~} t_k \big[ \rho (t_k - 1) + c_1(\beta_kt_k + \gamma) \big] \norms{w^{k+1} }^2 \vspace{1ex}\\
&& - {~} t_k \big[ (1 - 2c_1)(\beta_kt_k + \gamma) - 2\rho (t_k-1) \big] \iprods{w^{k+1}, w^k} \vspace{1ex}\\
&& - {~} \frac{t_k}{4}\big(\frac{\beta_kt_k + \gamma}{c_1} + \frac{\gamma + \beta_k}{c_2}\big)\norms{e^k}^2.
\end{array}
\end{equation*}
For any $c_1 \in (0, 1/2)$, since $\beta_k$ is chosen as in \eqref{eq:DFEG4NI_param}, we have $(1 - 2c_1)(\beta_kt_k + \gamma) - 2\rho (t_k-1) = 0$.
Substituting this expression into $\Tc_{[4]}$ and using $\beta_k = \frac{2\rho(t_k-1)}{(1-2c_1)t_k} - \frac{\gamma}{t_k}$ from \eqref{eq:DFEG4NI_param}, we can simplify it as 
\begin{equation*}
\arraycolsep=0.2em
\begin{array}{lcl}
\Tc_{[4]} & \geq & \eta t_k^2 \iprods{ \hat{w}^{k+1}, w^{k+1} }  - \eta t_k(t_k-1) \iprods{ \hat{w}^{k+1}, w^k } \vspace{1ex}\\
&& - {~} \big[ \gamma (1+c_2)(t_k - 1) - \frac{\rho(t_k-1)[ t_k - 2(1+c_2)]}{1-2c_1} \big] \norms{w^k}^2 \vspace{1ex}\\
&& - {~} \frac{\rho t_k(t_k-1)}{1- 2c_1}\norms{w^{k+1} }^2 - \frac{(t_k-1)}{4(1-2c_1)}\big[ \frac{2\rho t_k}{c_1} + \frac{2\rho + \gamma(1-2c_1)}{c_2} \big]\norms{e^k}^2.
\end{array}
\end{equation*}
Since $t_{k-1} - t_k + 1 = 1- \mu$ due to \eqref{eq:DFEG4NI_param}, adding $(1-\mu)\iprods{w^k, x^k - \bar{x}^k}$ to both sides of $\Tc_{[4]}$ and using the identity 
\begin{equation*}
\arraycolsep=0.2em
\begin{array}{lcl}
\eta t_k(t_k-1)\iprods{\hat{w}^{k+1}, w^k} & = & \frac{\eta t_k^2}{2}\norms{\hat{w}^{k+1}}^2 + \frac{\eta(t_k-1)^2}{2}\norms{w^k}^2 - \frac{\eta}{2}\norms{t_k\hat{w}^{k+1} - (t_k-1)w^k}^2
\end{array}
\end{equation*}
into $\Tc_{[4]}$, we can show that
\begin{equation*}
\arraycolsep=0.2em
\begin{array}{lcl}
\Tc_{[5]} &:= & t_{k-1}\iprods{w^k, x^k - \bar{x}^k}  - t_k\iprods{w^{k+1}, x^{k+1} - \bar{x}^{k+1} }  \vspace{1ex}\\
& \geq &  \frac{\eta}{2}\norms{t_k\hat{w}^{k+1} - (t_k-1)w^k}^2 +  (1-\mu) \iprods{w^k, x^k - \bar{x}^k} \vspace{1ex}\\
&& - {~} \big[ \frac{\eta(t_k-1)^2}{2} + \gamma(1+c_2)(t_k - 1) - \frac{\rho(t_k-1)[t_k - 2(1+c_2)]}{1-2c_1} \big] \norms{w^k}^2 \vspace{1ex}\\
&& - {~} \frac{\rho t_k(t_k-1)}{1- 2c_1}\norms{w^{k+1} }^2 - \frac{(t_k-1)}{4(1-2c_1)}\big[ \frac{2\rho t_k}{c_1} + \frac{2\rho + \gamma(1-2c_1)}{c_2} \big]\norms{e^k}^2 \vspace{1ex}\\
&& + {~} \eta t_k^2 \iprods{ \hat{w}^{k+1}, w^{k+1} }   - \frac{\eta t_k^2}{2}\norms{\hat{w}^{k+1}}^2.
\end{array}
\end{equation*}
Now, since $\hat{\eta}_k = \eta(1-\tau_k) = \frac{\eta(t_k-1)}{t_k}$ due to  \eqref{eq:DFEG4NI_param}, from the first and  second lines of \eqref{eq:DFEG4NI}, we can show that 
\begin{equation*}
\arraycolsep=0.2em
\begin{array}{lcl}
t_k(x^{k+1} - y^k) = -\eta t_k \hat{w}^{k+1} + t_k \hat{\eta}_kz^k = -\eta \left[ t_k\hat{w}^{k+1} - (t_k - 1)w^k - (t_k-1)e^k \right].
\end{array}
\end{equation*}
By the $L$-Lipschitz continuity of $F$, for any $\omega \geq 0$, $c > 0$, and $M^2 := (1+\omega)(1+c)L^2$, applying Young's inequality, the last expression leads to 
\begin{equation*}
\arraycolsep=0.15em
\begin{array}{lcl}
(1+\omega)t_k^2\norms{w^{k+1} - \hat{w}^{k+1}}^2 & = & (1+\omega)t_k^2\norms{Fx^{k+1} - Fy^k }^2 \vspace{1ex}\\
& \leq & (1+\omega)L^2 \norms{t_k( x^{k+1} - y^k ) }^2 \vspace{1ex}\\
& = & (1+\omega)L^2\eta^2 \norms{t_k \hat{w}^{k+1} - (t_k-1)w^k - (t_k-1)e^k}^2 \vspace{1ex}\\
& \leq & M^2\eta^2\norms{t_k\hat{w}^{k+1} - (t_k-1)w^k}^2 + \frac{M^2\eta^2(t_k-1)^2}{c}\norms{e^k}^2.
\end{array}
\end{equation*}
This inequality implies that
\begin{equation*}
\arraycolsep=0.2em
\begin{array}{lcl}
0 &\geq & \frac{\eta t_k^2}{2} \norms{w^{k+1}}^2 + \frac{\eta t_k^2}{2} \norms{\hat{w}^{k+1}}^2 - \eta t_k^2\iprods{w^{k+1}, \hat{w}^{k+1}}  + \frac{\eta\omega t_k^2}{2}  \norms{w^{k+1} - \hat{w}^{k+1}}^2 \vspace{1ex}\\
&& - {~} \frac{M^2\eta^3}{2} \norms{t_k\hat{w}^{k+1} - (t_k-1)w^k}^2  -  \frac{M^2\eta^3(t_k-1)^2}{2c}\norms{e^k}^2.
\end{array}
\end{equation*}
Adding this inequality to $\Tc_{[5]}$ above, we get
\begin{equation*}
\arraycolsep=0.2em
\begin{array}{lcl}
\Tc_{[5]} &:= & t_{k-1}\iprods{w^k, x^k - \bar{x}^k} - t_k\iprods{w^{k+1}, x^{k+1} - \bar{x}^{k+1}}  \vspace{1ex}\\
& \geq &  \frac{\omega\eta t_k^2}{2} \norms{w^{k+1} - \hat{w}^{k+1}}^2 + \frac{\eta(1 - M\rmark{^2}\eta^2)}{2}\norms{ t_k\hat{w}^{k+1} - (t_k-1)w^k}^2  \vspace{1ex}\\
&& + {~} (1-\mu) \iprods{w^k, x^k - \bar{x}^k} +  \frac{t_k}{2}\big[  \eta t_k - \frac{2\rho (t_k-1)}{1- 2c_1} \big] \norms{w^{k+1}}^2 \vspace{1ex}\\
&& - {~} \big[ \frac{\eta(t_k-1)^2}{2} + \gamma(1+c_2)(t_k - 1) - \frac{\rho(t_k-1)[ t_k - 2(1 + c_2)]}{1-2c_1} \big] \norms{w^k}^2 \vspace{1ex}\\
&& - {~} \Big\{ \frac{M^2\eta^3(t_k-1)^2}{2c} +   \frac{(t_k-1)}{4(1-2c_1)}\big[ \frac{2\rho t_k}{c_1} + \frac{2\rho + \gamma(1-2c_1)}{c_2} \big] \Big\} \norms{e^k}^2.
\end{array}
\end{equation*}
Next, since $\bar{x}^{k+1} - \bar{x}^k = -\gamma z^k = -\gamma (w^k + e^k)$ from the third line of \eqref{eq:DFEG4NI}, for any $\hat{c} > 0$, by Young's inequality, we can prove that
\begin{equation*}
\arraycolsep=0.2em
\begin{array}{lcl}
\Tc_{[6]} &:= & \frac{(1-\mu)}{2\gamma}\norms{\bar{x}^k - x^{\star}}^2 - \frac{(1-\mu)}{2\gamma} \norms{\bar{x}^{k+1} - x^{\star}}^2 \vspace{1ex}\\
& = &  -\frac{(1-\mu)}{\gamma}\iprods{\bar{x}^{k+1} - \bar{x}^k, \bar{x}^k - x^{\star}} - \frac{(1-\mu)}{2\gamma}\norms{\bar{x}^{k+1} - \bar{x}^k}^2 \vspace{1ex}\\
& = & (1-\mu)\iprods{w^k + e^k, \bar{x}^k - x^{\star}} - \frac{(1-\mu)\gamma}{2}\norms{w^k + e^k}^2 \vspace{1ex}\\
& \geq & (1-\mu)\iprods{w^k, \bar{x}^k - x^{\star}}  +  (1-\mu)\iprods{e^k, \bar{x}^k - x^{\star}} \vspace{1ex}\\
&& - {~} \frac{(1+\hat{c})(1-\mu)\gamma}{2}\norms{w^k}^2 - \frac{(1+\hat{c})(1-\mu)\gamma}{2\hat{c}}\norms{e^k}^2.
\end{array}
\end{equation*}
Summing up the last two expressions $\Tc_{[5]}$ and $\Tc_{[6]}$, we get
\begin{equation*}
\arraycolsep=0.2em
\begin{array}{lcl}
\Tc_{[7]} &:= & t_{k-1}\iprods{w^k, x^k - \bar{x}^k} - t_k\iprods{w^{k+1}, x^{k+1} - \bar{x}^{k+1}} \vspace{1ex}\\
&& + {~}  \frac{(1-\mu)}{2\gamma }\norms{\bar{x}^k - x^{\star}}^2 - \frac{(1-\mu)}{2\gamma} \norms{\bar{x}^{k+1} - x^{\star}}^2 \vspace{1ex}\\
& \geq &  \frac{\omega \eta t_k^2}{2} \norms{w^{k+1} - \hat{w}^{k+1}}^2 + \frac{\rmark{\eta}(1 - M^2\eta ^2)}{2}\norms{t_k\hat{w}^{k+1} - (t_k-1)w^k }^2 \vspace{1ex} \\
&& + {~}  \frac{t_k}{2}\big[  \eta t_k - \frac{2\rho (t_k-1)}{1- 2c_1} \big] \norms{w^{k+1}}^2  - \frac{\hat{a}_k}{2} \norms{w^k}^2 \vspace{1ex}\\
&& - {~} \frac{1}{2}\Big\{ \frac{M^2\eta^3(t_k-1)^2}{c} +   \frac{(t_k-1)}{2(1-2c_1)}\big[ \frac{2\rho t_k}{c_1} + \frac{2\rho + \gamma(1-2c_1)}{c_2} \big] + \frac{(1+\hat{c})(1-\mu)\gamma}{\hat{c}} \Big\} \norms{e^k}^2 \vspace{1ex}\\
&& + {~} (1-\mu)\iprods{e^k, \bar{x}^k - x^{\star}} +  (1-\mu) \iprods{w^k, x^k - x^{\star}}.
\end{array}
\end{equation*}
where $\hat{a}_k := \eta(t_k-1)^2 + 2\gamma(1+c_2)(t_k - 1) - \frac{2\rho(t_k-1)[ t_k - 2(1+c_2)]}{1-2c_1} +  (1+\hat{c})(1-\mu)\gamma$.
Rearranging this expression and using $\Vc_k$ from \eqref{eq:DFEG4NI_potential_func}, we  obtain \eqref{eq:DFEG4NI_descent_property}.
\Eproof

\beforesubsec
\subsection{\mytb{The Proof of Lemma~\ref{le:DFEG4NI_Vk_lowerbound}}}\label{apdx:le:DFEG4NI_Vk_lowerbound}
\aftersubsec
From the $\rho$-co-hypomonotonicity of $\Phi$ and $x^{\star} \in \zer{\Phi}$, we can show that $\iprods{w^k, x^k - x^{\star}} \geq -\rho\norms{w^k}^2$.
Using this inequality and  $\Vc_k$ from \eqref{eq:DFEG4NI_potential_func}, for any $\tilde{c} > 0$, we have
\begin{equation*}
\arraycolsep=0.15em
\begin{array}{lcl}
\Vc_k &:= & \frac{a_k}{2} \norms{w^k}^2 + t_{k-1}\iprods{w^k, x^k - \bar{x}^k} +  \frac{(1-\mu)}{2\gamma }\norms{\bar{x}^k - x^{\star}}^2 \vspace{1ex}\\
& = & \frac{\tilde{c}\eta t_{k-1}^2}{2} \norms{w^k}^2 - t_{k-1}\iprods{w^k, \bar{x}^k - x^{\star} } +  \frac{1}{2 \tilde{c} \eta }\norms{\bar{x}^k - x^{\star}}^2 + t_{k-1}\iprods{w^k, x^k - x^{\star} } \vspace{1ex}\\
&& + \frac{(a_k - \tilde{c}\eta t_{k-1}^2)}{2}\norms{w^k}^2 + \frac{[(1-\mu)\tilde{c}\eta - \gamma]}{2\tilde{c} \gamma \eta }\norms{\bar{x}^k - x^{\star}}^2 \vspace{1ex}\\
& \geq & \frac{1}{2\tilde{c}\eta} \norms{\bar{x}^k - x^{\star} - \tilde{c}\eta t_{k-1}w^k}^2  +  \frac{[(1-\mu)\tilde{c}\eta - \gamma]}{2\tilde{c} \gamma \eta }\norms{\bar{x}^k - x^{\star}}^2 \vspace{1ex}\\
&& + {~} \frac{(a_k - \tilde{c}\eta t_{k-1}^2 - 2\rho t_{k-1})}{2}\norms{w^k}^2 \vspace{1ex}\\
& \geq & \frac{[(1-\mu)\tilde{c}\eta - \gamma]}{2\tilde{c}\gamma\eta}\norms{\bar{x}^k - x^{\star}}^2 + \frac{(a_k - \tilde{c}\eta t_{k-1}^2 - 2\rho t_{k-1})}{2}\norms{w^k}^2,
\end{array}
\end{equation*}
which proves \eqref{eq:DFEG4NI_Vk_lower_bound}.
\Eproof

\beforesubsec
\subsection{\mytb{The Proof of Theorem~\ref{th:DFEG4NI_convergence1}}}\label{apdx:th:DFEG4NI_convergence1}
\aftersubsec
For simplicity of our analysis, let us fix $c_1 := \frac{1}{4}$, $c_2 := \frac{1}{2}$, $c := 1$, and $\hat{c} := 1$ in Lemma~\ref{le:DFEG4NI_descent_pro}.
We also choose $\gamma := \zeta(1-\mu)\eta$ for some $\zeta \in (0, 1]$.

First, from \eqref{eq:DFEG4NI_param}, we have 
\begin{equation*} 
\arraycolsep=0.2em
\begin{array}{ll}
& \tau_k := \frac{1}{t_k} = \frac{1}{\mu(k+r)}, \quad \hat{\eta}_k := \eta(1 - \tau_k) = \frac{\eta(t_k-1)}{t_k},  \vspace{1ex}\\
& \beta_k := -\gamma\tau_k + 4\rho(1 - \tau_k) =  -\frac{\zeta(1-\mu)\eta}{t_k} + \frac{4\rho(t_k-1)}{t_k},
\end{array}
\end{equation*}
as stated in \rmark{\eqref{eq:DFEG4NI_param_update1}}.

Next, using $c_1 = \frac{1}{4}$, $c_2 = \frac{1}{2}$, $c = 1$, $\hat{c} = 1$, $\gamma = \zeta(1- \rmark{\mu})\eta$, and $t_k = t_{k-1} + \mu$ in \eqref{eq:DFEG4NI_descent_property}, we can show that 
\begin{equation}\label{eq:DFEG4NI_th41_proof1a}
\hspace{-0.5ex}
\arraycolsep=0.1em
\left\{\begin{array}{lcl}
a_k 
& = &  \eta  t_{k-1}^2 -  4\rho t_{k-1}(t_{k-1} - 1),  \vspace{1ex}\\
\hat{a}_k 
& = & a_k - \psi_k, \vspace{1ex}\\
b_k 
& = & M^2\eta^3(t_k-1)^2  +   (t_k-1) \big[  8\rho t_k + 4\rho + \zeta(1-\mu)\eta  \big] + 2\zeta(1-\mu)^2\eta,
\end{array}\right.
\hspace{-3ex}
\end{equation}
where 
\begin{equation}\label{eq:DFEG4NI_th41_proof1b}
\arraycolsep=0.2em
\begin{array}{lcl}
\psi_k & := & \big[ (2 - 3\zeta)(1-\mu)\eta - 4\rho(3-2\mu)  \big]t_{k-1} \vspace{1ex}\\
&& - {~} (1-\mu)[(1-\zeta)(1-\mu)\eta - 4\rho(3-\mu)].
\end{array}
\end{equation}
Let us choose $\zeta := \frac{1-2\mu}{6(1-\mu)}$, provided that $0 < \mu < \frac{1}{2}$.
Then, we obtain $\gamma = \frac{(1-2\mu)\eta}{6}$ as stated in \eqref{eq:DFEG4NI_eta_choice1}.
Moreover, from \eqref{eq:DFEG4NI_th41_proof1b}, to guarantee $\psi_k \geq 0$, we first need to choose $\eta > 8\rho$ as stated in \eqref{eq:DFEG4NI_eta_choice1}.
Moreover, we can simplify $\psi_k$ as
\begin{equation}\label{eq:DFEG4NI_th41_proof1c}
\arraycolsep=0.2em
\begin{array}{lcl}
\psi_k & = & \frac{(3-2\mu)(\eta - 8\rho)}{2}t_{k-1} -  \frac{(1-\mu)[ (5 -4\mu)\eta - 24\rho(3-\mu)]}{6} \vspace{1ex}\\
& \geq &  \frac{(3-2\mu)(\eta - 8\rho)}{2}(t_{k-1} - 1 + \mu) = \frac{(3-2\mu)(\eta - 8\rho)}{2}(t_k - 1) .
\end{array}
\end{equation}
Hence, if $\mu r \geq 1$, then $t_k \geq 1$, leading to $\psi_k \geq 0$.

Now, utilizing \eqref{eq:DFEG4NI_th41_proof1a}, we can simplify \eqref{eq:DFEG4NI_descent_property} as
\begin{equation}\label{eq:DFEG4NI_descent_pro_proof1} 
\arraycolsep=0.2em
\begin{array}{lcl}
\Vc_k & \geq & \Vc_{k+1}  + \frac{\psi_k}{2}\norms{w^k}^2  + \frac{\rmark{\eta}(1 - M^2\eta^2)}{2}\norms{t_k\hat{w}^{k+1} - (t_k-1)w^k }^2 \vspace{1ex} \\
&& + {~}  \frac{\omega\eta t_k^2}{2} \norms{w^{k+1} - \hat{w}^{k+1}}^2 + (1 - \mu ) \iprods{w^k, x^k - x^{\star} } \vspace{1ex}\\
&& + {~} (1-\mu)\iprods{e^k, \bar{x}^k - x^{\star}} -  \frac{b_k}{2} \norms{e^k}^2.
\end{array}
\end{equation}
Since $e^{k+1} = u^{k+1} - Fx^{k+1}$, from \eqref{eq:DFEG4NI_u_cond}, for any $\lambda > 0$, we have
\begin{equation*} 
\arraycolsep=0.2em
\begin{array}{lcl}
\frac{\lambda \eta t_k^2}{2} \norms{e^{k+1}}^2 & \leq & \frac{ \lambda\kappa \eta t_k^2}{2} \norms{w^{k+1} - \hat{w}^{k+1}}^2  + \frac{\lambda\hat{\kappa}\eta\rmark{^3}}{2} \norms{t_k \hat{w}^{k+1} - (t_k-1)w^k}^2.
\end{array}
\end{equation*}
In addition, by the $\rho$-co-hypomonotonicity of $\Phi$ and $x^{\star} \in \zer{\Phi}$, we have $\iprods{w^k, x^k - x^{\star} } \geq -\rho\norms{w^k}^2$.
Combining \eqref{eq:DFEG4NI_descent_pro_proof1}  and the last two inequalities, and then applying Young's inequality to $\iprods{e^k, \bar{x}^k - x^{\star}}$, for $\nu > 0$, we can show that
\begin{equation}\label{eq:DFEG4NI_descent_pro_proof2} 
\hspace{-0.5ex}
\arraycolsep=0.1em
\begin{array}{lcl}
\Vc_k + \frac{(b_k + \eta \nu t_k^2)}{2}\norms{e^k}^2 & \geq & \Vc_{k+1} + \frac{\lambda \eta t_k^2}{2}\norms{e^{k+1}}^2  +  \frac{[\psi_k - 2(1-\mu)\rho] }{2}\norms{w^k}^2 \vspace{1ex}\\
&& + {~} \frac{\rmark{\eta}(1 - M^2\eta^2 - \lambda\hat{\kappa}\eta^2)}{2}\norms{t_k\hat{w}^{k+1} - (t_k-1)w^k }^2  \vspace{1ex}\\
&& +  {~} \frac{\eta(\omega - \lambda\kappa) t_k^2}{2} \norms{w^{k+1} - \hat{w}^{k+1}}^2 - \frac{(1-\mu)^2}{2\nu\eta t_k^2}\norms{\bar{x}^k - x^{\star}}^2.
\end{array}
\hspace{-2ex}
\end{equation}
Since $\eta > 8\rho$ by \eqref{eq:DFEG4NI_eta_choice1}, $\zeta = \frac{1-2\mu}{6(1-\mu)}$, and $t_k = t_{k-1} + \mu$, we have from \eqref{eq:DFEG4NI_th41_proof1a} that
\begin{equation*} 
\arraycolsep=0.2em
\begin{array}{lcl}
b_k + \eta\nu t_k^2 
& \leq & M^2\eta^3(t_k-1)^2 + \eta\nu t_k^2 + (t_k-1) \big[  \eta t_k  + \frac{\eta}{2} + \frac{(1-2\mu)\eta}{6}  \big]  + 2\zeta(1-\mu)^2\eta \vspace{1ex}\\
& = & \eta(M^2\eta^2 + \nu + 1)t_{k-1}^2 + \eta S_k,
\end{array}
\end{equation*}
where 
\begin{equation}\label{eq:DFEG4NI_descent_pro_proof2b}  
\arraycolsep=0.2em
\begin{array}{lcl}
S_k & := & \big(2\nu\mu - 2(1-\mu)M^2\eta^2 + \frac{5\mu-1}{3}\big)t_{k-1} \vspace{1ex}\\
&& + {~} \big[(1-\mu)^2M^2\eta^2 + \nu\mu^2 - \frac{(1-\mu)(1+4\mu)}{3}\big].
\end{array}
\end{equation}
From the last expression, if we impose
\begin{equation}\label{eq:DFEG4NI_descent_pro_proof3} 
\arraycolsep=0.1em
\begin{array}{lcl}
\lambda & \geq & M^2\eta^2 +  2\nu + 1 \quad \textrm{and} \quad S_k \leq \nu t_{k-1}^2,
\end{array}
\end{equation}
then we can show that $b_k + \eta\nu t_k^2 \leq \eta\lambda t_{k-1}^2$.

\noindent
Since $b_k + \eta\nu t_k^2 \leq \eta\lambda t_{k-1}^2$, utilizing $\Lc_k$ from \eqref{eq:DFEG4NI_potential_func},  \eqref{eq:DFEG4NI_descent_pro_proof2} implies that
\begin{equation}\label{eq:DFEG4NI_descent_pro_proof4} 
\hspace{-0.5ex}
\arraycolsep=0.1em
\begin{array}{lcl}
\Lc_k - \Lc_{k+1} & \geq & \frac{[\psi_k - 2(1-\mu)\rho] }{2}\norms{w^k}^2 + \frac{\rmark{\eta}(1 - M^2\eta^2 - \lambda\hat{\kappa}\eta^2)}{2}\norms{t_k\hat{w}^{k+1} - (t_k-1)w^k }^2  \vspace{1ex}\\
&& +  {~} \frac{\eta(\omega - \lambda\kappa) t_k^2}{2} \norms{w^{k+1} - \hat{w}^{k+1}}^2 - \frac{(1-\mu)^2}{2\nu\eta t_k^2}\norms{\bar{x}^k - x^{\star}}^2.
\end{array}
\hspace{-2ex}
\end{equation}
Next, from Lemma~\ref{le:DFEG4NI_Vk_lowerbound}, $a_k$ from \eqref{eq:DFEG4NI_th41_proof1a}, and $\gamma = \rmark{\zeta (1-\mu) \eta = } \frac{(1-2\mu)\eta}{6}$, we also have
\begin{equation}\label{eq:DFEG4NI_descent_pro_proof5} 
\arraycolsep=0.2em
\begin{array}{lcl}
\Lc_k & \geq & \Vc_k \geq  \frac{[(1-\mu)\tilde{c} - \zeta(1-\mu)]}{2\tilde{c}\zeta(1-\mu) \eta}\norms{\bar{x}^k - x^{\star}}^2 + \frac{[\eta(1-\tilde{c}) - 4\rho] t_{k-1}^2 + 2\rho t_{k-1}}{2}\norms{w^k}^2 \vspace{1ex}\\
& \geq &  \frac{[6(1-\mu)\tilde{c} - (1-2\mu)]}{2\tilde{c}(1-2\mu) \eta}\norms{\bar{x}^k - x^{\star}}^2 + \frac{[\eta(1-\tilde{c}) - 4\rho] t_{k-1}^2}{2}\norms{w^k}^2.
\end{array}
\end{equation}
Let us first choose $\tilde{c} := \frac{1}{2}$.
Since $\eta > 8\rho$, \eqref{eq:DFEG4NI_descent_pro_proof5} reduces to $\Lc_k \geq \frac{(2- \mu)}{(1-2\mu)\eta}\norms{ \bar{x}^k - x^{\star} }^2 $.
Substituting this estimate into \eqref{eq:DFEG4NI_descent_pro_proof4}, we can show that
\begin{equation}\label{eq:DFEG4NI_descent_pro_proof6} 
\hspace{-2ex}
\arraycolsep=0.1em
\begin{array}{lcl}
\Lc_{k+1} & \leq & \big[ 1 + \frac{(1-\mu)^2(1-2\mu)}{2\nu(2-\mu) t_k^2}\big] \cdot \Lc_k -  \frac{[\psi_k - 2(1-\mu)\rho] }{2}\norms{w^k}^2 \vspace{1ex}\\
&& - {~} \frac{\eta(\omega - \lambda\kappa) t_k^2}{2} \norms{w^{k+1} - \hat{w}^{k+1}}^2 \vspace{1ex}\\
&& - {~} \frac{\rmark{\eta}(1 - M^2\eta^2 - \lambda\hat{\kappa}\eta^2)}{2}\norms{t_k\hat{w}^{k+1} - (t_k-1)w^k }^2.
\end{array}
\hspace{-5ex}
\end{equation}
Suppose that 
\begin{equation}\label{eq:DFEG4NI_th41_proof7} 
\arraycolsep=0.2em
\begin{array}{lcl}
\omega \geq \lambda\kappa \quad \textrm{and} \quad (M^2 + \lambda\hat{\kappa})\eta^2 \leq 1.
\end{array}
\end{equation}
Then \eqref{eq:DFEG4NI_descent_pro_proof6} reduces to
\begin{equation}\label{eq:DFEG4NI_th41_proof8} 
\arraycolsep=0.2em
\begin{array}{lcl}
0 \leq \Lc_{k+1} & \leq & \big[ 1 + \frac{(1-2\mu)(1-\mu)^2}{2\nu(2-\mu) t_k^2} \big] \cdot \Lc_k - \frac{[\psi_k - 2(1-\mu)\rho]}{2}\norms{w^k}^2.
\end{array}
\end{equation}
By Lemma~\ref{le:A2_sum_bounds}, for $\Lambda_K := \prod_{k=0}^K \big(1 + \frac{\theta\mu^2}{ t_k^2} \big) = \prod_{k=0}^K \big(1 + \frac{\theta}{(k+r)^2} \big) \leq \Lambda := e^{\frac{\theta}{r}}$, where $\theta := \frac{(1-2\mu)(1-\mu)^2}{2\nu(2-\mu)\mu^2}$, we obtain from \eqref{eq:DFEG4NI_th41_proof8} \rmark{and Lemma \ref{le:A1_descent}}  that 
\begin{equation*} 
\arraycolsep=0.2em
\begin{array}{lcl}
\Lc_K & \leq & e^{\frac{\theta}{r}} \cdot \Lc_0 \quad \textrm{and}\quad  \sum_{k=0}^\rmark{\infty}\rmark{\frac{\psi_k - 2(1-\mu)\rho}{2}}\norms{w^k}^2  \leq  \big(  1 + \frac{\theta}{r} e^{\frac{\theta}{r}} \big) \cdot \Lc_0.
\end{array}
\end{equation*}
Note that since $u^0 := Fx^0$ and $x^0 = \bar{x}^0$, we have $\Lc_0 = \Vc_0 = \frac{a_0}{2}\norms{w^0}^2 + \frac{(1-\mu)}{2\gamma}\norms{x^0 - x^{\star}}^2 \leq \frac{\eta \mu^2(r-1)^2}{2}\norms{Fx^0+\xi^0}^2 + \frac{3(1-\mu)}{(1-2\mu)\eta}\norms{x^0 - x^{\star}}^2 =: \Rc_0^2$.
Combining this bound and the last expressions, we get 
\begin{equation}\label{eq:DFEG4NI_th43_proof9} 
\arraycolsep=0.2em
\begin{array}{lcl}
\Lc_K & \leq & e^{\frac{\theta}{ r}} \cdot \Rc_0^2, \vspace{1ex}\\
\sum_{k=0}^{\infty} [ \psi_k - 2(1-\mu)\rho] \norms{w^k}^2 & \leq &   2\big(  1 + \frac{\theta}{r} e^{\frac{\theta}{r}} \big) \cdot \Rc_0^2.
\end{array}
\end{equation}
If we choose $\tilde{c} = \frac{1-2\mu}{6(1-\mu)}$ in  \eqref{eq:DFEG4NI_descent_pro_proof5}, then $\Lc_k \geq \frac{[\eta(5-4\mu) - 24\rho(1-\mu)] \mu^2(k+r-1)^2}{12(1-\mu)}\norms{w^k}^2 \geq \frac{\mu^2\eta(k+r-1)^2}{6}\norms{w^k}^2$ since $\eta \geq 8\rho$ \rmark{and $\mu \in (0, \frac{1}{2})$}.
Combining this bound and the first line of \eqref{eq:DFEG4NI_th43_proof9},  we obtain  \eqref{eq:DFEG4NI_th41_convergence1}.

From \eqref{eq:DFEG4NI_th41_proof1c}, one can also show that $\psi_k - 2(1-\mu)\rho \geq \frac{(3-2\mu)\rmark{(\eta - 8\rho)}}{2}(t_k - 1)$.
Using this bound into the second line of \eqref{eq:DFEG4NI_th43_proof9}, we obtain \eqref{eq:DFEG4NI_th41_summable1}.

Finally, let us choose $\lambda := \frac{\omega}{\kappa}$, $\nu := 1$, and $\omega := 4\nu\kappa = 4\kappa$.
Then, the first condition of \eqref{eq:DFEG4NI_th41_proof7} automatically holds, while the second one becomes $\big[2(1+4\kappa)L^2 + 4\hat{\kappa} \big] \eta^2 \leq 1$.
Combining this condition and $M^2\eta^2 + 2\nu + 1 = 2(1 + 4\kappa)L^2\eta^2 +  3 \leq \lambda = 4$ from \eqref{eq:DFEG4NI_descent_pro_proof3}, we get
\begin{equation*} 
\arraycolsep=0.2em
\begin{array}{lcl}
\eta \leq  \bar{\eta} := \min\set{ \frac{1}{\sqrt{2(1+4\kappa)L^2 + 4\hat{\kappa}}}, \frac{1}{L\sqrt{2(1+4\kappa)}} } = \rmark{\frac{1}{\sqrt{2(1+4\kappa)L^2 + 4\hat{\kappa}}}}.
\end{array}
\end{equation*}
This condition is exactly the one in \eqref{eq:DFEG4NI_eta_choice1}.
In addition, since $\nu = 1$, $M^2\eta^2 < 1$ and $\mu \in \rmark{(0, \frac{1}{2})}$, the second condition of \eqref{eq:DFEG4NI_descent_pro_proof3}  holds if $t_{k-1} \geq 2$, which is guaranteed if we choose $r \geq 1 + \frac{2}{\mu}$.
Combining this condition and $r \geq \frac{1}{\mu}$ above, we get $r \geq 1 + \frac{2}{\mu}$.
\Eproof

\beforesubsec
\subsection{\mytb{The Proof of Theorem~\ref{th:DFEG4NI_small_o_rates}}}\label{apdx:th:DFEG4NI_small_o_rates}
\aftersubsec
First, the condition $r \geq 1 + \frac{2}{\mu}$ is sufficiently tight to guarantee $S_k  \leq (1-\alpha) t_{k-1}^2$ for some sufficiently small $\alpha \in (0, 1)$ instead of the second condition of \eqref{eq:DFEG4NI_descent_pro_proof3}.
Hence, we have $b_k + \nu\eta t_k^2 \leq \lambda \eta t_{k-1}^2 - \alpha\eta t_{k-1}^2$.
Similar to the proof of \eqref{eq:DFEG4NI_descent_pro_proof6}, but using this condition, we have
\begin{equation}\label{eq:DFEG4NI_small_o_rate_proof1}
\hspace{-2ex}
\arraycolsep=0.1em
\begin{array}{lcl}
\Lc_{k+1} & \leq & \big(1 + \frac{\theta\mu^2}{t_k^2}\big)\Lc_k -  \frac{[\psi_k - 2(1-\mu)\rho] }{2}\norms{w^k}^2 - \frac{\eta(\omega - \lambda\kappa) t_k^2}{2} \norms{w^{k+1} - \hat{w}^{k+1}}^2 \vspace{1ex}\\
&& - {~} \frac{\rmark{\eta}(1 - M^2\eta^2 - \lambda\hat{\kappa}\eta^2)}{2}\norms{t_k\hat{w}^{k+1} - (t_k-1)w^k }^2 - \frac{\alpha\eta t_{k-1}^2}{2}\norms{e^k}^2.
\end{array}
\hspace{-5ex}
\end{equation}
Next, since we assume that $\eta < \bar{\eta}$ in \eqref{eq:DFEG4NI_eta_choice1}, we can show from \eqref{eq:DFEG4NI_th41_proof7} that $\omega > \lambda\kappa$ and $(M^2 + \lambda\hat{\kappa})\eta^2 < 1$.
Hence, we can derive from \eqref{eq:DFEG4NI_small_o_rate_proof1} that
\begin{equation}\label{eq:DFEG4NI_small_o_rate_proof2}
\arraycolsep=0.2em
\begin{array}{lcl}
\sum_{k=0}^{\infty} \norms{t_k\hat{w}^{k+1} - (t_k-1)w^k}^2 & < &  +\infty, \vspace{1ex}\\
\sum_{k=0}^{\infty} t_k^2 \norms{\hat{w}^{k+1} - w^{k+1}}^2 & < & +\infty, \vspace{1ex}\\
\sum_{k=0}^{\infty} t_k^2 \norms{e^k}^2  & < &  +\infty.
\end{array}
\end{equation}
The second and third bounds of \eqref{eq:DFEG4NI_small_o_rate_proof2} are respectively the first two lines of \eqref{eq:DFEG4NI_small_o_rates_summable_results}.
Since $t_k(x^{k+1} - y^k) = -\eta\big(t_k\hat{w}^{k+1} - (t_k-1)z^k\big)$ due to \eqref{eq:DFEG4NI} and $e^k = z^k - w^k$,  by Young's inequality, we can show that
\begin{equation*} 
\arraycolsep=0.2em
\begin{array}{lcl}
t_k\norms{z^k}^2 & \leq & 2t_k\norms{w^k}^2 + 2t_k\norms{z^k - w^k}^2 = 2t_k\norms{w^k}^2 + 2t_k\norms{e^k}^2, \vspace{1ex}\\
t_k^2\norms{w^{k+1} - w^k}^2  & \leq & 3\norms{t_k\hat{w}^{k+1} - (t_k-1)w^k}^2 + 3t_k^2\norms{\hat{w}^{k+1} - w^{k+1}}^2 + 3\norms{w^k}^2, \vspace{1ex}\\
t_k^2\norms{x^{k+1} - y^k}^2 & = & \eta^2\norms{t_k\hat{w}^{k+1} - (t_k-1)z^k}^2 \vspace{1ex}\\
& \leq & 2\eta^2\norms{t_k\hat{w}^{k+1} - (t_k-1)w^k}^2 + 2\eta^2(t_k-1)^2\norms{e^k}^2.
\end{array}
\end{equation*}
All the terms on the right-hand side of these inequalities are summable due to \eqref{eq:DFEG4NI_small_o_rate_proof2} and \eqref{eq:DFEG4NI_th41_summable1}, we get
\begin{equation}\label{eq:DFEG4NI_small_o_rate_proof2b}
\arraycolsep=0.2em
\begin{array}{lcl}
\sum_{k=0}^{\infty} t_k \norms{z^k}^2 & < & +\infty, \vspace{1ex}\\
\sum_{k=0}^{\infty} t_k^2 \norms{w^{k+1} - w^k}^2 & < & +\infty, \vspace{1ex}\\
\sum_{k=0}^{\infty} t_k^2 \norms{x^{k+1} - y^k}^2 & < & +\infty.
\end{array}
\end{equation}
Now, from the first line of \eqref{eq:DFEG4NI} and $\tau_k = \frac{1}{t_k}$, we can easily derive that
\begin{equation*} 
\arraycolsep=0.2em
\begin{array}{lcl}
v^k & := &  \bar{x}^k - x^k = t_k(y^k - x^k) +   [(\eta-4\rho)(t_k - 1) + \gamma]z^k.
\end{array}
\end{equation*}
Using this expression, \rmark{the third line of \eqref{eq:DFEG4NI}}, and $x^{k+1} - y^k = -\frac{\eta}{t_k}(t_k\hat{w}^{k+1} - (t_k-1)z^k)$, we can derive
\begin{equation*} 
\arraycolsep=0.15em
\begin{array}{lcl}
v^{k+1} - \frac{(t_k-1)}{t_k}v^k & = &  \bar{x}^{k+1} - \bar{x}^k - (x^{k+1} - y^k) + \frac{[(\eta-4\rho)(t_k - 1) + \gamma]}{t_k}z^k \vspace{1ex}\\
& = & -\gamma z^k + \eta\hat{w}^{k+1} - \frac{\eta(t_k-1)}{t_k}z^k + \frac{[(\eta-4\rho)(t_k - 1) + \gamma]}{t_k}z^k \vspace{1ex}\\
& = &  \eta\big( \hat{w}^{k+1} - \frac{t_k-1}{t_k}w^k \big)  + \frac{(t_k-1)}{t_k}(\eta - 4\rho - \gamma)z^k - \frac{\rmark{\eta}(t_k-1)}{t_k}e^k.
\end{array}
\end{equation*}
Let us denote $d^k := \eta (t_k \hat{w}^{k+1} - (t_k-1)w^k)  + (\eta - 4\rho - \gamma)(t_k-1)w^k - \rmark{\eta}(t_k-1)e^k$.
Then, the last estimate is rewritten as $v^{k+1} = \big(1 - \frac{1}{t_k}\big)v^k + \frac{1}{t_k}d^k$.
By convexity of $\norms{\cdot}^2$, $\frac{1}{t_k} \in (0, 1]$, and Young's inequality, we have 
\begin{equation*} 
\arraycolsep=0.2em
\begin{array}{lcl}
\norms{v^{k+1}}^2 & \leq & \frac{t_k-1}{t_k}\norms{v^k}^2 + \frac{1}{t_k}\norms{d^k}^2 \vspace{1ex}\\
& \leq &  \norms{v^k}^2 - \frac{1}{t_k}\norms{v^k}^2 + \frac{3\eta^2}{t_k} \norms{t_k w^{k+1} - (t_k-1)w^k}^2 \vspace{1ex}\\
&& +  {~} \frac{3(\eta - 4\rho - \gamma)^2(t_k-1)^2}{t_k} \norms{w^k}^2 + \frac{3\rmark{\eta}^2(t_k-1)^2}{t_k}\norms{e^k}^2.
\end{array}
\end{equation*}
Since the last three terms of this inequality are summable due to \eqref{eq:DFEG4NI_small_o_rate_proof2} and \eqref{eq:DFEG4NI_th41_summable1}, one can easily prove that
\begin{equation}\label{eq:DFEG4NI_small_o_rate_proof3} 
\lim_{k\to\infty} \norms{v^k}^2 = \lim_{k\to\infty} \norms{x^k - \bar{x}^k}^2 = 0, \ \ \textrm{and} \ \ \sum_{k=0}^{\infty}\tfrac{1}{t_k}\norms{v^k}^2 < +\infty.
\end{equation}
Next, applying Young's inequality again, we have 
\begin{equation*} 
\arraycolsep=0.2em
\begin{array}{lcl}
t_k\norms{y^k - x^k}^2 & \leq & \frac{2}{t_k}\norms{v^k}^2 + \frac{2[(\eta-4\rho)(t_k-1) + \gamma]^2}{t_k} \norms{z^k}^2, \vspace{1ex}\\
t_k\norms{x^{k+1} - x^k}^2 & \leq & 2t_k\norms{x^{k+1} - y^k}^2 +  2t_k\norms{y^k - x^k}^2.
\end{array}
\end{equation*}
Since the two terms on the right-hand side of \rmark{the first} inequality are summable due to \eqref{eq:DFEG4NI_small_o_rate_proof3}  and \eqref{eq:DFEG4NI_small_o_rate_proof2b}, we get 
\begin{equation}\label{eq:DFEG4NI_small_o_rate_proof4} 
\arraycolsep=0.0em
\begin{array}{lcl}
\sum_{k=0}^{\infty} t_k \norms{y^k - x^k}^2 < +\infty \quad \text{and} \quad \sum_{k=0}^{\infty} t_k \norms{x^{k+1} - x^k}^2 < +\infty.
\end{array}
\end{equation}
These results lead to the third and fourth lines of \eqref{eq:DFEG4NI_small_o_rates_summable_results}.

We have proven in Theorem~\ref{th:DFEG4NI_convergence1} that $\frac{(2- \mu)}{(1-2\mu)\eta}\norms{ \bar{x}^k - x^{\star} }^2 \leq \Lc_k \leq \Lc_0$.
Hence, $\sets{\norms{\bar{x}^k - x^{\star}}}$ is bounded, i.e. there exists $M > 0$ such that $\norms{\bar{x}^k - x^{\star}} \leq M$ for all $k \geq 0$.
Using this fact and $\iprods{w^k, x^k - x^{\star}} \geq -\rho\norms{w^k}^2$, we can deduce that
\begin{equation*} 
\arraycolsep=0.2em
\begin{array}{lcl}
\norms{\bar{x}^{k+1} - x^{\star}}^2 & = & \norms{\bar{x}^k - x^{\star}}^2 + 2\iprods{\bar{x}^{k+1} - \bar{x}^k, \bar{x}^k - x^{\star}} + \norms{\bar{x}^{k+1} - \bar{x}^k}^2 \vspace{1ex}\\
& = & \norms{\bar{x}^k - x^{\star}}^2 - 2\gamma\iprods{z^k, \bar{x}^k - x^{\star}} + \gamma^2 \norms{z^k}^2 \vspace{1ex}\\
& = & \norms{\bar{x}^k - x^{\star}}^2 - 2\gamma\iprods{e^k, \bar{x}^k - x^{\star}} - 2\gamma\iprods{w^k, x^k - x^{\star}} \vspace{1ex}\\
&& - {~} 2\gamma\iprods{w^k, \bar{x}^k - x^k} + \gamma^2 \norms{z^k}^2 \vspace{1ex}\\
& \leq & \norms{\bar{x}^k - x^{\star}}^2 + \gamma t_k^2 \norms{e^k}^2 + \frac{\rmark{\gamma}}{t_k^2}\norms{\bar{x}^k - x^{\star}}^2 + 2\gamma \rho\norms{w^k}^2 \vspace{1ex}\\
&& + {~} \gamma t_k \norms{w^k}^2 + \frac{\gamma}{t_k}\norms{\bar{x}^k - x^k}^2 + \gamma^2 \norms{z^k}^2 \vspace{1ex}\\
& \leq &  \norms{\bar{x}^k - x^{\star}}^2 + \gamma t_k^2 \norms{e^k}^2 + \frac{M^2\rmark{\gamma}}{t_k^2} + \gamma(t_k + 2 \rho)\norms{w^k}^2 \vspace{1ex}\\
&& + {~}  \frac{\gamma}{t_k}\norms{v^k}^2 + \gamma^2 \norms{z^k}^2.
\end{array}
\end{equation*}
Since the last five terms on the right-hand side are summable, from this inequality, we conclude that $\lim_{k\to\infty}\norms{\bar{x}^k - x^{\star}}^2$ exists.
However, since  $\lim_k\norms{x^k - \bar{x}^k}^2 = 0$ by \eqref{eq:DFEG4NI_small_o_rate_proof3}, using the triangle inequality $\vert \norms{x^k - x^{\star}} - \norms{\bar{x}^k - x^{\star}} \vert \leq \norms{x^k - \bar{x}^k}$, we conclude that $\lim_{k\to\infty}\norms{x^k - x^{\star}}^2$ also exists.

Next, from \eqref{eq:DFEG4NI_th41_proof8}, we can easily show that $\lim_{k\to\infty}\Lc_k$ exists.
However, since $\Lc_k = \Vc_k + \frac{\lambda \eta t_k^2}{2}\norms{e^k}^2$ due to \eqref{eq:DFEG4NI_potential_func} and $\lim_{k\to\infty}t_k^2\norms{e^k}^2 = 0$ by the last line of \eqref{eq:DFEG4NI_small_o_rate_proof2}, we conclude that $\lim_{k\to\infty}\Vc_k$ also exists.

By \eqref{eq:DFEG4NI_th41_convergence1}, there exists $\bar{M} > 0$ such that $t_{k-1}\norms{w^k} \leq \bar{M}$.
Using this inequality and $\lim_{k\to\infty} \norms{v^k} = 0$, we can show that
\begin{equation*} 
\arraycolsep=0.2em
\begin{array}{lcl}
\rmark{t_{k-1}} \vert \iprods{w^k, x^k - \bar{x}^k} \vert &\leq & t_{k-1}\norms{w^k}\norms{v^k} \leq \bar{M}\norms{v^k} \to 0 \quad \text{as}~k \to \infty.
\end{array}
\end{equation*}
From \eqref{eq:DFEG4NI_potential_func}, we have
\begin{equation*} 
\arraycolsep=0.2em
\begin{array}{lcl}
\Vc_k = \frac{a_k}{2}\norms{w^k}^2 + t_{k-1}\iprods{w^k, x^k - \bar{x}^k} + \frac{(1-\mu)}{2\gamma}\norms{\bar{x}^k - x^{\star}}^2.
\end{array}
\end{equation*}
Using the above limits, we can say that $\lim_{k\to\infty}a_k\norms{w^k}^2$ exists.
Since $a_k = (\eta - 4\rho)t_{k-1}^2 + 4\rho t_{k-1}$, we conclude that $\lim_{k\to\infty}t_k^2\norms{w^k}^2$ exists.
However, since $\sum_{k=0}^{\infty}t_k\norms{w^k}^2 < +\infty$ due to \eqref{eq:DFEG4NI_th41_summable1}, we can easily argue that $\lim_{k\to\infty}t_k^2\norms{w^k}^2 = 0$ as stated in the first line of \eqref{eq:DFEG4NI_small_o_rates}.
Note that 
\begin{equation*} 
\arraycolsep=0.2em
\begin{array}{lcl}
t_k^2\norms{y^k - x^k}^2 & \leq & 3\norms{v^k}^2 + 3[(\eta-4\rho)(t_k - 1) + \gamma]^2\norms{w^k}^2 \vspace{1ex}\\
&& + {~} 3[(\eta-4\rho)(t_k - 1) + \gamma]^2\norms{e^k}^2.
\end{array}
\end{equation*}
The limit of three terms on the right-hand side are zero, leading to $\lim_{k\to\infty}t_k^2\norms{y^k - x^k}^2 = 0$, which is the second line of \eqref{eq:DFEG4NI_small_o_rates}.
The third line of \eqref{eq:DFEG4NI_small_o_rates} follows from Young's inequality,  the second line of \eqref{eq:DFEG4NI_small_o_rates}, and $\lim_{k\to\infty}t_k^2\norms{x^{k+1} - y^k}^2 = 0$.

Finally, since $\lim_{k\to\infty}\norms{x^k - x^{\star}}^2$ exists and $\gra{\Phi}$ is closed (see the proof of Theorem~\ref{apdx:th:NGEAG4NI_V1_convergence3} below), and one can easily show that any cluster point of $\sets{x^k}$ is a solution of \eqref{eq:NI}, we conclude that $\sets{x^k}$ converges to $x^{\star}$, a solution of \eqref{eq:NI}.
Since $\lim_{k\to\infty}\norms{x^k - \bar{x}^k} = 0$ and $\lim_{k\to\infty}\norms{x^k - y^k} = 0$, both $\sets{\bar{x}^k}$ and $\sets{y^k}$ also converge to $x^{\star}$.
\Eproof

\beforesubsec
\subsection{\mytb{The Proof of Theorem~\ref{th:DFEG4NI_convergence2}}}\label{apdx:th:DFEG4NI_convergence2}
\aftersubsec
First, since $u^k := Fx^k$, we have $e^k = 0$. 
Thus we can choose $c_1 = c_2 = \hat{c} = c = 0$ and $\omega = 0$ in Lemma~\ref{le:DFEG4NI_descent_pro}.

In this case, from \eqref{eq:DFEG4NI_param}, we have $t_k$ and $\beta_k$ are updated as in \eqref{eq:DFEG4NI_param_update2}, i.e.:
\begin{equation*} 
\arraycolsep=0.2em
\begin{array}{lcl}
t_k := \mu(k + t_0) = \mu(k + r),  \quad \textrm{and} \quad \beta_k := \frac{2\rho(t_k-1) - \gamma}{t_k} = -\gamma \tau_k + 2\rho(1-\tau_k).
\end{array}
\end{equation*}
In addition, $a_k$ and $\hat{a}_k$ in \eqref{eq:DFEG4NI_descent_property} respectively reduce to
\begin{equation}\label{eq:DFEG4NI_th42_proof2} 
\arraycolsep=0.1em
\begin{array}{lcl}
a_k &:= &  t_{k-1}\big[  \eta t_{k-1} - 2\rho (t_{k-1}-1) \big] \quad \textrm{and} \quad \hat{a}_k = a_k - \psi_k,
\end{array}
\end{equation}
where 
\begin{equation}\label{eq:DFEG4NI_th42_proof2b} 
\arraycolsep=0.1em
\begin{array}{lcl}
\psi_k 
& := & 2[(1-\mu)(\eta-2\rho) - \gamma] (t_k - 1) + \rmark{(1-\mu)[(1-\mu)\eta + 2\mu \rho - \gamma]}.
\end{array}
\end{equation}
Second, since $c = 0$, $\omega = 0$, and $e^k = 0$, we have $M^2 = (1 + c)(1+\omega)L^2 = L^2$ and $\Lc_k = \Vc_k + \frac{\lambda \eta t_k^2}{2}\norms{e^k}^2 = \Vc_k$ in \eqref{eq:DFEG4NI_potential_func}.
Moreover, due to the $\rho$-co-hypomonotonicity of $\Phi$ and $x^{\star} \in \zer{\Phi}$, we have $\iprods{w^k, x^k - x^{\star}} \geq -\rho\norms{w^k}^2$.
Using these facts and \eqref{eq:DFEG4NI_th42_proof2}, \eqref{eq:DFEG4NI_descent_property} reduces to
\begin{equation}\label{eq:DFEG4NI_th42_proof3} 
\arraycolsep=0.2em
\begin{array}{lcl}
\Lc_k - \Lc_{k+1} & \geq &   \frac{\rmark{\eta}(1 - L^2\eta^2)}{2}\norms{t_k\hat{w}^{k+1} - (t_k-1)w^k }^2 + \frac{\psi_k}{2}\norms{w^k}^2 \vspace{1ex}\\
&& + {~}  (1-\mu)\iprods{w^k, x^k - x^{\star} } \vspace{1ex}\\
& \geq & \frac{\rmark{\eta}(1 - L^2\eta^2)}{2}\norms{t_k\hat{w}^{k+1} - (t_k-1)w^k }^2 + \frac{(\psi_k - 2(1-\mu)\rho)}{2}\norms{w^k}^2.
\end{array}
\end{equation}
To guarantee the nonnegativity of the right-hand \rmark{side} of \eqref{eq:DFEG4NI_th42_proof3}, we require $L\eta \leq 1$ and $\psi_k \geq \rmark{2(1-\mu)\rho}$.
The first condition holds if $\eta \leq \frac{1}{L}$, while the second one holds if $\eta > 2\rho$, $\gamma < (1-\mu)(\eta - 2\rho)$, and $t_k \geq 1$.
The first two conditions are stated in \eqref{eq:DFEG4NI_eta_choice2}, while $t_k \geq 1$ holds if $\mu r \geq 1$ as stated in Theorem~\ref{th:DFEG4NI_convergence2}.
Since $2L\rho < 1$, the condition $2\rho < \eta \leq \frac{1}{L}$ is well-defined.

Third, by induction, we obtain from \eqref{eq:DFEG4NI_th42_proof3} that
\begin{equation}\label{eq:DFEG4NI_th42_proof4}
\arraycolsep=0.2em
\left\{\begin{array}{lcl}
\Lc_k & \leq & \Lc_0, \quad \forall k \geq 0, \vspace{1ex}\\
\rmark{\eta}(1-L^2\eta^2) \sum_{k=0}^{\infty} \norms{t_k\hat{w}^{k+1} - (t_k-1)w^k }^2   & \leq & 2\Lc_0, \vspace{1ex}\\
\sum_{k=0}^{\infty}[ \psi_k - 2(1-\mu)\rho] \norms{w^k }^2   & \leq & 2\Lc_0.
\end{array}\right.
\end{equation}
By Lemma~\ref{le:DFEG4NI_Vk_lowerbound} and \eqref{eq:DFEG4NI_potential_func}, with $\tilde{c} := \frac{\gamma}{\eta(1-\mu)}$, we have 
\begin{equation}\label{eq:DFEG4NI_th42_proof5}
\arraycolsep=0.2em
\begin{array}{lcl}
\Lc_k & \geq & \frac{[(1-\mu)(\eta - 2\rho) - \gamma] t_{k-1}^2 }{\rmark{2(1-\mu)}} \norms{w^k}^2, \vspace{1ex}\\
\Lc_0 & = & \frac{\mu(r-1)[ \mu(r-1)(\eta - 2\rho) + 2\rho]}{2}\norms{w^0}^2 + \frac{(1-\mu)}{2\gamma}\norms{x^0 - x^{\star}}^2.
\end{array}
\end{equation}
Combining the first line of \eqref{eq:DFEG4NI_th42_proof4} and \eqref{eq:DFEG4NI_th42_proof5}, we obtain the first line of \eqref{eq:DFEG4NI_convergence2}.

Finally, since $\psi_k - 2(1-\mu)\rho = [(1-\mu)(\eta-2\rho) - \gamma] \rmark{(2t_k - 1 - \mu)}$ due to \eqref{eq:DFEG4NI_th42_proof2b}, we obtain the second line of \eqref{eq:DFEG4NI_convergence2} from the third line of \eqref{eq:DFEG4NI_th42_proof4} and the second line of \eqref{eq:DFEG4NI_th42_proof5}.
\Eproof

\beforesec
\section{The Proof of Technical Results in Section~\ref{sec:AEG4NI}}\label{apdx:sec:AEG4NI}
\aftersec
This appendix presents the full proof of technical results in Section~\ref{sec:AEG4NI}.

\beforesubsec
\subsection{\mytb{The Proof of Lemma~\ref{le:NesEAG4NI_key_estimate1}}}\label{apdx:le:NesEAG4NI_key_estimate1}
\aftersubsec
First, by inserting a zero term $(t_k-1)\hat{x}^{k+1} - (t_k-1)\hat{x}^{k+1}$, we can expand $\Tc_{[1]} := \norms{\hat{x}^k - x^{\star} + t_k(y^k - \hat{x}^k) }^2$ as follows:
\begin{equation*} 
\arraycolsep=0.2em
\begin{array}{lcl}
\Tc_{[1]} & := & \norms{\hat{x}^k - x^{\star} + t_k(y^k - \hat{x}^k) }^2   \vspace{1ex}\\
& = & \norms{\hat{x}^{k+1} - x^{\star} + (t_k - 1)(\hat{x}^{k+1} - \hat{x}^k) + t_k(y^k - \hat{x}^{k+1} ) }^2   \vspace{1ex}\\
&= &\norms{\hat{x}^{k+1} - x^{\star}}^2 + (t_k - 1)^2\norms{\hat{x}^{k+1} - \hat{x}^k}^2 +  t_k^2\norms{y^k - \hat{x}^{k+1}}^2 \vspace{1ex}\\
&& + {~} 2(t_k -1)\iprods{\hat{x}^{k+1} - \hat{x}^k, \hat{x}^{k+1} - x^{\star}} + 2t_k\iprods{y^k - \hat{x}^{k+1}, \hat{x}^{k+1} - x^{\star}} \vspace{1ex}\\
&& + {~} 2(t_k-1)t_k\iprods{y^k - \hat{x}^{k+1}, \hat{x}^{k+1} - \hat{x}^k}.
\end{array} 
\end{equation*}
Similarly, using $y^{k+1} - \hat{x}^{k+1} = \theta_k (\hat{x}^{k+1} - \hat{x}^k) + \nu_k (y^k - \hat{x}^{k+1})$ from the last line of \eqref{eq:NesEAG4NI}, we can expand $\Tc_{[2]} := \norms{\hat{x}^{k+1}  -  x^{\star} + t_{k+1}(y^{k+1} - \hat{x}^{k+1}) }^2$ as
\begin{equation*} 
\arraycolsep=0.2em
\begin{array}{lcl}
\Tc_{[2]} & := &  \norms{\hat{x}^{k+1}  -  x^{\star} + t_{k+1}(y^{k+1} - \hat{x}^{k+1}) }^2  \vspace{1ex}\\ 
& = &  \norms{\hat{x}^{k+1}  -  x^{\star} + t_{k+1}\theta_k(\hat{x}^{k+1} - \hat{x}^k) + t_{k+1}\nu_k(y^k - \hat{x}^{k+1}) }^2  \vspace{1ex}\\ 
&= & \norms{\hat{x}^{k+1} - x^{\star}}^2 + t_{k+1}^2\theta_k^2\norms{\hat{x}^{k+1} - \hat{x}^k}^2 +  t_{k+1}^2\nu_k^2\norms{y^k  - \hat{x}^{k+1} }^2 \vspace{1ex}\\
&& + {~} 2t_{k+1}\theta_k \iprods{\hat{x}^{k+1} - \hat{x}^k, \hat{x}^{k+1} - x^{\star} } + 2t_{k+1}\nu_k \iprods{y^k - \hat{x}^{k+1}, \hat{x}^{k+1} - x^{\star}} \vspace{1ex}\\
&& + {~} 2t_{k+1}^2\nu_k \theta_k \iprods{y^k - \hat{x}^{k+1}, \hat{x}^{k+1} - \hat{x}^k}.
\end{array} 
\end{equation*}
Combining the last two expressions, we can show that
\begin{equation*} 
\arraycolsep=0.2em
\begin{array}{lcl}
\Tc_{[3]}  &:= &   \norms{\hat{x}^k  - x^{\star} + t_k(y^k - \hat{x}^k) }^2 - \norms{\hat{x}^{k+1}  - x^{\star} + t_{k+1}(y^{k+1} - \hat{x}^{k+1})  }^2 \vspace{1ex}\\
&= & \left[ (t_k-1)^2 - t_{k+1}^2\theta_k^2 \right] \norms{\hat{x}^{k+1} - \hat{x}^k }^2 + (t_k^2  - t_{k+1}^2 \nu_k^2 )\norms{y^k - \hat{x}^{k+1} }^2 \vspace{1ex}\\
&& + {~} 2(t_k - 1 - t_{k+1}\theta_k )\iprods{\hat{x}^{k+1} - \hat{x}^k, \hat{x}^{k+1} - x^{\star}} \vspace{1ex}\\
&& + {~} 2(t_k - t_{k+1}\nu_k )\iprods{y^k - \hat{x}^{k+1}, \hat{x}^{k+1} - x^{\star}} \vspace{1ex}\\
&& + {~} 2 \left[ t_k(t_k-1) - t_{k+1}^2\theta_k \nu_k \right] \iprods{ y^k - \hat{x}^{k+1}, \hat{x}^{k+1} - \hat{x}^k }.
\end{array} 
\end{equation*}
Now, if we define the following function:
\begin{equation*} 
\arraycolsep=0.2em
\begin{array}{lcl}
\Qc_k & := &  b_k\iprods{w^k,  \hat{x}^k - y^k} + \norms{\hat{x}^k - x^{\star} + t_k(y^k - \hat{x}^k) }^2,
\end{array}
\end{equation*}
then utilizing this function $\Qc_k$ and  $\hat{x}^{k+1} - y^{k+1} = -\theta_k (\hat{x}^{k+1} - \hat{x}^k) - \nu_k (y^k - \hat{x}^{k+1})$ from \eqref{eq:NesEAG4NI},  we can further derive  from $\Tc_{[3]}$ that
\begin{equation}\label{eq:NesEAG4NI_proof3} 
\hspace{-1ex}
\arraycolsep=0.2em
\begin{array}{lcl}
\Qc_k - \Qc_{k+1}
&= &  \left[ (t_k-1)^2 - t_{k+1}^2\theta_k ^2 \right] \norms{\hat{x}^{k+1} - \hat{x}^k }^2  \vspace{1ex}\\
&& + {~} (t_k^2  - t_{k+1}^2\nu_k^2 )\norms{y^k - \hat{x}^{k+1} }^2 +  b_k\iprods{w^{k+1} - w^k, \hat{x}^{k+1} - \hat{x}^k}   \vspace{1ex}\\
&& + {~} b_{k+1}\iprods{\nu_k w^{k+1} - \theta_k w^k, y^k - \hat{x}^{k+1}} \vspace{1ex}\\
&& + {~} \left( b_{k+1}\theta_k - b_k\right) \big[ \iprods{w^{k+1}, \hat{x}^{k+1} - \hat{x}^k}  + \iprods{w^k, y^k - \hat{x}^{k+1}} \big]  \vspace{1ex}\\
&& + {~} 2(t_k - 1 - t_{k+1}\theta_k )\iprods{\hat{x}^{k+1} - \hat{x}^k, \hat{x}^{k+1} - x^{\star}} \vspace{1ex}\\ 
&& + {~} 2(t_k - t_{k+1}\nu_k )\iprods{y^k - \hat{x}^{k+1}, \hat{x}^{k+1} - x^{\star}} \vspace{1ex}\\
&& + {~} 2 \left[ t_k(t_k-1) - t_{k+1}^2\theta_k \nu_k \right] \iprods{ y^k - \hat{x}^{k+1}, \hat{x}^{k+1} - \hat{x}^k}.
\end{array} 
\hspace{-2ex}
\end{equation}
Since $\theta_k = \frac{t_k - 1}{t_{k+1}}$ and $\nu_k = \frac{t_k}{t_{k+1}}$ from \eqref{eq:NesEAG4NI_para_choice}, and we choose $b_{k+1} = \frac{b_k}{\theta_k}$, we have
\begin{equation}\label{eq:NesEAG4NI_para_cond1}
\arraycolsep=0.2em
\left\{\begin{array}{lclclcl}
t_k - t_{k+1}\nu_k & = &  0, \ \ &  &\quad t_k(t_k-1) -  t_{k+1}^2\theta_k\nu_k & = & 0, \vspace{1ex}\\
t_k - 1  - t_{k+1} \theta_k & =&  0, \quad &  \text{and} \quad &\quad  b_{k+1}\theta_k - b_k & = &  0.
\end{array}\right.
\end{equation}
Substituting \eqref{eq:NesEAG4NI_para_cond1} into \eqref{eq:NesEAG4NI_proof3}, it reduces to
\begin{equation}\label{eq:NesEAG4NI_proof5} 
\arraycolsep=0.2em
\begin{array}{lcl}
\Qc_k - \Qc_{k+1} &= & b_k\iprods{w^{k+1} - w^k, \hat{x}^{k+1} - \hat{x}^k} \vspace{1ex}\\
&& + {~}   b_{k+1} \iprods{\nu_k w^{k+1} - \theta_k w^k, y^k - \hat{x}^{k+1} }.
\end{array} 
\end{equation}
Applying $\hat{x}^k = x^k - \lambda z^k$ from \eqref{eq:NesEAG4NI} to $\myeqc{1}$, the $\rho$-co-hypomonotonicity of $\Phi$ to $\myeqc{2}$, and Young's inequality to $\myeqc{3}$, for any $\hat{c} > 0$, we can derive that
\begin{equation*} 
\arraycolsep=0.2em
\begin{array}{lcl}
\Tc_{[4]}& := &\iprods{w^{k+1} - w^k, \hat{x}^{k+1} - \hat{x}^k} \vspace{1ex}\\
& \overset{\tiny\myeqc{1}}{=} & \iprods{w^{k+1} - w^k, x^{k+1} - x^k} - \lambda\iprods{w^{k+1} - w^k, z^{k+1} - z^k} \vspace{0.8ex}\\
&  \overset{\tiny\myeqc{2}}{ \geq } & -(\lambda + \rho)\norms{w^{k+1} - w^k}^2 - \lambda\iprods{w^{k+1} - w^k, z^{k+1} - w^{k+1}} \vspace{0.5ex}\\
&& + {~} \lambda \iprods{w^{k+1} - w^k, z^k - w^k} \vspace{0.5ex}\\
& \overset{\tiny\myeqc{3}}{ \geq } & - \big[ (1 + \hat{c}) \lambda + \rho \big] \norms{w^{k+1} - w^k}^2 - \frac{\lambda}{2\hat{c}}\norms{z^{k+1} - w^{k+1}}^2 - \frac{\lambda}{2\hat{c}}\norms{z^k - w^k}^2  \vspace{0.8ex}\\
&= &  - \big[ (1+\hat{c})\lambda + \rho\big] \big[ \norms{w^{k+1}}^2 + \norms{w^k}^2 - 2\iprods{w^{k+1}, w^k} \big]  \vspace{1ex}\\
&& - {~} \frac{\lambda}{2\hat{c}}\norms{u^{k+1} - Fx^{k+1}}^2 - \frac{\lambda}{2\hat{c}}\norms{u^k - Fx^k}^2.
\end{array} 
\end{equation*}
Alternatively, combining the first two lines of   \eqref{eq:NesEAG4NI} we get
\begin{equation*} 
\arraycolsep=0.2em
\begin{array}{lcl}
y^k - \hat{x}^{k+1} &= &  \lambda z^{k+1} + \eta \hat{w}^{k+1}  - \eta\gamma_k z^k \vspace{1ex}\\
&= &  \lambda w^{k+1} + \eta \hat{w}^{k+1} -  \eta\gamma_k w^k + \lambda(z^{k+1} - w^{k+1}) - \eta\gamma_k (z^k - w^k) \vspace{1ex}\\
&= &  \lambda w^{k+1} + \eta \hat{w}^{k+1} - \eta\gamma_k w^k + \lambda(u^{k+1} - Fx^{k+1}) - \eta\gamma_k (u^k - Fx^k).
\end{array}
\end{equation*}
Utilizing this expression and Young's inequality in $\myeqc{1}$, for any $\beta > 0$, we obtain
\begin{equation*} 
\arraycolsep=0.2em
\begin{array}{lcl}
\Tc_{[5]} & := & \iprods{\nu_k w^{k+1} - \theta_k w^k, y^k - \hat{x}^{k+1} } \vspace{1ex}\\ 
&= & \iprods{\nu_k  w^{k+1} - \theta_k w^k, \lambda w^{k+1} + \eta \hat{w}^{k+1} - \eta\gamma_k w^k} \vspace{1ex}\\
&& + {~} \lambda \iprods{\nu_k w^{k+1} - \theta_k w^k, u^{k+1} - Fx^{k+1}}   -  \eta \gamma_k \iprods{\nu_k w^{k+1} - \theta_k w^k, u^k - Fx^k} \vspace{0.5ex}\\
& \overset{\tiny\myeqc{1}}{ \geq } &  \lambda\nu_k \norms{w^{k+1}}^2 + \eta\gamma_k\theta_k \norms{w^k}^2 - \left(\eta\gamma_k\nu_k + \lambda\theta_k \right)\iprods{w^{k+1}, w^k}  \vspace{1ex}\\
&& - {~} \eta\theta_k \iprods{w^k, \hat{w}^{k+1}} + \eta\nu_k \iprods{w^{k+1}, \hat{w}^{k+1}} - \frac{\beta }{\nu_k }\norms{\nu_k w^{k+1} - \theta_k w^k}^2  \vspace{1ex}\\
&& - {~} \frac{\lambda^2\nu_k }{2\beta }\norms{u^{k+1} - Fx^{k+1}}^2 - \frac{\eta^2\gamma_k^2 \nu_k }{2\hat{\beta} }\norms{u^k - Fx^k}^2 \vspace{1ex}\\
&= & \big( \lambda - \beta  \big)\nu_k \norms{w^{k+1}}^2 + \theta_k \big( \eta\gamma_k  - \frac{\beta \theta_k }{\nu_k } \big)  \norms{w^k}^2  +  \eta(\nu_k - \theta_k )\iprods{\hat{w}^{k+1}, w^{k+1}} \vspace{1ex}\\
&& - {~} \big[ \eta\gamma_k\nu_k  + \lambda\theta_k - 2\beta \theta_k \big] \iprods{w^{k+1}, w^k} - \frac{\lambda^2 \nu_k }{2\beta }\norms{u^{k+1} - Fx^{k+1}}^2   \vspace{1ex}\\
&& - {~} \frac{\eta^2\gamma_k^2\nu_k }{2 \beta }\norms{u^k - Fx^k}^2 + \eta\theta_k \iprods{\hat{w}^{k+1}, w^{k+1} - w^k}.
\end{array}
\end{equation*}
Substituting $\Tc_{[4]}$ and $\Tc_{[5]}$ above  into \eqref{eq:NesEAG4NI_proof5}, and using the facts that $b_k = b_{k+1}\theta_k$ and $\gamma_k = \frac{t_k-1}{t_k} = \frac{\theta_k}{\nu_k}$ from  \eqref{eq:NesEAG4NI_para_choice}, we can prove that
\begin{equation}\label{eq:NesEAG4NI_proof9} 
\hspace{-2ex}
\arraycolsep=0.2em
\begin{array}{lcl}
\Qc_k - \Qc_{k+1} 
& \geq & - {~} \frac{ b_k }{2} \big( \frac{\lambda}{\hat{c}}  + \frac{ \eta^2\gamma_k }{  \beta } \big) \norms{u^k - Fx^k}^2 -  \frac{ b_k}{2} \big( \frac{\lambda}{\hat{c}} + \frac{\lambda^2 }{ \beta \gamma_k } \big) \norms{u^{k+1} - Fx^{k+1}}^2 \vspace{1ex}\\
&& - {~} b_k \big[ (1+\hat{c})\lambda + \rho + \beta \gamma_k   - \eta\gamma_k \big] \norms{w^k}^2 \vspace{1ex}\\
&& + {~} b_k \big[\frac{  \lambda - \beta   }{ \gamma_k } - (1+\hat{c})\lambda - \rho \big]\norms{w^{k+1}}^2 \vspace{1ex}\\
&& + {~} b_k \big[ (1+ 2\hat{c})\lambda + 2\rho + 2\beta    - \eta  \big] \iprods{w^{k+1}, w^k} \vspace{1ex}\\
&& + {~} \eta b_k\iprods{\hat{w}^{k+1}, w^{k+1} - w^k} +  \frac{\eta b_k(1-\gamma_k) }{\gamma_k} \iprods{\hat{w}^{k+1}, w^{k+1}}.
\end{array} 
\hspace{-7ex}
\end{equation}
Next, by the Lipschitz continuity of $F$, the first line of \eqref{eq:NesEAG4NI}, Young's inequality, and $z^k - w^k = u^k - Fx^k$, for any $c > 0$, one can show that
\begin{equation*}
\arraycolsep=0.2em
\begin{array}{lcl}
\norms{w^{k+1} - \hat{w}^{k+1}}^2 &= & \norms{Fx^{k+1} - Fy^k}^2 \le L^2\norms{x^{k+1} - y^k}^2 \vspace{1ex}\\
&=& L^2\eta^2 \norms{ \hat{w}^{k+1} - \gamma_k w^k - \gamma_k(z^k - w^k)}^2  \vspace{1ex}\\
&\le&  (1+c)L^2\eta^2\norms{ \hat{w}^{k+1} - \gamma_k w^k}^2 + \frac{(1+c)L^2\eta^2\gamma_k^2}{c}\norms{u^k - Fx^k}^2.
\end{array} 
\end{equation*}
Denote $M_c := (1 + c)L^2$.
Partially expanding the last expression yields
\begin{equation*} 
\arraycolsep=0.2em
\begin{array}{lcl}
0 & \geq & \norms{w^{k+1}}^2 + (1 - M_c\eta^2)\norms{\hat{w}^{k+1}}^2  - M_c\eta^2\gamma_k^2 \norms{w^k}^2 - \frac{M_c\eta^2\gamma_k^2}{c} \norms{u^k - Fx^k}^2 \vspace{1ex}\\
&& - {~} 2(1-M_c\eta^2\gamma_k)\iprods{w^{k+1}, \hat{w}^{k+1}} - 2M_c\eta^2\gamma_k\iprods{\hat{w}^{k+1}, w^{k+1} - w^k}.
\end{array}
\end{equation*}
Multiplying both sides of this inequality by $\frac{b_k}{ 2M_c\eta\gamma_k }$ and adding the result to \eqref{eq:NesEAG4NI_proof9}, we arrive at
\begin{equation*}
\arraycolsep=0.2em
\begin{array}{lcl}
\Qc_k - \Qc_{k+1}  &\geq &  -  \frac{ b_k}{ 2 } \big(  \frac{\lambda}{\hat{c}} + \frac{\eta^2\gamma_k }{ \beta } + \frac{ \eta \gamma_k}{c} \big) \norms{u^k - Fx^k}^2 \vspace{1ex}\\
&& - {~} \frac{ b_k }{2} \big( \frac{\lambda}{\hat{c}}  + \frac{\lambda^2 }{\beta \gamma_k  } \big) \norms{u^{k+1} - Fx^{k+1}}^2 \vspace{1ex}\\
&& - {~} b_k\big[ (1+\hat{c})\lambda + \rho +  \beta \gamma_k  - \frac{ \eta \gamma_k}{2} \big] \norms{w^k}^2 \vspace{1ex}\\
&& + {~} b_k \big[ \frac{1}{2M_c\eta\gamma_k } + \frac{ \lambda - \beta }{ \gamma_k } - (1+\hat{c})\lambda - \rho \big] \norms{w^{k+1}}^2 \vspace{1ex}\\
&& - {~} b_k \big[  \eta - (1+2\hat{c})\lambda - 2\rho - 2 \beta   \big] \iprods{w^{k+1}, w^k}  \vspace{1ex}\\
&& + {~} \frac{b_k(1 - M_c\eta^2)}{ 2M_c\eta\gamma_k } \big[ \norms{\hat{w}^{k+1}}^2 - 2\iprods{\hat{w}^{k+1}, w^{k+1}} \big].
\end{array} 
\end{equation*}
Assume that we choose $\mu \geq 0$ such that $\eta - (1+2\hat{c})\lambda - 2\rho - 2\beta  = 2\mu$.
This choice leads to $\eta := (1+2\hat{c})\lambda + 2\rho + 2\beta  + 2\mu$ as stated in \eqref{eq:NesEAG4NI_para_choice}.
Moreover, utilizing this relation and the following two identities
\begin{equation*}
\begin{array}{lcl}
2\iprods{w^{k+1}, w^k} & = &  \norms{w^{k+1}}^2 + \norms{w^k}^2 - \norms{w^{k+1} - w^k}^2, \vspace{1ex}\\
\norms{\hat{w}^{k+1}}^2 - 2\iprods{\hat{w}^{k+1}, w^{k+1}} & = & \norms{w^{k+1} - \hat{w}^{k+1}}^2 - \norms{w^{k+1}}^2,
\end{array}
\end{equation*}
the last inequality becomes
\begin{equation}\label{eq:NesEAG4NI_proof10}
\arraycolsep=0.2em
\begin{array}{lcl}
\Qc_k - \Qc_{k+1}  &\geq &  -   \frac{ b_k}{ 2 } \big(  \frac{\lambda}{\hat{c}} + \frac{\eta^2\gamma_k }{ \beta  } + \frac{ \eta \gamma_k}{c} \big) \norms{u^k - Fx^k}^2 \vspace{1ex}\\
&& - {~}  \frac{ b_k }{2} \big( \frac{\lambda}{\hat{c}} + \frac{\lambda^2 }{\beta \gamma_k  } \big) \norms{u^{k+1} - Fx^{k+1}}^2 \vspace{1ex}\\
&& - {~}  \frac{ a_k }{2} \norms{w^k}^2 + \frac{ \hat{a}_k}{2} \norms{w^{k+1}}^2 + \mu b_k \norms{w^{k+1} - w^k}^2 \vspace{1ex}\\
&& + {~} \frac{ b_k (1 - M_c\eta^2) }{ 2M_c\eta\gamma_k } \norms{w^{k+1} - \hat{w}^{k+1}}^2,
\end{array} 
\end{equation}
where
\begin{equation*}
\arraycolsep=0.2em
\left\{\begin{array}{lcl}
a_k & := & 2b_k \big[ (1+\hat{c})\lambda + \rho + \beta \gamma_k   -  \frac{ \eta\gamma_k}{2}  + \mu \big], \vspace{1ex}\\
\hat{a}_k & := &  2b_k \big[ \frac{\eta}{2\gamma_k} + \frac{ \lambda - \beta }{ \gamma_k } - (1+\hat{c})\lambda - \rho - \mu \big].
\end{array}\right.
\end{equation*}
Finally, using \eqref{eq:NesEAG4NI_para_choice}, one can easily  show that $a_k =  \frac{ b_k}{t_k } \big[  \lambda  t_k  +  (1+2\hat{c})\lambda + 2\rho + 2\mu \big]$ and $\hat{a}_k = \frac{b_{k+1} }{  t_{k+1} } \big[  \lambda t_{k + 1}    +   (1+2\hat{c})\lambda + 2\rho + 2\mu  \big] = a_{k+1}$. 
In this case, substituting $\Pc_k$ from \eqref{eq:NesEAG4NI_Pk_func} into \eqref{eq:NesEAG4NI_proof10}, we obtain \eqref{eq:NesEAG4NI_key_est1}.
\Eproof

\beforesubsec
\subsection{\mytb{The Proof of Lemma~\ref{le:NesEAG4NI_descent_property}}}\label{apdx:le:NesEAG4NI_descent_property}
\aftersubsec
Substituting \eqref{eq:NesEAG4NI_para_choice} and $b_k = b_{k+1}\theta_k$ into \eqref{eq:NesEAG4NI_key_est1}, we get
\begin{equation*} 
\arraycolsep=0.2em
\begin{array}{lcl}
\Pc_k - \Pc_{k+1}  &\geq & \frac{ c_{k+1}}{2} \norms{u^{k+1} - Fx^{k+1} }^2 - \frac{c_k}{2} \norms{u^k - Fx^k}^2 \vspace{1ex}\\
&& - {~} \frac{1}{2} \big[ c_{k+1} + b_k\big( \frac{\lambda}{\hat{c}} + \frac{\lambda^2}{\beta\gamma_k} \big) \big]  \norms{u^{k+1} - Fx^{k+1}}^2 \vspace{1ex}\\
&& + {~}  \mu b_k \norms{w^{k+1} - w^k}^2 + \frac{ b_k (1 - M_c\eta^2) }{ 2M_c\eta\gamma_k } \norms{w^{k+1} - \hat{w}^{k+1}}^2.
\end{array} 
\end{equation*}
Since $b_{k+1} = \frac{b_k}{\theta_k} = \frac{b_kt_{k+1}}{t_k-1}$ and $t_{k+1} = t_k + 1$ from \eqref{eq:NesEAG4NI_para_choice}, one has
\begin{equation*} 
\arraycolsep=0.2em
\begin{array}{lcl}
c_{k+1} + b_k\big( \frac{\lambda}{\hat{c}} + \frac{\lambda^2}{\beta\gamma_k} \big) = \frac{b_kt_k}{t_{k-1}}\big( \frac{2\lambda}{\hat{c}} + \frac{\eta^2}{\beta} + \frac{\eta}{c} + \frac{\lambda}{\beta} \big).
\end{array} 
\end{equation*}
Using this relation, \eqref{eq:NesEAG4NI_u_cond}, and $\Lc_k$ from \eqref{eq:NesEAG4NI_Lyapunov_func} into the last inequality and noting that $Fx^{k+1} - Fy^k = w^{k+1} - \hat{w}^{k+1}$ and $\frac{t_0-1}{t_0} \leq \frac{ t_{k-1}}{t_k} \leq 1$,  we can show that
\begin{equation*} 
\arraycolsep=0.2em
\begin{array}{lcl}
\Lc_k - \Lc_{k+1}  &\geq &   \frac{ b_k t_k }{2t_{k-1}} \big[  \frac{  1 - M_c\eta^2  }{ M_c\eta\ } - \kappa \big( \frac{ 2\lambda }{\hat{c}}  + \frac{\lambda^2 }{\beta   }  +  \frac{\eta^2 }{ \beta  } + \frac{ \eta   }{c  } \big) \big]  \norms{w^{k+1} - \hat{w}^{k+1}}^2 \vspace{1ex}\\
&& + {~}  \frac{b_kt_k}{2t_{k-1}} \big[ \frac{2\mu(t_0-1)}{t_0} -  \hat{\kappa} \big( \frac{ 2\lambda }{\hat{c}} + \frac{\lambda^2 }{\beta   }  +  \frac{\eta^2 }{ \beta  } + \frac{ \eta   }{c  } \big)  \big] \norms{w^{k+1} - w^k}^2.
\end{array} 
\end{equation*}
This exactly proves \eqref{eq:NesEAG4NI_descent_property}.
\Eproof

\beforesubsec
\subsection{\mytb{The Proof of Lemma~\ref{le:NesNAEG4NI_lower_bound_Lk}}}\label{apdx:le:NesNAEG4NI_lower_bound_Lk}
\aftersubsec
By Young's inequality and $z^k - w^k = u^k -Fx^k$, for any $\omega > 0$, we get
\begin{equation*}
\begin{array}{lcl}
-\iprods{w^k, z^k} &=& -\iprods{w^k, z^k - w^k} - \norms{w^k}^2 \ge -\frac{2+\omega}{2}\norms{w^k}^2 - \frac{1}{2\omega}\norms{Fx^k - u^k}^2.
\end{array} 
\end{equation*}
Utilizing this inequality, $\hat{x}^k = x^k - \lambda z^k$ from \eqref{eq:NesEAG4NI}, and  $\iprods{w^k, x^k - x^{\star}} \geq -\rho\norms{w^k}^2$ from the $\rho$-co-hypomonotonicity of $\Phi$, we can show from \eqref{eq:NesEAG4NI_Lyapunov_func} that
\begin{equation}\label{eq:NesEAG4NI_L_lowerbound} 
\arraycolsep=0.2em
\begin{array}{lcl}
\Lc_k  & = & \norms{\hat{x}^k + t_k(y^k - \hat{x}^k) - x^{\star} - \frac{b_k}{2t_k}w^k}^2 + \big( \frac{a_k}{2} - \frac{b_k^2}{4t_k^2} \big) \norms{w^k}^2   \vspace{1ex}\\
&& + {~} \frac{ c_k}{2} \norms{Fx^k - u^k}^2 + \frac{b_k}{t_k}\iprods{w^k, x^k - x^{\star}} - \frac{\lambda b_k}{t_k}\iprods{w^k, z^k} \vspace{1ex}\\  
& \geq & \frac{1}{2} \big[ a_k - \frac{b_k^2}{2t_k^2} - \frac{[ \lambda(2+\omega)  + 2\rho ] b_k}{t_k} \big] \norms{w^k}^2 
+ \frac{1}{2} \big( c_k  - \frac{\lambda b_k}{\omega t_k} \big) \norms{Fx^k - u^k}^2.
\end{array}
\end{equation}
Since $t_{k+1} = t_k + 1 = k + t_0 + 1$ and $b_{k+1} = \frac{b_k}{\theta_k} = \frac{b_kt_{k+1}}{t_k - 1} = \frac{b_kt_{k+1}t_k}{t_kt_{k-1}}$\rmark{,
by} induction, we get $b_k = \frac{b_0 t_kt_{k-1}}{t_0t_{-1}} = \frac{b_0t_kt_{k-1}}{t_0(t_0-1)}$.
\rmark{Moreover}, we also have $a_k = \frac{b_k}{t_k}\big( \lambda t_k +  (1+2\hat{c})\lambda + 2\rho + 2\mu  \big)$ and $c_k := b_k \big(  \frac{\lambda}{\hat{c}} + \frac{\eta^2\gamma_k }{ \beta  } + \frac{ \eta \gamma_k}{c} \big) = \frac{b_k}{t_k}\big[ \frac{\lambda}{\hat{c}} t_k + (\frac{\eta}{\beta} + \frac{1}{c})\eta t_{k-1} \big]$.

Using these formulae, and choosing $\omega := \frac{1}{4}$ and $b_0 := \frac{3\lambda t_0(t_0-1)}{2}$, we can easily show that
\begin{equation*}
\arraycolsep=0.2em
\left\{ \begin{array}{lcl}
a_k - \frac{b_k^2}{2t_k^2} - \frac{ [ \lambda(2+\omega) + 2\rho ] b_k}{t_k} &=& \frac{b_k}{t_k}\big[ \frac{\lambda t_k }{4} + (2\hat{c} - \frac{1}{2})\lambda + 2\mu  \big], \vspace{1ex}\\
c_k - \frac{\lambda b_k}{\omega t_k} &=& \frac{b_k}{t_k}\big[  \lambda ( \frac{t_k}{\hat{c}} - 4) + (\frac{\eta}{\beta} + \frac{1}{c})\eta t_{k-1}  \big].
\end{array}\right.
\end{equation*}
Substituting these expressions into \eqref{eq:NesEAG4NI_L_lowerbound}, we obtain \eqref{eq:NesEAG4NI_lower_bound_Lk}.
\Eproof

\beforesubsec
\subsection{\mytb{The Proof of Theorem~\ref{th:NesEAG4NI_convergence}}}\label{apdx:th:NesEAG4NI_convergence}
\aftersubsec
Let us choose $c := 1$, $\hat{c} := \frac{1}{4}$, $\beta := \frac{\eta}{4}$, and $\mu := \frac{\lambda + 4\sigma \rho}{4}$ in Lemma~\ref{le:NesEAG4NI_descent_property}, where $\sigma := \frac{5}{24}$.
Then, we have $\eta := (1+2\hat{c})\lambda + 2\rho + 2\beta + 2\mu = 2 \lambda + \frac{\eta}{2} + 2(1+\sigma)\rho$, leading to $\eta := 4[\lambda + (1+\sigma)\rho]$ as shown in \eqref{eq:NesEAG4NI_para_choice}.

Using these choices,  we can also show that
$\Lambda := \frac{ 2\lambda }{\hat{c}}  + \frac{\lambda^2 }{\beta   }  +  \frac{\eta^2 }{ \beta  } + \frac{ \eta   }{c  } = 28\lambda + 20(1+\sigma)\rho + \frac{\lambda^2}{\lambda + (1+\sigma)\rho}$.

Since $t_k := k + r$ (i.e. $t_0 = r$), to guarantee the nonnegativity of the right-hand side of \eqref{eq:NesEAG4NI_descent_property}, we needs to enforce the following two conditions:
\begin{equation}\label{eq:NesEAG4NI_para_cond3}
\arraycolsep=0.2em
\begin{array}{lcl}
\frac{1}{M_c\eta} - \eta - \kappa\Lambda \geq 0 \quad \textrm{and} \quad \frac{(r - 1)(\lambda + 4\sigma\rho)}{2r } - \hat{\kappa}\Lambda \geq 0.
\end{array} 
\end{equation}
The first condition of \eqref{eq:NesEAG4NI_para_cond3} becomes 
\begin{equation*}
\arraycolsep=0.2em
\begin{array}{lcl}
4(7\kappa + 1)\lambda + 4(1+\sigma)(5\kappa + 1)\rho + \frac{\kappa \lambda^2}{\lambda+ (1+\sigma)\rho} \leq \frac{1}{4M_c[ \lambda + (1+\sigma)\rho]}.
\end{array} 
\end{equation*}
Since $M_c = 2L^2$ and $\frac{\lambda}{\lambda + (1+\sigma)\rho} \leq 1$, the last condition holds if
\begin{equation*}
\arraycolsep=0.2em
\begin{array}{lcl}
8L^2[ \lambda + (1+\sigma) \rho]\big[ (29\kappa + 4)\lambda + 4(1+\sigma)(5\kappa + 1) \rho \big] \leq 1.
\end{array} 
\end{equation*}
This condition holds if
\begin{equation*}
\arraycolsep=0.2em
\begin{array}{lcl}
4(1+\sigma)\sqrt{2( 5\kappa + 1)}L\rho < 1 \quad \text{and} \quad 0 < \lambda \leq \bar{\lambda} := \frac{2\bar{c}}{ \bar{b} + \sqrt{ \bar{b}^2 + 4\bar{c}}},
\end{array} 
\end{equation*}
where $\bar{b} := \frac{(1+\sigma)(49\kappa + 8)\rho}{29\kappa + 4}$ and $\bar{c} := \frac{1 - 32(1+\sigma)^2( 5\kappa + 1) L^2\rho^2}{8L^2(29\kappa + 4)}$ as given in  \eqref{eq:NesEAG4NI_lambda_bar}.

Since $\sigma = \frac{5}{24}$, the condition $4(1+\sigma)\sqrt{2( 5\kappa + 1)}L\rho < 1$ is equivalent to $L\rho < \frac{ 6 }{ 29 \sqrt{2(5\kappa + 1)} }$ as stated in  \eqref{eq:NesEAG4NI_para_choice}. 

Using again $\sigma = \frac{5}{24}$, we can easily bound $\Lambda = 28\lambda + 20(1+\sigma)\rho + \frac{\lambda^2}{\lambda + (1+\sigma) \rho} \leq 29\lambda + 20(1+\sigma)\rho = 29( \lambda + 4\sigma\rho\big)$.
Thus if we choose $0 \leq \hat{\kappa} \leq \frac{r - 1}{58 r}$ as stated in \eqref{eq:NesEAG4NI_para_choice}, then $0 \leq \hat{\kappa} \leq \frac{(r - 1)(\lambda + 4\sigma\rho)}{2r \Lambda}$.
Hence, the second condition of \eqref{eq:NesEAG4NI_para_cond3} holds.

Under the above conditions, \eqref{eq:NesEAG4NI_descent_property} reduces to $\Lc_{k+1} \leq \Lc_k$ for all $k \geq 0$.
Since $t_0 := r > 1$, $b_k := \frac{ 3\lambda}{2} t_kt_{k-1}$ as computed in Lemma~\ref{le:NesNAEG4NI_lower_bound_Lk}, $c := 1$, $\hat{c} := \frac{1}{4}$, $\mu := \frac{\lambda + 4\sigma \rho}{4}$, $\beta := \frac{\eta}{4}$, and $\eta := 4[\lambda + (1+\sigma)\rho]$, from \eqref{eq:NesEAG4NI_lower_bound_Lk} we can prove that
\begin{equation}\label{eq:NesEAG4NI_convergence_proof1} 
\begin{array}{lcl}
\Lc_k & \geq & \frac{b_k}{8 t_k}\big[ \lambda(t_k + 2) + 8\sigma\rho \big]   \norms{w^k}^2 +  \frac{ 2[6\lambda + 5(1+\sigma)\rho] b_k t_{k-1}}{ t_k } \norms{u^k - Fx^k}^2 \vspace{1ex} \\
& \geq & \frac{3\lambda^2 (k+r - 1)(k + r +2)}{16} \norms{w^k}^2.
\end{array}
\end{equation}
Finally, since $y^{-1} = \hat{x}^{-1} = \hat{x}^0 := x^0$, using the last line of \eqref{eq:NesEAG4NI} at $k = -1$, we can easily show that $y^0 = \hat{x}^0$, leading to $\Lc_0 = \frac{ a_0}{2} \norms{Fx^0 + \xi^0}^2 + \norms{x^0 - x^\star}^2$.
Since $\Lc_{k+1} \leq \Lc_k$, by induction, we have $\Lc_k \leq \Lc_0 = \frac{a_0}{2} \norms{Fx^0 + \xi^0}^2 + \norms{x^0  - x^\star}^2 = \frac{3\lambda(r - 1)[ (r + 2)\lambda + 2(1+\sigma)\rho]}{4}\norms{Fx^0 + \xi^0}^2 + \norms{x^0 - x^\star}^2 =: \frac{\Rc_0^2}{16}$.
Combining this bound and \eqref{eq:NesEAG4NI_convergence_proof1}, we obtain 
\begin{equation*}
\arraycolsep=0.2em
\begin{array}{lcl}
\norms{w^k}^2 &\leq & \frac{ \Rc_0^2 }{3\lambda^2 (k+ r - 1)(k+ r + 2)},
\end{array}
\end{equation*}
which is exactly \eqref{eq:NesEAG4NI_convergence} after substituting $w^k := Fx^k + \xi^k$.
\Eproof

\beforesubsec
\subsection{\mytb{The Proof of Corollary~\ref{cor:GAEG4NI_Fx^k_Fyk-1_convergence}}}\label{apdx:cor:GAEG4NI_Fx^k_Fyk-1_convergence}
\aftersubsec
$\mathrm{(i)}$~Since we choose $u^k := Fx^k$, \eqref{eq:NesEAG4NI_u_cond} holds with $\kappa = \hat{\kappa} = 0$.
We can choose $c := 0$, $\hat{c} := 0$, $\beta := 0$, and $\mu := 0$ in Lemma~\ref{le:NesEAG4NI_descent_property}.
In this case, we get $\eta := \lambda + 2\rho$, $M_c = L^2$, and $c_k = 0$.
Moreover,  \eqref{eq:NesEAG4NI_descent_property} reduces to
\begin{equation*} 
\hspace{-1ex}
\begin{array}{lcl}
\Lc_k - \Lc_{k+1}  &\geq &   \frac{ b_k t_k }{2t_{k-1}} \big(  \frac{  1  }{ L^2\eta} - \eta  \big)  \norms{w^{k+1} - \hat{w}^{k+1}}^2.
\end{array}
\end{equation*}
Therefore, we can choose $0 < \lambda \leq \frac{1}{L} - 2\rho$, provided that $2L\rho < 1$.
Furtheremore, from \eqref{eq:NesEAG4NI_lower_bound_Lk}, we also have
\begin{equation*} 
\begin{array}{lcl}
\frac{3\lambda^2 (k+r - 2)(k+r-1)}{16} \norms{w^k}^2 & \leq & \Lc_k \leq \Lc_0 = \frac{a_0}{2}\norms{w^0}^2 + \norms{x^0 - x^{\star}}^2.
\end{array}
\end{equation*}
Since $a_0 := \frac{3\lambda}{2}(r-1) [ \lambda(r + 1) + 2\rho ]$, the last inequality implies \eqref{eq:NesEAG4NI_Fx^k_convergence}.

$\mathrm{(ii)}$~Since we choose $u^k := Fy^{k-1}$, \eqref{eq:NesEAG4NI_u_cond} holds with $\kappa = 1$ and $\hat{\kappa} = 0$.
In this case, we choose $c := 1$, $\hat{c} := \frac{1}{4}$, $\beta := \frac{\eta}{4}$, and $\mu := 0$.
Then, we get $\eta = 3 \lambda + 4 \rho$ and $M_c = 2L^2$.
Moreover, \eqref{eq:NesEAG4NI_descent_property} reduces to 
\begin{equation*} 
\hspace{-1ex}
\begin{array}{lcl}
\Lc_k - \Lc_{k+1}  &\geq &   \frac{ b_k t_k }{2t_{k-1}} \big( \frac{  1  }{ 2L^2 \eta} - 6\eta -   8\lambda   - \frac{4\lambda^2 }{ \eta  } \big)  \norms{w^{k+1} - \hat{w}^{k+1}}^2.
\end{array}
\end{equation*}
To guarantee $\Gamma :=  \frac{  1  }{ 2L^2 \eta} - 6\eta -   8\lambda   - \frac{4\lambda^2 }{ \eta  }  \geq 0$, we need to impose $8\sqrt{3} L\rho < 1$ and $0 < \lambda \leq \bar{\lambda} := \frac{2\bar{c}}{\bar{b} + \sqrt{\bar{b}^2 + 4\bar{c}}}$ as stated, where $\bar{b} := \frac{272 \rho}{123}$ and $\bar{c} := \frac{1 - 129L^2\rho^2}{164L^2}$.
Moreover, we also have $\Lc_{k+1} \leq \Lc_k \leq \Lc_0$ by induction.
Using this relation and  \eqref{eq:NesEAG4NI_lower_bound_Lk}, we can conclude that
\begin{equation*} 
\begin{array}{lcl}
\frac{3\lambda^2 (k+r-1)(k+r)}{16} \norms{w^k}^2 & \leq & \Lc_k \leq \Lc_0 = \frac{a_0}{2}\norms{w^0}^2 + \norms{x^0  - x^{\star}}^2.
\end{array}
\end{equation*}
Since $a_0 := \frac{3\lambda}{4}(r-1)[ \lambda(2r + 3)  + 4\rho]$, the last inequality implies \eqref{eq:NesEAG4NI_Fy^k-1_convergence}.
\Eproof

\beforesec
\section{The Proof of Technical Results in Section~\ref{sec:NGEAG4NI}}\label{apdx:sec:NGEAG4NI}
\aftersec
This appendix provides the proof of all results in Section~\ref{sec:NGEAG4NI}.

\beforesubsec
\subsection{\mytb{The Proof of Lemma~\ref{le:NGEAG4NI_descent_property1}}}\label{apdx:le:NGEAG4NI_descent_property1}
\aftersubsec
First, from the first line $y^k = x^{k+1} + \eta d^k$ of \eqref{eq:NGEAG4NI}, we have
\begin{equation*} 
\arraycolsep=0.2em
\begin{array}{lcl}
\Tc_{[1]} &:= & \norms{ r(x^k - x^{\star}) + t_k(y^k - x^k)  }^2 \vspace{1ex}\\
& = & \norms{r(x^k - x^{\star}) + t_k(x^{k+1} - x^k + \eta d^k) }^2 \vspace{1ex} \\
& = & \norms{r (x^{k+1} - x^{\star} ) + (t_k - r)(x^{k+1} - x^k) + \eta t_kd^k }^2 \vspace{1ex} \\
& = & r^2 \norms{x^{k+1} - x^{\star}}^2 + (t_k - r)^2 \norms{ x^{k+1} - x^k }^2 + \eta^2 t_k^2 \norms{d^k}^2 \vspace{1ex} \\
&& + {~} 2r(t_k - r)\iprods{x^{k+1} - x^k, x^{k+1} - x^{\star}} + 2\eta r t_k\iprods{d^k, x^{k+1} - x^{\star}} \vspace{1ex} \\
&& + {~} 2\eta t_k(t_k - r) \iprods{d^k, x^{k+1} - x^k}.
\end{array}
\end{equation*}
Second, from the second line $y^{k+1} - x^{k+1} = \theta_k(x^{k+1} - x^k) - p^k$ of \eqref{eq:NGEAG4NI} and $t_{k+1}\theta_k = t_k - r - \mu$, we can also derive that
\begin{equation*} 
\arraycolsep=0.2em
\begin{array}{lcl}
\Tc_{[2]} & := &  \norms{ r( x^{k+1} - x^{\star} ) + t_{k+1}(y^{k+1} - x^{k+1})  }^2 \vspace{1ex}\\
& = & \norms{r(x^{k+1} - x^{\star}) + t_{k+1}\theta_k(x^{k+1} - x^k) - t_{k+1} p^k }^2 \vspace{1ex} \\
& = & r^2 \norms{x^{k+1} - x^{\star}}^2 + (t_k- r - \mu)^2 \norms{x^{k+1} - x^k}^2 + t_{k+1}^2\norms{p^k}^2 \vspace{1ex} \\
&& + {~} 2r(t_k - r - \mu) \iprods{x^{k+1} - x^k, x^{k+1} - x^{\star}} \vspace{1ex} \\
&& - {~} 2rt_{k+1} \iprods{p^k, x^{k+1} - x^{\star}} - 2t_{k+1}(t_k - r - \mu) \iprods{p^k, x^{k+1} - x^k}.
\end{array}
\end{equation*}
Third, combining $\Tc_{[1]}$ and $\Tc_{[2]}$, and the identity $\mu r \norms{x^k - x^{\star}}^2 - \mu r \norms{x^{k+1} - x^{\star}}^2 =  \mu r \norms{x^{k+1} - x^k}^2 - 2r\mu\iprods{x^{k+1} - x^k, x^{k+1} - x^{\star}}$,  we can prove that
\begin{equation}\label{eq:NGEAG4NI_lm1_proof3} 
\hspace{-2ex}
\arraycolsep=0.2em
\begin{array}{lcl}
\Tc_{[3]} &:= & \norms{ r(x^k - x^{\star}) + t_k(y^k - x^k)  }^2 -  \norms{ r( x^{k+1} - x^{\star} ) + t_{k+1}(y^{k+1} - x^{k+1})  }^2 \vspace{1ex}\\
&& + {~} \mu r \norms{x^k - x^{\star} }^2 - \mu r\norms{x^{k+1} - x^{\star} }^2 \vspace{1ex}\\
& = &   \mu(2t_k - r - \mu)  \norms{x^{k+1} - x^k}^2 +   \eta^2 t_k^2 \norms{d^k}^2 - t_{k+1}^2\norms{p^k}^2 \vspace{1ex} \\
&& + {~} 2r \iprods{ \eta t_kd^k + t_{k+1}  p^k, x^{k+1} - x^{\star}} \vspace{1ex} \\
&& + {~} 2\iprods{\eta (t_k-r)t_k  d^k +  t_{k+1}(t_k-r-\mu) p^k, x^{k+1} - x^k}.
\end{array}
\hspace{-6ex}
\end{equation}
Fourth, we process the last two terms \eqref{eq:NGEAG4NI_lm1_proof3} as follows.
We first utilize $d^k$ and $p^k$ from \eqref{eq:NGEAG4NI} and the update rule \eqref{eq:NGEAG4NI_params} to show that
\begin{equation*}
\arraycolsep=0.2em
\begin{array}{lcl}
\Tc_{[4]} &:= & \iprods{\eta t_kd^k + t_{k+1}  p^k, x^{k+1} - x^{\star}}  \vspace{1ex} \\
& = & [(\eta-\beta)t_k - \delta] \iprods{ z^{k+1}, x^{k+1} - x^{\star}}  - [(\eta-\beta)t_k - \eta(r-1)] \iprods{z^k, x^k - x^{\star}} \vspace{1ex}\\
&& - {~} [(\eta-\beta)t_k - \eta(r-1)] \iprods{z^k, x^{k+1} - x^k}.
\end{array}
\end{equation*}
Next, using again \eqref{eq:NGEAG4NI_params}, we get
\begin{equation*}
\arraycolsep=0.2em
\begin{array}{lcl}
\Tc_{[5]} &:= & \iprods{\eta t_k(t_k-r) d^k + t_{k+1}(t_k-r-\mu)  p^k, x^{k+1} - x^k } \vspace{1ex} \\
& = &  \big[ ((\eta-\beta)t_k - \delta)(t_k - r - \mu)  + \mu\eta t_k \big] \iprods{ z^{k+1}, x^{k+1} - x^k} \vspace{1ex}\\
&&  - {~} \big[ \eta(t_k-r)(t_k-r+1) - \beta t_k(t_k-r-\mu) \big] \iprods{ z^k, x^{k+1} - x^k} \vspace{1ex}\\
&& + {~} \mu \eta t_k \iprods{\hat{w}^{k+1} - z^{k+1}, x^{k+1} - x^k}.
\end{array}
\end{equation*}
Then, substituting $\Tc_{[4]}$ and $\Tc_{[5]}$ into $\Tc_{[3]}$ of \eqref{eq:NGEAG4NI_lm1_proof3}, we arrive at
\begin{equation}\label{eq:NGEAG4NI_lm1_proof5}
\arraycolsep=0.2em
\begin{array}{lcl}
\Tc_{[3]} & = &  \mu(2t_k - r - \mu) \norms{x^{k+1} - x^k}^2  +  \eta^2 t_k^2 \norms{d^k}^2 - t_{k+1}^2\norms{p^k}^2 \vspace{1ex} \\
&& + {~} 2r [(\eta-\beta)t_k - \delta]  \iprods{ z^{k+1}, x^{k+1} - x^{\star}} \vspace{1ex}\\
&& - {~} 2r  [(\eta-\beta)t_k - \eta(r-1)] \iprods{ z^k, x^k - x^{\star}} \vspace{1ex}\\
&& + {~}  2\big[ ((\eta-\beta)t_k - \delta)(t_k - r - \mu)  + \mu\eta t_k \big] \iprods{z^{k+1} - z^k, x^{k+1} - x^k } \vspace{1ex}\\
&& - {~} 2  \big[  ((\eta-\beta)r  + \delta)t_k - \eta(r-1)t_k - \delta(r + \mu)  \big] \iprods{ z^k, x^{k+1} - x^k} \vspace{1ex}\\
&& + {~} 2 \mu \eta t_k \iprods{\hat{w}^{k+1} - z^{k+1}, x^{k+1} - x^k}.
\end{array}
\hspace{-3ex}
\end{equation}
Fifth, exploiting again $y^k = x^{k+1} + \eta d^k$ and $y^{k+1} - x^{k+1} = \theta_k(x^{k+1} - x^k) - p^k$ from \eqref{eq:NGEAG4NI}, one can express
\begin{equation*}
\arraycolsep=0.2em
\begin{array}{lcl}
\Tc_{[6]} & = & 2c_k\iprods{z^k, y^k - x^k} - 2c_{k+1}\iprods{z^{k+1}, y^{k+1} - x^{k+1}} \vspace{1ex} \\
& = & 2c_k \iprods{z^k, x^{k+1} - x^k + \eta d^k} - 2c_{k+1} \iprods{z^{k+1}, \theta_k(x^{k+1} - x^k) - p^k} \vspace{1ex} \\
& = & 2c_k \iprods{z^k, x^{k+1} - x^k} - 2c_{k+1}\theta_k \iprods{z^{k+1}, x^{k+1} - x^k} \vspace{1ex}\\
&& + {~} 2\eta c_k\iprods{z^k, d^k} + 2c_{k+1}\iprods{z^{k+1}, p^k}.
\end{array}
\end{equation*}
Sixth, adding $\Tc_{[6]}$ to $\Tc_{[3]}$ in \eqref{eq:NGEAG4NI_lm1_proof5}, and noticing that $\Pc_k - \Pc_{k+1} = \Tc_{[3]} + \Tc_{[6]}$, we further obtain
\begin{equation*}
\arraycolsep=0.2em
\begin{array}{lcl}
\Pc_k - \Pc_{k+1} & = &   \eta^2 t_k^2 \norms{d^k}^2 - t_{k+1}^2\norms{p^k}^2 + 2\eta c_k\iprods{z^k, d^k} + 2c_{k+1}\iprods{z^{k+1}, p^k} \vspace{1ex} \\
&& + {~} \mu(2t_k - r - \mu) \norms{x^{k+1} - x^k}^2 + 2 \mu \eta t_k \iprods{\hat{w}^{k+1} - z^{k+1}, x^{k+1} - x^k} \vspace{1ex}\\
&& + {~}  2r [(\eta-\beta)t_k - \delta]  \iprods{ z^{k+1}, x^{k+1} - x^{\star}} \vspace{1ex}\\
&& - {~} 2r  [(\eta-\beta)t_k - \eta(r-1)] \iprods{ z^k, x^k - x^{\star}} \vspace{1ex}\\
&& + {~} 2  \big[ \omega (t_k - \mu - r) + \mu \eta \big]t_k  \iprods{ z^{k+1} - z^k, x^{k+1} - x^k}.
\end{array}
\end{equation*}
Finally, substituting $\bmark{\Ec_k}$ from \eqref{eq:NGEAG4NI_Pk_and_Ek} into the last expression, we get \eqref{eq:NGEAG4NI_key_property1}.
\Eproof

\beforesubsec
\subsection{\mytb{The Proof of Lemma~\ref{le:NGEAG4NI_Ek_simplification}}}\label{apdx:le:NGEAG4NI_Ek_simplification}
\aftersubsec
Exploiting $d^k := \hat{w}^{k+1} - \gamma_k z^k$ and $p^k := \eta_k z^{k+1} - \lambda_k \hat{w}^{k+1} + \nu_k z^k$ from \eqref{eq:NGEAG4NI_dir}, we can expand $\Ec_k$ in \eqref{eq:NGEAG4NI_Pk_and_Ek}, grouping each term and using $\lambda_k = \frac{\eta t_k}{t_{k+1}}$ yield 
\begin{equation*}
\arraycolsep=0.2em
\begin{array}{lcl}
\Ec_k & \overset{\tiny\eqref{eq:NGEAG4NI_Pk_and_Ek}}{ = } & \eta^2 t_k^2 \norms{d^k}^2 - t_{k+1}^2\norms{p^k}^2 +  2\eta c_k  \iprods{z^k, d^k} +   2 c_{k+1} \iprods{z^{k+1}, p^k} \vspace{1ex}\\
& = & \eta_k \big( 2 c_{k+1}  - t_{k+1}^2\eta_k \big) \norms{z^{k+1}}^2  +  \big( \eta^2t_k^2\gamma_k^2 - t_{k+1}^2\nu_k^2 - 2\eta c_k\gamma_k \big) \norms{z^k}^2 \vspace{1ex}\\
&& + {~} 2 \lambda_k \big( t_{k+1}^2 \eta_k -  c_{k+1} \big) \iprods{z^{k+1}, \hat{w}^{k+1}}  \vspace{1ex}\\
&& - {~} 2 \nu_k\big( t_{k+1}^2 \eta_k -  c_{k+1} \big) \iprods{z^{k+1}, z^k}  \vspace{1ex}\\
&& - {~} 2 \big( \eta^2 t_k^2\gamma_k - t_{k+1}^2\lambda_k\nu_k - \eta c_k \big) \iprods{\hat{w}^{k+1}, z^k}.
\end{array}
\end{equation*}
By \eqref{eq:NGEAG4NI_params} and \eqref{eq:NGEAG4NI_params2}, we can compute each coefficient of the last expression to get
\begin{equation*}
\arraycolsep=0.2em
\begin{array}{lcl}
\Ec_k & = & \eta^2 t_k^2 \norms{d^k}^2 - t_{k+1}^2\norms{p^k}^2  +  2\eta c_k\iprods{z^k, d^k} + 2c_{k+1}\iprods{z^{k+1}, p^k} \vspace{1ex}\\
& = & \big[ (1-\mu)\psi t_k  - \delta \big][(\eta - \beta)t_k - \delta \big] \norms{z^{k+1}}^2  \vspace{1ex} \\
&& + {~}  \big\{ \eta(t_k - r + 1)\big[ (\eta - 2\psi)(t_k - r + 1) + 2(r-1)\beta \big] - \beta^2t_k^2 \big\}   \norms{z^k}^2 \vspace{1ex}\\
&& + {~} 2\eta \omega t_k^2   \iprods{z^{k+1}, \hat{w}^{k+1}} -  2 \beta \omega t_k^2 \iprods{z^{k+1}, z^k} \vspace{1ex} \\
&& - {~} 2 \eta\omega t_k(t_k - r + 1) \iprods{\hat{w}^{k+1}, z^k}.
\end{array}
\end{equation*}
This is exactly \eqref{eq:NAEG4NI_Ek_simplification}.
\Eproof

\beforesubsec
\subsection{\mytb{The Proof of Lemma~\ref{le:NGEAG4NI_Ek_lower_bound}}}\label{apdx:le:NGEAG4NI_Ek_lower_bound}
\aftersubsec
First, by \eqref{eq:NGAEG4NI_w_quantities}, \eqref{eq:NGEAG4NI}, and  the $L$-Lipschitz continuity of $F$, we have $\norms{w^{k+1} - \hat{w}^{k+1}} = \norms{Fx^{k+1} - Fy^k} \leq L\norms{x^{k+1} - y^k} = L\eta\norms{d^k}$.
Using this relation, Young's inequality, for any $c_1 > 0$, $\phi > 0$, and $\hat{\phi} \geq 0$, we get
\begin{equation*}
\arraycolsep=0.1em
\begin{array}{lcl}
\bar{\Tc}_{[1]} & := & \norms{z^{k+1} - \hat{w}^{k+1}}^2 + \phi\norms{z^{k+1} - \hat{w}^{k+1}}^2 + \hat{\phi}\norms{w^{k+1} - \hat{w}^{k+1} }^2 \vspace{1ex}\\
& \leq & \big[ (1 + \phi)(1+c_1) + \hat{\phi} \big] \norms{w^{k+1} - \hat{w}^{k+1}}^2 + \frac{(1+\phi)(1+c_1)}{c_1}\norms{z^{k+1} - w^{k+1} }^2 \vspace{1ex}\\
& \leq & \big[ (1 + \phi)(1+c_1) + \hat{\phi} \big] L^2\eta^2 \norms{\hat{w}^{k+1} - \gamma_k z^k}^2 + \frac{(1+\phi)(1+c_1)}{c_1}\norms{z^{k+1} - w^{k+1} }^2.
\end{array}
\end{equation*}
Let $M^2 := \big[ (1 + \phi)(1+c_1) + \hat{\phi} \big] L^2$ as stated.
Expanding the first term on the left-hand side, and the first term on the right-hand side of $\bar{\Tc}_{[1]}$ above, and then using $\gamma_k := \frac{t_k - r + 1}{t_k}$ from \eqref{eq:NGEAG4NI_params}, rearranging the result, we get
\begin{equation*}
\arraycolsep=0.1em
\begin{array}{lcl}
0 & \geq & t_k^2 \norms{z^{k+1}}^2 + (1 - M^2\eta^2)t_k^2 \norms{\hat{w}^{k+1}}^2 - M^2\eta^2(t_k - r +1 )^2\norms{z^k}^2 \vspace{1ex}\\
&& - {~} 2t_k^2 \iprods{z^{k+1}, \hat{w}^{k+1}} +  2M^2\eta^2t_k(t_k - r + 1 )\iprods{\hat{w}^{k+1}, z^k} + \hat{\phi} t_k^2 \norms{w^{k+1} - \hat{w}^{k+1}}^2 \vspace{1ex}\\
&& + {~}  \phi t_k^2 \norms{z^{k+1} - \hat{w}^{k+1}}^2 -  \frac{(1+\phi)(1+c_1)t_k^2 }{c_1}\norms{z^{k+1} - w^{k+1} }^2.
\end{array}
\end{equation*}
Multiplying this inequality by $\eta\omega$ with $\omega := \frac{\mu(\eta - \beta)}{\mu + 1}$, and adding the result to $\Ec_k$ in \eqref{eq:NAEG4NI_Ek_simplification} of Lemma~\ref{le:NGEAG4NI_Ek_simplification}, we can show that
\begin{equation*} 
\arraycolsep=0.1em
\begin{array}{lcl}
\Ec_k & \geq & \big\{ \big[ (1-\mu)\psi t_k  - \delta \big][(\eta - \beta)t_k - \delta] + \eta\omega t_k^2 \big\} \norms{z^{k+1}}^2  \vspace{1ex} \\
&& + {~}  \big\{ \eta(t_k \! - \! r \! + \! 1)\big[ \big(\eta - 2\psi - \omega M^2\eta^2\big)(t_k \! - \! r \! +  \! 1) + 2(r-1)\beta \big] - \beta^2t_k^2  \big\}   \norms{z^k}^2 \vspace{1ex}\\
&& - {~} 2 \beta \omega t_k^2    \iprods{z^{k+1}, z^k} +  \eta\omega  (1 - M^2\eta^2)t_k^2 \norms{\hat{w}^{k+1}}^2 \vspace{1ex} \\
&& - {~} 2 \eta\omega (1-M^2\eta^2) t_k(t_k - r + 1)  \iprods{\hat{w}^{k+1}, z^k} +  \eta \omega \hat{\phi} t_k^2 \norms{w^{k+1} - \hat{w}^{k+1}}^2  \vspace{1ex}\\
&& + {~} \eta \omega \phi t_k^2 \norms{z^{k+1} - \hat{w}^{k+1}}^2 -  \frac{\eta\omega (1+\phi)(1+c_1)t_k^2 }{c_1}\norms{z^{k+1} - w^{k+1} }^2.
\end{array}
\end{equation*}
Since $t_k\hat{w}^{k+1} - (t_k-r+1)z^k = t_kd^k$ due to \eqref{eq:NGEAG4NI_dir}, we can easily show that
\begin{equation*}
\arraycolsep=0.2em
\begin{array}{ll}
& 2\iprods{z^{k+1}, z^k} = \norms{z^k}^2 + \norms{z^{k+1}}^2 - \norms{z^{k+1} - z^k }^2, \vspace{1ex}\\
& t_k^2 \norms{\hat{w}^{k+1}}^2 - 2t_k(t_k-r+1)\iprods{\hat{w}^{k+1}, z^k} + (t_k-r+1)^2 \norms{z^k}^2 = t_k^2\norms{d^k}^2.
\end{array} 
\end{equation*}
Utilizing these two identities into the last inequality, and rearranging the result, we can further lower bound
\begin{equation}\label{eq:NGEAG4NI_lm33_proof100}
\arraycolsep=0.2em
\begin{array}{lcl}
\Ec_k & \geq &  \hat{\Lambda}_{k+1} \norms{z^{k+1}}^2 - \Lambda_k  \norms{z^k}^2 +  \beta \omega t_k^2 \norms{z^{k+1} - z^k}^2  \vspace{1ex}\\
&& + {~}    \eta \omega(1 - M^2\eta^2) t_k^2 \norms{ d^k}^2 +  \eta\omega \hat{\phi} t_k^2 \norms{w^{k+1} - \hat{w}^{k+1}}^2 \vspace{1ex}\\
&& + {~} \eta \omega \phi t_k^2 \norms{z^{k+1} - \hat{w}^{k+1}}^2 -  \frac{\eta\omega (1+\phi)(1+c_1)t_k^2 }{c_1}\norms{z^{k+1} - w^{k+1} }^2,
\end{array}
\end{equation}
where $\hat{\Lambda}_{k+1}$ and $\Lambda_k$ are respectively given by
\begin{equation*}
\arraycolsep=0.2em
\left\{\begin{array}{lcl}
\hat{\Lambda}_{k+1} &:= &  \big[ (1-\mu)\psi t_k  - \delta \big][(\eta - \beta)t_k - \delta] + \omega(\eta - \beta) t_k^2, \vspace{1ex}\\
\Lambda_k &:= & 2\eta(t_k - r + 1)\big[ \psi (t_k - r + 1) - (r-1)\beta \big] - \beta^2t_k^2.
\end{array}\right.
\end{equation*}
If we simplify $\Lambda_k$, then we obtain it as in \eqref{eq:NGEAG4NI_ak_and_ahat_k}.
Furthermore, we can easily prove that $\hat{\Lambda}_{k+1} =  \Lambda_{k+1} + S_k$ for $S_k$ defined by \eqref{eq:NGEAG4NI_ak_and_ahat_k}.

Finally, using $\hat{\Lambda}_{k+1}$, $\Lambda_k$, and $S_k$ from \eqref{eq:NGEAG4NI_ak_and_ahat_k}, we obtain \eqref{eq:NAEG4NI_key_estimate2b} from \eqref{eq:NGEAG4NI_lm33_proof100}.
\Eproof

\beforesubsec
\subsection{\mytb{The Proof of Lemma~\ref{le:NGEAG4NI_V1_key_estimate1}}}\label{apdx:le:NGEAG4NI_V1_key_estimate1}
\aftersubsec
First, since $u^k := Fx^k$ in \eqref{eq:NGEAG4NI}, $\Pc_k$ and $\bmark{\Ec_k}$ in \eqref{eq:NGEAG4NI_Pk_and_Ek} reduce to
\begin{equation*}
\arraycolsep=0.2em
\begin{array}{lcl}
\Pc_k &:= & \norms{r(x^k - x^{\star}) + t_k(y^k - x^k)}^2 + r \mu \norms{x^k - x^{\star}}^2 + 2c_k \iprods{w^k, y^k - x^k}, \vspace{1ex} \\
\bmark{\Ec_k} &:= & \eta^2 t_k^2 \norms{d^k}^2 - t_{k+1}^2\norms{p^k}^2  +  2\eta c_k\iprods{w^k, d^k} + 2c_{k+1}\iprods{w^{k+1}, p^k}.
\end{array}
\end{equation*}
In this case, \eqref{eq:NGEAG4NI_key_property1} becomes
\begin{equation}\label{eq:NGEAG4NI_V1_lm33_proof2}
\arraycolsep=0.2em
\begin{array}{lcl}
\Pc_k - \Pc_{k+1} & = & \Ec_k +  \mu(2t_k - r - \mu) \norms{x^{k+1} - x^k}^2\vspace{1ex} \\
&& + {~} 2r [(\eta-\beta)t_k - \delta]  \iprods{ w^{k+1}, x^{k+1} - x^{\star}} \vspace{1ex}\\
&& - {~} 2r  [(\eta-\beta)t_k - \eta(r-1)] \iprods{ w^k, x^k - x^{\star}} \vspace{1ex}\\
&& + {~} 2 t_k \big[ \omega (t_k - \mu - r) + \mu \eta \big]  \iprods{ w^{k+1} - w^k, x^{k+1} - x^k} \vspace{1ex} \\
&&+ {~} 2 \mu \eta t_k \iprods{\hat{w}^{k+1} - w^{k+1}, x^{k+1} - x^k}.
\end{array}
\end{equation}
Next, since $z^k = w^k$, we can choose $c_1 := 0$ and $\phi := 0$ and \eqref{eq:NAEG4NI_key_estimate2b} reduces to
\begin{equation}\label{eq:NAEG4NI_V1_lm33_proof3}
\arraycolsep=0.2em
\begin{array}{lcl}
\bmark{\Ec_k} & \geq & \Lambda_{k+1} \norms{w^{k+1}}^2 - \Lambda_k \norms{w^k}^2 + S_k\norms{w^{k+1}}^2  +   \beta \omega t_k^2 \norms{w^{k+1} - w^k}^2  \vspace{1ex} \\
&& + {~} \eta \omega(1 - M^2\eta^2)t_k^2 \norms{ d^k}^2 +  \eta\omega \hat{\phi} t_k^2 \norms{ w^{k+1} - \hat{w}^{k+1}}^2,
\end{array}
\end{equation}
where $M^2 :=  \big(1 + \hat{\phi}) L^2$ for any  $\hat{\phi} \geq 0$.

Now, substituting \eqref{eq:NAEG4NI_V1_lm33_proof3} into \eqref{eq:NGEAG4NI_V1_lm33_proof2}, and using $\iprods{w^{k+1} - w^k, x^{k+1} - x^k} \geq -\rho\norms{w^{k+1} - w^k}^2$ from the $\rho$-co-hypomonotonicity of $\Phi$, we get
\begin{equation}\label{eq:NGEAG4NI_V1_lm33_proof3} 
\hspace{-1ex}
\arraycolsep=0.2em
\begin{array}{lcl}
\Pc_k - \Pc_{k+1} & \geq & \Lambda_{k+1} \norms{w^{k+1}}^2 - \Lambda_k \norms{w^k}^2  +  \mu(2t_k - r - \mu) \norms{x^{k+1} - x^k}^2\vspace{1ex} \\
&& + {~} 2r [ (\eta-\beta)t_{k+1} - \eta(r-1) ]  \iprods{ w^{k+1}, x^{k+1} - x^{\star}} \vspace{1ex}\\
&& - {~} 2r  [(\eta-\beta)t_k - \eta(r-1)] \iprods{ w^k, x^k - x^{\star}} \vspace{1ex}\\
&& + {~} 2r(r-2)\omega  \iprods{ w^{k+1}, x^{k+1} - x^{\star}} + S_k\norms{w^{k+1}}^2 \vspace{1ex}\\
&& + {~}  t_k\big\{ \beta \omega t_k - 2 \rho \big[ \omega (t_k - \mu - r) + \mu \eta \big] \big\}  \norms{ w^{k+1} - w^k}^2 \vspace{1ex} \\
&& + {~} \eta \omega(1 - M^2\eta^2)t_k^2 \norms{ d^k}^2 +  \eta\omega \hat{\phi} t_k^2 \norms{ w^{k+1} - \hat{w}^{k+1}}^2 \vspace{1ex}\\
&&+ {~} 2 \mu \eta t_k \iprods{\hat{w}^{k+1} - w^{k+1}, x^{k+1} - x^k}.
\end{array}
\hspace{-4ex}
\end{equation}
Utilizing $\Lc_k$ from \eqref{eq:NGEAG4NI_Lyapunov_func1}, and Young's inequality for the last term of \eqref{eq:NGEAG4NI_V1_lm33_proof3}, we can show from this expression that
\begin{equation*} 
\arraycolsep=0.2em
\begin{array}{lcl}
\Lc_k - \Lc_{k+1} & \geq &  \mu\big(2t_k - r - \mu - \frac{\mu\eta}{\hat{\phi}\omega}\big) \norms{x^{k+1} - x^k}^2 +  \eta \omega(1 - M^2\eta^2)t_k^2 \norms{ d^k}^2 \vspace{1ex} \\
&& + {~} 2\omega r(r-2)  \iprods{ w^{k+1}, x^{k+1} - x^{\star}} + S_k\norms{w^{k+1}}^2 \vspace{1ex}\\
&& + {~}  t_k\big[ \omega(\beta - 2\rho)t_k + 2 \rho(\mu \omega + r\omega  - \mu \eta) \big]  \norms{ w^{k+1} - w^k}^2.
\end{array}
\end{equation*}
Substituting $\iprods{ w^{k+1}, x^{k+1} - x^{\star}} \geq -\rho\norms{w^{k+1}}^2$ from the $\rho$-co-hypomonotonicity of $\Phi$ into the last estimate, we obtain \eqref{eq:NGEAG4NI_V1_key_estimate1}.
\Eproof

\beforesubsec
\subsection{\mytb{The Proof of Lemma~\ref{le:NGEAG4NI_V1_key_estimate2}}}\label{apdx:le:NGEAG4NI_V1_key_estimate2}
\aftersubsec
From \eqref{eq:NGEAG4NI_Lyapunov_func1} and  $\iprods{ w^k, x^k - x^{\star}} \geq -\rho\norms{w^k}^2$, we can show that
\begin{equation*} 
\arraycolsep=0.2em
\begin{array}{lcl}
\Lc_k &: = & \norms{r(x^k - x^{\star}) + t_k(y^k - x^k)}^2 + r \mu \norms{x^k - x^{\star}}^2 + \Lambda_k\norms{w^k}^2 \vspace{1ex}\\
&& + {~} 2 \big[ \psi(t_k - r + 1) - (r-1)\beta \big] \iprods{w^k, r(x^k - x^{\star}) +  t_k(y^k - x^k) } \vspace{1ex}\\
&& + {~}  2\omega r(t_k - r + 1) \iprods{w^k, x^k - x^{\star}} \vspace{1ex}\\
& \geq & \norms{r(x^k - x^{\star}) + t_k(y^k - x^k) + [\psi(t_k - r + 1) - (r-1)\beta] w^k}^2  \vspace{1ex}\\
&& + {~}  \sigma_k \norms{w^k}^2 + r \mu \norms{x^k - x^{\star}}^2,
\end{array}
\end{equation*}
where $\sigma_k := \Lambda_k - [\psi(t_k - r + 1) - (r-1)\beta]^2 - 2\rho \omega r(t_k - r + 1)$.
Since $\Lambda_k$ is given by \eqref{eq:NGEAG4NI_ak_and_ahat_k}, we can simplify $\sigma_k$ as $\sigma_k = \mu\big[ \psi(t_k - r + 1) - r\rho\big]^2 - \mu\big[r^2\rho^2 - \frac{(r-1)^2\beta(\eta - \beta) }{\mu+1}\big]$.
Substituting this value into the last estimate, we obtain \eqref{eq:NGEAG4NI_V1_key_estimate2}.

Finally, the proof of \eqref{eq:NGEAG4NI_V1_key_estimate3} directly follows from the definition of $\Lc_k$ in \eqref{eq:NGEAG4NI_Lyapunov_func1}, the Cauchy-Schwarz inequality and $y^0 = x^0$.
\Eproof

\beforesubsec
\subsection{\mytb{The Proof of Theorem~\ref{th:NGEAG4NI_V1_convergence}}}\label{apdx:th:NGEAG4NI_V1_convergence}
\aftersubsec
For simplicity of our presentation, we choose $\mu := 1$.
Since $\mu = 1$ and $\hat{\phi} := \frac{1 - L^2\eta^2}{2L^2\eta^2} > 0$ as in \eqref{eq:NGEAG4NI_V1_params0}, we have $1 - M^2 \eta^2 = \frac{1 - L^2 \eta^2}{2}$.
From this relation and   the choice of $t_0$, we can derive from \eqref{eq:NGEAG4NI_V1_key_estimate1} that
\begin{equation*} 
\arraycolsep=0.2em
\begin{array}{lcl}
\Lc_k - \Lc_{k+1} & \geq &  \big(2t_k - r - 1 - \frac{\eta}{\hat{\phi}\omega}\big) \norms{x^{k+1} - x^k}^2 +  \frac{ \eta \omega(1 - L^2\eta^2)t_k^2 }{2} \norms{ d^k}^2 \vspace{1ex} \\
&& + {~}  t_k\big[ \omega(\beta - 2\rho)t_k + 2 \rho(\omega + r\omega  - \eta) \big]  \norms{ w^{k+1} - w^k}^2 \vspace{1ex}\\
&& + {~}  \frac{(r-2)}{2}\big[ (\eta-\beta)^2t_k - \frac{2\Gamma}{r-2} - 4\rho \omega r  \big] \norms{w^{k+1}}^2.
\end{array}
\end{equation*}
Note that, since $t_{k+1} = t_k + 1$ from \eqref{eq:NGEAG4NI_params}, we get $t_k = k + t_0$.
Summing up the last inequality from $k := 0$ to $k := K$, \rmark{then} taking the limit as $K\to\infty$ and noting that $\Lc_k \geq 0$ and $t_k = k + t_0$, we obtain \eqref{eq:NGEAG4NI_V1_result1}.
Here, the last summable estimate of \eqref{eq:NGEAG4NI_V1_result1} follows from the fact that $\omega(\beta - 2\rho)t_k + 2 \rho(\omega + r\omega  - \eta) \geq \frac{\omega(\beta - 2\rho)t_k}{2}$ due to the choice of $t_0$ in \eqref{eq:NGEAG4NI_V1_R02}.

Finally, $\sigma_k$ in Lemma~\ref{le:NGEAG4NI_V1_key_estimate2} can be directly lower bounded by 
\begin{equation*} 
\arraycolsep=0.2em
\begin{array}{lcl}
\sigma_k & \geq & \frac{(\eta-\beta)^2(t_k - r + 1)^2}{8} = \frac{(\eta-\beta)^2(k+ t_0 - r + 1)^2}{8}.
\end{array}
\end{equation*}
Since $\sigma_k\norms{w^k}^2 \leq \Lc_k$ from \eqref{eq:NGEAG4NI_V1_key_estimate2} and $\Lc_k  \leq \Lc_0 \leq \Rc_0^2$ by induction and \eqref{eq:NGEAG4NI_V1_key_estimate3}, combining these estimates, we obtain \eqref{eq:NGEAG4NI_V1_result2}.
\Eproof

\beforesubsec
\subsection{\mytb{The Proof of Theorem~\ref{th:NGEAG4NI_V1_convergence2}}}\label{apdx:th:NGEAG4NI_V1_convergence2}
\aftersubsec
From \eqref{eq:NGEAG4NI}, \eqref{eq:NGEAG4NI_dir}, and \eqref{eq:NGEAG4NI_params}, since $\mu = 1$, we have
\begin{equation*}
\arraycolsep=0.2em
\begin{array}{lcl}
y^{k+1} - x^{k+1} + \eta_kw^{k+1} &= & \theta_k(y^k - x^k + \eta_{k-1}w^k) - (\eta\theta_k-\lambda_k)d^k \vspace{1ex}\\
&& - {~} (\theta_k\eta_{k-1} + \nu_k - \lambda_k\gamma_k)w^k \vspace{1ex}\\
& = & \theta_k(y^k - x^k + \eta_{k-1}w^k) + \frac{\eta(r+1)}{t_{k+1}}d^k + \frac{s_k}{t_{k+1}}w^k,
\end{array}
\end{equation*}
where $s_k :=  \frac{(r+4)(\eta-\beta)t_k -  2(r+1)(\eta - \beta+\delta)}{2t_k}$.

Let us denote $v^k := y^k - x^k + \eta_{k-1}w^k$.
Then, the last expression can be rewritten as follows:
\begin{equation*}
\arraycolsep=0.2em
\begin{array}{lcl}
v^{k+1} &= & \theta_kv^k + (1-\theta_k)\big[ \frac{\eta(r+1)}{(1-\theta_k)t_{k+1}}d^k + \frac{s_k}{(1-\theta_k)t_{k+1}}w^k\big].
\end{array}
\end{equation*}
Since $\theta_k \in (0, 1)$, using convexity of $\norms{\cdot}^2$, we have
\begin{equation*}
\arraycolsep=0.2em
\begin{array}{lcl}
\norms{v^{k+1}}^2 & \leq & \theta_k\norms{v^k}^2 + \frac{1}{1-\theta_k}\norms{ \frac{\eta(r+1)}{t_{k+1}}d^k + \frac{s_k}{t_{k+1}}w^k}^2 \vspace{1ex}\\
& \leq & \theta_k\norms{v^k}^2 + \frac{2\eta^2(r+1)^2}{(1-\theta_k)t_{k+1}^2}\norms{d^k}^2 + \frac{2s_k^2}{(1-\theta_k)t_{k+1}^2}\norms{w^k}^2.
\end{array}
\end{equation*}
Noting that $s_k \leq \bar{s} := \frac{(r+4)(\eta-\beta) + 1}{2}$ for any $t_k \ge 0$.
Multiplying both sides of the last expression by $t_{k+1}^2$ and using this relation and $\theta_k = \frac{t_k-r-1}{t_{k+1}}$, we get 
\begin{equation}\label{eq:NGEAG4NI_V1_th32_proof0}
\hspace{-1ex}
\arraycolsep=0.2em
\begin{array}{lcl}
t_{k+1}^2 \norms{v^{k+1}}^2 & \leq & t_k^2 \norms{v^k}^2 - (rt_k + r + 1) \norms{v^k}^2 + \frac{2\eta^2(r+1)^2t_{k+1}}{(r+2)}\norms{d^k}^2 \vspace{1ex}\\
&& + {~} \frac{ 2 \bar{s}^2 t_{k+1}}{r+2}\norms{w^k}^2.
\end{array}
\hspace{-2ex}
\end{equation}
Since $\sum_{k=0}^{\infty}\frac{2\eta^2(r+1)^2t_{k+1}}{(r+2)}\norms{d^k}^2 < +\infty$ and $\sum_{k=0}^{\infty}\frac{2\bar{s}^2 t_{k+1}}{r+2}\norms{w^k}^2 < +\infty$ due to \eqref{eq:NGEAG4NI_V1_result1}, the limit $\lim_{k\to\infty}t_k^2 \norms{v^k}^2$ exists.
Moreover, since $\sum_{k=0}^{\infty} [(r-2)t_k + r-1] \norms{v^k}^2 < +\infty$, we can show that $\lim_{k\to\infty}t_k^2 \norms{y^k- x^k - \eta_{k-1}w^k}^2 = 0$.

Next, using again  \eqref{eq:NGEAG4NI}, \eqref{eq:NGEAG4NI_dir}, and \eqref{eq:NGEAG4NI_params} with $\mu = 1$, we can also derive
\begin{equation}\label{eq:NGEAG4NI_V1_th32_proof1}
\arraycolsep=0.2em
\begin{array}{lcl}
y^{k+1} - x^{k+1} &=  & \theta_k (y^k - x^k) - p^k -  \frac{\eta (t_k-r-1)}{t_{k+1}} d^k \vspace{1ex}\\
& = & \theta_k (y^k - x^k)  + \frac{r + 1}{t_{k+1}} d^k - \frac{r (\eta - \beta)}{2 t_{k+1}}w^{k+1}  \\ [1ex]
&& - {~} \frac{[(\eta - \beta)t_k - (r - 1)\eta ]}{t_{k+1}} (w^{k+1} - w^k) \vspace{1ex}\\
& = & \theta_k(y^k - x^k) + \frac{1}{t_{k+1}} \zeta_k - \frac{s_k}{t_{k+1}} (w^{k+1} - w^k),
\end{array}
\end{equation}
where $\zeta^k :=  (r + 1)d^k -  \frac{r(\eta - \beta)}{2} w^{k+1}$ and $s_k := (\eta - \beta)t_k - (r - 1)\eta$.

Utilizing $y^k - x^k = x^{k+1} - x^k + \eta d^k$, the last expression leads to
\begin{equation*}
\arraycolsep=0.2em
\begin{array}{lcl}
\norms{y^{k+1} - x^{k+1}}^2 & =  & \theta_k^2\norms{y^k - x^k}^2 + \frac{2\theta_k}{t_{k+1}}\iprods{\zeta^k - s_k(w^{k+1} - w^k), y^k - x^k} \vspace{1ex}\\
&& + {~} \frac{1}{t_{k+1}^2}\norms{\zeta^k - s_k(w^{k+1} - w^k)}^2 \vspace{1ex}\\
& = & \theta_k^2\norms{y^k - x^k}^2 + \frac{2\theta_k}{t_{k+1}}\iprods{\zeta^k, y^k - x^k} \vspace{1ex}\\
&&  - {~} \frac{2\theta_ks_k}{t_{k+1}}\iprods{w^{k+1} - w^k, x^{k+1} - x^k} - \frac{2\eta\theta_ks_k}{t_{k+1}}\iprods{w^{k+1} - w^k, d^k} \vspace{1ex}\\
&& + {~}  \frac{1}{t_{k+1}^2}\norms{\zeta^k - s_k(w^{k+1} - w^k) }^2.
\end{array}
\end{equation*}
By $\iprods{w^{k+1} - w^k, x^{k+1} - x^k} \geq -\rho\norms{w^{k+1} - w^k}^2$ from the $\rho$-co-hypomonotonicity of $\Phi$ and Young's inequality, we can prove that
\begin{equation}\label{eq:NGEAG4NI_V1_th32_proof2}
\hspace{-2ex}
\arraycolsep=0.2em
\begin{array}{lcl}
\norms{y^{k+1} - x^{k+1}}^2 & \leq & \big( \theta_k^2 + \frac{\theta_k}{t_{k+1}}  \big) \norms{y^k - x^k}^2 +  \frac{1}{t_{k+1}^2}\big( 2 +  t_{k+1}\theta_k \big) \norms{\zeta^k}^2 \vspace{1ex}\\
&&  + {~} \frac{(\eta + 2\rho) t_{k+1}s_k\theta_k + 2 s_k^2}{t_{k+1}^2} \norms{w^{k+1} - w^k}^2 + \frac{\eta\theta_ks_k}{t_{k+1}}\norms{d^k}^2.
\end{array}
\hspace{-4ex}
\end{equation}
We can easily check that $\theta_k^2 + \frac{\theta_k}{t_{k+1}} = \frac{t_k^2 - (2r+1)t_k + r(r+1)}{t_{k+1}^2}$, and 
\begin{equation*}
\arraycolsep=0.2em
\begin{array}{ccl}
(\eta + 2\rho) t_{k+1}s_k\theta_k + 2s_k^2  &\leq& (3\eta - 2\beta + 2\rho)(\eta - \beta)t_k^2, \\ [2ex]
\eta \theta_k s_k  &\leq& \eta (\eta - \beta) t_k^2,
\end{array}
\end{equation*}
where the first inequality holds if $t_k \geq \frac{(r^2 - 1) \eta + 2 (r - 1)^2 \eta^2}{(\eta + 2\rho)[ (\eta - \beta)(r + 1) + (r - 1) \eta ] + 4(r - 1)\eta (\eta - \beta)}$ and the second one requires $t_k \geq \frac{(r^2 - 1)\eta}{(r - 1) \eta + (r + 1)(\eta - \beta)}$, which holds as $t_k \to\infty$.

Multiplying \eqref{eq:NGEAG4NI_V1_th32_proof2} by $t_{k+1}^2$ and using these bounds, we further obtain
\begin{equation}\label{eq:NGEAG4NI_V1_th32_proof3}
\hspace{-2ex}
\arraycolsep=0.2em
\begin{array}{lcl}
t_{k+1}^2\norms{y^{k+1} - x^{k+1}}^2 & \leq & t_k^2 \norms{y^k - x^k}^2 - [(2r+1)t_k - r(r+1)] \norms{y^k - x^k}^2 \vspace{1ex}\\
&&  + {~}  (t_k - r + 1) \norms{\zeta^k}^2 +  \eta (\eta-\beta)t_k^2 \norms{d^k}^2 \vspace{1ex}\\
&& + {~}  ( 3\eta - 2\beta + 2\rho )(\eta-\beta)t_k^2 \norms{w^{k+1} - w^k}^2.
\end{array}
\hspace{-4ex}
\end{equation}
By Young's inequality again, one has
\begin{equation*}
\arraycolsep=0.2em
\begin{array}{lcl}
\norms{\zeta^k}^2 & \leq & 2(r+1)^2 \norms{d^k}^2 +   \frac{r^2(\eta - \beta)^2}{2} \norms{w^{k+1}}^2.
\end{array}
\end{equation*}
Substituting this inequality into \eqref{eq:NGEAG4NI_V1_th32_proof3}, we get
\begin{equation*} 
\hspace{-2ex}
\arraycolsep=0.2em
\begin{array}{lcl}
t_{k+1}^2\norms{y^{k+1} - x^{k+1}}^2 & \leq & t_k^2 \norms{y^k - x^k}^2 - [(2r+1)t_k - r(r+1)] \norms{y^k - x^k}^2 \vspace{1ex}\\
&&  + {~} \big[  \eta  (\eta-\beta)t_k^2 +  2(r+1)^2  (t_k - r + 1)\big] \norms{d^k}^2 \vspace{1ex}\\
&& + {~}  ( 3\eta - 2\beta + 2\rho )(\eta-\beta)t_k^2 \norms{w^{k+1} - w^k}^2 \vspace{1ex}\\
&&+ {~}  \frac{r^2(\eta - \beta)^2}{2} (t_k - r + 1)\norms{w^{k+1}}^2.
\end{array}
\hspace{-4ex}
\end{equation*}
Note that the last three terms are summable due to \eqref{eq:NGEAG4NI_V1_result1}.
This implies that $\lim_{k\to\infty}t_k^2\norms{y^k - x^k}^2 = 0$ as before, which proves the last limit in \eqref{eq:NGEAG4NI_V1_result3}.

Now, utilizing Young's inequality, one can show that
\begin{equation*}
\arraycolsep=0.2em
\begin{array}{lcl}
[(\eta-\beta)t_{k-1} - \delta]^2\norms{w^k}^2 & = &  t_k^2\eta_{k-1}^2 \norms{w^k}^2 \vspace{1ex}\\
& \leq & 2t_k^2\norms{y^k - x^k + \eta_{k-1}w^k}^2 + 2t_k^2\norms{y^k - x^k}^2.
\end{array}
\end{equation*}
Combining this inequality, $\lim_{k\to\infty}t_k^2 \norms{y^k- x^k - \eta_{k-1}w^k}^2 = 0$, and $\lim_{k\to\infty}t_k^2\norms{y^k - x^k}^2 = 0$, we get $\lim_{k\to\infty} [(\eta-\beta)t_{k-1} - \delta]^2\norms{w^k}^2 = 0$, which also implies $\lim_{k\to\infty} t_k^2 \norms{w^k}^2 = 0$.
This proves the first line of \eqref{eq:NGEAG4NI_V1_result3}.

Finally, using Young's inequality one more time, we get
\begin{equation*}
\arraycolsep=0.2em
\begin{array}{lcl}
t_k^2\norms{x^{k+1} - x^k}^2 = t_k^2\norms{y^k - x^k - \eta d^k}^2 \leq 2t_k^2\norms{y^k - x^k}^2 + 2\eta^2t_k^2\norms{d^k}^2.
\end{array}
\end{equation*}
Since $\lim_{k\to\infty}t_k^2\norms{y^k - x^k}^2 = 0$ and $\lim_{k\to\infty}t_k^2\norms{d^k}^2 = 0$, this inequality implies $\lim_{k\to\infty}t_k^2\norms{x^{k+1} - x^k}^2 = 0$, which leads to the second line of \eqref{eq:NGEAG4NI_V1_result3}.
\Eproof

\beforesubsec
\subsection{\mytb{The Proof of Theorem~\ref{th:NGEAG4NI_V1_convergence3}}}\label{apdx:th:NGEAG4NI_V1_convergence3}
\aftersubsec
From the proof of Theorem~\ref{th:NGEAG4NI_V1_convergence}, we have $r\norms{x^k - x^{\star}}^2 \leq \Lc_k \leq \Lc_0 \leq \Rc_0^2$, we conclude that $\sets{x^k}$ is bounded, and hence, it has a cluster point.
Let $x^{*}$ be a cluster point of $\sets{x^k}$ and $\sets{x^{k_i}}$ be a subsequence converging to $x^{*}$.

Consider the subsequence $\sets{z^{k_i}}$ with $z^{k_i} := Fx^{k_i} + \xi^{k_i}$  for $\xi^{k_i} \in Tx^{k_i}$.
By \eqref{eq:NGEAG4NI_V1_result3}, we have $\lim_{i\to\infty}\norms{Fx^{k_i} + \xi^{k_i}}^2 = \lim_{i\to\infty}\norms{z^{k_i}}^2 = 0$, implies that $z^{k_i} \to 0$.
Since $\gra{T}$ is closed and $F$ is $L$-Lipschitz continuous, $\gra{\Phi} = \gra{F+T}$ is also closed and $(x^{k_i}, z^{k_i}) \in \gra{\Phi}$ converges to $(x^{\star}, 0)$.
By the closedness of $\gra{\Phi}$, we have $(x^{*}, 0)\in \gra{\Phi}$, which means that $0 \in \Phi{x^{*}} = Fx^{*} + Tx^{*}$.

Next, we prove that $\lim_{k\to\infty}\norms{x^k - x^{\star}}^2$ exists.
By the $\rho$-co-hypomonotonicity of $\Phi$, the Cauchy-Schwarz inequality, and $r\norms{x^k - x^{\star}}^2 \leq \Rc_0^2$, we have
\begin{equation*} 
\arraycolsep=0.2em
\begin{array}{lcl}
-\rho t_k \norms{w^k}^2 \leq t_k \iprods{w^k, x^k - x^{\star}} \leq t_k \norms{w^k}\norms{x^k - x^{\star}} \leq \frac{\Rc_0 }{\sqrt{r}} t_k \norms{w^k}.
\end{array}
\end{equation*}
Note that $\lim_{k\to\infty}t_k\norms{w^k} = 0$ due to the first line of \eqref{eq:NGEAG4NI_V1_result3}, and $\lim_{k\to\infty}t_k\norms{w^k}^2 = 0$ due to the third summable expression in \eqref{eq:NGEAG4NI_V1_result1}, we conclude that 
\begin{equation}\label{eq:NGEAG4NI_V1_convergence3_proof2}
\arraycolsep=0.2em
\begin{array}{lcl}
\lim_{k\to\infty}\vert t_k \iprods{w^k, x^k - x^{\star}} \vert = 0.
\end{array}
\end{equation}
From \eqref{eq:NGEAG4NI_dir} and  \eqref{eq:NGEAG4NI_params}, with $z^k = w^k$, we have
\begin{equation*} 
\arraycolsep=0.2em
\begin{array}{lcl}
w^{k+1} - \hat{w}^{k+1} & = & w^{k+1} - w^k - d^k + \frac{r-1}{t_k}w^k, \vspace{1ex}\\
t_{k+1}p^k &= & [(\eta-\beta)t_k-\delta](w^{k+1} - w^k) - \eta t_kd^k + [\eta(r-1)-\delta]w^k.
\end{array}
\end{equation*}
Applying Young's inequality, these relations lead to
\begin{equation*} 
\arraycolsep=0.2em
\begin{array}{lclcl}
0 & \leq & t_k^2\norms{w^{k+1} - \hat{w}^{k+1}}^2 &\leq & 3t_k^2 \norms{ w^{k+1} - w^k}^2 + 3t_k^2 \norms{ d^k}^2 + 3(r-1)^2t_k\norms{w^k}^2, \vspace{1ex}\\
0 & \leq & t_{k+1}^2\norms{p^k}^2 &\leq & 3[(\eta-\beta)t_k -\delta]^2\norms{w^{k+1} - w^k}^2 + 3\eta^2t_k^2\norms{d^k}^2 \vspace{1ex}\\
& & && + {~} 3[\eta(r-1) - \delta]^2\norms{w^k}^2.
\end{array}
\end{equation*}
Each term on the right-hand side of each inequality is summable, we obtain 
\begin{equation}\label{eq:NGEAG4NI_V1_convergence3_proof3}
\arraycolsep=0.2em
\begin{array}{lcl}
\sum_{k=0}^{+\infty}t_k^2\norms{w^{k+1} - \hat{w}^{k+1}}^2 <+ \infty, \vspace{1ex}\\
\sum_{k=0}^{+\infty}t_{k+1}^2\norms{p^k}^2 <+ \infty.
\end{array}
\end{equation}
Let us define 
\begin{equation}\label{eq:NGEAG4NI_V1_convergence3_proof4}
\arraycolsep=0.2em
\begin{array}{lcl}
\Qc_k &:= & \norms{r(x^k - x^{\star}) + t_k(y^k - x^k)}^2 + r\norms{x^k - x^{\star}}^2 \vspace{1ex}\\
&& + {~} 2r  [(\eta-\beta)t_k - \eta(r-1)] \big[\iprods{ w^k, x^k - x^{\star}} + \rho\norms{w^k}^2].
\end{array}
\end{equation}
Then, similar to the proof of \eqref{eq:NGEAG4NI_lm1_proof5} in Lemma~\ref{le:NGEAG4NI_descent_property1}, with $\mu = 1$, we can show that
\begin{equation*}
\arraycolsep=0.2em
\begin{array}{lcl}
\Qc_k - \Qc_{k+1} & = &  (2t_k - r - 1) \norms{x^{k+1} - x^k}^2  +  \eta^2 t_k^2 \norms{d^k}^2 - t_{k+1}^2\norms{p^k}^2 \vspace{1ex} \\
&& + {~} 2r [\eta(r-2) - \delta]  \iprods{ w^{k+1}, x^{k+1} - x^{\star}} \vspace{1ex}\\
&& + {~}  2\big[ ((\eta-\beta)t_k - \delta)(t_k - r - 1)  + \eta t_k \big] \iprods{w^{k+1} - w^k, x^{k+1} - x^k } \vspace{1ex}\\
&& - {~} 2  \big[  ((\eta-\beta)r  + \delta)t_k - \eta(r-1)t_k - \delta(r + 1)  \big] \iprods{ w^k, x^{k+1} - x^k} \vspace{1ex}\\
&& + {~} 2 \eta t_k \iprods{\hat{w}^{k+1} - w^{k+1}, x^{k+1} - x^k} \vspace{1ex}\\
&& + {~} 2\rho r  [(\eta-\beta)t_k - \eta(r-1)]\norms{w^k}^2 \vspace{1ex}\\
&& - {~} 2\rho r  [(\eta-\beta)t_{k+1} - \eta(r-1)]\norms{w^{k+1}}^2.
\end{array}
\end{equation*}
Next, by the $\rho$-co-hypomonotonicity of $\Phi$ and Young's inequality, we have
\begin{equation*} 
\arraycolsep=0.2em
\begin{array}{lcl}
\iprods{w^{k+1} - w^k, x^{k+1} - x^k}  \geq -\rho\norms{w^{k+1} - w^k}^2, \vspace{1ex}\\
2\eta t_k \vert \iprods{\hat{w}^{k+1} - w^{k+1}, x^{k+1} - x^k}\vert \geq - \eta t_k^2 \norms{\hat{w}^{k+1} - w^{k+1}}^2 - \eta \norms{x^{k+1} - x^k}^2.
\end{array}
\end{equation*}
Substituting these inequalities into $\Qc_k - \Qc_{k+1}$ above, one can derive that
\begin{equation*}
\arraycolsep=0.2em
\begin{array}{lcl}
\Qc_{k+1} - \Qc_k & \leq & t_{k+1}^2\norms{p^k}^2 - \eta^2 t_k^2 \norms{d^k}^2 \vspace{1ex}\\
&& + {~} 2r [\eta(r-2) - \delta]  \vert \iprods{ w^{k+1}, x^{k+1} - x^{\star} } \vert \vspace{1ex}\\
&& + {~} 2\rho \big[ ((\eta-\beta)t_k - \delta)(t_k - r - 1)  + \eta t_k \big]\norms{w^{k+1} - w^k}^2 \vspace{1ex}\\
&& + {~} \big[  (\eta -\beta r + \delta)t_k - \delta(r + 1)  \big] \norms{w^k}^2  \vspace{1ex}\\
&& + {~}  \big[  (\eta -\beta r + \delta - 2)t_k - (\delta - 1)(r + 1)  \big] \norms{x^{k+1} - x^k}^2 \vspace{1ex}\\
&& + {~} 2\rho r  [(\eta-\beta)t_{k+1} - \eta(r-1)]\norms{w^{k+1}}^2 \vspace{1ex}\\
&& + {~} \eta t_k^2\norms{w^{k+1} - \hat{w}^{k+1}}^2 + \eta \norms{x^{k+1} - x^k}^2.
\end{array}
\end{equation*}
From \eqref{eq:NGEAG4NI_V1_result1}, \eqref{eq:NGEAG4NI_V1_convergence3_proof2}, and \eqref{eq:NGEAG4NI_V1_convergence3_proof3}, we can see that all the terms on the right-hand side of this inequality are summable and $\Qc_k \geq 0$.
We conclude that $\lim_{k\to\infty}\Qc_k$ exists due to \cite[Proposition 5.31]{Bauschke2011}.

By  \eqref{eq:NGEAG4NI_V1_convergence3_proof4}, the existence of $\lim_{k\to\infty}\Qc_k$, and \eqref{eq:NGEAG4NI_V1_convergence3_proof2}, we conclude that 
\begin{equation}\label{eq:NGEAG4NI_V1_convergence3_proof5}
\arraycolsep=0.2em
\begin{array}{lcl}
\lim_{k\to\infty} \big[ \norms{r(x^k - x^{\star}) + t_k(y^k - x^k)}^2 + r\norms{x^k - x^{\star}}^2 \big] \ \text{exists}.
\end{array}
\end{equation}
Now, since $\norms{x^k - x^{\star}} \leq \frac{\Rc_0}{\sqrt{r}}$ and $\lim_{k\to\infty}t_k\norms{y^k - x^k} = 0$ due to \eqref{eq:NGEAG4NI_V1_result3}, we get
\begin{equation*} 
\arraycolsep=0.2em
\begin{array}{lcl}
\vert t_k\iprods{x^k - x^{\star}, y^k - x^k} \vert \leq t_k\norms{x^k-x^{\star}}\norms{y^k - x^k} \leq  \frac{\Rc_0}{\sqrt{r}} t_k\norms{y^k - x^k} \to 0,~k\to\infty.
\end{array}
\end{equation*}
As a consequence,  we conclude that $\lim_{k\to\infty}t_k\iprods{x^k - x^{\star}, y^k - x^k}  = 0$.
\rmark{Now, we consider} 
\begin{equation*} 
\arraycolsep=0.2em
\begin{array}{lcl}
\bar{\Tc}_{[2]} &:= & \norms{r(x^k - x^{\star}) + t_k(y^k - x^k)}^2 + r\norms{x^k - x^{\star}}^2 \vspace{1ex}\\
& = & (r^2 + r)\norms{x^k - x^{\star}}^2 + 2rt_k\iprods{x^k - x^{\star}, y^k - x^k} + t_k^2\norms{y^k - x^k}^2.
\end{array}
\end{equation*}
The limit on the left-hand side exists due to \eqref{eq:NGEAG4NI_V1_convergence3_proof5}, while the limits of the last two terms $t_k\iprods{x^k - x^{\star}, y^k - x^k}$ and $t_k^2\norms{y^k - x^k}^2$  on the right hand side are zero, we conclude that $\lim_{k\to\infty} \norms{x^k - x^{\star}}^2$ exists.

Finally, since any cluster point of $\sets{x^k}$ is in $\zer{\Phi}$, we conclude that $\sets{x^k}$ converges to $x^{\star} \in \zer{\Phi}$.
Combining the convergence of $\sets{x^k}$ and $\lim_{k\to\infty}\norms{x^k - y^k} = 0$, we can say that $\sets{y^k}$ also converges to $x^{\star} \in \zer{\Phi}$ due to the well-known Opial's lemma.
\Eproof

\beforesubsec
\subsection{\mytb{The Proof of Lemma~\ref{le:NGEAG4NI_key_estimate3}}}\label{apdx:le:NGEAG4NI_key_estimate3}
\aftersubsec
Denote $\hat{a}_k := \omega (t_k + 1 - r - \mu) + \mu \eta$.
First, since $\omega = \frac{\mu(\eta-\beta)}{\mu+1}$ and $\eta \geq \frac{(r+\mu-1)\beta}{r-2}$, we have $\hat{a}_k \leq \omega t_k$.
Similarly, since $r > 2$, we can show that $\eta \geq \frac{(r+\mu-1)\beta}{r-2} \geq \frac{(r+\mu)\beta}{r-1}$.
Thus we also get $a_k \leq \omega t_k^2$.

Next, utilizing \eqref{eq:NGEAG4NI}, we can express $\Fc_k$ from \eqref{eq:NGEAG4NI_Fk_quantity} as
\begin{equation*} 
\begin{array}{lcl}
\Fc_k & = & 2a_k \iprods{w^{k+1} - w^k, x^{k+1} - x^k} + 2a_k\iprods{z^{k+1} - w^{k+1}, x^{k+1} - x^k} \vspace{1ex}\\
&& - {~} 2a_k\iprods{z^k - w^k, x^{k+1} - x^k} \vspace{1ex}\\
& = & 2a_k \iprods{w^{k+1} - w^k, x^{k+1} - x^k} + 2a_{k+1}\iprods{z^{k+1} - w^{k+1}, \theta_k(x^{k+1} - x^k)} \vspace{1ex}\\
&& + {~}  2(a_k - a_{k+1}\theta_k)\iprods{z^{k+1} - w^{k+1}, x^{k+1} - x^k}\vspace{1ex}\\
&& - {~} 2a_k\iprods{z^k - w^k, y^k - x^k - \eta d^k} \vspace{1ex}\\
& = & 2a_k \iprods{w^{k+1} - w^k, x^{k+1} - x^k} + 2a_{k+1}\iprods{z^{k+1} - w^{k+1}, y^{k+1} - x^{k+1}} \vspace{1ex}\\
&& - {~} 2a_k\iprods{z^k - w^k, y^k - x^k} + 2(a_k - a_{k+1}\theta_k)\iprods{z^{k+1} - w^{k+1}, x^{k+1} - x^k}\vspace{1ex}\\
&& + {~} 2a_{k+1}\iprods{z^{k+1} - w^{k+1}, p^k} + 2\eta a_k\iprods{z^k - w^k, d^k}.
\end{array}
\end{equation*}
By the definition of $p^k$ from \eqref{eq:NGEAG4NI_dir}, we have
\begin{equation*} 
\begin{array}{lcl}
t_{k+1}p^k & = &  t_{k+1}\eta_kz^{k+1} - t_{k+1}\lambda_k\hat{w}^{k+1} + t_{k+1}\nu_kz^k \vspace{1ex}\\
& = & [(\eta-\beta)t_k - \delta]z^{k+1} - \eta t_k \hat{w}^{k+1} + \beta t_kz^k \vspace{1ex}\\
& = & \eta t_k(z^{k+1} - \hat{w}^{k+1})  - \beta t_k(z^{k+1} - z^k)  -  \delta z^{k+1}.
\end{array}
\end{equation*}
Therefore, applying Young's inequality, for any $\Delta > 0$, we can derive that
\begin{equation*} 
\begin{array}{lcl}
\hat{\Tc}_{[1]} &:= & a_{k+1}\iprods{z^{k+1} - w^{k+1}, p^k} = \hat{a}_k  \iprods{z^{k+1} - w^{k+1}, t_{k+1} p^k}\vspace{1ex}\\
& = &  \hat{a}_k  \iprods{ \eta t_k( z^{k+1} - \hat{w}^{k+1})  - \beta t_k(z^{k+1} - z^k)  -  \delta z^{k+1}, z^{k+1} - w^{k+1}} \vspace{1ex}\\
& \geq & -  \hat{a}_k  [(\eta + \beta)t_k + \delta]  \norms{z^{k+1} - w^{k+1}}^2 - \frac{\eta \hat{a}_k  t_k}{4}\norms{ z^{k+1} - \hat{w}^{k+1}}^2 \vspace{1ex}\\
&& - {~} \frac{\beta \hat{a}_k  t_k}{4}\norms{ z^{k+1} - z^k }^2 - \frac{\delta  \hat{a}_k  }{4}\norms{ z^{k+1}}^2, \vspace{1ex}\\
\hat{\Tc}_{[2]} &:= & \eta a_k\iprods{z^k - w^k, d^k} \geq - \frac{\Delta \eta r a_k}{2}\norms{z^k - w^k}^2 - \frac{\eta a_k}{2\Delta r}\norms{d^k}^2.
\end{array}
\end{equation*}
Substituting these two inequalities of $\hat{\Tc}_{[1]}$ and $\hat{\Tc}_{[2]}$ into $\Fc_k$ above, we get
\begin{equation*} 
\begin{array}{lcl}
\Fc_k & \geq & 2a_{k+1}\iprods{z^{k+1} - w^{k+1}, y^{k+1} - x^{k+1}} - 2a_k\iprods{z^k - w^k, y^k - x^k}  \vspace{1ex}\\
&& - {~} 2\hat{a}_k  [(\eta + \beta)t_k + \delta] \norms{z^{k+1} - w^{k+1}}^2 -  \frac{ \eta \hat{a}_k  t_k}{2} \norms{ z^{k+1} - \hat{w}^{k+1}}^2 \vspace{1ex}\\
&& - {~} \frac{\beta \hat{a}_k  t_k}{2} \norms{ z^{k+1} - z^k }^2 - \frac{ \delta  \hat{a}_k}{2}  \norms{ z^{k+1}}^2 - \frac{\eta a_k }{\Delta r}\norms{d^k}^2 \vspace{1ex}\\
&& - {~} \Delta \eta r a_k  \norms{z^k - w^k}^2 +  2a_k \iprods{w^{k+1} - w^k, x^{k+1} - x^k} \vspace{1ex}\\
&& + {~} 2(a_k - a_{k+1}\theta_k) \iprods{z^{k+1} - w^{k+1}, x^{k+1} - x^k}.
\end{array}
\end{equation*}
Finally, substituting $b_k := a_k - a_{k+1}\theta_k = \omega (r + \mu - 1) (t_k - r - \mu) + \mu \eta (r + \mu)$, $\hat{a}_k \leq \omega t_k$, and $a_k \leq \omega t_k^2$ into the last inequality, we can further lower bound it and obtain \eqref{eq:NGEAG4NI_key_estimate3}.
\Eproof

\beforesubsec
\subsection{\mytb{The Proof of Lemma~\ref{le:NGEAG4NI_key_bound100}}}\label{apdx:le:NGEAG4NI_key_bound100}
\aftersubsec
First, substituting $\Fc_k$ from  \eqref{eq:NGEAG4NI_Fk_quantity} and $\bmark{\Ec_k}$ from \eqref{eq:NAEG4NI_key_estimate2b}  into \eqref{eq:NGEAG4NI_key_property1}, and choosing $c_1 := 1$ in $\bmark{\Ec_k}$, we get
\begin{equation*} 
\arraycolsep=0.2em
\begin{array}{lcl}
\Pc_k - \Pc_{k+1} & \geq & \big( S_k - \frac{\delta\omega t_k}{2}  \big) \norms{z^{k+1}}^2 + \mu (2t_k - r - \mu) \norms{x^{k+1} - x^k}^2  \vspace{1ex}\\
&& + {~} \Lambda_{k+1}\norms{z^{k+1} }^2 - \Lambda_k\norms{z^k}^2 +   \frac{\beta \omega t_k^2 }{2}  \norms{z^{k+1} - z^k}^2  \vspace{1ex} \\
&& + {~} \eta \omega \big( 1- \frac{1}{\Delta r} - M^2\eta^2 \big) t_k^2 \norms{ d^k}^2 +  \eta\omega \hat{\phi} t_k^2 \norms{ w^{k+1} - \hat{w}^{k+1}}^2 \vspace{1ex}\\
&& + {~} \eta\omega\big(\phi - \frac{1}{2}\big) t_k^2 \norms{z^{k+1} - \hat{w}^{k+1}}^2 - \Delta r \eta\omega t_k^2  \norms{z^k - w^k}^2  \vspace{1ex}\\
&& -  {~} 2\omega t_k \big[  (\eta (2+\phi) + \beta)t_k  + \delta \big] \norms{z^{k+1} - w^{k+1} }^2 \vspace{1ex}\\
&& + {~} 2r [  (\eta - \beta)t_{k+1} - \eta(r-1) ]  \iprods{ z^{k+1}, x^{k+1} - x^{\star}} \vspace{1ex}\\
&& - {~} 2r [ (\eta - \beta)t_k - \eta(r-1) ]  \iprods{ z^k, x^k - x^{\star}} \vspace{1ex} \\
&& + {~} 2a_{k+1}\iprods{z^{k+1} - w^{k+1}, y^{k+1} - x^{k+1}}  - 2a_k\iprods{z^k - w^k, y^k - x^k} \vspace{1ex}\\
&& + {~}  2b_k \iprods{z^{k+1} - w^{k+1}, x^{k+1} - x^k} +  2 \mu \eta t_k \iprods{\hat{w}^{k+1} - z^{k+1}, x^{k+1} - x^k} \vspace{1ex}\\
&& + {~}  2a_k \iprods{w^{k+1} - w^k, x^{k+1} - x^k} + 2\omega r(r-2)  \iprods{ z^{k+1}, x^{k+1} - x^{\star}}.
\end{array}
\end{equation*}
Next, applying Young's inequality, we have
\begin{equation*} 
\arraycolsep=0.2em
\begin{array}{lcl}
2\mu\eta t_k \iprods{\hat{w}^{k+1} - z^{k+1}, x^{k+1} - x^k}  & \geq &   - \frac{2\mu^2\eta}{\omega}\norms{x^{k+1} - x^k}^2 - \frac{\eta \omega t_k^2}{2}\norms{z^{k+1} - \hat{w}^{k+1}}^2.
\end{array}
\end{equation*}
Since $\eta \geq \frac{(r+\mu-1)\beta}{r-2}$, we have $b_k \leq (r+\mu-1)\omega t_k$ from Lemma~\ref{le:NGEAG4NI_key_estimate3}.
Applying Young's inequality again, and using $b_k \leq (r+\mu-1)\omega t_k$, we can show that
\begin{equation*} 
\arraycolsep=0.2em
\begin{array}{lcl}
\hat{\Tc}_{[3]} &:= & 2b_k \iprods{z^{k+1} - w^{k+1}, x^{k+1} - x^k}   \vspace{1ex}\\
& \geq &  - \frac{b_k}{\omega(r+\mu-1)}\norms{x^{k+1} - x^k}^2 - \omega(r+\mu-1)b_k \norms{z^{k+1} - w^{k+1} }^2 \vspace{1ex}\\
& \geq & t_k\norms{x^{k+1} - x^k}^2 - \omega^2(r+\mu-1)^2t_k \norms{z^{k+1} - w^{k+1} }^2.
\end{array}
\end{equation*}
Substituting the last two inequalities into the first estimate, we obtain
\begin{equation*} 
\arraycolsep=0.2em
\begin{array}{lcl}
\Pc_k - \Pc_{k+1} & \geq & \big( S_k - \frac{\delta\omega t_k}{2}  \big) \norms{z^{k+1}}^2 + \mu \big( t_k - r - \mu - \frac{2\mu\eta}{\omega} \big) \norms{x^{k+1} - x^k}^2  \vspace{1ex}\\
&& + {~} \Lambda_{k+1}\norms{z^{k+1} }^2 - \Lambda_k\norms{z^k}^2 +   \frac{ \beta \omega t_k^2}{2} \norms{z^{k+1} - z^k}^2  \vspace{1ex} \\
&& + {~} \eta\omega\big(1 - \frac{1}{\Delta r} - M^2\eta^2 \big) t_k^2  \norms{ d^k}^2 +  \eta\omega \hat{\phi} t_k^2 \norms{ w^{k+1} - \hat{w}^{k+1}}^2 \vspace{1ex}\\
&& + {~} \eta\omega ( \phi - 1 ) t_k^2 \norms{z^{k+1} - \hat{w}^{k+1}}^2 - \Delta r \eta\omega t_k^2  \norms{z^k - w^k}^2  \vspace{1ex}\\
&& - {~}  \omega t_k \big[  2(\eta (2+\phi) + \beta)t_k  + 2\delta + \omega(r + \mu - 1)^2 \big] \norms{z^{k+1} - w^{k+1} }^2 \vspace{1ex}\\
&& + {~} 2r [  (\eta - \beta)t_{k+1} - \eta(r-1) ]  \iprods{ z^{k+1}, x^{k+1} - x^{\star}} \vspace{1ex}\\
&& - {~} 2r [ (\eta - \beta)t_k - \eta(r-1) ]  \iprods{ z^k, x^k - x^{\star}} \vspace{1ex} \\
&& + {~} 2a_{k+1}\iprods{z^{k+1} - w^{k+1}, y^{k+1} - x^{k+1}}  - 2a_k\iprods{z^k - w^k, y^k - x^k} \vspace{1ex}\\
&& + {~} 2\omega r(r-2)  \iprods{ z^{k+1}, x^{k+1} - x^{\star}} +  2a_k \iprods{w^{k+1} - w^k, x^{k+1} - x^k}.
\end{array}
\end{equation*}
Utilizing $\Gc_k$ from \eqref{eq:NGEAG4NI_Gk_func}, the last inequality leads to
\begin{equation}\label{eq:NGEAG4NI_V2_Gk_proof3} 
\hspace{-3ex}
\arraycolsep=0.2em
\begin{array}{lcl}
\Gc_k - \Gc_{k+1} & \geq & \big( S_k - \frac{\delta\omega t_k}{2}  \big) \norms{z^{k+1}}^2 + \mu \big( t_k - r - \mu - \frac{2\mu\eta}{\omega} \big) \norms{x^{k+1} - x^k}^2  \vspace{1ex}\\
&& + {~} \eta\omega\big(1 - \frac{1}{\Delta r} - M^2\eta^2 \big) t_k^2  \norms{ d^k}^2 +  \eta\omega \hat{\phi} t_k^2 \norms{ w^{k+1} - \hat{w}^{k+1}}^2 \vspace{1ex}\\
&& + {~} \eta\omega ( \phi - 1 ) t_k^2 \norms{z^{k+1} - \hat{w}^{k+1}}^2 - \Delta r \eta\omega t_k^2  \norms{z^k - w^k}^2  \vspace{1ex}\\
&& + {~} 2\omega r(r-2)  \iprods{ z^{k+1}, x^{k+1} - x^{\star}} +  2a_k \iprods{w^{k+1} - w^k, x^{k+1} - x^k} \vspace{1ex}\\
&& - {~} \omega t_k \big[  2(\eta (2+\phi) + \beta)t_k  + 2\delta + \omega(r + \mu - 1)^2 \big] \norms{z^{k+1} - w^{k+1} }^2 \vspace{1ex}\\
&& + {~}   \frac{ \beta \omega t_k^2}{2} \norms{z^{k+1} - z^k}^2.
\end{array}
\hspace{-6ex}
\end{equation}
Now, applying again Young's inequality, the $\rho$-co-hypomonotonicity of $\Phi$, and $a_k \leq \omega t_k^2$, we can prove that
\begin{equation*} 
\arraycolsep=0.2em
\begin{array}{lcl}
\hat{\Tc}_{[4]} & := & 2a_k\iprods{w^{k+1} - w^k, x^{k+1} - x^k} \vspace{1ex}\\
& \geq &  2a_k \big[ \iprods{w^{k+1} - w^k, x^{k+1} - x^k} + \rho \norms{w^{k+1} - w^k}^2 \big] \vspace{1ex}\\
&&  - {~}  4\rho a_k \norms{z^{k+1} - z^k}^2 - 8\rho a_k \norms{z^{k+1} - w^{k+1}}^2 -  8 \rho a_k \norms{z^k - w^k}^2 \vspace{1ex}\\
& \geq &- 4\rho\omega t_k^2 \norms{z^{k+1} - z^k}^2 - 8\rho\omega t_k^2 \norms{z^{k+1} - w^{k+1}}^2 - 8\rho\omega t_k^2 \norms{z^k - w^k}^2, \vspace{1ex}\\
\hat{\Tc}_{[5]} &:= & 2\omega r(r-2) \iprods{ z^{k+1}, x^{k+1} - x^{\star}} \vspace{1ex}\\
& = &  2 \omega r(r-2) \iprods{ w^{k+1}, x^{k+1} - x^{\star}} + 2 \omega r(r-2) \iprods{ z^{k+1} - w^{k+1}, x^{k+1} - x^{\star}}  \vspace{1ex}\\ 
& \geq & - 2 \rho\omega r(r-2) \norms{w^{k+1}}^2 - \eta\omega t_k^2 \norms{z^{k+1} - w^{k+1}}^2 - \frac{\omega r^2(r-2)^2}{ \eta t_k^2}\norms{x^{k+1} - x^{\star}}^2 \vspace{1ex}\\
& \geq & - 4 \rho\omega r(r-2) \norms{z^{k+1}}^2 - (\eta\omega t_k^2 + 4 \rho\omega r(r-2) ) \norms{z^{k+1} - w^{k+1}}^2 \vspace{1ex}\\
&& - {~} \frac{\omega r^2(r-2)^2}{ \eta t_k^2}\norms{x^{k+1} - x^{\star}}.
\end{array}
\end{equation*}
Substituting $\hat{\Tc}_{[4]}$ and $\hat{\Tc}_{[5]}$ into the expression $\Gc_k - \Gc_{k+1}$, we can further derive
\begin{equation*} 
\arraycolsep=0.2em
\begin{array}{lcl}
\Gc_k - \Gc_{k+1} & \geq & \big[  S_k - \frac{\delta \omega t_k}{2} - 4\rho \omega r(r-2) \big]  \norms{z^{k+1}}^2 - \frac{ \omega r^2(r-2)^2 }{ \eta t_k^2} \norms{x^{k+1} - x^{\star} }^2 \vspace{1ex}\\
&& + {~} \eta\omega\big(1 - \frac{1}{\Delta r} - M^2\eta^2 \big) t_k^2  \norms{ d^k}^2 +  \eta\omega \hat{\phi} t_k^2 \norms{ w^{k+1} - \hat{w}^{k+1}}^2 \vspace{1ex}\\
&& + {~} \eta\omega ( \phi - 1 ) t_k^2 \norms{z^{k+1} - \hat{w}^{k+1}}^2 - \big( \Delta r \eta + 8\rho \big) \omega t_k^2  \norms{z^k - w^k}^2  \vspace{1ex}\\
&& - {~} \hat{\Theta}_k \norms{z^{k+1} - w^{k+1} }^2 +  \frac{ (\beta - 8\rho) \omega t_k^2}{2} \norms{z^{k+1} - z^k}^2 \vspace{1ex}\\
&& + {~} \mu \big( t_k - r - \mu - \frac{2\mu\eta}{\omega} \big) \norms{x^{k+1} - x^k}^2, 
\end{array}
\end{equation*}
where $\hat{\Theta}_k :=   \omega t_k \big[   (\eta (5+2\phi) + 2 \beta +  8\rho ) t_k  + 2\delta + \omega(r + \mu - 1)^2 \big]  + 4 \rho\omega r(r-2)$.
This exactly proves \eqref{eq:NGEAG4NI_key_bound100}.
\Eproof

\beforesubsec
\subsection{\mytb{The Proof of Lemma~\ref{le:NGEAG4NI_V2_descent_of_Lyapunov}}}\label{apdx:le:NGEAG4NI_V2_descent_of_Lyapunov}
\aftersubsec
Choosing $\phi := 1$ in \eqref{eq:NGEAG4NI_key_bound100}, and $\beta := 8\rho + 2\epsilon$, \eqref{eq:NGEAG4NI_key_bound100} reduces to 
\begin{equation}\label{eq:NGEAG4NI_V2_Lyapunov_proof1} 
\hspace{-1ex}
\arraycolsep=0.2em
\begin{array}{lcl}
\Gc_k - \Gc_{k+1} & \geq & \big[  S_k - \frac{\delta \omega t_k}{2} - 4\rho \omega r(r-2) \big]  \norms{z^{k+1}}^2 - \frac{ \omega r^2(r-2)^2 }{ \eta t_k^2} \norms{x^{k+1} - x^{\star} }^2 \vspace{1ex}\\
&& + {~} \eta\omega\big(1 - \frac{1}{\Delta r} - M^2\eta^2 \big) t_k^2  \norms{ d^k}^2 - \big( \Delta r \eta + 8\rho \big) \omega t_k^2  \norms{z^k - w^k}^2  \vspace{1ex}\\
&& + {~} \big( \Delta r \eta + 8\rho \big) \omega t_{k+1}^2 \norms{z^{k+1} - w^{k+1}}^2 \vspace{1ex}\\
&& + {~} \mu \big( t_k - r - \mu - \frac{2\mu\eta}{\omega} \big) \norms{x^{k+1} - x^k}^2  \vspace{1ex}\\
&& + {~}  \eta\omega \hat{\phi} t_k^2 \norms{ w^{k+1} - \hat{w}^{k+1}}^2  - A_k \norms{z^{k+1} - w^{k+1} }^2 \vspace{1ex}\\
&& + {~} \hat{\epsilon} \omega \eta  t_{k+1}^2 \norms{z^{k+1} - w^{k+1}}^2 + \epsilon\omega t_k^2 \norms{z^{k+1} - z^k}^2,
\end{array}
\hspace{-4ex}
\end{equation}
where $A_k := \Theta_k + \big( \Delta r \eta + 8\rho + \hat{\epsilon}\eta \big) \omega t_{k+1}^2$ for some $\hat{\epsilon} \geq 0$.

Now, since $\beta = 8\rho + 2\epsilon  \geq 8\rho$, $\phi = 1$, and $\eta \geq \frac{(r+\mu - 1)\beta}{r-2}$, we can show that
\begin{equation*} 
\arraycolsep=0.2em
\begin{array}{lcl}
A_k & = & \omega t_k \big[   (\eta (5+2\phi) + 2 \beta +  8\rho ) t_k  + 2\delta + \omega(r + \mu - 1)^2 \big]  + 4 \rho\omega r(r-2) \vspace{1ex}\\
&& + {~} \omega (\Delta r \eta + 8\rho + \hat{\epsilon}\eta ) (t_k^2 + 2t_k + 1) \vspace{1ex}\\
& \leq &  \big[   8 + \frac{4(r-2)}{r+\mu-1} + \Delta r +  \hat{\epsilon} \big] \eta \omega t_k^2 \vspace{1ex}\\
& \leq &   \eta \omega \Theta t_k^2,
\end{array}
\end{equation*}
where $\Theta := 12  + \Delta r + \hat{\epsilon}$, 
provided that $t_k \geq  \frac{D + \sqrt{D^2 + 4 \eta E}}{2 \eta}$ with $D := 2\delta + \omega(r + \mu - 1)^2 + 2 \Delta r \eta + 2 \beta + 2\hat{\epsilon} \eta$ and $E := 4 \rho r(r-2) + \Delta r \eta + 8\rho + \hat{\epsilon} \eta$.

We also have $S_k - \frac{\delta \omega t_k}{2} - 4\rho \omega r(r-2) =  \frac{\omega[3(r-2)\psi - (r-1)\beta]t_k - \hat{\Gamma}}{2}$, where $\hat{\Gamma} := 2\Gamma + 8\rho \omega r(r-2)$.

Moreover, from \eqref{eq:NGEAG4NI_u_cond}, and $z^{k+1} := u^{k+1} + \xi^{k+1}$ in \eqref{eq:NGAEG4NI_w_quantities}, we can write 
\begin{equation*} 
\arraycolsep=0.2em
\begin{array}{lcl}
\norms{z^{k+1} - w^{k+1}}^2 & \leq & \kappa\norms{w^{k+1} - \hat{w}^{k+1}}^2 + \hat{\kappa}\norms{d^k}^2.
\end{array}
\end{equation*}
Substituting the last three expressions into \eqref{eq:NGEAG4NI_V2_Lyapunov_proof1}, rearranging the result, and using \eqref{eq:NGEAG4NI_V2_Lyapubov_func}, we can prove that
\begin{equation*} 
\hspace{-2ex}
\arraycolsep=0.2em
\begin{array}{lcl}
\hat{\Lc}_k - \hat{\Lc}_{k+1} & \geq & \frac{\omega[3(r-2)\psi - (r-1)\beta]t_k - \hat{\Gamma}}{2}    \norms{z^{k+1}}^2 - \frac{ \omega r^2(r-2)^2 }{\eta t_k^2} \norms{x^{k+1} - x^{\star} }^2 \vspace{1ex}\\
&& + {~}  \eta\omega (\hat{\phi} - \kappa \Theta ) t_k^2 \norms{ w^{k+1} - \hat{w}^{k+1}}^2 \vspace{1ex}\\ 
&& + {~} \mu \big(t_k - r - \mu - \frac{2\mu\eta}{\omega} \big) \norms{x^{k+1} - x^k}^2  \vspace{1ex}\\
&& + {~} \eta \omega  \big(1 - \frac{1}{\Delta r} - \hat{\kappa}\Theta - M^2\eta^2\big) t_k^2 \norms{ d^k}^2\vspace{1ex}\\
&& + {~}  \hat{\epsilon} \omega \eta  t_{k+1}^2 \norms{z^{k+1} - w^{k+1}}^2 + \epsilon\omega t_k^2 \norms{z^{k+1} - z^k}^2.
\end{array}
\hspace{-2ex}
\end{equation*}
This is exactly \eqref{eq:NGEAG4NI_V2_Lyapubov_descent}.
\Eproof

\beforesubsec
\subsection{\mytb{The Proof of Lemma~\ref{le:NGEAG4NI_Lyapunov_lower_bound}}}\label{apdx:le:NGEAG4NI_Lyapunov_lower_bound}
\aftersubsec
First, we can easily show that
\begin{equation*} 
\arraycolsep=0.1em
\begin{array}{lcl}
\tilde{\Tc}_{[1]} &:= & \norms{r(x^k - x^{\star}) + t_k(y^k - x^k)}^2 + \mu r \norms{x^k - x^{\star}}^2 \vspace{1ex}\\
& = & (r^2+ r\mu)\norms{x^k - x^{\star}}^2 + 2rt_k \iprods{x^k - x^{\star}, y^k - x^k} + t_k^2\norms{y^k - x^k}^2 \vspace{1ex}\\
& = & r \big( \mu - \frac{1}{2}\big) \norms{x^k {\!} - {\!} x^{\star}}^2 + \frac{2r+1}{2r} \norms{r(x^k {\!} - {\!} x^{\star}) + \frac{2rt_k}{2r+1}(y^k {\!} - {\!} x^k)}^2 +  \frac{ t_k^2}{2r+1} \norms{y^k {\!} - {\!} x^k}^2.
\end{array}
\end{equation*}
Denote $s_k := (\eta - \beta)t_k - \eta(r-1)$.
Then, substituting $\tilde{\Tc}_{[1]}$ into \eqref{eq:NGEAG4NI_V2_Lyapubov_func}, and using Young's inequality and $a_k \leq \omega t_k^2$, we can derive that
\begin{equation}\label{eq:NGEAG4NI_lm36_proof1}
\hspace{-1ex}
\arraycolsep=0.2em
\begin{array}{lcl}
\hat{\Lc}_k &:= &  \norms{r(x^k - x^{\star}) + t_k(y^k - x^k)}^2 + r \mu \norms{x^k - x^{\star}}^2 +  \Lambda_k\norms{z^k}^2  \vspace{1ex}\\
&& + {~} 2rs_k \iprods{z^k, x^k - x^{\star}} +  2c_k \iprods{z^k, y^k - x^k}  \vspace{1ex}\\
&& + {~} 2a_k \iprods{z^k - w^k, y^k - x^k} + \alpha_k\norms{z^k - w^k}^2 \vspace{1ex}\\
& \geq & r \big( \mu - \frac{1}{2}\big) \norms{x^k - x^{\star}}^2 + \frac{ t_k^2}{2r+1} \norms{y^k - x^k}^2 + \Lambda_k\norms{z^k}^2  \vspace{1ex}\\
&& + {~} \frac{2r+1}{2r} \norms{r(x^k - x^{\star}) + \frac{2r t_k}{2r+1}(y^k - x^k)}^2 \vspace{1ex}\\
&& + {~} \frac{(2r+1)c_k}{r t_k}\iprods{z^k, r(x^k - x^{\star}) + \frac{2rt_k}{2r+1}(y^k - x^k)} \vspace{1ex}\\
&& + {~}  \big[ 2r s_k - \frac{(2r+1)c_k}{t_k}\big] \iprods{z^k, x^k - x^{\star}} \vspace{1ex}\\
&& - {~} \frac{a_k}{(2r+1)\omega}\norms{y^k - x^k}^2 + \big[ \alpha_k - (2r+1)\omega a_k \big] \norms{z^k - w^k}^2 \vspace{1ex}\\
&=& r \big( \mu - \frac{1}{2}\big) \norms{x^k - x^{\star}}^2 + \big[ \Lambda_k - \frac{(2r + 1)c_k^2}{2r t_k^2} \big]\norms{z^k}^2 \vspace{1ex}\\
&& + {~} \frac{2r+1}{2r} \norms{r(x^k - x^{\star}) + \frac{2r t_k}{2r+1}(y^k - x^k) + \frac{c_k}{t_k} z^k}^2 \vspace{1ex} \\
&& + {~}  \big[ 2r s_k - \frac{(2r+1)c_k}{t_k}\big] \iprods{z^k, x^k - x^{\star}}  \vspace{1ex}\\
&& + {~} \big[ \alpha_k - (2r+1)\omega a_k \big] \norms{z^k - w^k}^2.
\end{array}
\hspace{-2ex}
\end{equation}
We note that if $t_k \geq \frac{\mu}{\omega}(r-1)[ \eta - (2r+1)\omega ]$, then 
\begin{equation}\label{eq:NGEAG4NI_lm36_proof2}
\hspace{-1ex}
\arraycolsep=0.2em
\begin{array}{lcl}
n_k &:= & 2rs_k- \frac{(2r+1)c_k}{t_k} \leq 2r\omega t_k. 
\end{array}
\hspace{-4ex}
\end{equation}
By Young's inequality and $\iprods{w^k, x^k - x^{\star}} \geq -\rho\norms{w^k}^2$, we have
\begin{equation*} 
\arraycolsep=0.2em
\begin{array}{lcl}
\tilde{\Tc}_{[2]} &:= & n_k \iprods{z^k, x^k - x^{\star}} \vspace{1ex}\\
& = &  n_k \iprods{z^k - w^k, x^k - x^{\star}} + n_k \iprods{w^k, x^k - x^{\star}}  \vspace{1ex}\\
& \geq &  - n_k \big( \frac{n_k}{2r} + 2 \rho) \norms{z^k - w^k}^2 - \frac{r}{2}\norms{x^k - x^{\star}}^2 -  2 \rho n_k \norms{z^k}^2.
\end{array}
\end{equation*}
Substituting this inequality $\tilde{\Tc}_{[2]}$ into \eqref{eq:NGEAG4NI_lm36_proof1}, we can show that
\begin{equation}\label{eq:NGEAG4NI_lm36_proof3}
\arraycolsep=0.2em
\begin{array}{lcl}
\hat{\Lc}_k &\geq& r (\mu - 1) \norms{x^k - x^{\star}}^2 + \big[ \Lambda_k - \frac{(2r + 1)c_k^2}{2r t_k^2} -  2 \rho n_k \big]\norms{z^k}^2 \vspace{1ex}\\
&& + {~} \frac{2r+1}{2r} \norms{r(x^k - x^{\star}) + \frac{2r t_k}{2r+1}(y^k - x^k) + \frac{c_k}{t_k} z^k}^2 \vspace{1ex}\\
&& + {~} \big[ \alpha_k - (2r+1)\omega a_k - \frac{n_k^2}{2r} - 2\rho n_k\big] \norms{z^k - w^k}^2.
\end{array}
\end{equation}
One the one hand, we can lower bound
\begin{equation*} 
\arraycolsep=0.2em
\begin{array}{lcl}
\tilde{\Tc}_{[3]} &:= & \Lambda_k - \frac{(2r+1) c_k^2}{2r t_k^2} - 2\rho n_k \geq \frac{(\eta - \beta)^2t_k^2}{2(\mu+1)^2}. 
\end{array}
\end{equation*}
provided that $\mu \ge 1$, $r > 2$, and $t_k \geq \hat{t}_0$, where $\hat{t}_0$ is defined as in \eqref{eq:NGEAG4NI_Lyapunov_lower_bound_t0_hat}.

On the other hand, using \eqref{eq:NGEAG4NI_lm36_proof2} and $a_k \leq \omega t_k$, we can show that
\begin{equation*} 
\arraycolsep=0.2em
\begin{array}{lcl}
\tilde{\Tc}_{[4]} &:= & \alpha_k - (2r+1)\omega a_k - \frac{n_k^2}{2r} - 2\rho n_k  \geq \big[ \frac{\Delta r \eta}{\omega} - 4r - 1 \big] \omega^2 t_k^2,
\end{array}
\end{equation*}
due to $t_k \geq r > \frac{r}{2}$.
Substituting $\tilde{\Tc}_{[3]}$ and $\tilde{\Tc}_{[4]}$ into \eqref{eq:NGEAG4NI_lm36_proof3}, we obtain \eqref{eq:NGEAG4NI_Lyapunov_lower_bound}.

Finally, if we choose $\Delta := 3 \ge \frac{(\eta - \beta)[1 + 2(4r + 1)\mu]}{2r \eta \mu (\mu + 1)}$, then we have $\big[ \frac{\Delta r \eta}{\omega} - 4r - 1 \big] \omega^2 \geq \frac{(\eta - \beta)^2}{2(\mu+1)^2}$.
Consequently, \eqref{eq:NGEAG4NI_Lyapunov_lower_bound} leads to
\begin{equation*} 
\begin{array}{lcl}
\hat{\Lc}_k & \geq &  r(\mu - 1) \norms{x^k - x^{\star}}^2 + \frac{(\eta - \beta)^2t_k^2}{2(\mu+1)^2}\big[ \norms{z^k}^2 + \norms{z^k - w^k}^2\big] \vspace{1ex}\\
& \geq &  r(\mu - 1) \norms{x^k - x^{\star}}^2 + \frac{(\eta - \beta)^2t_k^2}{4(\mu+1)^2} \norms{w^k}^2,
\end{array}
\end{equation*}
which proves \eqref{eq:NGEAG4NI_Lyapunov_lower_bound2}.
\Eproof

\beforesubsec
\subsection{\mytb{The Proof of Theorem~\ref{th:NGEAG4NI_V2_convergence1}}}\label{apdx:th:NGEAG4NI_V2_convergence1}
\aftersubsec
First, since $\Delta = 3$, we have $\Theta = 12 + \Delta r + \hat{\epsilon} = 12 + 3r + \hat{\epsilon}$ in Lemma~\ref{le:NGEAG4NI_V2_descent_of_Lyapunov}.
Let us choose $\hat{\phi} = \kappa\Theta = \kappa(12 + 3r + \hat{\epsilon})$ in Lemma~\ref{le:NGEAG4NI_V2_descent_of_Lyapunov}.
In this case, $M^2 = ((1+ \phi)(1 + c_1) + \hat{\phi} )L^2 = (4 +  12 \kappa  + 3 \kappa r + \kappa \hat{\epsilon}) L^2$ in Lemma~\ref{le:NGEAG4NI_Ek_lower_bound}.

Next, to guarantee $1 - \frac{1}{\Delta r} - \hat{\kappa}\Theta > 0$, we need to impose $0 \leq \hat{\kappa} < \frac{3r - 1}{3r(12 + 3r + \hat{\epsilon})}$, which is exactly the condition \eqref{eq:NGEAG4NI_V2_kappa_hat_cond}.

Since we require $1 - \frac{1}{\Delta r} - \hat{\kappa}\Theta - M^2\eta^2 \geq 0$, we need to choose 
\begin{equation*} 
\begin{array}{lcl}
\eta \leq \bar{\eta} := \frac{\Psi}{L} =  \frac{1}{L\sqrt{4 +  12 \kappa  + 3\kappa r + \kappa \hat{\epsilon}} } \big[1 - \frac{1}{3r} - \hat{\kappa}(12 + 3r + \hat{\epsilon})\big],
\end{array}
\end{equation*}
where $\Psi$ is given in \eqref{eq:NGEAG4NI_V2_Psi_cond}.
Moreover, we have imposed $\eta \geq \frac{\beta(r+\mu-1)}{r-2}$ in Lemma~\ref{le:NGEAG4NI_V2_descent_of_Lyapunov}, leading to the update rule of $\eta$ as in \eqref{eq:NGEAG4NI_V2_choice_of_constants}.
The choice of $\beta := 8\rho + 2\epsilon$ was also enforced in Lemma~\ref{le:NGEAG4NI_V2_descent_of_Lyapunov} for some $\epsilon \geq 0$.

Now, let us denote 
\begin{equation}\label{eq:NGEAG4NI_V2_convergence_proof1}
\hspace{-2ex}
\arraycolsep=0.2em
\begin{array}{lcl}
\Hc_k & := & \frac{\omega[3(r-2)\psi - (r-1)\beta]t_k - \hat{\Gamma}}{2}    \norms{z^{k+1}}^2 + \eta \omega M^2(\bar{\eta}^2 - \eta^2\big) t_k^2  \norms{ d^k}^2 \vspace{1ex}\\
&& + {~} \mu \big(t_k - r - \mu - \frac{2\mu\eta}{\omega} \big) \norms{x^{k+1} - x^k}^2\vspace{1ex}\\
&& + {~} \hat{\epsilon} \omega \eta t_{k+1}^2 \norms{z^{k+1} - w^{k+1}}^2 + \epsilon\omega t_k^2 \norms{z^{k+1} - z^k}^2.
\end{array}
\hspace{-2ex}
\end{equation}
Then, under the above parameter selections, \eqref{eq:NGEAG4NI_V2_Lyapubov_descent} reduces to
\begin{equation*} 
\arraycolsep=0.2em
\begin{array}{lcl}
\hat{\Lc}_k - \hat{\Lc}_{k+1} & \geq & \Hc_k - \frac{ \omega r^2(r-2)^2 }{\eta t_k^2} \norms{x^{k+1} - x^{\star} }^2 \geq  \Hc_k - \frac{ \omega r(r-2)^2 }{\eta (\mu-1) t_k^2} \hat{\Lc}_{k+1},
\end{array}
\end{equation*}
where we have used $\hat{\Lc}_{k+1} \geq r(\mu - 1) \norms{x^{k+1} - x^{\star}}^2$ from \eqref{eq:NGEAG4NI_Lyapunov_lower_bound2} in the second inequality.

Denote $\tau_{k+1} :=  \frac{ \omega r(r-2)^2 }{\eta (\mu-1) t_k^2 } =  \frac{ \omega r(r-2)^2 }{\eta (\mu-1) (k+t_0)^2 }$.
Then, by the choice of $t_0$ in \eqref{eq:NGEAG4NI_V2_Psi_cond}, we have $\tau_{k+1} \in (0, 1)$.
Thus the last inequality leads to 
\begin{equation*} 
\arraycolsep=0.2em
\begin{array}{lcl}
(1 - \tau_{k+1})\hat{\Lc}_{k+1} \leq \big(1 + \frac{\tau_k}{1-\tau_k} \big) (1-\tau_k)\hat{\Lc}_k - \Hc_k
\end{array}
\end{equation*}
Since $\sum_{k=0}^{\infty} \frac{\tau_k}{1-\tau_k} < \infty$, $(1-\tau_k)\hat{\Lc}_k \geq 0$, and $\Hc_k \geq 0$, applying \cite[Lemma~5.31]{Bauschke2011}, we can show that $\sum_{k=0}^{\infty}\Hc_k <+\infty$ and $\lim_{k\to\infty}\hat{\Lc}_k$ exists.

Since $t_k - r - \mu - \frac{2\mu\eta}{\omega} \geq \frac{t_k}{2}$ due to the choice of $t_0$, the summable result $\sum_{k=0}^{\infty}\Hc_k <+\infty$ implies the first five summable results in \eqref{eq:NGEAG4NI_V2_convergence1}.
The last summable result follows from $\norms{w^k}^2 \leq 2\norms{z^k - w^k}^2 + 2\norms{z^k}^2$.

From the proof of \cite[Lemma~5.31]{Bauschke2011}, we also have
\begin{equation*} 
\arraycolsep=0.2em
\begin{array}{lcl}
\hat{\Lc}_k \leq \frac{1}{(1-\tau_k)} \hat{\Lc}_{k-1} \leq \prod_{i = 1}^{k} \frac{1}{1 - \tau_i} \hat{\Lc}_0 \leq \prod_{k = 1}^{\infty} \frac{1}{1 - \tau_k} \hat{\Lc}_0.
\end{array}
\end{equation*}
Denoting $\tau := \prod_{k=1}^{\infty} \frac{1}{1 - \tau_k}$.
Then, using an elementary proof, we can  show that $1 \leq \tau \leq \Omega := \exp\left(\frac{c^2}{(1 + t_0)^2 - c^2}\right) \left( \frac{1 + t_0 + c}{1 + t_0 - c} \right)^{\frac{c}{2}}$, where $c^2 := \frac{\omega r (r - 2)^2}{\eta (\mu - 1)}$.
Hence, the last inequality leads to $0 \leq \hat{\Lc}_k \leq \Omega\hat{\Lc}_0$ for all $k\geq 0$.

Finally, since $z^0 := w^0$ and $y^0 := x^0$, it is obvious to show that $\hat{\Lc}_0 \leq \Rc_0^2$ for $\Rc_0^2$ defined by \eqref{eq:NGEAG4NI_V1_R02}.
By \eqref{eq:NGEAG4NI_Lyapunov_lower_bound2} and $\hat{\Lc}_k \leq   \Omega \hat{\Lc}_0 \leq \Omega \Rc_0^2$, we have $\frac{(\eta - \beta)^2t_k^2}{4(\mu+1)^2} \norms{w^k}^2 \leq \hat{\Lc}_k \leq \Omega \Rc_0^2$.
This relation implies  \eqref{eq:NGEAG4NI_V2_convergence2}.
\Eproof

\beforesubsec
\subsection{\mytb{The Proof of Theorem~\ref{th:NGEAG4NI_V2_convergence2}}}\label{apdx:th:NGEAG4NI_V2_convergence2}
\aftersubsec
Denote $\hat{v}^k := y^k - x^k + \eta_{k-1}z^k$.
Similar to the proof of \eqref{eq:NGEAG4NI_V1_th32_proof0}, we have
\begin{equation*} 
\hspace{-1ex}
\arraycolsep=0.2em
\begin{array}{lcl}
t_{k+1}^2 \norms{\hat{v}^{k+1}}^2 & \leq & t_k^2 \norms{\hat{v}^k}^2  - [(r-2)t_k + r-1] \norms{\hat{v}^k}^2 + \frac{2\eta^2(r+1)^2t_{k+1}}{(r+2)}\norms{d^k}^2 \vspace{1ex}\\
&& + {~} \frac{2[ (r+4)(\eta-\beta) + 1 ]^2 t_{k+1}}{2(r+2)}\norms{z^k}^2.
\end{array}
\hspace{-2ex}
\end{equation*}
This inequality together with \eqref{eq:NGEAG4NI_V2_convergence1} imply 
\begin{equation}\label{eq:NGEAG4NI_V2_convergence2_proof1}
\arraycolsep=0.2em
\begin{array}{ll}
& \lim_{k\to\infty}t_k^2\norms{  y^k - x^k + \eta_{k-1}z^k }^2 = 0, \vspace{1ex}\\
& \sum_{k=0}^{\infty} t_k \norms{ y^k - x^k + \eta_{k-1}z^k }^2 < +\infty.
\end{array}
\end{equation}
Next, we write
\begin{equation*} 
\arraycolsep=0.2em
\begin{array}{lcl}
\eta\theta_kd^k + p^k & = & (\eta\theta_k - \lambda_k)\hat{w}^{k+1} + \eta_kz^{k+1} - (\eta\theta_k\gamma_k - \nu_k)z^k \vspace{1ex}\\
&= & -\frac{\eta(r+\mu)}{t_{k+1}}\hat{w}^{k+1} + \frac{[(\eta-\beta)t_k-\delta]}{t_{k+1}}z^{k+1} - \big[ \frac{(\eta-\beta)\theta_k t_{k-1}}{t_k} - \frac{s_k}{t_{k+1}}\big]z^k,
\end{array}
\end{equation*}
where $s_k :=  \frac{ [ \eta(r-2) + \beta(r+\mu  +  1) ](t_k - r - \mu)  + (r+\mu)^2\beta }{t_k}$.

\noindent 
Using this expression and \eqref{eq:NGEAG4NI}, we can show that
\begin{equation*} 
\arraycolsep=0.2em
\begin{array}{lcl}
\zeta^{k+1} &:= & y^{k+1} - x^{k+1} + \eta_k (z^{k+1} - w^{k+1}) \vspace{1ex}\\
&=&  \theta_k [(y^k - x^k) + \eta_{k-1} (z^k - w^k)] - \theta_k \eta_{k-1} (z^k - w^k) - \eta \theta_k d^k \vspace{1ex}\\
&& - {~} p^k + \eta_k (z^{k+1} - w^{k+1})  \vspace{1ex}\\
& = & \theta_k\zeta^k   - \eta_k(w^{k+1} - w^k) + \frac{\eta(r+\mu)}{t_{k+1}} d^k +  \frac{ \hat{s}_k }{t_{k+1}} (z^k - w^k) - \frac{ \tilde{s}_k }{t_{k+1}} z^k,
\end{array}
\end{equation*}
where $\hat{s}_k := \frac{(r + \mu + 1)(\eta - \beta) t_k - (r + \mu)(\eta - \beta + \delta)}{t_k} \leq \hat{s} := (r + \mu + 1)(\eta - \beta)$ and $\tilde{s}_k :=  \eta (r - 1) - \delta := \tilde{s} > 0$.

Denote $h^k := \eta(r+\mu)d^k + \hat{s}_k(z^k - w^k) - \tilde{s}_kz^k$.
Using the last expression and Young's inequality, we can derive
\begin{equation*} 
\arraycolsep=0.1em
\begin{array}{lcl}
\norms{\zeta^{k+1}}^2 & = & \theta_k^2\norms{\zeta^k}^2 - 2\eta_k\theta_k \iprods{w^{k+1} - w^k, \zeta^k} + \frac{2\theta_k}{t_{k+1}}\iprods{\zeta^k, h^k} - \frac{2\eta_k}{t_{k+1}} \iprods{h^k, w^{k+1} - w^k} \vspace{1ex}\\
&& + {~} \eta_k^2\norms{w^{k+1} - w^k}^2 + \frac{1}{t_{k+1}^2}\norms{h^k}^2 \vspace{1ex}\\
& \leq & \big(\theta_k^2 + \frac{\theta_k}{t_{k+1}}\big)\norms{\zeta^k}^2  - 2\eta_k\theta_k\iprods{w^{k+1} - w^k, x^{k+1} - x^k} \vspace{1ex}\\
&& - {~} 2\eta_k\theta_k \iprods{w^{k+1} - w^k, \eta d^k + \eta_{k-1}(z^k - w^k)} \vspace{1ex}\\
&& + {} \big(\frac{1}{t_{k+1}^2} + \frac{\theta_k}{t_{k+1}}  + \frac{\eta_k}{t_{k+1}} \big) \norms{h^k}^2 + \left( \eta_k^2  + \frac{\eta_k}{t_{k+1}} \right) \norms{w^{k+1} - w^k}^2 \vspace{1ex}\\
&\leq & \big(\theta_k^2 + \frac{\theta_k}{t_{k+1}}\big)\norms{\zeta^k}^2  + \big( 2\rho \eta_k\theta_k + \eta_k\theta_k + \eta_k^2 + \frac{\eta_k}{t_{k+1}} \big) \norms{w^{k+1} - w^k}^2 \vspace{1ex}\\
&& + {~}  2\eta^2 \eta_k\theta_k \norms{d^k}^2  + 2\eta_k\theta_k \eta_{k-1}^2\norms{z^k - w^k}^2 + \big(\frac{1}{t_{k+1}^2} + \frac{\theta_k}{t_{k+1}} + \frac{\eta_k}{t_{k+1}} \big) \norms{h^k}^2,
\end{array}
\end{equation*}
where  we have used $\iprods{w^{k+1} - w^k, x^{k+1} - x^k} \geq -\rho\norms{w^{k+1} - w^k}^2$ from the $\rho$-co-hypomonotonicity of $\Phi$ in the last inequality.

Let us denote
\begin{equation}\label{eq:NGEAG4NI_V2_convergence2_proof2}
\arraycolsep=0.2em
\begin{array}{lcl}
\Sigma_k & := & \big( 2\rho \eta_k\theta_k + \eta_k\theta_k + \eta_k^2  + \frac{\eta_k}{t_{k+1}} \big) \norms{w^{k+1} - w^k}^2  + 2\eta^2 \eta_k\theta_k \norms{d^k}^2  \vspace{1ex}\\
&& + {~} 2\eta_k\theta_k \eta_{k-1}^2\norms{z^k - w^k}^2 + \big(\frac{1}{t_{k+1}^2} + \frac{\theta_k}{t_{k+1}} + \frac{\eta_k}{t_{k+1}}\big) \norms{h^k}^2.
\end{array}
\end{equation}
Then, the last inequality can be rewritten as
\begin{equation*}
\arraycolsep=0.2em
\begin{array}{lcl}
\norms{\zeta^{k+1} }^2 & \leq & \big(\theta_k^2 + \frac{\theta_k}{t_{k+1}}\big)\norms{\zeta^k}^2 + \Sigma_k.
\end{array}
\end{equation*}
Multiplying this inequality by $t_{k+1}^2$, we can show that
\begin{equation}\label{eq:NGEAG4NI_V2_convergence2_proof3}
\arraycolsep=0.2em
\begin{array}{lcl}
t_{k+1}^2\norms{\zeta^{k+1} }^2 & \leq & t_k^2 \norms{\zeta^k}^2   + t_{k+1}^2\Sigma_k \vspace{1ex}\\
&& - {~} [(2r+ 2\mu - 1)t_k - (r+\mu)(r+\mu-1)]\norms{\zeta^k}^2.
\end{array}
\end{equation}
Our next step is to prove that $t_{k+1}^2\Sigma_k$ is summable.
Indeed, from \eqref{eq:NGEAG4NI_V2_convergence2_proof2}, by Young's inequality, we have
\begin{equation*} 
\arraycolsep=0.2em
\begin{array}{lcl}
\Sigma_k & \leq & 3 \big( 2\rho \eta_k\theta_k + \eta_k\theta_k + \eta_k^2 + \frac{\eta_k}{t_{k+1}} \big) \big[ \norms{z^{k+1} - w^{k+1}}^2 +  \norms{z^{k+1} - z^k }^2  \big] \vspace{1ex}\\
&& + {~} \big[ 3\big( 2\rho \eta_k\theta_k + \eta_k\theta_k + \eta_k^2 + \frac{\eta_k}{t_{k+1}} \big) + 2\eta_k\theta_k \eta_{k-1}^2 \big]\norms{z^k - w^k}^2   \vspace{1ex}\\
&& + {~} \big(\frac{1}{t_{k+1}^2} + \frac{\theta_k}{t_{k+1}} + \frac{\eta_k}{t_{k+1}}\big) \norms{ \eta(r+\mu)d^k + \hat{s}_k(z^k - w^k) - \tilde{s}_kz^k }^2 \vspace{1ex}\\
&& + {~} 2\eta^2 \eta_k\theta_k \norms{d^k}^2 \vspace{1ex}\\
& \leq & 3 \big( 2\rho \eta_k\theta_k + \eta_k\theta_k + \eta_k^2 + \frac{\eta_k}{t_{k+1}} \big) \big[ \norms{z^{k+1} - w^{k+1}}^2 +  \norms{z^{k+1} - z^k }^2  \big] \vspace{1ex}\\
&& + {~} B_k \norms{z^k - w^k}^2  + 3\tilde{s}_k^2\big(\frac{1}{t_{k+1}^2} + \frac{\theta_k}{t_{k+1}}  + \frac{\eta_k}{t_{k+1}} \big) \norms{z^k}^2 \vspace{1ex}\\
&& + {~} \Big[ 2\eta^2 \eta_k\theta_k + 3\eta^2(r+\mu)^2 \big(\frac{1}{t_{k+1}^2} + \frac{\theta_k}{t_{k+1}} + \frac{\eta_k}{t_{k+1}} \big) \Big] \norms{d^k}^2,
\end{array}
\end{equation*}
where $B_k := 3\big( 2\rho \eta_k\theta_k + \eta_k\theta_k + \eta_k^2 + \frac{\eta_k}{t_{k+1}} \big) + 2\eta_k\theta_k \eta_{k-1}^2 + 3\hat{s}_k^2 \big(\frac{1}{t_{k+1}^2} + \frac{\theta_k}{t_{k+1}} + \frac{\eta_k}{t_{k+1}} \big)$.

Using the facts that $\eta_k \leq \frac{(\eta-\beta)t_k}{t_{k+1}} \leq \eta-\beta$, $\theta_k\leq \frac{t_k}{t_{k+1}}$, $\hat{s}_k \leq \hat{s} := (r + \mu + 1)(\eta - \beta)$, and $\tilde{s}_k = \tilde{s} := \eta (r - 1) - \delta$, we can prove that
\begin{equation*} 
\arraycolsep=0.2em
\begin{array}{lcl}
t_{k+1}^2\Sigma_k & \leq & 3 (\eta - \beta) \big[ (2\rho + 1 + \eta - \beta) t_k^2 + t_k \big] \big[ \norms{z^{k+1} - w^{k+1}}^2 +  \norms{z^{k+1} - z^k }^2  \big] \vspace{1ex}\\
&&  + {~} \Big[ 2\eta^2 (\eta - \beta) t_k^2 + 3\eta^2(r+\mu)^2 (\eta - \beta + 1)t_k +  3\eta^2(r+\mu)^2 \Big] \norms{d^k}^2 \vspace{1ex}\\ 
&& + {~} 3\tilde{s}^2 \big[(\eta - \beta + 1)t_k + 1 \big] \norms{z^k}^2 + \tilde{B}_k \norms{z^k - w^k}^2,
\end{array}
\end{equation*}
where $\tilde{B}_k :=  (\eta - \beta) \big[ (6\rho + 3 + 3\eta - 3\beta + 2 (\eta - \beta)^2) t_k^2 + 3t_k \big] + 3[\hat{s}^2 (\eta - \beta + 1) + \eta - \beta]t_k + 3\hat{s}^2 = \BigO{t_k^2}$.

Now, applying  \eqref{eq:NGEAG4NI_V2_convergence1} to this inequality, one can establish that
\begin{equation}\label{eq:NGEAG4NI_V2_convergence2_proof3b}
\arraycolsep=0.2em
\begin{array}{lcl}
\sum_{k=0}^{\infty}t_{k+1}^2\Sigma_k < +\infty.
\end{array}
\end{equation}
Then, applying \eqref{eq:NGEAG4NI_V2_convergence2_proof3b}, we can prove from \eqref{eq:NGEAG4NI_V2_convergence2_proof3} that
\begin{equation}\label{eq:NGEAG4NI_V2_convergence2_proof3c}
\arraycolsep=0.2em
\begin{array}{ll}
& \lim_{k\to\infty}t_k^2\norms{ y^k - x^k +  \eta_{k-1}(z^k - w^k) }^2 = 0, \vspace{1ex}\\ 
& \sum_{k=0}^{\infty} t_k \norms{ y^k - x^k +  \eta_{k-1}(z^k - w^k) }^2 < +\infty.
\end{array}
\end{equation}
In addition, by Young's inequality, we get
\begin{equation*}
\arraycolsep=0.2em
\begin{array}{ll}
\eta_{k-1}^2\norms{w^k}^2 & \leq 2\norms{y^k - x^k +  \eta_{k-1}(z^k - w^k)}^2 + 2\norms{y^k - x^k +  \eta_{k-1}z^k}^2.
\end{array}
\end{equation*}
Using \eqref{eq:NGEAG4NI_V2_convergence2_proof2} and \eqref{eq:NGEAG4NI_V2_convergence2_proof3c} into this inequality, we obtain $\lim_{k\to\infty } t_k^2\eta_{k-1}^2 \norms{w^k}^2 = 0$.
Since $\eta_{k-1} = \frac{(\eta-\beta)t_{k-1} - \delta }{ t_k } = \eta - \beta - \frac{ (\eta-\beta - \delta)}{t_k} \geq \delta$ as $ t_k \geq 1$, the last limit leads to $\lim_{k\to\infty } t_k^2  \norms{w^k}^2 = 0$, which proves the first line of \eqref{eq:NGEAG4NI_V2_convergence3}.

By Young's inequality again, we also have
\begin{equation*}
\arraycolsep=0.2em
\begin{array}{lcl}
t_k^2 \norms{y^k - x^k}^2 & \leq &  2t_k^2 \norms{y^k - x^k +  \eta_{k-1}(z^k - w^k)}^2 + 2\eta_{k-1}^2t_k^2 \norms{z^k - w^k}^2.
\end{array}
\end{equation*}
This inequality together with  \eqref{eq:NGEAG4NI_V2_convergence2_proof3c} and the summability of $t_k^2 \norms{z^k - w^k}^2$ imply the second line of \eqref{eq:NGEAG4NI_V2_convergence3}.

Finally, since $t_k^2 \norms{x^{k+1} - x^k}^2 \leq 2\eta^2 t_k^2 \norms{d^k}^2 + 2t_k^2 \norms{y^k - x^k}^2$, applying this relation, the second line of \eqref{eq:NGEAG4NI_V2_convergence3}, and the summability of $t_k^2 \norms{d^k}^2$, we get  the third line of \eqref{eq:NGEAG4NI_V2_convergence3}.
\Eproof

\beforesubsec
\subsection{\mytb{Technical Lemma~\ref{le:NGEAG4NI_V2_lm10} and Lemma~\ref{le:NGEAG4NI_V2_lm11}}}\label{apdx:le:NGEAG4NI_V2_lm10}
\aftersubsec
First, let us define the following quantities:
\begin{equation}\label{eq:NGEAG4NI_V2_lm10_Qhat_k}
\hspace{-1ex}
\arraycolsep=0.2em
\begin{array}{lcl}
\Ac_k &:= & \norms{r(x^k - x^{\star}) + t_k(y^k - x^k) + t_k\eta_{k-1}(z^k - w^k)}^2 + r\mu \norms{x^k - x^{\star}}^2, \vspace{1ex}\\
\hat{\Qc}_k &:= & \Ac_k +  2r t_k\eta_{k-1}  \big[\iprods{ w^k, x^k - x^{\star}} + \rho\norms{w^k}^2].
\end{array}
\hspace{-2ex}
\end{equation}
Clearly, by the $\rho$-co-hypomonotonicity of $\Phi$,  we have $\hat{\Qc}_k \geq 0$.

\begin{lemma}\label{le:NGEAG4NI_V2_lm10}
Under the same settings as in Theorem~\ref{th:NGEAG4NI_V2_convergence2}, $\hat{\Qc}_k$ defined by \eqref{eq:NGEAG4NI_V2_lm10_Qhat_k} satisfies the following inequality:
\begin{equation}\label{eq:NGEAG4NI_V2_lm10_Qhatk_bound}
\hspace{-2ex}
\arraycolsep=0.2em
\begin{array}{lcl}
\hat{\Qc}_k - \hat{\Qc}_{k+1} & \geq &    \mu(2t_k - r - \mu)  \norms{x^{k+1} - x^k}^2  - 2[(\eta-\beta)t_k - \delta]^2\norms{z^{k+1} - w^{k+1}}^2 \vspace{1ex}\\
&& - {~}  2t_{k+1}^2 \norms{p^k}^2  - 2\rho r t_{k+1}\eta_k \norms{w^{k+1} }^2 \vspace{1ex} \\
&& + {~} 2 [(\eta-\beta)t_k-\delta](t_k-r-\mu) \iprods{ w^{k+1} - w^k, x^{k+1} - x^k} \vspace{1ex}\\
&& + {~}  2r  [\eta (r-2) + \beta - \delta] \iprods{z^k, x^{k+1} - x^{\star}}  +  2 \mu\eta t_k \iprods{\hat{w}^{k+1}, x^{k+1} - x^k}\vspace{1ex}\\
&& + {~} 2s^3_k \iprods{ z^k - w^k, x^{k+1} - x^k}  - 2\hat{s}^3_k \iprods{ z^k, x^{k+1} - x^k},
\end{array}
\hspace{-9ex}
\end{equation}
where 
\begin{equation*}
\arraycolsep=0.2em
\left\{\begin{array}{lcl}
s^3_k & := & (\eta - \beta) (r + \mu - 1) t_k - (r + \mu) \delta, \vspace{1ex}\\
\hat{s}^3_k & := & [(\mu + 1)\eta + \delta - r \beta] t_k - r(2\eta - \beta + 2\delta - r \eta) - \delta \mu.
\end{array}\right.
\end{equation*}
\end{lemma}

\begin{proof}
Denote $g^k :=  \eta d^k + \eta_{k-1}(z^k - w^k)$.
Using this $g^k$ and $x^{k+1} = y^k - \eta d^k$ from \eqref{eq:NGEAG4NI} and \eqref{eq:NGEAG4NI_dir}, similar to the proof of Lemma~\ref{le:NGEAG4NI_descent_property1}, we have
\begin{equation*}
\arraycolsep=0.2em
\begin{array}{lcl}
\breve{\Tc}_{[1]} &:= & \norms{r(x^k - x^{\star}) + t_k(y^k - x^k) + t_k\eta_{k-1}(z^k - w^k)}^2 \vspace{1ex}\\
& = & r^2\norms{ x^{k+1} - x^{\star} }^2   + (t_k-r)^2\norms{x^{k+1} - x^k}^2 +  t_k^2 \norms{ g^k }^2 \vspace{1ex}\\
&& + {~} 2r(t_k-r)\iprods{ x^{k+1} - x^{\star}, x^{k+1} - x^k } + 2 r t_k \iprods{ g^k , x^{k+1} - x^{\star}} \vspace{1ex}\\
&& + {~} 2t_k(t_k - r)\iprods{ g^k, x^{k+1} - x^k}.
\end{array}
\end{equation*}
Alternatively, denote $h^k :=  p^k - \eta_k (z^{k+1} - w^{k+1})$.
Then, similar to the proof of Lemma~\ref{le:NGEAG4NI_descent_property1}, we can show from \eqref{eq:NGEAG4NI} that
\begin{equation*}
\arraycolsep=0.2em
\begin{array}{lcl}
\breve{\Tc}_{[2]} &:= & \norms{r(x^{k+1} - x^{\star}) + t_{k+1}(y^{k+1} - x^{k+1}) + t_{k+1}\eta_k (z^{k+1} - w^{k+1})}^2 \vspace{1ex}\\
& = & r^2 \norms{x^{k+1} - x^{\star}}^2 + t_{k+1}^2\theta_k^2\norms{x^{k+1} - x^k}^2 + t_{k+1}^2\norms{h^k}^2 \vspace{1ex}\\
&& + {~} 2rt_{k+1}\theta_k \iprods{ x^{k+1} - x^{\star}, x^{k+1} - x^k} - 2rt_{k+1} \iprods{h^k, x^{k+1} - x^{\star}} \vspace{1ex}\\
&& - {~} 2t_{k+1}^2\theta_k \iprods{h^k, x^{k+1} - x^k}.
\end{array}
\end{equation*}
Combining both expressions $\breve{\Tc}_{[1]}$ and $\breve{\Tc}_{[2]}$, and the identity $\mu r \norms{x^k - x^{\star}}^2 - \mu r \norms{x^{k+1} - x^{\star}}^2 =  \mu r \norms{x^{k+1} - x^k}^2 - 2r\mu\iprods{x^{k+1} - x^k, x^{k+1} - x^{\star}}$,  we can derive from \eqref{eq:NGEAG4NI_V2_lm10_Qhat_k} that
\begin{equation}\label{eq:NGEAG4NI_V2_lm10_proof3} 
\arraycolsep=0.2em
\begin{array}{lcl}
\Ac_k - \Ac_{k+1} &= &    \mu(2t_k - r - \mu)  \norms{x^{k+1} - x^k}^2 +   t_k^2 \norms{g^k}^2 - t_{k+1}^2 \norms{h^k}^2 \vspace{1ex} \\
&& + {~}  2\iprods{t_k(t_k-r)g^k + t_{k+1}^2\theta_k h^k, x^{k+1} - x^k} \vspace{1ex}\\
&& + {~} 2r \iprods{ t_k g^k + t_{k+1} h^k, x^{k+1} - x^{\star}}.
\end{array}
\end{equation}
Next, we note that
\begin{equation*}
\arraycolsep=0.2em
\begin{array}{lcl}
t_k g^k + t_{k+1} h^k & = & \eta t_k\hat{w}^{k+1} - \eta t_k\gamma_kz^k + t_k \eta_{k-1}(z^k - w^k) + t_{k+1}\eta_k z^{k+1} \vspace{1ex}\\
&& - {~}  t_{k+1}\lambda_k\hat{w}^{k+1}  +  t_{k+1}\nu_kz^k  - t_{k+1}\eta_k(z^{k+1} - w^{k+1}) \vspace{1ex}\\
& = & ( \eta t_k -  t_{k+1}\lambda_k )\hat{w}^{k+1} + t_{k+1}\eta_kw^{k+1}  -  t_k\eta_{k-1}w^k \vspace{1ex}\\
&& - {~} (\eta t_k\gamma_k - t_{k+1}\nu_k - t_k\eta_{k-1}) z^k \vspace{1ex}\\
& = &  t_{k+1}\eta_kw^{k+1}  -  t_k\eta_{k-1}w^k  + [\eta (r-2) + \beta - \delta] z^k.
\end{array}
\end{equation*}
Therefore, one can prove that
\begin{equation*}
\arraycolsep=0.2em
\begin{array}{lcl}
\breve{\Tc}_{[3]} &:= & \iprods{ t_k g^k + t_{k+1} h^k, x^{k+1} - x^{\star}} \vspace{1ex}\\
& = &  t_{k+1}\eta_k\iprods{w^{k+1}, x^{k+1} - x^{\star}}  -  t_k\eta_{k-1}\iprods{w^k, x^k - x^{\star}} \vspace{1ex}\\
&& - {~} t_k\eta_{k-1}\iprods{w^k, x^{k+1} - x^k }  - (\eta r + \beta - \delta)\iprods{z^k, x^{k+1} - x^{\star}} \vspace{1ex}\\
& = &  t_{k+1}\eta_k\iprods{w^{k+1}, x^{k+1} - x^{\star}}  -  t_k\eta_{k-1}\iprods{w^k, x^k - x^{\star}} \vspace{1ex}\\
&& + {~} t_k\eta_{k-1}\iprods{z^k - w^k, x^{k+1} - x^k } -  t_k\eta_{k-1}\iprods{z^k, x^{k+1} - x^k}  \vspace{1ex}\\
&& + {~}   [\eta (r-2) + \beta - \delta] \iprods{z^k, x^{k+1} - x^{\star}}.
\end{array}
\end{equation*}
Similarly, we also get
\begin{equation*}
\arraycolsep=0.2em
\begin{array}{lcl}
\breve{\Tc}_{[4]} &:= & t_k(t_k-r)g^k + t_{k+1}^2\theta_kh^k \vspace{1ex}\\
& = & \eta t_k(t_k-r) \hat{w}^{k+1} - \eta t_k(t_k-r)\gamma_kz^k + t_k(t_k-r) \eta_{k-1}(z^k - w^k) \vspace{1ex}\\
&& + {~} t_{k+1}^2\theta_k\eta_k z^{k+1} - t_{k+1}^2\theta_k\lambda_k\hat{w}^{k+1} +  t_{k+1}^2\theta_k \nu_kz^k  - t_{k+1}^2\theta_k\eta_k(z^{k+1} - w^{k+1}) \vspace{1ex}\\
%
& = & \mu\eta t_k\hat{w}^{k+1} + [(\eta-\beta)t_k-\delta](t_k-r-\mu)(w^{k+1} - w^k) \vspace{1ex}\\
&& +  {~} [ (\mu-1)(\eta-\beta)t_k - \mu\delta + r(\eta-\beta)] (z^k - w^k) \vspace{1ex}\\
&& + {~} \big\{[ \eta(r-\mu-1) - \delta]t_k - \eta r(r-1) + \delta(r+\mu)\big\} z^k.
\end{array}
\end{equation*}
This expression leads to
\begin{equation*}
\arraycolsep=0.2em
\begin{array}{lcl}
\breve{\Tc}_{[5]} &:= & \iprods{t_k(t_k-r)g^k + t_{k+1}^2\theta_kh^k, x^{k+1} - x^k } \vspace{1ex}\\
& = & \mu\eta t_k \iprods{\hat{w}^{k+1}, x^{k+1} - x^k} + [(\eta-\beta)t_k-\delta](t_k-r-\mu) \iprods{ w^{k+1} - w^k, x^{k+1} - x^k} \vspace{1ex}\\
&& + {~} [ (\mu-1)(\eta-\beta)t_k - \mu\delta + r(\eta-\beta)] \iprods{ z^k - w^k, x^{k+1} - x^k} \vspace{1ex}\\
&& + {~} \big\{ [ \eta(r-\mu-1) - \delta]t_k - \eta r(r-1) + \delta(r+\mu) \big\} \iprods{ z^k, x^{k+1} - x^k}.
\end{array}
\end{equation*}
Substituting $\breve{\Tc}_{[3]}$ and $\breve{\Tc}_{[5]}$ into \eqref{eq:NGEAG4NI_V2_lm10_proof3}, we can show that
\begin{equation*} 
\arraycolsep=0.2em
\begin{array}{lcl}
\Ac_k - \Ac_{k+1} &= &    \mu(2t_k - r - \mu)  \norms{x^{k+1} - x^k}^2 +   t_k^2 \norms{g^k}^2 - t_{k+1}^2 \norms{h^k}^2 \vspace{1ex} \\
&&  + {~} 2r t_{k+1}\eta_k\iprods{w^{k+1}, x^{k+1} - x^{\star}}  -  2r t_k\eta_{k-1}\iprods{w^k, x^k - x^{\star}} \vspace{1ex}\\
&& + {~}  2r   [\eta (r-2) + \beta - \delta] \iprods{z^k, x^{k+1} - x^{\star}}  +  2 \mu\eta t_k \iprods{\hat{w}^{k+1}, x^{k+1} - x^k}\vspace{1ex}\\
&& + {~} 2 [(\eta-\beta)t_k-\delta](t_k-r-\mu) \iprods{ w^{k+1} - w^k, x^{k+1} - x^k} \vspace{1ex}\\
&& + {~} 2s^3_k \iprods{ z^k - w^k, x^{k+1} - x^k}  - 2\hat{s}^3_k \iprods{ z^k, x^{k+1} - x^k},
\end{array}
\end{equation*}
where 
\begin{equation*} 
\arraycolsep=0.2em
\begin{array}{lcl}
s^3_k & := & (\eta - \beta) (r + \mu - 1) t_k - (r + \mu) \delta \vspace{1ex}\\
\hat{s}^3_k & := & [(\mu + 1)\eta + \delta - r \beta] t_k - r(2\eta - \beta + 2\delta - r \eta) - \delta \mu.
\end{array}
\end{equation*}
Finally, using the definition of $\hat{\Qc}_k$ from \eqref{eq:NGEAG4NI_V2_lm10_Qhat_k} and $t_{k+1}^2\norms{h^k}^2 \leq 2t_{k+1}^2\norms{p^k}^2 + 2[(\eta-\beta)t_k - \delta]^2\norms{z^{k+1} - w^{k+1}}^2$, and neglecting the nonnegative terms $t_k^2\norms{g^k}^2$ and $2\rho r t_k \eta_{k-1} \norms{w^k}^2$, the last expression leads to \eqref{eq:NGEAG4NI_V2_lm10_Qhatk_bound}.
\Eproof
\end{proof}

\begin{lemma}\label{le:NGEAG4NI_V2_lm11}
Under the same settings as in Theorem~\ref{th:NGEAG4NI_V2_convergence2}, we have
\begin{equation}\label{eq:NGEAG4NI_V2_summable_results}
\begin{array}{lcl}
\sum_{k=0}^{\infty} t_k\vert \iprods{z^k, x^{k+1} - x^k } \vert & < & + \infty, \vspace{1ex}\\
\sum_{k=0}^{\infty} t_k\vert \iprods{\hat{w}^{k+1}, x^{k+1} - x^k } \vert & < & + \infty, \vspace{1ex}\\
\sum_{k=0}^{\infty} t_k\vert \iprods{z^k - w^k, x^{k+1} - x^k } \vert & < & + \infty, \vspace{1ex}\\
\sum_{k=0}^{\infty} t_k^2 \norms{w^{k+1} - w^k}^2 & < & + \infty, \vspace{1ex}\\
\sum_{k=0}^{\infty} t_k^2 \norms{p^k}^2 & < & + \infty.
\end{array}
\end{equation}
In addition, for $\Omega$ given in Theorem~\ref{th:NGEAG4NI_V2_convergence1}, and $\Rc_0^2$ in \eqref{eq:NGEAG4NI_V1_R02}, we also have
\begin{equation}\label{eq:NGEAG4NI_V2_lower_bound5}
\begin{array}{lcl}
2\iprods{z^k, x^{k+1} - x^{\star}} & \geq & -2\rho\norms{w^k}^2 - \norms{z^k }^2 - \norms{x^{k+1} - x^k}^2 \vspace{1ex}\\
&& - {~} t_k^2\norms{z^k-w^k}^2 - \frac{\Omega\Rc_0^2}{r(\mu-1)t_k^2}.
\end{array}
\end{equation}
\end{lemma}

\begin{proof}
Since  $2t_k\vert \iprods{z^k, x^{k+1} - x^k} \vert \leq t_k\norms{z^k}^2 + t_k\norms{x^{k+1} - x^k }^2$, we obtain the first line of \eqref{eq:NGEAG4NI_V2_summable_results} from this relation and the summable results in \eqref{eq:NGEAG4NI_V2_convergence1}.

Next, we have $\norms{\hat{w}^{k+1}}^2 \leq 2\norms{d^k}^2 +  2\gamma_k^2\norms{z^k}^2 \leq  2\norms{d^k}^2 + 2\norms{z^k}^2$.
Therefore, we can show that 
\begin{equation*} 
\begin{array}{lcl}
2t_k\vert \iprods{\hat{w}^{k+1}, x^{k+1} - x^k} \vert & \leq & t_k\norms{\hat{w}^{k+1} }^2 + t_k\norms{x^{k+1} - x^k }^2 \vspace{1ex}\\
& \leq & 2t_k\norms{d^k}^2 + 2t_k\norms{z^k}^2 + t_k\norms{x^{k+1} - x^k }^2.
\end{array}
\end{equation*}
This inequality together with the summable results in \eqref{eq:NGEAG4NI_V2_convergence1} imply the second line of \eqref{eq:NGEAG4NI_V2_summable_results}.
The third line of \eqref{eq:NGEAG4NI_V2_summable_results} is proven similarly, and hence, we omit.

Since $\norms{w^{k+1} - w^k}^2 \leq 3\norms{z^{k+1} - w^{k+1}}^2 + 3\norms{z^{k+1} - z^k}^2 + 3\norms{z^k - w^k}^2$ due to Young's inequality, we obtain the fourth  line of \eqref{eq:NGEAG4NI_V2_summable_results} by combining this inequality and \eqref{eq:NGEAG4NI_V2_convergence1}.

We also have 
\begin{equation*} 
\begin{array}{lcl}
t_{k+1}^2\norms{p^k}^2 & \leq & 3[(\eta-\beta)t_k -\delta]^2\norms{z^{k+1} - z^k}^2 \vspace{1ex}\\
&& + {~} 3\eta^2t_k^2\norms{d^k}^2 + 3[\eta(r-1) - \delta]^2\norms{z^k}^2.
\end{array}
\end{equation*}
Since all the terms on the right-hand side of this inequality are summable due to  \eqref{eq:NGEAG4NI_V2_convergence1}, we obtain the last line of \eqref{eq:NGEAG4NI_V2_summable_results}.

Finally, by Young's inequality, we have  
\begin{equation*}
\arraycolsep=0.2em
\begin{array}{lcl}
2\iprods{z^k, x^{k+1} - x^{\star}}  & = & 2\iprods{w^k, x^k - x^{\star}} +  2\iprods{z^k - w^k, x^k - x^{\star}}  + 2\iprods{z^k, x^{k+1} - x^k} \vspace{1ex}\\
&\geq & -2\rho\norms{w^k}^2 - 2\norms{z^k - w^k}\norms{x^k - x^{\star}} - \norms{z^k }^2 - \norms{x^{k+1} - x^k}^2 \vspace{1ex}\\
& \geq & -2\rho\norms{w^k}^2 - t_k^2\norms{z^k-w^k}^2 - \frac{\norms{x^k - x^{\star}}^2 }{t_k^2} \vspace{1ex}\\
&& - {~} \norms{z^k }^2 - \norms{x^{k+1} - x^k}^2.
\end{array}
\end{equation*}
Since $r(\mu-1)\norms{x^k - x^{\star}}^2 \leq \hat{\Lc}_k \leq \Omega\hat{\Lc}_0 \leq \Omega\Rc_0^2$ as in the proof of Theorem \ref{th:NGEAG4NI_V2_convergence1}, the last inequality implies \eqref{eq:NGEAG4NI_V2_lower_bound5}.
\Eproof
\end{proof}

\beforesubsec
\subsection{\mytb{The Proof of Theorem~\ref{th:NGEAG4NI_V2_convergence3}}}\label{apdx:th:NGEAG4NI_V2_convergence3}
\aftersubsec
From the proof of Theorem~\ref{th:NGEAG4NI_V2_convergence1}, we have $r(\mu-1)\norms{x^k - x^{\star}}^2 \leq \Lc_k \leq \Omega \Rc_0^2$, we conclude that $\sets{x^k}$ is bounded, and hence, it has a cluster point.
Let $x^{*}$ be a cluster point of $x^k$ and $x^{k_i}$ be a subsequence converging to $x^{*}$.
Similar to the proof of Theorem~\ref{th:NGEAG4NI_V1_convergence3}, we have $0 \in Fx^{*} + Tx^{*}$ (i.e. $x^{*} \in \zer{\Phi}$).

Our next step is to show that $\lim_{k\to\infty}\norms{x^k - x^{\star}}^2$ exists.
Combining \eqref{eq:NGEAG4NI_V2_lm10_Qhatk_bound}, \eqref{eq:NGEAG4NI_V2_summable_results}, and \eqref{eq:NGEAG4NI_V2_lower_bound5}, we can derive that
\begin{equation*} 
\arraycolsep=0.2em
\begin{array}{lcl}
\hat{\Qc}_{k+1} -  \hat{\Qc}_k  &\leq & t_{k+1}^2 \norms{p^k}^2 + 2[(\eta-\beta)t_k - \delta]^2\norms{z^{k+1} - w^{k+1}}^2   \vspace{1ex} \\
&& + {~} 2\rho [(\eta-\beta)t_k-\delta](t_k-r-\mu) \norms{ w^{k+1} - w^k}^2  \vspace{1ex}\\
&& + {~}  r  [\eta (r-2) + \beta - \delta] \Big[  2\rho\norms{w^k}^2 + \norms{z^k }^2 + \norms{x^{k+1} - x^k}^2  \vspace{1ex}\\
&& \qquad\qquad\qquad\qquad \qquad + {~} t_k^2\norms{z^k-w^k}^2 + \frac{\Omega\Rc_0^2}{r(\mu-1)t_k^2} \Big] \vspace{1ex}\\
&& +  {~}  2 \mu\eta t_k \vert \iprods{\hat{w}^{k+1}, x^{k+1} - x^k} \vert + 2\rho r t_{k+1}\eta_k \norms{w^{k+1} }^2 \vspace{1ex}\\
&& +  {~} 2s^3_k \vert \iprods{ z^k - w^k, x^{k+1} - x^k}\vert  + 2\hat{s}^3_k \vert \iprods{ z^k, x^{k+1} - x^k} \vert,
\end{array}
\end{equation*}
where we have used $\iprods{w^{k+1} - w^k, x^{k+1} - x^k} \geq -\rho\norms{w^{k+1} - w^k}^2$.

It is clear that since $s_k^3 = \BigO{t_k}$ and $\hat{s}^3_k = \BigO{t_k}$, all the terms on the right-hand sides of the last inequality are summable due to \eqref{eq:NGEAG4NI_V2_convergence1} and \eqref{eq:NGEAG4NI_V2_summable_results}.
Moreover, we have $\hat{\Qc}_k \geq 0$.
Applying \cite[Lemma~5.31]{Bauschke2011}, we can conclude that the limit $\lim_{k\to\infty}\hat{\Qc}_k$ exists.

Similar to \eqref{eq:NGEAG4NI_V1_convergence3_proof2}, we have $\lim_{k\to\infty} \vert t_k \iprods{ w^k, x^k - x^{\star}} \vert = 0$.
Since $\hat{\Qc}_k = \Ac_k + 2r t_k\eta_{k-1} \big[\iprods{ w^k, x^k - x^{\star}} + \rho\norms{w^k}^2 \big]$, this limit also implies that  $\lim_{k\to\infty}\Ac_k$ exists.

From the definition of $\Ac_k$ in \eqref{eq:NGEAG4NI_V2_lm10_Qhat_k}, we can expand it as
\begin{equation*} 
\arraycolsep=0.2em
\begin{array}{lcl}
\Ac_k  &= & (r^2 + \mu r)\norms{x^k - x^{\star}}^2 + t_k^2\norms{y^k - x^k}^2 + t_k^2\eta_{k-1}^2\norms{z^k - w^k}^2 \vspace{1ex}\\
&&  + {~} 2rt_k\iprods{x^k - x^{\star}, y^k - x^k } + 2rt_k\eta_{k-1}\iprods{z^k - w^k, x^k - x^{\star}} \vspace{1ex}\\
&& + {~} 2t_k^2\eta_{k-1}\iprods{y^k - x^k, z^k - w^k}.
\end{array}
\end{equation*}
It is not hard to show that the limits of the last five terms are all zeros.
Combining these facts and the existence of $\lim_{k\to\infty}\Ac_k$, we conclude that $\lim_{k\to\infty}\norms{x^k - x^{\star}}^2$ exists.
Thus we have $\lim_{k\to\infty}x^k = x^{\star} \in \zer{\Phi}$ due to the well-known Opial's lemma.
Finally, since $\lim_{k\to\infty}\norms{y^k - x^k}^2 = 0$, we also have $\lim_{k\to\infty}y^k = x^{\star} \in \zer{\Phi}$.
\Eproof

\bibliographystyle{plain}

\end{document}